\documentclass[11pt,a4paper,reqno]{amsart}
\usepackage[english]{babel}
\usepackage[T1]{fontenc}
\usepackage{verbatim}
\usepackage{mathabx}
\usepackage{palatino}
\usepackage{amsmath}
\usepackage{amssymb}
\usepackage{amsthm}
\usepackage{amsfonts}
\usepackage{graphicx}
\usepackage{color}
\usepackage{mathtools}
\usepackage{dsfont}
\usepackage{mathrsfs}
\usepackage{comment}

\usepackage[colorlinks = true, citecolor = black]{hyperref}
\title{Sub-elliptic boundary value problems in flag domains}
\author{Tuomas Orponen and Michele Villa}
\address{University of Helsinki, Department of Mathematics and Statistics}
\email{tuomas.t.orponen@jyu.fi \\ michele.villa@helsinki.fi}
\date{\today}
\subjclass[2010]{35R03 (Primary), 31E05, 35H20, 35S15, 42B20 (Secondary)}
\thanks{T.O. is supported by the Academy of Finland via the project \emph{Quantitative rectifiability in Euclidean and non-Euclidean spaces}, grant Nos. 309365, 314172. Both authors are supported by the Academy of Finland via the project \emph{Incidences on Fractals}, grant No. 321896.}
\keywords{Sub-elliptic partial differential equations, Kohn-Laplacian, Dirichlet problem, Neumann problem, Heisenberg group, Singular integrals}

\newcommand{\R}{\mathbb{R}}
\newcommand{\W}{\mathbb{W}}
\newcommand{\He}{\mathbb{H}}
\newcommand{\N}{\mathbb{N}}

\newcommand{\C}{\mathbb{C}}
\newcommand{\Z}{\mathbb{Z}}
\newcommand{\tn}{\mathbb{P}}

\newcommand{\calD}{\mathcal{D}}
\newcommand{\calH}{\mathcal{H}}

\newcommand{\calN}{\mathcal{N}}
\newcommand{\calB}{\mathcal{B}}

\newcommand{\calS}{\mathcal{S}}

\newcommand{\spt}{\operatorname{spt}}

\newcommand{\spa}{\operatorname{span}}
\newcommand{\diam}{\operatorname{diam}}

\newcommand{\dist}{\operatorname{dist}}

\newcommand{\sgn}{\operatorname{sgn}}

\newcommand{\pv}{\mathrm{p.v. }}
\newcommand{\Div}{\mathrm{div}_{\He}}

\newcommand{\V}{\mathbb{V}}

\newcommand{\Lap}{\bigtriangleup}

\newcommand{\ve}{\epsilon}
\newcommand{\wt}{\widetilde}

\newcommand{\calM}{\mathcal{M}}

\newcommand{\Le}{\mathbb{L}}
\newcommand{\calK}{\mathcal{K}}
\newcommand{\ch}{\mathds{1}}

\def\Barint_#1{\mathchoice
          {\mathop{\vrule width 6pt height 3 pt depth -2.5pt
                  \kern -8pt \intop}\nolimits_{#1}}%
          {\mathop{\vrule width 5pt height 3 pt depth -2.6pt
                  \kern -6pt \intop}\nolimits_{#1}}%
          {\mathop{\vrule width 5pt height 3 pt depth -2.6pt
                  \kern -6pt \intop}\nolimits_{#1}}%
          {\mathop{\vrule width 5pt height 3 pt depth -2.6pt
                  \kern -6pt \intop}\nolimits_{#1}}}

\numberwithin{equation}{section}

\theoremstyle{plain}
\newtheorem{thm}[equation]{Theorem}

\newtheorem{lemma}[equation]{Lemma}

\newtheorem{ex}[equation]{Example}
\newtheorem{cor}[equation]{Corollary}
\newtheorem{proposition}[equation]{Proposition}
\newtheorem{question}{Question}

\theoremstyle{definition}

\newtheorem{definition}[equation]{Definition}

\theoremstyle{remark}
\newtheorem{remark}[equation]{Remark}

\begin{document}

\begin{abstract} A \emph{flag domain in $\R^{3}$} is a subset of $\R^{3}$ of the form $\{(x,y,t) : y < A(x)\}$, where $A \colon \R \to \R$ is a Lipschitz function. We solve the Dirichlet and Neumann problems for the sub-elliptic Kohn-Laplacian $\bigtriangleup^{\flat} = X^{2} + Y^{2}$ in flag domains $\Omega \subset \R^{3}$, with $L^{2}$-boundary values. We also obtain improved regularity for solutions to the Dirichlet problem if the boundary values have first order $L^{2}$-Sobolev regularity. Our solutions are obtained as sub-elliptic single and double layer potentials, which are best viewed as integral operators on the first Heisenberg group. We develop the theory of these operators on flag domains, and their boundaries.

\end{abstract}

\maketitle

\tableofcontents

\section{Introduction}\label{s:intro} 

\subsection{Dirichlet and Neumann problems}  Let $X$ and $Y$ be the following vector fields in $\R^{3}$:
\begin{equation}\label{XY} X := \partial_{x} - \tfrac{y}{2}\partial_{t} \quad \text{and} \quad Y := \partial_{y} + \tfrac{x}{2}\partial_{t}, \end{equation}
and let $\bigtriangleup^{\flat} = X^{2} + Y^{2}$ be the \emph{Kohn Laplacian}. This paper studies boundary value problems for $\bigtriangleup^{\flat}$ in \emph{flag domains} domains $\Omega \subset \R^{3}$. These are, by definition, domains of the form
\begin{displaymath} \Omega := \{(x,y,t) \in \R^{3} : x < A(y)\}, \end{displaymath}
where $A \colon \R \to \R$ is a Lipschitz function. In brief, we will solve the Dirichlet and Neumann problems for $\bigtriangleup^{\flat}$ in flag domains with boundary values $g,h \in L^{2}(\partial \Omega)$. Formally, the Dirichlet problem is
\begin{equation}\label{dirichlet} \begin{cases} \bigtriangleup^{\flat} u = 0 & \text{in } \Omega, \\ u = g & \text{in } \partial \Omega, \end{cases} \end{equation}
and the Neumann problem is
\begin{equation}\label{neumann} \begin{cases} \bigtriangleup^{\flat} u = 0 & \text{in } \Omega, \\ \nabla_{\nu}u = h & \text{in } \partial \Omega. \end{cases} \end{equation}
Here $\nabla_{\nu}u$ stands for the normal derivative of $u$. Since we only assume that $g,h \in L^{2}(\partial \Omega)$, interpretation is needed to make sense of these equations, and we return to this point in a moment. If the reader is familiar with the $L^{2}$-theory for the Dirichlet and Neumann problems for the standard Laplacian $\bigtriangleup = \partial_{x}^{2} + \partial_{y}^{2} + \partial_{z}^{2}$, the paper can be summarised as follows: for flag domains $\Omega \subset \R^{3}$, the $L^{2}$-theory of $\bigtriangleup^{\flat}$ matches the $L^{2}$-theory of $\bigtriangleup$.

We will solve the problems \eqref{dirichlet}-\eqref{neumann} by means of \emph{single and double layer potentials} in the first Heisenberg group $\He = (\R^{3},\cdot)$. These are the operators 
\begin{displaymath} \mathcal{S}f(p) := \int_{\partial \Omega} G(q^{-1} \cdot p)f(q) \, d\sigma(q) \quad \text{and} \quad \mathcal{D}f(p) := \int_{\partial \Omega} \langle \nabla G(p^{-1} \cdot q),\nu(q) \rangle f(q) \, d\sigma(q), \end{displaymath}  
where $\sigma = \mathcal{H}^{2}|_{\partial \Omega}$, and $G$ is the fundamental solution of $-\bigtriangleup^{\flat}$, found by Folland \cite{MR0315267}. Throughout the paper, the notation "$\nabla$" will be reserved for the \emph{horizontal gradient} $\nabla g := (Xg,Yg) \in \R^{2}$. The symbol "$\nu$" refers to the \emph{inward horizontal normal} of $\Omega$, see Section \ref{s:normals} for details. The integrals defining $\mathcal{S}f(p)$ and $\mathcal{D}f(p)$ are absolutely convergent for all $p \in \R^{3} \, \setminus \, \partial \Omega$ whenever $f \in L^{2}(\sigma)$. Moreover, $\mathcal{S}f$ and $\mathcal{D}f$ define $\bigtriangleup^{\flat}$-harmonic functions in $\He \, \setminus \, \partial \Omega$, and hence provide candidates for solving \eqref{dirichlet}-\eqref{neumann}.

\begin{remark}\label{harmonicityRemark} It is not obvious that $\mathcal{D}f$ is $\bigtriangleup^{\flat}$-harmonic, since the vector fields $X$ and $Y$ do not commute with each other. However, they do commute with the vector fields
\begin{equation}\label{XRYR} X^{R} := \partial_{x} + \tfrac{y}{2}\partial_{t} \quad \text{and} \quad Y^{R} := \partial_{y} - \tfrac{x}{y}\partial_{t}, \end{equation}
and one can check that $ZG(p^{-1} \cdot q) = -Z^{R}G(q^{-1} \cdot p)$ for $Z \in \{X,Y\}$. Consequently, 
\begin{displaymath} \bigtriangleup^{\flat}_{p}ZG(p^{-1} \cdot q) = -\bigtriangleup^{\flat}_{p}Z^{R}G(q^{-1} \cdot p) = -Z^{R}\bigtriangleup^{\flat}_{p}G(q^{-1} \cdot p) = 0, \quad p \neq q. \end{displaymath}
This implies the $\bigtriangleup^{\flat}$-harmonicity of $\mathcal{D}f$. This definition of the double layer potential also appears in the paper \cite{MR3600064} of Ruzhansky and Suragan. The application of the single layer potential to the $\bigtriangleup^{\flat}$-Dirichlet problem \eqref{dirichlet} goes back to Jerison \cite{MR639800}. More generally, the technique of layer potentials has a long history of applications to boundary value problems, mainly for elliptic and parabolic equations, but also to sub-elliptic equations in lesser extent. We will scratch the surface of this area in Section \ref{s:previousWork}.  \end{remark}
Our solution to the Dirichlet problem \eqref{dirichlet} will have the form $u = \mathcal{D}f_{g}$ for a certain $f_{g} \in L^{2}(\sigma)$. Similarly, the solution the Neumann problem \eqref{neumann} will have the form $u = \mathcal{S}f_{h}$ for some $f_{h} \in L^{2}(\sigma)$. To understand how exactly $f_{g},f_{h} \in L^{2}(\sigma)$ should to be chosen, we need expressions for $(\mathcal{D}f)|_{\partial \Omega}$ and $\nabla_{\nu}\mathcal{S}f$, given by the following theorem:
\begin{thm}\label{main1} Let $\Omega \subset \He$ be a flag domain, and let $f \in L^{p}(\sigma)$ for some $1 < p < \infty$. Then, $\mathcal{D}f$ and $\nabla \mathcal{S}f$ have non-tangential limits $\sigma$ a.e. on $\partial \Omega$, denoted $(\mathcal{D} f)|_{\partial \Omega}$ and $(\nabla \mathcal{S}f)|_{\partial \Omega}$. Moreover,
\begin{displaymath} (\mathcal{D}f)|_{\partial \Omega} = (\tfrac{1}{2}I + D)f \quad \text{and} \quad \nabla_{\nu}\mathcal{S}f := \langle (\nabla \mathcal{S}f)|_{\partial \Omega},\nu \rangle = (-\tfrac{1}{2}I + D^{t})f. \end{displaymath} 
\end{thm}
To be accurate, the formulae above describe the non-tangential limits of $\mathcal{D}f$ and $\nabla \mathcal{S}f$ from within $\Omega$. Similar formulae, with appropriate sign changes, also hold for non-tangential limits from within $\R^{3} \, \setminus \, \overline{\Omega}$, see Theorem \ref{t:rJumps} and Corollary \ref{c:allJumps}. The symbols "$D$" and "$D^{t}$" refer to certain principal value operators, see Definition \ref{boundaryDoubleLayers}. Their existence is a non-trivial matter, and is discussed extensively in Section \ref{s:principal}.

From Theorem \ref{main1}, we see that, in order to solve the equations \eqref{dirichlet}-\eqref{neumann} as $u = \mathcal{D}f_{g}$ and $u = \mathcal{S}f_{h}$, respectively, the functions $g$ and $h$ need to satisfy
\begin{displaymath} (\tfrac{1}{2}I + D)f_{g} = g \quad \text{and} \quad (-\tfrac{1}{2}I + D^{t})f_{h} = h. \end{displaymath}
Since $g,h \in L^{2}(\sigma)$ are arbitrary, solving $f_{g},f_{h}$ from these equations is equivalent to proving that $\tfrac{1}{2}I + D$ and $-\tfrac{1}{2}I + D^{t}$ are invertible operators on $L^{2}(\sigma)$:
\begin{thm}\label{main2} Let $\Omega \subset \R^{3}$ be a flag domain. Then, the operators $\tfrac{1}{2}I \pm D$ and $\tfrac{1}{2}I \pm D^{t}$ are invertible on $L^{2}(\sigma)$. \end{thm} 
The proof of Theorem \ref{main2} occupies Sections \ref{s:injectivity} and \ref{s:surjectivity}. A combination of Theorems \ref{main1} and \ref{main2} allows us to solve the problems \eqref{dirichlet}-\eqref{neumann} in the following sense:
\begin{thm}\label{main3} Let $\Omega \subset \R^{3}$ be a flag domain, and $g,h \in L^{2}(\sigma)$. Then, there exist $\bigtriangleup^{\flat}$-harmonic functions $u,v \in C^{\infty}(\Omega)$ with the following properties. First, the non-tangential maximal functions of both $u$ and $\nabla v$ are in $L^{2}(\sigma)$, with
\begin{equation}\label{form176} \|\mathcal{N}_{\theta}u\|_{L^{2}(\sigma)} \lesssim \|g\|_{L^{2}(\sigma)} \quad \text{and} \quad \|\mathcal{N}_{\theta}(\nabla v)\|_{L^{2}(\sigma)} \lesssim \|h\|_{L^{2}(\sigma)}, \qquad \theta \in (0,1). \end{equation}
Second, $u$ and $\nabla v$ have interior non-tangential limits $\sigma$ a.e. on $\partial \Omega$, denoted $u|_{\partial \Omega}$ and $(\nabla v)|_{\partial \Omega}$, which satisfy
\begin{equation}\label{form196} g = u|_{\partial \Omega} \quad \text{and} \quad h = \nabla_{\nu}v := \langle (\nabla v)|_{\partial \Omega},\nu \rangle. \end{equation}
Moreover, $u$ and $v$ have the form $u = \mathcal{D}f_{g}$ and $v = \mathcal{S}f_{h}$, where $f_{g},f_{h} \in L^{2}(\sigma)$ satisfy
\begin{equation}\label{form195} (\tfrac{1}{2}I + D)f_{g} = g \quad \text{and} \quad (-\tfrac{1}{2}I + D^{t})f_{h} = h. \end{equation}
\end{thm}

For the definitions of the non-tangential maximal function $\mathcal{N}_{\theta}$, and non-tangential limits, see Section \ref{s:nt}. Note that once $f_{g},f_{h}$ have been defined by \eqref{form195}, the boundary behaviour of $u = \mathcal{D}f_{g}$ and $\nabla v = \nabla \mathcal{S}f_{h}$, stated in \eqref{form196}, is a direct consequence of Theorem \ref{main1}. Also the "regularity" estimates \eqref{form176} follow from the form of the solutions $u = \mathcal{D}f_{g}$ and $v = \mathcal{S}f_{h}$, see Corollary \ref{K-MaxNT}. So, Theorem \ref{main3} is a corollary of Theorems \ref{main1} and \ref{main2}.

\subsection{The regularity problem} The reason for using the double layer potential "$\mathcal{D}$" in solving the Dirichlet problem \eqref{dirichlet} is that the single layer potential $\mathcal{S}$ is a "smoothing" operator: the space of boundary values $\{(\mathcal{S}f)|_{\partial \Omega} : f \in L^{2}(\sigma)\}$ is neither contained in, nor contains, $L^{2}(\sigma)$, so the equation $(\mathcal{S}f)|_{\partial \Omega} = g$, with $g \in L^{2}(\sigma)$, does not always have a solution $f \in L^{2}(\sigma)$. In fact, the following result characterises $\{(\mathcal{S}f)|_{\partial \Omega} : f \in L^{2}(\sigma)\}$:
\begin{thm}\label{main4} Let $\Omega \subset \R^{3}$ be a flag domain, and $f \in L^{2}(\sigma)$. Then, the non-tangential limits of $\mathcal{S}f$ exist $\sigma$ a.e. on $\partial \Omega$ and equal
\begin{displaymath} Sf(p) := \int_{\partial \Omega} G(q^{-1} \cdot p)f(q) \, d\sigma(q). \end{displaymath}
Moreover, $S \colon L^{2}(\sigma) \to L^{2}_{1,1/2}(\sigma)$ is an invertible operator.   \end{thm}
The definition of the (homogeneous) Sobolev space $L^{2}_{1,1/2}(\sigma)$ is somewhat involved, so we postpone the full details to Section \ref{s:singleInvert}, where also the proof of Theorem \ref{main4} can be found. Informally, this space consists of those functions in $L^{2}(\sigma) + L^{\infty}(\sigma)$ with a "tangential horizontal derivative" in $L^{2}(\sigma)$, and which, additionally, have a $\tfrac{1}{2}$-order $t$-derivative in $L^{2}(\sigma)$. To make this rigorous, it is most convenient to use the existence of a canonical "graph map" $\Gamma \colon \R^{2} \to \partial \Omega$, defined below \eqref{curveGamma}. Then, one may define $L^{2}_{1,1/2}(\sigma)$ to consist of those functions $f \in L^{2}(\sigma) + L^{\infty}(\sigma)$ whose precomposition with $\Gamma$ lies in a certain non-isotropic homogeneous Sobolev space, see Definition \ref{def:L2112}.

As a corollary of Theorem \ref{main4}, we obtain the following regularity result for the Dirichlet problem \eqref{dirichlet} when the boundary values lie in $L^{2}_{1,1/2}(\sigma)$:
\begin{thm}\label{main5} Let $\Omega \subset \R^{3}$ be a flag domain, and $g \in L^{2}_{1,1/2}(\sigma)$. Then, there exists a $\bigtriangleup^{\flat}$-harmonic function $u \in C^{\infty}(\Omega)$ such that the non-tangential limits of $u$ on $\partial \Omega$ coincide $\sigma$ a.e. with $g$, and the non-tangential maximal function of $\nabla u$ is in $L^{2}(\sigma)$, with
\begin{equation}\label{form177} \|\mathcal{N}_{\theta}(\nabla u)\|_{L^{2}(\sigma)} \lesssim \|g\|_{L^{2}_{1,1/2}(\sigma)}, \qquad \theta \in (0,1). \end{equation}
Moreover, $u = \mathcal{S}f$ for some $f \in L^{2}(\sigma)$ with $\|f\|_{L^{2}(\sigma)} \lesssim \|g\|_{L^{2}_{1,1/2}(\sigma)}$. \end{thm}
The main news compared to Theorem \ref{main2} is the regularity of $\nabla u$: the bound \eqref{form177} follows easily from the representation of $u$ as a single layer potential, see Corollary \ref{K-MaxNT}. Theorems \ref{main3} and \ref{main5} imply that if the boundary data "$g$" in the Dirichlet problem \eqref{dirichlet} is in $L^{2}(\sigma)$ or $L^{2}_{1,1/2}(\sigma)$, then the non-tangential maximal function of "$u$" or "$\nabla u$", respectively, lies in $L^{2}(\sigma)$. It is also true, but takes some extra work to show, that if 
\begin{displaymath} g \in \mathbf{L}^{2}_{1,1/2}(\sigma) := L^{2}(\sigma) \cap L^{2}_{1,1/2}(\sigma), \end{displaymath}
then the estimates for "$u$" and "$\nabla u$" hold simultaneously:
\begin{thm}\label{main7} Let $\Omega \subset \R^{3}$ be a flag domain, and $g \in \mathbf{L}^{2}_{1,1/2}(\sigma)$. Then, the double layer potential solution $u \in C^{\infty}(\Omega)$ to the Dirichlet problem \eqref{dirichlet}, whose existence is given by Theorem \ref{main3}, satisfies
\begin{displaymath} \max\{\|\mathcal{N}_{\theta}u\|_{L^{2}(\sigma)}, \|\mathcal{N}_{\theta}(\nabla u)\|_{L^{2}(\sigma)}\} \lesssim \|g\|_{\mathbf{L}^{2}_{1,1/2}(\sigma)} \qquad \theta \in (0,1). \end{displaymath} 
Here $\|g\|_{\mathbf{L}^{2}_{1,1/2}(\sigma)} := \|g\|_{L^{2}(\sigma)} + \|g\|_{L^{2}_{1,1/2}(\sigma)}$. \end{thm}

The proof of Theorem \ref{main7} can be found in Section \ref{s:Invert3}. As part of the argument, we will verify that the operator $\tfrac{1}{2}I + D$ is invertible on $\mathbf{L}^{2}_{1,1/2}(\sigma)$, see Theorem \ref{t:Invert}.

\begin{ex}\label{ex:Green} Theorem \ref{main7} applies to the boundary values $g_{p}(q) := G(q^{-1} \cdot p)$, where $p \in \Omega$ is fixed. Checking that $g_{p} \in L^{2}(\sigma)$ is straightforward, using $g_{p}(q) \sim d(p,q)^{-2}$, and we will verify that $g_{p} \in L^{2}_{1,1/2}(\sigma)$ in Remark \ref{r:SLP}. Therefore, by Theorem \ref{main7}, there exists a $\bigtriangleup^{\flat}$-harmonic function $u_{p} = \mathcal{D}f_{p} \in C^{\infty}(\Omega)$, with $f_{p} \in \mathbf{L}^{2}_{1,1/2}(\sigma)$, whose non-tangential limits exist $\sigma$ a.e. and equal $g_{p}$. Moreover,
\begin{displaymath} \max\{\|\mathcal{N}_{\theta}u_{p}\|_{L^{2}(\sigma)}, \|\mathcal{N}_{\theta}(\nabla u_{p})\|_{L^{2}(\sigma)}\} \lesssim \|g_{p}\|_{\mathbf{L}^{2}_{1,1/2}(\sigma)} \qquad \theta \in (0,1). \end{displaymath} 
\end{ex}

\subsection{Uniqueness in the Dirichlet problem} Theorems \ref{main3}, \ref{main5}, and \ref{main7} are existence results. Since the solutions are not continuous up to $\partial \Omega$, and the domains we consider are unbounded, the questions of uniqueness are not trivial. We only have a satisfactory answer for the Dirichlet problem \eqref{dirichlet}, contained in the following theorem:
\begin{thm}\label{main6} Let $\Omega \subset \R^{3}$ be a flag domain, and let $u \in C^{\infty}(\Omega)$ be $\bigtriangleup^{\flat}$-harmonic with radial maximal function $\mathcal{N}_{\mathrm{rad}}u \in L^{2}(\sigma)$, and with vanishing radial limits $\sigma$ a.e. on $\partial \Omega$. Then $u \equiv 0$. \end{thm}

The proof can be found in Section \ref{s:uniqueness}.

\subsection{Previous work}\label{s:previousWork} First, an apology: we can only scratch the surface of the literature!

\subsubsection{Some history of elliptic and parabolic problems} The method of layer potentials was applied in the late 70s by Fabes, Jodeit, and Rivi\`ere \cite{MR501367} to solve the Dirichlet and Neumann problems for the (elliptic) Laplacian with $L^{p}$-boundary values, in bounded subdomains of $\R^{n}$ with $C^{1}$-boundaries. The technique was extended to the case of Lipschitz boundaries by Verchota \cite{MR769382} in the early 80s. These developments were enabled by nearly contemporary works of Calder\'on \cite{Calderon} and Coifman, Macintosh, and Meyer \cite{CMM}, which established the $L^{p}$-boundedness of the Riesz transform, and the double layer potential, on boundaries of Lipschitz domains. For more information on the case of the standard Laplacian, and more general elliptic problems, we refer to the book of Kenig \cite{MR1282720}.

The method of layer potentials was applied to the (parabolic) heat equation first by Fabes and Rivi\`ere \cite{MR545307} and Brown \cite{MR987761} in the cases of $C^{1}$ and Lipschitz cylinders in $\R^{n + 1}$. The cylinder setup is close to our flag domains, with the technical difference that Fabes, Rivi\`ere, and Brown considered cylinders with a compact base. The layer potential technique was generalised to the case of "time-varying domains" by Lewis and Murray \cite{MR1323804} and Hofmann and Lewis \cite{MR1418902} in mid-90s. The $L^{p}$-boundedness of the relevant singular integral operators in full generality was established in a separate paper \cite{MR1484857} by Hofmann. Rather informally, the "time-varying domains" mentioned here are subdomains of $\R^{n + 1}$ bounded by "parabolic Lipschitz graphs". For precise definitions, see \cite{MR1418902}. As a more recent reference where (more general) parabolic equations are solved by means of layer potentials in time-varying domains, see the paper \cite{2018arXiv180507270D} of Dindo{\v{s}}, Dyer, and Hwang.

\subsubsection{Some history of sub-elliptic problems} Layer potentials were applied by Jerison \cite{MR639800,MR633978} in the early 80s to study the boundary regularity of solutions to the Dirichlet problem \eqref{dirichlet} in bounded $C^{\infty}$-domains $\Omega \subset \R^{2n + 1} \cong \He^{n}$. We also mention the earlier papers of Derridj \cite{MR601055}, and Kohn and Nirenberg \cite{MR181815}, where same problem was studied with other techniques. Jerison's work mainly concentrated on a phenomenon which is characteristic to sub-elliptic problems: the existence of \emph{characteristic points} on the boundary. If $\Omega \subset \R^{3}$ is a $C^{1}$-domain, these are the points $p \in \partial \Omega$ where the tangent space of $\partial \Omega$ is spanned by $\{X_{p},Y_{p}\}$. Jerison demonstrated that whenever $\Omega \subset \R^{2n + 1}$ is a bounded $C^{\infty}$-domain, then the solutions of the Dirichlet problem \eqref{dirichlet}, with $\beta$-H\"older continuous boundary values remain $\beta$-H\"older continuous near non-characteristic points, see \cite[Theorem (7.1)]{MR639800}. Similar results with smooth boundary values are contained in \cite{MR601055,MR181815}. On the other hand, Jerison \cite[Theorem (3.1)]{MR633978} gave counterexamples to show that the conclusion of \cite[Theorem (7.1)]{MR639800} can fail near characteristic points. The delicate case of regularity near "isolated" characteristic points was also studied by Jerison \cite[\S 5]{MR633978}.

Summarising, Jerison was not trying to solve the Dirichlet problem \eqref{dirichlet} with minimal assumptions of $\partial \Omega$ and the boundary values $g$, but rather understand the phenomenon of characteristic points in the simplest case where $\partial \Omega$ is $C^{\infty}$-smooth, and, at least, $g \in C(\partial \Omega)$. 

The case $g \in L^{p}(\partial \Omega)$ was considered almost simultaneously by Lanconelli and Uguzzoni \cite{MR1620876} and Capogna, Garofalo, and Nhieu \cite{MR1890994} in the late 90s. The domains $\Omega \subset \He^{n}$ studied in \cite{MR1620876,MR1890994} were assumed to be bounded, and to satisfy the \emph{uniform outer ball condition} (UOBC), namely the following: there exists $R > 0$ such that each $p \in \partial \Omega$ lies on the boundary of a \emph{Kor\'anyi ball} $B_{d}(q,R) \subset \He^{n} \, \setminus \, \Omega$. In brief, Kor\'anyi balls are (left translates of) super level sets of the fundamental solution "$G$" of $-\bigtriangleup^{\flat}$. They also happen to be metric balls in the metric $d(p,q) := \|q^{-1} \cdot p\|$, where $\|(z,t)\| := \sqrt[4]{|z|^{4} + 16t^{2}}$ for $(z,t) \in \R^{2n} \times \R \cong \He^{n}$. This metric will be used throughout the present paper.

For example, all bounded convex set in $\He^{n}$ satisfy the UOBC, see \cite[Theorem 2.6]{MR1890994} or \cite[Corollary A.3]{MR1620876}. The UOBC implies very good properties for the $\bigtriangleup^{\flat}$-harmonic measure of $\Omega$ (see \cite[Theorem 4.3]{MR1620876} and \cite[Theorem 1.1]{MR1890994}), and in particular allows one to solve the Dirichlet problem \eqref{dirichlet} with $g \in L^{p}(\partial \Omega)$, $1 < p \leq \infty$. This was done explicitly by Capogna, Garofalo, and Nhieu in \cite{MR2500489}, partially relying on the previous work \cite{MR1658616}. The approach via harmonic measure is not suited for studying the Neumann problem \eqref{neumann}, and it has received little attention: see \cite[Theorem 4.1]{MR3511807}, which solves \eqref{neumann} in the case where $\Omega \subset \He^{n}$ is a Kor\'anyi ball, and the boundary data is (a bit better than) continuous. 

For non-smooth, or unbounded, domains, even less is known. We mention the paper \cite{MR3632681} of Garofalo, Ruzhansky, and Suragan, where Green's functions are explicitly computed for vertical half-spaces, and certain unbounded domains with "$l$-wedges" in Carnot groups. In $\He \cong \R^{3}$, an "$l$-wedge domain" would be $\{(x,y,t) : x,y > 0, \, t \in \R\}$ -- which is in particular a flag domain.

\subsubsection{The layer potential approach} We already mentioned that Jerison \cite{MR639800} applied layer potentials in his work. However, Jerison observed that even if $\partial \Omega$ is $C^{\infty}$-smooth, the "boundary double layer potential" $D$, which appears in Theorem \ref{main1}, is not a compact operator on $L^{p}(\partial \Omega)$. This is a major difference to the double layer potentials arising from elliptic and parabolic problems. Even in the case where $\Omega$ is a (vertical) half-space, the operator $D$ is a genuine singular integral operator, and one cannot e.g. use Fredholm theory to prove the invertibility of $\tfrac{1}{2}I \pm D$. For comparison, the elliptic boundary double layer potential is the zero operator if $\Omega \subset \R^{n}$ is a half-space.

The situation is actually worse than that: before one even gets to study the invertibility of $\tfrac{1}{2}I \pm D$ on $L^{p}(\partial \Omega)$, one needs make sure that the principal value operator $D$ is well-defined, and bounded on $L^{p}(\partial \Omega)$. It remains largely an open problem to determine for which domains $\Omega \subset \He^{n}$ this holds. We only solve these problems (in Section \ref{s:principal}) for flag domains in $\He \cong \R^{3}$. The solutions are based on the recent information, due to F\"assler and the first author \cite{2018arXiv181013122F}, that the \emph{$\He$-Riesz transform} is bounded on $L^{p}(\partial \Omega)$ in this generality. 

So, Jerison did not employ the double layer potential $\mathcal{D}$, and its use in boundary value problems for $\bigtriangleup^{\flat}$ has been quite limited, up to now. Two exceptions are the papers \cite{MR3430847,MR3600064} of Ruzhansky and Suragan, where e.g. \cite[Theorem 4.4]{MR3600064} records the \emph{jump relations} for $\mathcal{D}f$ in the case where $\partial \Omega$ is smooth, and $f \in C^{1}(\overline{\Omega})$, and \cite[Theorem 5.1]{MR3600064} solves a sub-elliptic version of \emph{Kac's boundary value problem}. 

Jerison only considered the Dirichlet problem \eqref{dirichlet} with H\"older-continuous boundary values. For this problem, it sufficed to study the invertibility of the single layer potential $S$, defined in Theorem \ref{main4}. Jerison did this by first proving the invertibility (between H\"older classes) of $S$ in the special case where $\partial \Omega$ is a $1$-codimensional \emph{vertical subgroup} of $\He^{n}$. Then, using a lifting technique, originally devised by Folland and Stein \cite{MR367477}, he managed to extend the results to small non-characteristic patches of smooth surfaces.

We are interested in $L^{2}$-boundary values in the Dirichlet and Neumann problems, so we are forced to study "$D$". We prove the invertibility of $\tfrac{1}{2}I \pm D$ on the boundaries of flag domains by adapting ideas of Verchota \cite{MR769382}, Brown \cite{MR987761}, and Hofmann and Lewis \cite{MR1418902}. For readers familiar with elliptic or parabolic theory, we say the following. The argument of Verchota is based on \emph{Rellich's identities}, which say that the $L^{2}(\partial \Omega)$-norms of certain "normal" and "tangential" derivatives of harmonic functions on $\Omega$ are comparable, up to a "lower order term". In the parabolic setting, Rellich's identities are still useful, but the "lower order term" becomes more difficult to control. Rellich's identities for $\bigtriangleup^{\flat}$ are also well-known, see \cite[p. 81]{BLU}, or Proposition \ref{p:rellich}. Interestingly, the "lower order term" looks quite similar to the one which one encounters in the parabolic situation. This clue was essential to us in proving the invertibility of $\tfrac{1}{2}I \pm D$ and $\tfrac{1}{2}I \pm D^{t}$.

\subsection{Open problems} The "umbrella" open problem is naturally to solve the Dirichlet and Neumann problems \eqref{dirichlet}-\eqref{neumann} in more generality than flag domains in $\R^{3} \cong \He$. And even in flag domains, it is reasonable to ask if the problems remain solvable for boundary values $g,h \in L^{p}(\partial \Omega)$ with $p \neq 2$. Besides these generic questions, which have little to do with the geometry of $\bigtriangleup^{\flat}$, we pose explicitly the following two problems:
\begin{question}\label{Q1} Let $\Omega \subset \He^{n}$ be a domain bounded by an intrinsic Lipschitz graph, in the sense of Franchi, Serapioni, and Serra Cassano \cite{MR2287539}. Are the Dirichlet and Neumann problems \eqref{dirichlet}-\eqref{neumann} solvable in $\Omega$ in the sense of Theorem \ref{main3}? \end{question}
As sub-questions, one could ask if the (principal value) operators $D,D^{t}$ exist, and are bounded on $L^{p}(\partial \Omega)$, $1 < p < \infty$. A more basic question, still, is whether the $\He$-Riesz transform is bounded on $L^{p}(\partial \Omega)$. For partial results, see \cite{CFO2,2018arXiv181013122F,2019arXiv191103223F}. The argument in Section \ref{s:principal} has been written so that the following point becomes explicit: the existence of the principal value $Df$ is easier to prove than the existence of $D^{t}f$: the latter result relies on the structure of flag domains, the former does not. However, to be clear, the flag domain assumption in the present paper is used for much more than the existence of $D^{t}f$: it is particularly indispensable in Section \ref{s:injectivity}, where the injectivity of the operators $\tfrac{1}{2}I \pm D^{t}$ is established. The proof employs fractional differentiation operators "$T^{\alpha}$" in the $t$-variable, which would be ill-defined if $\Omega$ and $\partial \Omega$ were not foliated by vertical lines.

The following special case of Question \ref{Q1} is worth stating separately:
\begin{question} Solve Question \ref{Q1} in case $\Omega \subset \He^{n}$ is a flag domain, and $n \geq 2$. \end{question}
There are two key reasons why we are confined to the case $n = 1$. The first one is that, for $n = 1$, boundaries of flag domains $\Omega \subset \He$ admit a nice "graph map" $\Gamma \colon \R^{2} \to \partial \Omega$, defined below \eqref{curveGamma}. This is not the usual graph map "$\Phi$" associated to every intrinsic Lipschitz graph in $\He^{n}$, see \cite{MR2287539}. Rather, $\Gamma = \Phi \circ \chi$, where $\chi \colon \R^{2} \to \R^{2}$ is known as the \emph{characteristic straightening map}. The map $\chi$, and problems regarding its existence and regularity, are well-known among specialists in the area, see e.g. \cite{MR2603594,MR3400438}. The map $\Gamma$ plays a prominent role in our arguments, and we do not know if it exists for boundaries of flag domains in $\He^{n}$, $n \geq 2$, let alone more general intrinsic Lipschitz graphs. 

The second reason is that proving the surjectivity of $\tfrac{1}{2}I \pm D^{t}$ forces us to deal with rather general singular integral operators on $\partial \Omega$, not just the $\He$-Riesz transform. To survive this issue, we rely on the main results of the recent paper \cite{2019arXiv191103223F}, which only works in $\He$. It seems non-trivial to generalise \cite{2019arXiv191103223F} to higher Heisenberg groups. In fact, these two problems mix nicely together: the explicit form of the map $\Gamma$ determines in Section \ref{s:surjectivity} the singular integrals we need to tackle, so without knowing how $\Gamma$ looks like in $\He^{n}$, we do not even know which singular integrals we need to understand!

\subsection{Acknowledgements} We would like to thank Katrin F\"assler for many useful discussions.

\section{Preliminaries}

\subsection{The Heisenberg group}\label{s:heisenberg} While the results described in the introduction can be formulated in $\R^{3}$, they might be best viewed as statements about the first Heisenberg group $\He = (\R^{3},\cdot)$, where "$\cdot$" is the non-commutative group law
\begin{displaymath} (x,y,t) \cdot (x',y',t') = (x + x',y + y',t + t' + \tfrac{1}{2}(xy' - yx')), \qquad (x,y,t),(x',y',t') \in \He. \end{displaymath}
We will sometimes abbreviate $\omega(z,z') := xy' - yx'$ for $z = (x,y)$ and $z' = (x',y')$, so $(z,t) \cdot (z',t') = (z + z',t + t' + \tfrac{1}{2}\omega(z,z'))$. Metric concepts in the paper, for example open balls $B(p,r)$, the non-tangential maximal functions $\mathcal{N}_{\theta}$ introduced in Section \ref{s:nt}, and principal value (singular) integral operators
\begin{displaymath} \pv \int K(p,q)f(q) \, d\mu(q) := \lim_{\epsilon \to 0} \int_{B(p,\epsilon)^{c}} K(p,q)f(q) \, d\mu(q) \end{displaymath}
will be defined with respect to the Kor\'anyi metric $d(p,q) := \|q^{-1} \cdot p\|$, where
\begin{displaymath} \|(z,t)\| := \sqrt[4]{|z|^{4} + 16t^{2}}, \qquad (z,t) \in \C \times \R. \end{displaymath} 
The Kor\'anyi norm is $1$-homogeneous with respect to the Heisenberg dilations
\begin{displaymath} \delta_{r}(z,t) := (rz,r^{2}t), \qquad r > 0, \, (z,t) \in \C \times \R. \end{displaymath}
The fundamental solution $G \in C^{\infty}(\He \, \setminus \, \{0\})$ of $-\bigtriangleup^{\flat}$, first computed by Folland \cite{MR0315267}, can be expressed as
\begin{equation}\label{fundamentalSolution} G(p) = \frac{c}{\|p\|^{2}}, \qquad p \in \He \, \setminus \, \{0\}, \end{equation}
where $c > 0$ is a constant. If $\Omega \subset \He$ is an open set, and $1 \leq k \leq \infty$, the notation $C^{k}(\Omega)$ means $k$ times continuously Euclidean differentiable functions. In particular, $C^{k}(\He) = C^{k}(\R^{3})$. Functions $f \in C^{k}(\Omega)$ with compact support $\spt f \subset \Omega$ are denoted $f \in C^{k}_{c}(\Omega)$.

The left-invariant vector fields $X$ and $Y$ were introduced in \eqref{XY}, and their right-invariant counterparts $X^{R}$ and $Y^{R}$ in \eqref{XRYR}. For $f \in C^{1}(\Omega)$, we write $\nabla f := (Xf,Yf) \in C(\Omega,\R^{2})$ for the \emph{horizontal gradient} of $f$. The Euclidean gradient, which very rarely appears in the paper, will be denoted $\nabla_{E}$. If $f \in C^{2}(\Omega)$, we write $\bigtriangleup^{\flat}f := X^{2}f + Y^{2}f \in C(\Omega)$.

Somewhat imprecisely, a \emph{horizontal vectorfield on $\Omega$} will be a vector field of the form $V = aX + bY$, where $a,b \colon \Omega \to \R$ are arbitrary coefficient functions. If $a,b \in C^{1}(\Omega)$, we define the \emph{horizontal divergence} $\mathrm{div}^{\flat} V := Xa + Yb \in C(\Omega)$. Horizontal vector fields with coefficients in $C^{k}(\Omega)$ (resp. $C^{k}_{c}(\Omega)$) will be denoted $C^{k}(\Omega,H\He)$ (resp. $C^{k}_{c}(\Omega,H\He)$).

If $(a,b) \colon \Omega \to \R^{2}$ is a function, we frequently identify $(a,b)$ with the horizontal vector field $aX + bY$. This happens particularly often for the horizontal gradient $\nabla f = (Xf,Yf) \colon \Omega \to \R^{2}$. With this identification in mind, we write $\mathrm{div}^{\flat}(a,b) := Xa + Yb \in C(\Omega)$ whenever $(a,b) \in C^{1}(\Omega,\R^{2})$. As a special case, $\mathrm{div}^{\flat} \nabla f = \bigtriangleup^{\flat} f$ for $f \in C^{2}(\Omega)$.

If $V = aX + bY$ and $W = cX + dY$ are two horizontal vector fields, we define their dot product
\begin{displaymath} \langle V(p),W(p) \rangle := (a(p),b(p)) \cdot (c(p),d(p)), \end{displaymath} 
where the RHS refers to the usual dot product in $\R^{2}$. The notation above will frequently be (ab)used in a situation where one of $V,W$ is an $\R^{2}$-valued function. A typical example is "$\langle \nabla f,X \rangle$". Under the identification $\nabla f \cong (Xf)X + (Yf)Y$, one finds $\langle \nabla f,X \rangle = Xf$. For a horizontal vector field $V$, we finally define $|V(p)| := \sqrt{\langle V(p),V(p) \rangle}$. 

\subsection{Surface measures, normal and tangent vectors}\label{s:normals} When $\Omega \subset \R^{3}$ is a flag domain, a natural surface measure "$\sigma$" on $\partial \Omega$ is the restriction of Euclidean $2$-dimensional Hausdorff measure, denoted $\calH^{2}_{E}$, to $\partial \Omega$. If the reader has limited background in Heisenberg groups, for much of the paper (including the introduction) he can think that "$\sigma$" refers to this measure. It is a fortunate coincidence that, for boundaries of flag domains, Euclidean Hausdorff measure equals a more intrinsic object in the Heisenberg group $\He$.

We recall from \cite[p. 482]{FSSC} that a measurable set $E \subset \He$ is a \emph{$\He$-Caccioppoli set} if the \emph{horizontal perimeter} $\sigma = |\partial E|_{\He}$ is a Radon measure. Whenever this holds, then there exists a $\sigma$ measurable horizontal unit vector field "$\nu$" (i.e. $|\nu(p)| = 1$ for $\sigma$ a.e. $p$), called the \emph{inward-pointing horizontal normal} of $E$, with the property that
\begin{displaymath} \int_{E} \mathrm{div}^{\flat} Z(p) \, dp = -\int \langle Z(p),\nu(p) \rangle \, d\sigma(p), \qquad Z \in C^{1}_{c}(\He,H\He). \end{displaymath}
If $E = \Omega$ happens to be a flag domain, then it turns out that $\sigma = \calH^{2}_{E}|_{\partial \Omega}$, and also $\nu$ has a simple relation to the Euclidean (inward) normal of $\partial \Omega$:

\begin{proposition}\label{prop3} Let $\Omega = \{(x,y,t) : x < A(y)\}$ be a flag domain, with $A \colon \R \to \R$ Lipschitz. Then $\Omega$ is a $\He$-Caccioppoli set, and $|\partial \Omega|_{\He} = \calH_{E}^{2}|_{\partial \Omega}$. Moreover, we have the formula 
\begin{displaymath} \nu(p) = n_{1}(A(y),y)X + n_{2}(A(y),y)Y, \qquad p := (A(y),y,t) \in \partial \Omega, \end{displaymath}
where $n = (n_{1},n_{2})$ is the Euclidean inward-pointing unit normal of $\{(x,y) : y < A(y)\} \subset \R^{2}$:
\begin{equation}\label{flagNormal} n_{1}(A(y),y) = \frac{-1}{\sqrt{1 + A'(y)^{2}}} \quad \text{and} \quad n_{2}(A(y),y) = \frac{A'(y)}{\sqrt{1 + A'(y)^{2}}}, \qquad y \in \R. \end{equation} 
In particular, the coordinates of $\nu(z,t)$ are independent of "$t$". It follows that if $Z = aX + bY$ is a $\sigma$ a.e. defined horizontal vector field, then
\begin{equation}\label{form45} \langle Z(z,t),\nu(z,t) \rangle = (a(z,t),b(z,t)) \cdot n(z) \qquad \text{for $\sigma$ a.e. } (z,t) \in \partial \Omega. \end{equation}
\end{proposition} 

\begin{proof} If $\Omega \subset \R^{3} \cong \He$ is a general (Euclidean) Lipschitz domain, the relationship between the horizontal perimeter $|\partial \Omega|_{\He}$ and the Euclidean $2$-dimensional Hausdorff measure is described by \cite[Proposition 2.14]{FSSC}:
\begin{displaymath} |\partial \Omega|_{\He} = |Cn| \cdot \calH^{2}_{E}|_{\partial \Omega},  \end{displaymath}
where "$n$" is the Euclidean (inward-pointing) normal of $\Omega$, and $C$ is the $(2 \times 3)$-matrix with rows $[1, \, 0, \, -\tfrac{1}{2}y]$ and $[0, \, 1, \, \tfrac{1}{2}x]$. In the special case where $\Omega$ is a flag domain, the Euclidean normal "$n$" evidently has the form $(n_{1},n_{2},0)$, where $n_{1},n_{2}$ are defined in \eqref{flagNormal}. In particular, $Cn = (n_{1},n_{2})$, and $|Cn| = |(n_{1},n_{2})| \equiv 1$. Therefore, $\sigma := |\partial \Omega|_{\He} = \calH^{2}_{E}|_{\partial \Omega}$.\footnote{Of course, such a relation is only true up to a normalising constant for $\calH^{2}_{E}$. The constant here is chosen so that the Euclidean divergence theorem holds in the form $\int_{\Omega} \mathrm{div} F = -\int_{\partial \Omega} F \cdot n \, d\calH^{2}_{E}$ for $F \in C^{1}_{c}(\R^{3})$.} For general Lipschitz domains $\Omega \subset \R^{3}$, the relationship between (inward) horizontal and Euclidean normals is given by $\nu = Cn/|Cn|$, see \cite[(3.3)]{MR3363670}. For flag domains we just observed that $|Cn| \equiv 1$, so $\nu = Cn = (n_{1},n_{2})$, as claimed.
\end{proof}

In addition to the horizontal normal, we also define a horizontal tangent vector field -- only in the case of flag domains:
\begin{definition}[Tangent vector field]\label{HTangent} Let $\Omega = \{(x,y,t) : x < A(y)\}$ be a flag domain. For $\sigma$ a.e. $p = (A(y),y,t) \in \partial \Omega$, define 
\begin{equation}\label{eq:tau} \tau(p) := n_{2}(A(y),y)X - n_{1}(A(y),y)Y. \end{equation}
The vector field $p \mapsto \tau(p)$ will be called the \emph{horizontal tangent vector field} of $\partial \Omega$.
\end{definition} 
Evidently "$\tau$" is a horizontal unit vector field with $\langle \nu,\tau \rangle \equiv 0$. We record that the integral curves of the horizontal vector field $\tau$ foliate $\partial \Omega$, and the perimeter measure $\sigma$ satisfies a Fubini-type theorem with respect to these integral curves:

\begin{lemma}\label{lemma5} Let $\Omega = \{(x,y,t) : x < A(y)\}$ be a flag domain. For every $t \in \R$, define the curve $\gamma_{t} \colon \R \to \partial \Omega$,
\begin{equation}\label{curveGamma} \gamma_{t}(y) = \left( A(y), y, t - \tfrac{y}{2}A(y) + \int_{0}^{y} A(r) \, dr \right), \end{equation}
and the map $\Gamma \colon \R^{2} \to \partial \Omega$, $\Gamma(y,t) := \gamma_{t}(y)$. Then the following area formula holds:
\begin{equation}\label{form16} \int g \, d\sigma = c\iint g(\Gamma(y,t))\sqrt{1 + A'(y)^{2}} \, dy \, dt, \qquad g \in L^{1}(\sigma), \end{equation}
where $c > 0$ is an absolute constant. \end{lemma}

\begin{proof} This is \cite[Lemma 7.5]{2019arXiv191103223F}. \end{proof}
The map "$\Gamma$" defined under \eqref{curveGamma} will play a major role in the paper as the natural "graph map" parametrising the boundary of a flag domain. The equation \eqref{form16} is an "area" formula for $\Gamma$, and it has the following corollary:

\begin{cor} Let $\Omega \subset \R^{3}$ be a flag domain, and let $V \in L^{1}(\sigma)$ be a horizontal vector field. Then, 
\begin{equation}\label{area} \int \langle V,\tau \rangle \, d\sigma = c\iint \langle V(\gamma_{t}(y)),\dot{\gamma}_{t}(y) \rangle \, dy \, dt, \end{equation} 
where $\gamma_{t} \colon \R \to \partial \Omega$ is the curve defined in \eqref{curveGamma}.
\end{cor}

Here "$\dot{\gamma}(y)$" is an abbreviation for the vector field $\dot{\gamma}_{1}\partial_{x} + \dot{\gamma}_{2}\partial_{y} + \dot{\gamma}_{3}\partial_{t}$ evaluated at $\gamma_{t}(y)$. It turns out that this vector field is horizontal, so the dot product in \eqref{area} is well-defined.

\begin{proof} Recall the formula for "$\tau(p)$" from \eqref{eq:tau}, and the formulae for $n_{1},n_{2}$ from \eqref{flagNormal}:
\begin{align}\label{form94} \tau(\gamma_{t}(y)) & = \tfrac{A'(y)}{\sqrt{1 + A'(y)^{2}}}X_{\gamma_{t}(y)} + \tfrac{1}{\sqrt{1 + A'(y)^{2}}}Y_{\gamma_{t}(y)}\\
& = \tfrac{A'(y)}{\sqrt{1 + A'(y)^{2}}}\left[\partial_{x} - \tfrac{y}{2} \partial_{t} \right] + \tfrac{1}{\sqrt{1 + A'(y)^{2}}}\left[ \partial_{y} +  \tfrac{A(y)}{2}\partial_{t} \right] = \frac{\dot{\gamma}_{t}(y)}{\sqrt{1 + A'(y)^{2}}}, \notag\end{align}
noting that $\dot{\gamma}_{t}(y) = A'(y)\partial_{x} + \partial_{y} + [\tfrac{1}{2}A(y) - \tfrac{y}{2}A(y)]\partial_{t}$. In particular, $\dot{\gamma}_{t}$ is a horizontal vector field, with $\dot{\gamma}(y) = \sqrt{1 + A'(y)}\tau(\gamma_{t}(y))$. Consequently, the "area" formula \eqref{form16} applied to $g = \langle V,\tau \rangle \in L^{1}(\sigma)$ gives
\begin{align*} \int \langle V,\tau \rangle \, d\sigma & = c\iint \langle V(\gamma_{t}(y)),\tau(\gamma_{t}(y)) \rangle \sqrt{1 + A'(y)} \, dy \, dt = c\iint \langle V(\gamma_{t}(y)),\dot{\gamma}_{t}(y) \rangle \, dy \, dt,  \end{align*}
as claimed. \end{proof}

\begin{remark} In addition to the "Heisenberg" area formula \eqref{form16}, the measure $\sigma = |\partial \Omega|_{\He} = \calH^{2}_{E}|_{\partial \Omega}$ also satisfies a "Euclidean" area formula, since $\sigma$ is the surface measure on the graph of the Euclidean Lipschitz function $(y,t) \mapsto (A(y),y,t)$. More precisely,
\begin{equation}\label{eq:sigmaFubini} \int g \, d\sigma = c \iint g(A(y),y,t)\sqrt{1 + A'(y)^{2}} \, dy \, dt = c\iint g(z,t) \, \calH_{E}^{1}|_{\partial U}(z) \, dt, \end{equation}
where $U = \{(x,y) : x < A(y)\}$ is the projection of $\Omega$ to the $xy$-plane. In particular, from \eqref{eq:sigmaFubini} we see that $\sigma$ is a product measure: $\sigma = c\calH^{1}_{E}|_{\partial U} \times \mathcal{L}^{1}$ for some constant $c > 0$. \end{remark}

\subsubsection{Regular and upper regular measures} A Radon measure $\mu$ on $\He$ will be called \emph{upper $s$-regular} if there exists a constant $C > 0$ such that
\begin{displaymath} \mu(B(x,r)) \leq Cr^{s}, \qquad x \in \He, \, r > 0. \end{displaymath}
It is worth emphasising that "$B(x,r)$" here refers to an (open) ball in the metric "$d$" introduced in Section \ref{s:heisenberg}. The measure $\mu$ is \emph{$s$-regular}, if it is upper $s$-regular, and further
\begin{displaymath} \mu(B(x,r)) > c r^{s}, \qquad x \in \spt \mu, \, 0 < r < \diam(\spt \mu) \end{displaymath}
for some constant $c > 0$.

\begin{ex} If $\Omega \subset \He$ is a flag domain, the measure $\sigma := |\partial \Omega|_{\He}$ is $3$-regular. This is easy to verify by hand, using that $\sigma = \calH^{2}_{E}|_{\partial \Omega}$. Alternatively, one can apply the general fact that the horizontal perimeter of a domain bounded by an \emph{intrinsic Lipschitz graph} is $3$-regular, see \cite[Theorem 3.13]{MR2287539} or \cite[Theorem 3.9]{MR3511465}. Intrinsic Lipschitz graphs will not appear prominently in this paper, but we note that the boundary of a flag domain $\{(x,y,t) : x < A(y)\}$ can be written in the form $\partial \Omega \{(0,y,t) \cdot (A(y),0,0) : y,t \in \R\}$. This expression shows that $\partial \Omega$ is the intrinsic graph of the intrinsic Lipschitz function $(0,y,t) \mapsto (A(y),0,0)$. \end{ex}

\subsection{Summary of key operators}\label{s:operatorList} We will use a number of integral and singular integral operators throughout the paper. Their motivation will naturally have to wait until later, but it may be instructive to see all of them once before the paper really begins.

\begin{definition}[Single layer potential] Let $\Omega \subset \He$ be an open set such that $\sigma := |\partial \Omega|_{\He}$ is upper $3$-regular, and let $1 \leq p < 3$. Then, for $f \in L^{p}(\sigma)$, we define the \emph{single layer potential}
\begin{align*}
    \calS f(x) :=  \int  G(y^{-1} \cdot x) f(y) \, d \sigma(y), \qquad x \in \He \, \setminus \, \partial \Omega. 
\end{align*}
\end{definition}
The integral defining $\mathcal{S}f(x)$ is absolutely convergent for all $x \in \He \, \setminus \, \partial \Omega$, observing that $y \mapsto G(y^{-1} \cdot x) \sim d(x,y)^{-2} \in L^{q}(\sigma)$ for all $q > \tfrac{3}{2}$ and using H\"older's inequality. Also, $\mathcal{S}$ is formally self-adjoint, since $G(y^{-1} \cdot x) = G(x^{-1} \cdot y)$ for $x \neq y$. As we verify in Proposition \ref{prop11}, the "boundary values" of $\mathcal{S}f$ on $\partial \Omega$ are given by the following operator:
\begin{definition}[Boundary single layer potential] Let $\Omega \subset \He$ be a domain such that $\sigma := |\partial \Omega|_{\He}$ is upper $3$-regular, and let $1 \leq p < 3$. For $f \in L^{p}(\sigma)$, we define the \emph{boundary single layer potential}
\begin{displaymath} Sf(x) := \int G(y^{-1} \cdot x)f(y) \, d\sigma(y), \qquad x \in \partial \Omega.  \end{displaymath} 
\end{definition}
It turns out (see Proposition \ref{prop11}) that the integral defining $Sf(x)$ is absolutely convergent for $\sigma$ a.e. $x \in \partial \Omega$, and $Sf \in L^{p}(\sigma) + L^{\infty}(\sigma)$. The mapping properties of $S$ (on boundaries of flag domains) are investigated in Theorem \ref{main4}.

\begin{definition} Let $\Omega \subset \He$ be an open set such that $\sigma := |\partial \Omega|_{\He}$ is upper $3$-regular, and let $1 \leq p < \infty$. Then, for a horizontal vector field $V \in L^{p}(\sigma)$, and a point $x \in \He \, \setminus \, \partial \Omega$, we define the \emph{Riesz transform} and the \emph{adjoint Riesz transform}
\begin{displaymath} \mathcal{R}V(x) := \int \langle \nabla G(y^{-1} \cdot x),V(y) \rangle \, d\sigma(y) \quad \text{and} \quad \mathcal{R}^{t}V(x) := \int \langle \nabla G(x^{-1} \cdot y),V(y) \rangle \, d\sigma(y). \end{displaymath} 
\end{definition}
The integrals defining $\mathcal{R}V(x)$ and $\mathcal{R}^{t}V(x)$ are absolutely convergent for all $x \in \He \, \setminus \, \partial \Omega$, because $|\nabla G(y)| \lesssim \|y\|^{-3}$ for all $y \in \He \, \setminus \, \{0\}$, and hence $y \mapsto \nabla G(y^{-1} \cdot x) \in L^{q}(\sigma)$ and $y \mapsto \nabla G(x^{-1} \cdot y) \in L^{q}(\sigma)$ for all $1 < q \leq \infty$. Furthermore, the adjoint Riesz transform $\mathcal{R}^{t}V$ defines a $\bigtriangleup^{\flat}$-harmonic function in $\He \, \setminus \, \Omega$, for the reason explained in Remark \ref{harmonicityRemark}.

The "boundary behaviour" of $\mathcal{R}f$ and $\mathcal{R}^{t}f$ is substantially more delicate than that of $\mathcal{S}f$, and we only have a complete result for flag domains, Theorem \ref{t:rJumps}. The formulae involve the following principal value operators:
\begin{definition} Let $\Omega \subset \He$ be a flag domain, $\sigma := |\partial \Omega|_{\He}$. For $1 < p < \infty$, a horizontal vector field $V \in L^{p}(\sigma)$, and $x \in \partial \Omega$, define 
\begin{displaymath} RV(x) := \pv \int \langle \nabla G(y^{-1} \cdot x),V(y) \rangle \, d\sigma(y) \quad \text{and} \quad R^{t}V(x) := \pv \int \langle \nabla G(x^{-1} \cdot y),V(y) \rangle \, d\sigma(y). \end{displaymath}
\end{definition}
Under the hypotheses of the definition above, we will show in Theorem \ref{t:pValues} that the principal values $RV(x)$ and $R^{t}V(x)$ exist for $\sigma$ a.e. $x \in \partial \Omega$, and $\|RV\|_{L^{p}(\sigma)} + \|R^{t}V\|_{L^{p}(\sigma)} \lesssim \|V\|_{L^{p}(\sigma)}$. The \emph{double layer potential}, which appeared prominently in the introduction, can be expressed in terms of the adjoint Riesz transform:
\begin{definition}[Double layer potential] Let $\Omega \subset \He$ be an open set such that $\sigma := |\partial \Omega|_{\He}$ is upper $3$-regular, and let $1 \leq p < \infty$. Then, for $f \in L^{p}(\sigma)$, we define the \emph{double layer potential}
\begin{displaymath} \mathcal{D}f(x) := \mathcal{R}^{t}(f\nu) = \int \langle \nabla G(x^{-1} \cdot y),\nu(y) \rangle f(y) \, d\sigma(y), \end{displaymath}
where "$\nu$" is the inward-pointing horizontal normal of $\Omega$. \end{definition}
The "boundary behaviour" of $\mathcal{D}f$ can be deduced from the boundary behaviour of $\mathcal{R}^{t}(f\nu)$, see Corollary \ref{c:allJumps}. The formula involves the principal value operator "$D$" from the definition below:
\begin{definition}[Boundary double layer potential]\label{boundaryDoubleLayers} Let $\Omega \subset \He$ be a flag domain, $\sigma := |\partial \Omega|_{\He}$. For $1 < p < \infty$ and $f \in L^{p}(\sigma)$, define the \emph{boundary double layer potential} 
\begin{displaymath} Df(x) := R^{t}(f\nu)(x) = \pv \int \langle \nabla G(x^{-1} \cdot y),\nu(y) \rangle f(y) \, d\sigma(y) \end{displaymath}
and the \emph{adjoint boundary double layer potential} 
\begin{displaymath} \quad D^{t}f(x) := \pv \int \langle \nabla G(y^{-1} \cdot x),\nu(x) \rangle f(y) \, d\sigma(y) = \langle R(fX)(x)X_{x} + R(fY)(x)Y_{x},\nu(x) \rangle.  \end{displaymath} 
\end{definition}
Under the hypotheses of the definition, it follows from the $\sigma$ a.e. existence of the principal values $R(fX),R(fY)$, and $R^{t}(f\nu)$, that the principal values defining $Df$ and $D^{t}f$ exist $\sigma$ a.e., and $\|Df\|_{L^{p}(\sigma)} + \|D^{t}f\|_{L^{p}(\sigma)} \lesssim \|f\|_{L^{p}(\sigma)}$.

%%%%%%%%%%%%%%%%%%%%%%%%%%

\section{Principal values of the Riesz transform}\label{s:principal}

The purpose of this section is to check that several singular integral operators defined on the boundaries of flag domains are well-defined and bounded on $L^{p}(\sigma)$. The boundedness questions are fully resolved by a previous result of F\"assler and the first author (see either \cite[Theorem 1.8]{2018arXiv181013122F} for a fairly simple argument based on the harmonicity of $G$, or \cite[Theorem 1.8]{2019arXiv191103223F} for a result for arbitrary horizontally odd $3$-dimensional kernels):
\begin{thm}\label{t:FO1} Let $\Omega = \{(x,y,t) : x < A(y)\}$ be a flag domain, and let $\sigma := \calH^{2}|_{\partial \Omega}$. Then the maximal (vectorial) Riesz transform
\begin{displaymath} \vec{R}^{\ast}f(p) := \sup_{\epsilon > 0} \left| \int_{\{\|q^{-1} \cdot q\| > \epsilon\}} \nabla G(q^{-1} \cdot p)f(q) \, d\sigma(q) \right|, \qquad f \in L^{p}(\sigma), \end{displaymath}
is a bounded operator $L^{p}(\sigma) \to L^{p}(\sigma)$ for all $1 < p < \infty$. \end{thm}

We remind the reader that the notation "$\nabla$" stands for the horizontal gradient $\nabla f = (Xf,Yf) \in \R^{2}$ in this paper. The main new result in this section is the following:
\begin{thm}\label{t:pValues} Let $\Omega \subset \He$ be a flag domain, $\sigma := |\partial \Omega|_{\He}$, let $1 < p < \infty$, and let $V \in L^{p}(\sigma)$ be a horizontal vector field. Then, the principal values
\begin{displaymath} RV(p) := \pv \int \langle \nabla G(q^{-1} \cdot p),V(q) \rangle \, d\sigma(q) \end{displaymath} 
and
\begin{displaymath} R^{t}V(p) := \pv \int \langle \nabla G(p^{-1} \cdot q),V(q) \rangle \, d\sigma(q) \end{displaymath}
exist for $\sigma$ a.e. $p \in \partial \Omega$, and $\max\{\|RV\|_{L^{p}(\sigma)},\|R^{t}V\|_{L^{p}(\sigma)}\} \lesssim \|V\|_{L^{p}(\sigma)}$. \end{thm}

\begin{remark}\label{rem:strategy} The strategy is the following. We will first prove the $\sigma$ a.e. existence of $RV$ and $R^{t}V$ for special vector fields of the form $V = f\nu$ and $V = f\tau$, where $f \in L^{p}(\sigma)$ is a complex valued function, and $\nu$ and $\tau$ are the horizontal normal and tangent vector fields, respectively, associated to $\Omega$, recall Section \ref{s:normals}. These vector fields are defined $\sigma$ a.e. on $\partial \Omega$, and $\{\nu(p),\tau(p)\}$ is an orthonormal basis for $\spa\{X_{p},Y_{p}\}$ for $\sigma$ a.e. $p \in \partial \Omega$. So, once the principal values $R(f\nu),R(f\tau),R^{t}(f\nu)$ and $R^{t}(f\tau)$ are known to exist for all $f \in L^{p}(\sigma)$, the theorem follows by decomposing a general vector field $V \in L^{p}(\sigma)$ as
\begin{displaymath} V = \langle V,\nu \rangle \nu + \langle V,\tau \rangle \tau =: f_{1}\nu + f_{2}\tau, \end{displaymath}
noting that $f_{1},f_{2} \in L^{p}(\sigma)$, and applying the linearity of $R$ and $R^{t}$. \end{remark}

It is worth recording (or rather repeating from Section \ref{s:operatorList}) the relations
\begin{equation}\label{form162} Df = R^{t}(f\nu) \quad \text{and} \quad D^{t}f = \langle R(fX)X + R(fY)Y,\nu \rangle, \end{equation}
where $Df$ and $D^{t}f$ are the (boundary) double layer potential and its adjoint, namely
\begin{align*}
    & Df(p) := \pv \int \left< \nabla G(p^{-1} \cdot q), \nu(q) \right> f(q) \, d\sigma(q), \\
    & D^t f(p) := \pv \int \left< \nabla G(q^{-1} \cdot p), \nu(p)\right> f(q) \, d\sigma(q).  
\end{align*}
From the relations \eqref{form162}, one sees that Theorem \ref{t:pValues} implies the $\sigma$ a.e. pointwise existence of $Df,D^{t}f \in L^{p}(\sigma)$ for all $f \in L^{p}(\sigma)$ (in the setting of flag domains).

The easiest sub-problem turns out to be the existence of $R^{t}(f\nu) = Df$. The structure of flag domains is only used in a mild way. In fact, as Lemma \ref{pvL2-K} shows, the principal values $Df$ exist $\sigma$ a.e. on $\partial \Omega$ whenever $\Omega \subset \He$ is a corkscrew domain (see Definition \ref{def:csDomain}) with $3$-regular boundary such that the maximal Riesz transform is bounded on $L^{p}(\sigma)$. The maximal Riesz transform might \textbf{always} be bounded on $L^{p}(\sigma)$ under these assumptions, but this is an open problem. We know the $L^{p}(\sigma)$ boundedness only when $\Omega$ is a flag domain, by Theorem \ref{t:FO1}. In fact, \cite{2018arXiv181013122F} contains a somewhat more general result, and the existence of $Df$ could indeed be extended to the class of domains treated in \cite{2018arXiv181013122F}.

The other three sub-problems, concerning the existence of $R^{t}(f\tau)$, $R(f\nu)$, and $R(f\tau)$, require arguments more tied to the structure of flag domains: for example, the area formula \eqref{area} is needed to survive $R^{t}(f\tau)$, and $R(f\nu),R(f\tau)$ require even more \emph{ad hoc} arguments involving right invariant vector fields. 

%To say something about the principal values of $Df$ and $D^tf$, it is useful to use the divergence theorem first, and, immediately after that, the fact that $G$ is a harmonic function away from the origin. However, already here the fact that $\He$ is a noncommutative group comes to the surface, as follows. While it is true that 
%\begin{align}\label{case1}
%    \Div (q \mapsto \nabla G(p^{-1} q)) = \bigtriangleup (q \mapsto G(p^{-1} q)) = \Lap G(p^{-1}q),
%\end{align}
%it is \textit{not} true that
%\begin{align}\label{case2}
%    \Div (q \mapsto \nabla G(q^{-1} p)) = \Lap (q \mapsto G(q^{-1} p)). 
%\end{align}
%The reason for this is that on one hand in \eqref{case1} we have a  \textit{left} translation by $p^{-1}$, and the divergence operator is defined via left invariant vector fields, thus $\Div (q \mapsto \nabla G(p^{-1} q)) = (\Div \nabla G) (p^{-1} q) = \Lap G(p^{-1} q)$. In the other case, we see a \textit{right} translation, and therefore one cannot derive the equality \eqref{case2}. 

\subsection{Existence of $R^{t}(f\nu)$} The aim of this section is to show that $R^{t}(f\nu)(p) = Df(p)$ exists for $\sigma$ a.e. $p \in \partial \Omega$ whenever $f \in L^{p}(\sigma)$, $1 < p < \infty$, and $\Omega \subset \He$ is a flag domain. We mainly keep to the notation $Df$ in this section. For $\ve>0$ and $f \in L^{p}(\sigma)$, we set
\begin{displaymath}
    D_\ve f(p) := \int_{B(p, \ve)^c }\left<\nabla G (p^{-1} \cdot q) , \nu(q) \right> f(q)\, d \sigma (q).
\end{displaymath}
%and
%\begin{align}\label{doubly-truncated}
%    D_{\ve, R} f (p) :=  \int_{A(p, \ve, R)} \left< \nabla G(p^{-1} \cdot q), \nu(q) \right> f(q)\, d\sigma(q),
%\end{align}
%where $A(p,\epsilon,R) = B(p,R) \, \setminus \, B(p,\epsilon)$.

\begin{definition}[Tangent points] Let $E \subset \He$ be a set. %A point $p \in E$ is called a \emph{tangent point} if there exists a vertical subgroup $\W \subset \He$ such that 
%\begin{displaymath} \sup \left\{ \tfrac{\dist(q,p \cdot \W)}{r}  : q \in E \cap B(p,r)\right\} \to 0 \end{displaymath}  
%as $r \to 0$. 
A point $p \in E$ is called a \emph{weak tangent point} if $\beta_{E}(p,r) \to 0$ as $r \to 0$, where $\beta_{E}(p,r)$ is the \emph{vertical $\beta$-number}
\begin{displaymath} \beta_{E}(p,r) := \inf_{z,\W} \sup \left\{ \tfrac{\dist(q,z \cdot \W)}{r} : q \in E \cap B(p,r) \right\}. \end{displaymath}
Here the "$\inf$" runs over all $z \in \He$ and vertical subgroups $\W \subset \He$. \end{definition}
%So, if $p \in E$ is a tangent point, there exists a fixed vertical plane $p \cdot \W$ which approximates $E$ increasingly well at small scales. If $p \in E$ is a weak tangent point, then the best-approximating planes are allowed to vary from scale to scale.
We record the following easy and well-known observations:
\begin{lemma}\label{lemma9} A point $p \in E$ is a weak tangent point if and only if
\begin{displaymath} \inf_{\W} \sup \left\{\tfrac{\dist(q,p \cdot \W)}{r} : q \in E \cap B(p,r) \right\} \to 0 \quad \text{as } r \to 0. \end{displaymath} \end{lemma}

\begin{proof} The statement means that if $\beta_{E}(p,r) \to 0$, then the "base point" $z \in \He$ in the definition of the $\beta$-number can be fixed to be $z = p$, and the slightly larger quantities still tend to $0$ as $r \to 0$. The reason is that if $\beta_{E}(p,r) \to 0$, then the point $p \in E$ itself lies at vanishing relative distance from the best-approximating plane $z_{r} \cdot \W_{r}$ at scale $r > 0$. Hence the $\tfrac{1}{r}$-normalised Hausdorff distance between $[z_{r} \cdot \W_{r}] \cap B(p,r)$ and $[p \cdot \W_{r}] \cap B(p,r)$ tends to $0$ as $r \to 0$, and the claim follows. We leave further details to the reader.  \end{proof}

In the sequel, we will abbreviate $\sigma_{p,r} := |\partial B(p,r)|_{\He}$ and $\sigma_{r} := \sigma_{0,r}$. It follows by easy calculations, using the definition of horizontal perimeter, that
\begin{equation}\label{eq:sigmas} \sigma_{p,r} = r^{3} \cdot \tau_{p\sharp} \sigma_{1}, \end{equation}
where $\tau_{p\sharp}$ refers to push-forward under the left translation $\tau_{p}(q) = p \cdot q$.

\begin{definition}[Corkscrews]\label{def:csDomain} An open set $\Omega \subset \He$ satisfies the \emph{corkscrew condition} if there exists a constant $0 < c < 1$ (the "corkscrew constant") such that the following holds. For every $p \in \partial \Omega$ and $0 < r < \diam(\partial \Omega)$, there exist points $q_{1} \in B(p,r) \cap \Omega$ and $q_{2} \in B(p,r) \, \setminus \, \overline{\Omega}$ such that
\begin{displaymath} \min\{\dist(q_{1},\partial \Omega),\dist(q_{2},\partial \Omega)\} \geq cr. \end{displaymath}
\end{definition}

Here is one reason why the corkscrew condition is useful:

\begin{proposition}\label{p:tangentPoints} Let $\Omega \subset \He$ be an open set of locally finite perimeter satisfying the corkscrew condition such that the perimeter measure $\sigma := |\partial \Omega|_{\He}$ is upper $3$-regular. Then $\sigma$ a.e. point $p \in \partial \Omega$ is a weak tangent point of $\Omega$. 
\end{proposition} 

\begin{proof} The corkscrew condition implies that $\partial \Omega$ equals the \emph{measure theoretic boundary} of $\Omega$, see \cite[Definition 7.4]{FSSC}. Therefore, \cite[Corollary 7.6]{FSSC} implies that $\sigma$ equals (up to a multiplicative constant) the $3$-dimensional spherical Hausdorff measure $\mathcal{S}^{3}$ restricted to $\partial \Omega$. In particular, the assumption that $\sigma$ is upper $3$-regular implies that $\partial \Omega$ is upper $3$-regular in the usual sense
\begin{displaymath} \mathcal{H}^{3}(B(p,r) \cap \partial \Omega) \lesssim r^{3}, \qquad p \in \He, \, r > 0. \end{displaymath}
Consequently, $\partial \Omega$ is a \emph{Semmes surface}, as in \cite[Definition 1.6]{2018arXiv180304819F}. The fact that $\sigma$ a.e. point $p \in \partial \Omega$ is a weak tangent point then follows immediately from \cite[Proposition 5.11]{2018arXiv180304819F}. \end{proof} 

\begin{lemma}\label{lemma10} Let $\Omega \subset \He$ satisfying the corkscrew condition with constant "$c$" such that $0 \in \partial \Omega$. Then, for every $\eta > 0$ there exists $\delta > 0$ (depending only on $c,\eta$) such that the following holds. If $r > 0$, and 
\begin{equation}\label{form146} \inf_{\W} \sup \left\{\dist(q,\W) : q \in \partial \Omega \cap \bar{B}(r) \right\} \leq \delta r, \end{equation} 
there exists a half-space $\He^{+}_{r}$ bounded by a vertical subgroup $\W_{r}$ such that $\sigma_{r}(\Omega \bigtriangleup \He_{r}^{+}) \leq \eta r^{3}$. \end{lemma}
\begin{proof} Let $\eta > 0$, and assume that \eqref{form146} holds with a constant $\delta > 0$ to be specified in the course of the proof. Let $\W$ be a vertical subgroup minimising \eqref{form146}, so $\partial \Omega \cap \bar{B}(r) \subset \overline{\W(\delta r)}$. This implies that the two components $U^{-}$ and $U^{+}$ of $U := B(r) \, \setminus \overline{\W(\delta r)}$ are contained in either $\Omega$ or $\He \, \setminus \, \overline{\Omega}$. In fact, we claim that one of $U^{-},U^{+}$ is contained in $\Omega$, and the other one in $\He \, \setminus \, \overline{\Omega}$. If this were not the case, then either $U \subset \Omega$ or $U \subset \He \, \setminus \, \overline{\Omega}$, and hence 
\begin{displaymath} B(r) \, \setminus \, \overline{\Omega} \subset \W(\delta r) \quad \text{or} \quad B(r) \cap \Omega \subset \W(\delta r). \end{displaymath}
For $\delta \ll c$ small enough, both options would violate the corkscrew condition at $0 \in \partial \Omega$. For example, if the first option took place, pick $q_{2} \in B(\tfrac{r}{2}) \, \setminus \, \overline{\Omega}$ with $\dist(q_{2},\partial \Omega) \geq \tfrac{cr}{2}$. Then $B(q_{2},\tfrac{cr}{2}) \subset B(r) \, \setminus \, \overline{\Omega} \subset \W(\delta r)$ (since $c < 1$), which is not possible if $\delta \ll c$.

 So, we may assume $U^{-} \subset \He \, \setminus \, \overline{\Omega}$ and $U^{+} \subset \Omega$. We define $\He^{+}$ to be the half-space bounded by $\W$ containing $U^{+}$. Then $[\Omega \bigtriangleup \He^{+}] \cap \partial B(r) \subset \overline{\W(\delta r)}$, and hence $\sigma_{r}(\Omega \bigtriangleup \He^{+}) \leq \sigma_{r}(\overline{\W(\delta r))})$. But using \eqref{eq:sigmas}, we infer that
\begin{displaymath} \sigma_{r}(\Omega \bigtriangleup \He^{+}) \leq r^{3} \cdot \sigma_{1}(\overline{\W(\delta)}). \end{displaymath} 
Now, it remains to choose $\delta = \delta(\eta) > 0$ so small that $\sigma_{1}(\overline{\W(\delta)}) \leq \eta$. \end{proof}

Here is the first result on the existence of principal values:

\begin{lemma}\label{pv-K} Let $\Omega \subset \He$ be an open set of locally finite $\He$-perimeter satisfying the corkscrew condition. If $p \in \partial \Omega$ is a weak tangent point of $\partial \Omega$, and $R > 0$, then 
\begin{align*}
   D\mathbf{1}_{B(p,R)}(p) := \lim_{\ve \to 0} D_{\ve} \mathbf{1}_{B(p,R)}(p) = \tfrac{1}{2} - \mathfrak{r}(\Omega,R, p), 
\end{align*}
where
\begin{equation}\label{K-pv10} \mathfrak{r}(\Omega,R,p) := \int_{\Omega} \langle \nabla G(p^{-1} \cdot q),\nu_{p,R}(q) \rangle \, d\sigma_{p,R}(q). \end{equation} 
\end{lemma}

We need one more auxiliary result, which is virtually a restatement of the fact that $G$ is the fundamental solution of $-\bigtriangleup^{\flat}$:
\begin{lemma}\label{K-ball} Let $\Omega \subset \He$ be an open set with finite $\He$-perimeter. Write $\sigma := |\partial \Omega|_{\He}$, and let $\nu$ be the inward-pointing horizontal normal of $\Omega$. Then,
\begin{equation}\label{kappa}
     \int_{\partial \Omega} \langle \nabla G(p^{-1} \cdot q), \nu(q) \rangle \, d\sigma(q) = 1, \qquad p \in \Omega.
\end{equation}
\end{lemma}
\begin{proof} Define $G_{\epsilon} := \psi_{\epsilon} \ast G$, where $\{\psi_{\epsilon}\}_{\epsilon > 0}$ is a (compactly supported) approximate identity, and observe that $G_{\epsilon}$ is a smooth function satisfying $\bigtriangleup^{\flat}(\psi_{\epsilon} \ast G) = -\psi_{\epsilon}$. Consequently, by the divergence theorem,
\begin{displaymath} \int_{\partial \Omega} \langle \nabla G_{\epsilon}(p^{-1} \cdot q),\nu(q) \rangle \, d\sigma(q) = \int_{\Omega} \psi_{\epsilon}(p^{-1} \cdot q) \, dq = 1 \end{displaymath}
for all $\epsilon > 0$ so small that $\spt (q \mapsto \psi_{\epsilon}(p^{-1} \cdot q)) \subset \Omega$. On the other hand, the LHS clearly tends to the LHS of \eqref{kappa}. \end{proof}

%A straightforward computation gives the formula
%\begin{displaymath} |\nabla G(p)| = \frac{2|z|}{\|p\|^{4}}, \qquad p \in \He \, \setminus \, \{0\}. \end{displaymath}
%We recall from \cite[Theorem 6.5]{FSSC} that the inward-pointing horizontal normal of a set of the form $\{f > c\}$ is given by $\nabla f/|\nabla f|$ when $f \in C^{1}$ in a neighbourhood of $\partial \{f > c\}$. Since $B(r) = \{G > r^{-2}\}$, this implies
%and consequently
%\begin{displaymath}  \int_{\partial B(r)} \langle \nabla G(p),\nu_{r}(p) \rangle \, d\sigma_{r}(p) =  \int_{\partial B(r)} \frac{2 |z|}{r^{4}} \, d\sigma_{r}(p).  \end{displaymath} 
%It remains to note that 
%\begin{displaymath} \int_{\partial B(r)} \frac{|z|}{r^{4}} \, d\sigma_{r}(p) = \frac{1}{r^3} \int_{\partial B(r)} |z| \, d \delta_{r^{-1}\sharp}(\sigma_r)(p) = \int_{\partial B(1)} |z| \, d\sigma_{1}(p), \end{displaymath}
%using \eqref{eq:sigmas} in the last equation.

We are then prepared to prove Lemma \ref{pv-K}:

 \begin{proof}[Proof of Lemma \ref{pv-K}] We assume with no loss of generality that $p = 0$, so we claim that if $0 \in \partial \Omega$ is a weak tangent point of $\partial \Omega$, then $D_{\epsilon}\mathbf{1}_{B(R)}(0) \to \tfrac{1}{2} + \mathfrak{r}(\Omega,R,0)$ as $\epsilon \to 0$. Set $\Omega_{\ve,R} := \Omega\cap A(\ve, R)$, where $A(\ve,R) := B(R) \, \setminus \, \bar{B}(\epsilon)$. We denote $\sigma_{\epsilon,R} := |\partial \Omega_{\epsilon,R}|_{\He}$, and write $\nu_{\ve, R}$ for the inward-pointing horizontal normal of $\Omega_{\epsilon,R}$. We remark that 
 \begin{itemize}
 \item $\sigma_{\epsilon,R}|_{\partial B(R) \cap \Omega} = \sigma_{R}|_{\Omega}$ and $\sigma_{\epsilon,R}|_{\partial B(\epsilon) \cap \Omega} = \sigma_{\epsilon}|_{\Omega}$ and $\sigma_{\epsilon,R}|_{\partial \Omega} = \sigma|_{\partial \Omega}$,
 \item $\nu_{\ve, R}|_{\partial B(R) \cap \Omega} = \nu_{R}|_{\Omega}$ and $\nu_{\epsilon,R}|_{\partial B(\epsilon) \cap \Omega} = -\nu_{\epsilon}|_{\Omega}$ and $\nu_{\epsilon,R}|_{\partial \Omega} = \nu|_{\partial \Omega}$,
\end{itemize}
where $\sigma,\nu$ are the horizontal perimeters and (inward) normals of $\Omega$, and $\sigma_{s},\nu_{s}$ are these objects for the ball $B(s)$, as before. These statements follow, most easily, from the locality of horizontal perimeter, see \cite[Corollary 2.5]{MR2678937} and \cite[Theorem 2.8]{MR2678937}. 
  
By first applying the divergence theorem in \cite{FSSC}, we find that
\begin{align}\label{div-to-lap}
 -\int_{\partial \Omega_{\epsilon,R}} \left< \nabla G(q), \nu_{\ve, R}(q) \right>
  \, d \sigma_{\ve, R}(q)  =  \int_{\Omega_{\epsilon,R}} \Lap^{\flat} G (q) \, dq = 0.
\end{align}
Then, using \eqref{div-to-lap} and the various relationships between the horizontal perimeters and normals listed above, we see that
\begin{align*}
    D_{\epsilon}\mathbf{1}_{B(R)}(0) = \int_{\partial \Omega \cap A(\ve, R)} \left< \nabla G(q), \nu(q) \right> \,
      d\sigma(q) & =  \int_{\partial B(\ve) \cap \Omega} \left< \nabla G(q), \nu_\ve(q) \right> d \sigma_{\epsilon}(q)  \nonumber \\
      & \quad -  \int_{\partial B(R) \cap
      \Omega} \left< \nabla G(q), \nu_R(q)\right> \, d \sigma_{R}(q). 
\end{align*}
The second term equals $\mathfrak{r}(\Omega,R,0)$ by definition, so it remains to show that the first term, denoted $I_{\epsilon}$, tends to $1/2$ as $\epsilon \to 0$. Since $0$ is a weak tangent point of $\partial \Omega$, by Lemma \ref{lemma9} we have
\begin{displaymath} \lim_{\epsilon \to 0} \inf_{\W} \left\{ \tfrac{\dist(q,\W)}{\epsilon} : q \in \partial \Omega \cap B(\epsilon) \right\} = 0. \end{displaymath} 
Since $\Omega$ and $\He \, \setminus \, \overline{\Omega}$ were assumed $4$-regular, Lemma \ref{lemma10} applies for $\epsilon > 0$ small enough and yields a half-space $\He^{+}_{\epsilon}$, bounded by a vertical subgroup $\W_{\epsilon}$, such that $\sigma_{\epsilon}(\Omega \bigtriangleup \He^{+}_{\epsilon}) \leq \epsilon^{3} \eta$. Here $\eta = \eta(\epsilon) \to 0$ as $\epsilon \to 0$. We may then rewrite $I_{\epsilon}$ as 
\begin{align}\label{e:pv11}
     \int_{\partial B(\ve) \cap \Omega} \left< \nabla G(q), \nu_\ve(q) \right> \, d \sigma_{\epsilon}(q)
      & = \int_{\partial B(\ve)} (\mathbf{1}_{\Omega}(q) - \mathbf{1}_{\He^+}(q))\left<
      \nabla G(q), \nu_\ve(q) \right>  \, d \sigma_{\epsilon}(q) \nonumber\\
      & \enskip + \int_{B(\ve)
      \cap \He^+} \left< \nabla G(q), \nu_\ve(q) \right> \, d \sigma_{\epsilon}(q)  =: I_{\epsilon}^1+ I_{\epsilon}^2.
\end{align}
We first claim that
\begin{align} \label{e:pv-halfball}
  I_{\epsilon}^{2} \equiv \int_{\partial B(0, \ve)\cap \He^+} \left< \nabla G(q), \nu_\ve(q) \right> \, d\sigma_\ve(q) = \frac{1}{2}.
\end{align}
To see this, one recalls (from e.g. Folland's original work \cite{MR0315267}) that $G(q) = c\|q\|^{-2}$ for some positive constant $c > 0$. Therefore, the sphere $\partial B(0,\epsilon)$ can be written as a level set $\{G > \kappa\}$ for some $\kappa = \kappa(\epsilon) > 0$. It then follows from \cite[Theorem 6.5]{FSSC} that the inward-pointing horizontal normal $\nu_{\epsilon}$ can be represented as $\nu_{\epsilon} = \nabla G/|\nabla G|$, and consequently 
\begin{displaymath} \langle \nabla G(q),\nu_{\epsilon}(q) \rangle =  \frac{\langle \nabla G,\nabla G \rangle}{|\nabla G|} = \frac{2c|z|}{\epsilon^{4}}, \qquad q = (z,t) \in \partial B(\epsilon), \end{displaymath} 
by direct computation. Next, we note that a $180^{\circ}$ rotation "$\mathcal{O}$" around the $t$-axis is an isometry which preserves the quantity displayed above, and the perimeter measure $\sigma_{\epsilon}$. Therefore,
\begin{align*}
\int_{\partial B(\ve) \cap \He^+} \left< \nabla G(q), \nu_\ve(q) \right>  \, d \sigma_\ve(q) & =
\int_{\partial B(\ve) \cap \He^+} \left< \nabla G(\mathcal{O}(q)), \nu_\ve(\mathcal{O}(q)) \right>  \, d \sigma_\ve(q) \\
& = \int_{\partial B(\ve) \cap \He^{-}} \left< \nabla G(q), \nu_\ve(q) \right>  \, d \mathcal{O}_{\sharp}\sigma_\ve(q)\\
& = \int_{\partial B(\ve) \cap \He^-} \left< \nabla G(q), \nu_\ve(q) \right> \, d \sigma_\ve(q),
\end{align*}
which gives \eqref{e:pv-halfball}. 

For the integral $I_{\epsilon}^1$ in \eqref{e:pv11}, we just need to estimate
\begin{displaymath} |I_{\epsilon}^{1}| \lesssim \frac{\sigma_{\epsilon}([\Omega \bigtriangleup \He^{+}])}{\epsilon^{3}} \leq \eta(\epsilon) \to 0 \quad \text{as } \epsilon \to 0. \end{displaymath}
Thus $\lim_{\epsilon \to 0} I_{\epsilon} = 1/2$, and the proof is complete.  \end{proof}

From Lemma \ref{pv-K} we infer that $Df$ exists for $f \in C_{c}^{\infty}(\He)$:
\begin{cor}\label{c:pv-R-smooth}
  Let $f \in C^\infty_c(\He)$, and let $p$ is a weak tangent point of $\partial \Omega$, where $\Omega \subset \He$ is an open set with locally finite $\He$-perimeter satisfying the corkscrew condition. Then, the principal value
  \begin{align*}
   Df(p) = \pv   \int_{B(p, \ve)^c} \left< \nabla G(p^{-1} \cdot q), \nu(q) \right> f(q)\, d \sigma(q)
  \end{align*}
  exists.
\end{cor}
\begin{proof} Choose a radius $R > 0$ such that $\spt(f) \subset B(p,R)$. Then $D_{\epsilon}f = D_{\epsilon}(f\mathbf{1}_{B(p,R)})$ for all $\epsilon > 0$, and
\begin{displaymath} D(f\mathbf{1}_{B(p,R)})(p) =  \int_{B(p,R)} \left< \nabla G(p^{-1} \cdot q), \nu(q) \right> (f(q)-f(p)) \, d\sigma(q) + f(p) \cdot D\mathbf{1}_{B(p,R)}(p).
\end{displaymath}
The first integral is absolutely convergent by the smoothness of $f$, and the second expression is well-defined by Lemma \ref{pv-K}. \end{proof}
We now extend Corollary  \ref{c:pv-R-smooth} to $f \in L^{p}(\sigma)$, for $1 < p < \infty$, via the following result:
\begin{thm}[\cite{duoandi}, Theorem 2.2] \label{t:general-pv}
  Let $(X,\mu)$ be a measure space, $1 < p,q < \infty$, and let $\{T_t\}_{t \in \R}$ be a family of linear operators on $L^p(X, \mu)$; set
  \begin{align*}
    T^* f(x) = \sup_{t} |T_tf(x)|.
  \end{align*}
  If $T^*$ is bounded $L^{p}(X) \to L^{q,\infty}(X)$, then the set
  \begin{align*}
    \left\{ f \in L^p(X, \mu) \,|\, \lim_{t \to 0}T_t f(x) \mbox{ exists } \, \mu-\mbox{a.e.} \right\}
  \end{align*}
  is closed in $L^p(X, \mu)$.
\end{thm}
The following lemma proves Theorem \ref{t:pValues} in the special case $Df = R^{t}(f\nu)$:

\begin{lemma}\label{pvL2-K} Let $\Omega \subset \He$ be an open set of locally finite $\He$-perimeter satisfying the corkscrew condition, such that the perimeter measure $\sigma := |\partial \Omega|_{\He}$ is upper $3$-regular. Assume that the maximal (vectorial) Riesz transform
\begin{displaymath} \vec{R}^{\ast}f(p) := \sup_{\epsilon > 0} \left|\int_{\{\|q^{-1} \cdot p\| > \epsilon\}} \nabla G(q^{-1} \cdot p)f(q) \, d\sigma(q) \right|, \qquad f \in L^{p}(\sigma), \end{displaymath} 
defines a bounded operator on $L^{p}(\sigma)$, for some $1 < p < \infty$. Then, the principal values $Df(p)$ exist for $f \in L^{p}(\sigma)$, for $\sigma$ a.e. $p \in \partial \Omega$. Moreover, $Df \in L^{p}(\sigma)$ with $\|Df\|_{L^{p}(\sigma)} \lesssim \|f\|_{L^{p}(\sigma)}$ (the implicit constant depends only on the operator norm of $\vec{R}^{\ast}$).
\end{lemma}

\begin{remark} Theorem \ref{t:FO1} states that the hypothesis concerning $\vec{R}^{\ast}$ is valid whenever $\Omega$ is a flag domain, and $1 < p < \infty$. \end{remark}

\begin{proof}[Proof of Lemma \ref{pvL2-K}] By Proposition \ref{p:tangentPoints}, the assumptions of the lemma imply that $\sigma$ a.e. $p \in \partial \Omega$ is a weak tangent point. Hence $Df(p)$ exists for all $f \in C^\infty_c (\He)$ for $\sigma$ a.e. $p \in \partial \Omega$ by Corollary \ref{c:pv-R-smooth}. Clearly $C_{c}^{\infty}(\He)$ is dense in $L^p(\sigma)$. So, it remains by Theorem \ref{t:general-pv} to prove that the maximal operator
\begin{displaymath} D^{\ast}f(p) := \sup_{\epsilon > 0} \left| \int_{B(p,\epsilon)^{c}} \langle \nabla G(p^{-1} \cdot q),\nu(q) \rangle f(q) \, d\sigma(q) \right|, \qquad f \in L^{p}(\sigma), \end{displaymath} 
is bounded $L^{p}(\sigma) \to L^{p}(\sigma)$. This follows rather immediately from our assumption on the $L^{p}(\sigma)$-boundedness of the maximal vectorial Riesz transform $\vec{R}^{\ast}$, since $|\nu(q)| \leq 1$ for $\sigma$ a.e. $q \in \partial \Omega$. The only small catch is that, to control $D^{\ast}$, we in fact need to know that maximal vectorial Riesz transform with the adjoint kernel $(p,q) \mapsto \nabla G(p^{-1} \cdot q)$ is bounded $L^{p}(\sigma) \to L^{p}(\sigma)$. The hypothesis on $\vec{R}^{\ast}$ implies that the $\epsilon$-truncated operators 
\begin{displaymath} \vec{R}_{\epsilon}f(p) := \int_{\{\|q^{-1} \cdot p\| > \epsilon\}} \nabla G(q^{-1} \cdot p)f(q) \, d\sigma(q)  \end{displaymath} 
are bounded uniformly on $L^{p}(\sigma)$, and hence are their adjoint operators with kernels $(p,q) \mapsto \nabla G(p^{-1} \cdot q)$. So, all we need is a Cotlar-type inequality saying that the adjoint maximal Riesz transform is also bounded on $L^{p}(\sigma)$. One option to use \cite[Theorem 7.1]{MR1626935}, which is available by the upper $3$-regularity of $\sigma$ (see also the Remark just under \cite[Theorem 1.1]{MR1626935}). Another possibility is to first observe that $\sigma$ is actually $3$-regular (not just upper $3$-regular) by \cite[Proposition 4.1]{2018arXiv180304819F}, so $(\partial \Omega,d,\sigma)$ is a space of homogeneous type. In this setting, any textbook proof of Cotlar's inequality will work.   \end{proof}

%    The main result in \cite{2018arXiv181013122F},
%  says that the Riesz transform is a bounded operator on $L^2(\sigma)$. It is immediate that the corresponding adjoint operator is also bounded. From a Cotlar inequality one then obtains that the operator pointfies defined by 
%  \begin{align*}
%     \sup_{\ve>0} \left| \int_{B(p, \ve)^c} \nabla G(p^{-1} q) f(q) \, d \sigma(q)\right|
%  \end{align*}
%  is bounded in $L^2(\sigma)$. Denote by $\nu^1(q)$ and $\nu^2(q)$ the coordinates functions of $\nu$ in the $\{X_q, Y_q\}$ basis. Then we write 
%\begin{align*}
%    \sup_{\ve>0} \left| \int_{B(p, \ve)^c} \left< \nabla G(p^{-1}q), \nu(q) \right> f(q) \, d \sigma(q)\right| & \leq \sup_{\ve >0 } \left| \int_{B(p, \ve)^c}  X G(p^{-1}q) (\nu^1(q)  f(q)) \, d \sigma(q)\right|\\
%    & + \sup_{\ve >0 } \left| \int_{B(p, \ve)^c}  Y G(p^{-1}q) (\nu^2(q)  f(q)) \, d \sigma(q)\right|.
%\end{align*}
%Thus, by the fact that $|\nu_1|, |\nu_2| \leq 1$, we have that 
%  \begin{align*}
%      \left\|\left( \sup_{\ve>0} \left| \int_{B(p, \ve)^c} \left< \nabla G(p^{-1}q), \nu(q) \right> f(q) \, d \sigma(q)\right| \right) \right\|_{L^2(\sigma)} \lesssim \|f\|_{L^2(\sigma)}.
%  \end{align*}
% Using Theorem \ref{t:general-pv}, we end the proof of the Lemma.

\subsection{Existence of $R^{t}(f\tau)$} As we discussed under the statement of Theorem \ref{t:pValues}, the proof consists of verifying separately the existence of
\begin{displaymath} Df = R^{t}(f\nu), \, R^{t}(f\tau), \, R(f\nu), \quad \text{and} \quad R(f\tau). \end{displaymath}
The first problem was solved by Lemma \ref{pvL2-K}. In this section we treat the existence of
\begin{displaymath} R^{t}(f\tau) = \pv  \int \langle \nabla G(p^{-1} \cdot p),\tau(q) \rangle f(q) \, d\sigma(q), \qquad f \in L^{p}(\sigma). \end{displaymath} 
The proof is similar to the one seen above, but we need to restrict attention to flag domains: we have not even defined the tangent vector field "$\tau$" in more generality, and we will need the area formula \eqref{area} in the first lemma below. Recall the integral curves $\gamma_{t} \colon \R \to \partial \Omega$ (of $\tau$) which were introduced in \eqref{curveGamma}.

\begin{lemma}\label{l:Ttau-pv} Let $\Omega \subset \He$ be a flag domain, and let $0 < \epsilon < R < \infty$. Then,
  \begin{align*}
    \int_{A(p, \ve, R)} \left< \nabla G(p^{-1}\cdot q), \tau (q) \right> \, d \sigma(q) = 0, \qquad p \in \He,
  \end{align*}
  where $A(p,\epsilon,R) := B(p,R) \, \setminus \, \bar{B}(p,\epsilon)$ is an open annulus around $p$.
  \end{lemma}
\begin{proof}
  Since $q \mapsto \nabla G(p^{-1} \cdot q) \in L^{1}(A(p,\epsilon,R),\sigma)$, we may apply the area formula \eqref{area}:
  \begin{align*}
    \int_{A(p, \ve, R)} \left< \nabla G (p^{-1} \cdot q), \tau(q) \right> \, d \sigma(q) = c\int_\R \int_{S_t}\left< \nabla G(p^{-1} \cdot \gamma_t (y)), \dot \gamma_t(y) \right> \, dy \, dt,
  \end{align*}
  where $S_t := \left\{ y : \gamma_{t}(y) \in A(p, \ve, R)\right\}$. Let $t \in \R$ so that $S_t \neq \emptyset$. Then $S_t$ is a bounded open set in $\R$ (boundedness follows from $\|\gamma_{y}(t)\| \geq |y|$), hence a disjoint union of its open component intervals, denoted $\mathcal{I}_{t}$. We may write
  \begin{displaymath} \int_{S_t}\left< \nabla G(p^{-1} \cdot \gamma_t (y)), \dot \gamma_t(y) \right> \, dy = \sum_{I \in \mathcal{I}_{t}} \int_{I}\left< \nabla G(p^{-1} \cdot \gamma_t (y)), \dot \gamma_t(y) \right> \, dy =: \sum_{I \in \mathcal{I}_{t}} E(I). \end{displaymath}
  Note that $\gamma_{t}(\partial I) \subset \partial A(p, \ve, R)$ for all $I \in \mathcal{I}_{t}$. Every interval $I = (a,b) \in \mathcal{I}_{t}$ has, therefore, one of the following four types:
  \begin{enumerate}
  \item $\gamma_{t}(\partial I) \subset \partial B(p,\epsilon)$,
  \item $\gamma_{t}(\partial I) \subset \partial B(p,R)$,
    \item $\gamma_{t}(a) \in \partial B(p,R)$ and $\gamma_{t}(b) \in \partial B(p,\epsilon)$.
  \item $\gamma_{t}(a) \in \partial B(p,\epsilon)$ and $\gamma_{t}(b) \in \partial B_{t}(p,R)$,
  \end{enumerate}
 For $I \in \mathcal{I}_{t}$ of type (1), the fundamental theorem of calculus and the radial symmetry of $G$ show that $E(I) = G(\epsilon)- G(\epsilon) = 0$, by a slight abuse of notation. The same argument shows that $E(I) = 0$ for $I \in \mathcal{I}_{t}$ of type (2). For intervals of type (3) and (4), the $E(I) = G(\epsilon) - G(R)$ and $E(I) = G(R) - G(\epsilon)$, respectively. However, we note that there are only finitely many (possibly zero) intervals of type (3) $\cup$ (4). Indeed, since $\gamma_{t}$ is a Lipschitz curve, the intervals of type (3) $\cup$ (4) have length $\sim R - \epsilon$, and there is only space for finitely many of them in the bounded set $S_{t}$. So, the intervals of type (3) $\cup$ (4) can be ordered, from left to right: $I_{1} < I_{2} < \ldots < I_{N}$. Now, we leave to the reader to check that $I_{1}$ is of type (3), $I_{N}$ is of type (4), and there is an alternating pattern: whenever $I_{j}$ is of type (3), then $I_{j + 1}$ is of type (4). In particular $N \in 2\N$, and 
 \begin{displaymath} \sum_{j = 1}^{N} E(I_{j}) = \sum_{j = 1}^{N/2} \left[ E(I_{j}) + E(I_{j + 1}) \right] = \sum_{j = 1}^{N/2} [G(\epsilon) - G(R) + G(R) - G(\epsilon)] = 0. \end{displaymath}
 This concludes the proof of the lemma.  \end{proof}
The existence of $R^{t}(f\tau)$ is an easy corollary:
\begin{cor}\label{c:pv-Htauf0} Let $\Omega \subset \He$ be a flag domain, and $f \in C^\infty_c(\He)$. Then the principal values 
\begin{displaymath}
     R^{t}(f\tau)(q) = \pv   \int \left< \nabla G(p^{-1} \cdot q), \tau(q) \right> f(q) \, d \sigma(q)
\end{displaymath}
exist for all $p \in \He$. If the hypothesis is relaxed to $f \in L^{p}(\sigma)$ for some $1 < p < \infty$, then the principal values $R^{t}(f\tau)(q)$ still exist for $\sigma$ a.e. $p \in \partial \Omega$.
\end{cor}
\begin{proof} Assume first that $f \in C^{\infty}_{c}(\He)$. Fix $p \in \He$, and let $R > 0$ be so large that $\spt(f) \subset B(p,R)$. Then,
    \begin{align*}
      \pv \int \left< \nabla G(p^{-1} \cdot q), \tau(q) \right> f(q) \, d \sigma(q) & = \int_{A(p,R) } \left< \nabla G(p^{-1} \cdot q) \, \tau(q) \right> (f(q)-f(p)) \, d \sigma(q)\\
      & \quad + f(p) \cdot \pv \int_{A(p, R)} \left< \nabla G(p^{-1} \cdot q), \tau (q) \right> \, d \sigma(q). 
    \end{align*}
   The first integral is absolutely convergent since $f \in C^{\infty}_{c}(\He)$. The second term vanishes by Lemma \ref{l:Ttau-pv}. This completes the case $f \in C^{\infty}_{c}(\He)$. The extension to $f \in L^{p}(\sigma)$ follows by the argument in Lemma \ref{pvL2-K} (since $\vec{R}^{\ast}$ is bounded on $L^{p}(\sigma)$ by Theorem \ref{t:FO1}). \end{proof}
   
\subsection{Existence of $R(f\nu)$ and $R(f\tau)$} Recall that
\begin{displaymath} RV(p) = \pv  \int \langle \nabla G(q^{-1} \cdot p),V(q) \rangle \, d\sigma(q), \end{displaymath} 
whenever the principal values exist. In this section, we show that they do if $\Omega \subset \He$ is a flag domain, $\sigma = |\partial \Omega|_{\He}$, and $V = f\nu$ or $V = f\tau$ with $f \in L^{p}(\sigma)$.

Before the details, a pause: what is the difference between, say, $R(f\nu)$ and $R^{t}(f\nu)$? The key to the existence of $R^{t}(f\nu) = Df$ was Lemma \ref{pv-K}, whose proof was based on combining the divergence theorem and the identity
\begin{displaymath} \mathrm{div}^{\flat}[q \mapsto G(p^{-1} \cdot q)] = \bigtriangleup^{\flat}G(p^{-1} \cdot q) = 0, \qquad p \neq q. \end{displaymath}
Unfortunately,
\begin{displaymath} \mathrm{div}^{\flat}[q \mapsto \nabla G(q^{-1} \cdot p)] \neq \bigtriangleup^{\flat} G(q^{-1} \cdot p), \end{displaymath} 
so the same approach will no longer work. Getting around this problem is where the vertical ruling of flag domains gets used most heavily. The key is Lemma \ref{l:right-left} below.

 We denote by $X^R$ and $Y^R$ the right invariant vector fields
\begin{displaymath}
    X^R = \partial_x + \tfrac{y}{2} \partial_t \quad \text{and} \quad Y^R = \partial_y - \tfrac{x}{2} \partial_t. 
\end{displaymath}
We also write $\nabla^{R}f := (X^{R}f,Y^{R}f)$ for $f \in C^{1}(\He)$. A direct computation reveals the following relationship between the left and right horizontal gradients of $G$: 
\begin{align}\label{nablaG-R-L}
    \nabla G(p^{-1}) = -\nabla^R G(p), \qquad p \in \He \, \setminus \, \{0\}. 
\end{align}

\begin{lemma}\label{l:right-left} Let $\Omega$ be a flag domain, and let $\sigma := |\partial \Omega|_{\He}$. Let $\eta = aX + bY \in L^{\infty}(\sigma)$ be a horizontal vector field with the property that the coefficient functions $a,b \in L^{\infty}(\sigma)$ only depend on the $z$-variable, that is, $a(z,t)=a(z)$ and $b(z,t)=b(z)$. Finally, let $\psi \colon \He \to \R$ be a bounded compactly supported radially symmetric Borel function with $0 \notin \spt \psi$. Then, for all $p \in \He$,
\begin{equation}\label{form161}
    \int \psi(q^{-1} \cdot p)\left< \nabla G(q^{-1} \cdot p),
  \eta(q) \right> \, d\sigma(q) = - \int \psi(q^{-1} \cdot p) \left<
  \nabla G(p^{-1} \cdot q), \eta(q) \right> \, d \sigma(q).
\end{equation}
\end{lemma}
\begin{remark} The relevant examples of "$\eta$" will be the normal and tangent vector fields $\eta = \nu$ and $\eta = \tau$: they satisfy the hypothesis of the lemma by Proposition \ref{prop3}, and the definition of $\tau$, see Definition \ref{HTangent}. The function $\psi$ is typically the characteristic function of an annulus. Note that since $\psi$ is radially symmetric, $\psi(q^{-1} \cdot p) \equiv \psi(p^{-1} \cdot q)$. \end{remark}
\begin{proof}[Proof of Lemma \ref{l:right-left}] Note that
\begin{equation}\label{e:right-left} X - X^R =  - y T \quad \text{and} \quad Y - Y^R = x T, \end{equation}
where $T = \partial_{t}$. Using \eqref{nablaG-R-L} and \eqref{e:right-left}, we see that
\begin{align*}
    \nabla G(q^{-1} \cdot p) & = \left(\nabla G(q^{-1} \cdot p) + \nabla G(p^{-1} \cdot q)\right) - \nabla G(p^{-1} \cdot q)\\
    & = \left(\nabla - \nabla^R \right)G(q^{-1} \cdot p) - \nabla G(p^{-1} \cdot q)\\
    & = \left[ (- (y'-y) T G(q^{-1} \cdot p), (x'-x) TG(q^{-1} \cdot p)) \right] - \nabla G(p^{-1} \cdot q)
\end{align*}
for $p=(x',y',t') $ and $q=(x,y, t)$. Moreover, $TG$ is an odd function (in fact $TG(z,t) = -16t/\|(z,t)\|^{6}$), so $TG(q^{-1} \cdot p) = -TG(p^{-1} \cdot q)$. So, to prove \eqref{form161}, it suffices to verify that
\begin{displaymath} \int \psi(q^{-1} \cdot p)\langle (- (y-y') T G(p^{-1} \cdot q), (x-x') TG(p^{-1} \cdot q)),\eta(q) \rangle \, d\sigma(q) = 0 \end{displaymath} 
for all $p = (x',y',t') \in \He$. With no loss of generality, we only prove this at $p=0$. Since the coefficients of $\eta$ do not depend on the $t$-variable, we may use the "product formula" $\sigma = c\mathcal{H}_{E}^{1}|_{\Gamma} \times \mathcal{L}^{1}$ (recall \eqref{eq:sigmaFubini}) to write
\begin{align*}
     \int \psi(q) \langle \left( -y T G(q), x TG(q) \right), \eta(q) \rangle \, d\sigma(q)
    = \int \left(-y  I(z) , x I(z)\right) \cdot \eta(z) \, d \mathcal{H}_{E}^1(z).
\end{align*}
where $\eta(z) := (a(z),b(z)) \in \R^{2}$, and 
\begin{align} \label{right-left-1}
    I(z) := \int \psi(z,t)TG(z,t) \, dt, \qquad z \in \R^{2}. \end{align}
Since $\psi(z,t) = \psi(z,-t)$ and $TG(z,t) = -TG(z,-t)$, the integrands in \eqref{right-left-1} are odd, and hence $\eqref{right-left-1} \equiv 0$. This completes the proof. \end{proof}

 \begin{cor}\label{c:pv-Htauf} Let $\Omega \subset \He$ be a flag domain, and $f \in C^\infty_c(\He)$. Then the principal values $R(f\nu)(p)$ and $R(f\tau)(p)$ exist at all weak tangent points $p \in \partial \Omega$. If the hypothesis is relaxed to $f \in L^{p}(\sigma)$, then the conclusion still holds for $\sigma$ a.e. $p \in \partial \Omega$.
\end{cor}
\begin{proof} Fix $p \in \He$, and let $R > 0$ be so large that $\spt(f) \subset B(p,R)$. Let $\eta \in \{\nu,\tau\}$. Then,
    \begin{displaymath} R(f\nu)(p) = \int_{B(p,R) } \left< \nabla G(q^{-1} \cdot p), \eta(q) \right> (f(q)-f(p)) \, d \sigma(q) + f(p) \cdot R(\mathbf{1}_{B(p,R)}\eta)(p). \end{displaymath} 
   The first integral is absolutely convergent since $f \in C^{\infty}_{c}(\He)$. By Lemma \ref{l:right-left}, the second term equals $-f(p) \cdot R^{t}(\mathbf{1}_{B(p,R)}\eta)(p)$. Furthermore,
\begin{itemize}
   \item $R^{t}(\mathbf{1}_{B(p,R)}\nu)(p)$ exists by Lemma \ref{pv-K} whenever $p \in \partial \Omega$ is a weak tangent point,
   \item $R^{t}(\mathbf{1}_{B(p,R)}\tau)(p) = 0$ for all $p \in \He$ by Lemma \ref{l:Ttau-pv}. 
\end{itemize}
This completes the case $f \in C^{\infty}_{c}(\He)$. The extension to $f \in L^{p}(\sigma)$ follows as in the proof of Lemma \ref{pvL2-K}. \end{proof}

The proof of Theorem \ref{t:pValues} is now complete, recalling the discussion in Remark \ref{rem:strategy}.

\subsection{The choice of truncations is irrelevant} Later on, it will be convenient to know that the principal values $RV$ and $R^{t}V$, whose existence we have now established in Theorem \ref{t:pValues}, can also be defined via smooth truncations of the kernels $\nabla G(q^{-1} \cdot p)$ and $\nabla G(p^{-1} \cdot q)$. The following proposition shows that this can be done very generally, even so that the values of $RV(p)$ and $R^{t}V(p)$ remain unchanged (whenever they exist).

\begin{proposition}\label{prop10} Let $(X,d,\mu)$ be a metric measure space such that $\mu(B(x,r)) \leq Cr^{n}$ for all $x \in X$, $r > 0$, for some $n > 0$. Let $K \colon X \times X \, \setminus \, \{x = y\} \to \C$ be a kernel satisfying $|K(x,y)| \leq Cd(x,y)^{-n}$. Let $\phi \colon [0,\infty) \to [0,1]$ be a non-decreasing function with $0 \notin \spt \phi$ and $\lim_{s \to \infty}\phi(s) = 1$. Assume that $f \in L^{p}(\mu)$ for some $1 < p < \infty$, and 
\begin{displaymath} \pv \, Kf(x) := \lim_{\epsilon \to 0} \int_{\{y \in X : d(x,y) \geq \epsilon\}} K(x,y)f(y) \, d\mu(y) \in \C \end{displaymath}
exists at some point $x \in X$. Then, writing $\phi_{\epsilon}(s) = \phi(s/\epsilon)$ for $\epsilon > 0$, we have
\begin{displaymath} \pv \, Kf(x) = \lim_{\epsilon \to 0} \int \phi_{\epsilon}(d(x,y))K(x,y)f(y) \, d\mu(y). \end{displaymath} \end{proposition}

\begin{proof} Define the probability measure $\tn := \tn_{\phi}$ on $[0,\infty)$ by the formula
\begin{displaymath} \tn[r,s] := \phi(s) - \phi(r), \qquad 0 \leq r \leq s < \infty, \end{displaymath}
and note that $0 \notin \spt \tn$, since $0 \notin \spt \phi$. Moreover, for all $\epsilon > 0$ and $s \geq 0$, we have
\begin{displaymath} \phi_{\epsilon}(s) = \tn[0,\tfrac{s}{\epsilon}] = \int_{0}^{\infty} \mathbf{1}_{[0,s/\epsilon]}(r) \, d\tn(r) = \int_{0}^{\infty} \mathbf{1}_{[\epsilon r,\infty)}(s) \, d\tn(r). \end{displaymath}
Consequently,
\begin{equation}\label{form163} \lim_{\epsilon \to 0} \int \phi_{\epsilon}(d(x,y)) K(x,y)f(y) \, d\mu(y) = \lim_{\epsilon \to 0} \int_{0}^{\infty}  \int_{\{d(x,y) \geq \epsilon r\}} K(x,y)f(y) \, d\mu(y) \, d\tn(r). \end{equation} 
The use of Fubini's theorem (for $\epsilon > 0$ fixed) was justified, since
\begin{equation}\label{form164} \int_{\{y \in X : d(x,y) \geq \epsilon r\}} |K(x,y)||f(y)| \, d\mu(y) \lesssim_{C,\epsilon,p,\phi} \|f\|_{L^{p}(\mu)}, \qquad r \in \spt \tn, \end{equation} 
applying H\"older's inequality, the growth and decay estimates for $\mu$ and $K$, and $0 \notin \spt \tn$. On the RHS of \eqref{form163}, we observe that
\begin{displaymath} \lim_{\epsilon \to 0} \int_{\{d(x,y) \geq \epsilon r\}} K(x,y)f(y) \, d\mu(y) = \pv Kf(p) \in \C  \end{displaymath} 
for all $r \in \spt \tn$, so in particular
\begin{equation}\label{form165} r \mapsto \sup_{\epsilon > 0} \left| \int_{\{d(x,y) \geq \epsilon r\}} K(x,y)f(y) \, d\mu(y) \right| = \sup_{\epsilon > 0} \left| \int_{\{d(x,y) \geq \epsilon\}} K(x,y)f(y) \, d\mu(y) \right| \end{equation}
is a (finite) constant function on $\spt \tn$ (we also use here the estimate \eqref{form164}, which shows that the only possible singularity of the expression in \eqref{form165} occurs when $\epsilon \to 0$). This justifies applying the dominated convergence theorem to the RHS of \eqref{form163}:
\begin{displaymath} \lim_{\epsilon \to 0} \int \phi_{\epsilon}(d(x,y))K(x,y)f(y) \, d\mu(y) = \int_{0}^{\infty} \pv Kf(x) \, d\tn(r) = \pv Kf(x), \end{displaymath}
recalling that $\tn$ is a probability measure. The proof is complete. \end{proof}

\section{Various maximal functions}

\subsection{Heisenberg cones and intrinsic Lipschitz graphs}\label{s:ILG} Before introducing any maximal functions, we define \textit{Heisenberg cones}. There are a number of roughly equivalent definitions in the literature, and we prefer the following as a starting point. Let $\mathbb{L} \subset \{(x,y,t) : t = 0\}$ be a line through the origin in the $xy$-plane (also known as a \emph{horizontal subgroup}), and let $\pi_{\mathbb{L}} \colon \He \cong \R^{3} \to \mathbb{L}$ be the orthogonal projection. The (open) \emph{Heisenberg cone at the origin with parameter $\lambda \in (0,1)$ and axis $\mathbb{L}$} is
\begin{align*}
    C_{\Le}(\lambda):= \{ p \in \He : \|\pi_{\mathbb{L}} (p)\| > \lambda \|p\| \}.
\end{align*}
These cones get smaller as $\lambda \nearrow 1$, and their intersection, over $\lambda \in (0,1)$, is the punctured axis $\mathbb{L} \, \setminus \, \{0\}$. An \emph{intrinsic $\lambda$-Lipschitz graph with axis $\mathbb{L}$} is a set $\Gamma \subset \He$ with the property that 
\begin{displaymath} p \cdot C_{\Le}(\lambda) \cap \Gamma = \emptyset, \qquad p \in \Gamma. \end{displaymath}
Boundaries of flag domains $\Omega = \{(x,y,t) : x < A(y)\} \subset \He$ are intrinsic $\lambda$-Lipschitz graphs with axis $\mathbb{L} = \{(x,0,0) : x \in \R\}$, for some parameter $\lambda > 0$ depending only on the Lipschitz constant of $A$. For much more information on intrinsic Lipschitz graphs, see for example \cite{2018arXiv180304819F,MR3511465,MR2287539,2019arXiv190406904R}.

On some occasions, it is useful to take the following definition of cones as a starting point:
\begin{displaymath} \widetilde{C}_{\Le}(\lambda) := \bigcup_{v \in \Le \, \setminus \, \{0\}} B(v,\lambda \|v\|). \end{displaymath} 
Rigot \cite[Proposition 3.9]{2019arXiv190406904R} has shown that the two notions of cones are equivalent: for every $\lambda > 0$ there exist $\lambda_{1},\lambda_{2} > 0$ such that
\begin{equation}\label{form168} C_{\mathbb{L}}(\lambda_{1}) \subset \widetilde{C}_{\mathbb{L}}(\lambda) \subset C_{\mathbb{L}}(\lambda_{2}) \end{equation}
for all horizontal subgroups $\mathbb{L} \subset \He$ (simultaneously). These inclusions lead to the following simple geometric observation:
\begin{lemma}\label{lemma12} Let $\Gamma \subset \He$ be an intrinsic Lipschitz graph with axis $\mathbb{L}$. Then, there exist a constants $\lambda,\theta \in (0,1)$ such that the following holds. If $p \in \Gamma$ and $q \in p \cdot C_{\mathbb{L}}(\lambda)$, then $\dist(q,\Gamma) > \theta d(p,q)$. 
\end{lemma}

\begin{proof} By definition of $\Gamma$ being an intrinsic Lipschitz graph, and the inclusions \eqref{form168}, we may find $\lambda,\lambda_{1} \in (0,1)$ such that $C_{\mathbb{L}}(\lambda) \subset \widetilde{C}_{\mathbb{L}}(\lambda_{1}/2)$, and
\begin{displaymath} p \cdot \widetilde{C}_{\mathbb{L}}(\lambda_{1}) \cap \Gamma = \emptyset, \qquad p \in \Gamma. \end{displaymath}
Let $q \in p \cdot C_{\mathbb{L}}(\lambda) \subset p \cdot \widetilde{C}_{\mathbb{L}}(\lambda_{1}/2)$, so $q \in B(p \cdot v,\tfrac{\lambda_{1}}{2}\|v\|)$ for some $v \in \mathbb{L}$, and 
\begin{displaymath} d(p,q) \leq d(p,p \cdot v) + d(p \cdot v,q) \leq (1 + \tfrac{\lambda_{1}}{2})\|v\| \leq 2\|v\|. \end{displaymath}
It follows that
\begin{displaymath} B(q,\tfrac{\lambda_{1}}{4}d(p,q)) \subset B(q,\tfrac{\lambda_{1}}{2}\|v\|) \subset B(p \cdot v,\lambda_{1}\|v\|) \subset p \cdot \widetilde{C}_{\mathbb{L}}(\lambda_{1}) \subset \Gamma^{c}, \end{displaymath}
which means that $\dist(q,\Gamma) \geq \tfrac{\lambda_{1}}{4}d(p,q)$. So, the claim holds with $\theta = \lambda_{1}/4$. \end{proof}

\subsection{Non-tangential and radial limits and maximal functions}\label{s:nt}

Let $\Gamma$ be a set in a metric space $(X,d)$, typically the boundary of a domain in $\He$. For $p \in \Gamma$ and $\theta \in (0,1)$, we define the \emph{non-tangential approach region}
\begin{displaymath} V_{\Gamma}(p,\theta) := \{q \in X : \dist(q,\Gamma) > \theta d(p,q)\} \subset X \, \setminus \, \Gamma. \end{displaymath}
\begin{definition}[Non-tangential maximal function] Let $\Gamma \subset X$ be a set, and let $u \colon X \, \setminus \, \Gamma \to B$ be a function taking values in some normed space $(B,|\cdot|)$ (typically $\R$ or $\R^{2}$). For $p \in \Gamma$ and $\theta > 0$, we define the quantity
\begin{displaymath} \calN_{\theta}  u(p) := \sup \{|u(q)| : q \in V_{\Gamma}(p,\theta)\}. \end{displaymath}
If $u$ is only defined on a subset of $X \, \setminus \, \Gamma$, we keep the same notation, but the "$\sup$" is only taken over $p \in \mathrm{dom}(u) \cap V_{\Gamma}(p,\theta)$. The function $\mathcal{N}_{\theta}u$ is the \emph{non-tangential maximal function of $u$ with aperture $\theta > 0$.} \end{definition}

\begin{definition}[Non-tangential limit] Let $\Gamma \subset X$ and $B$ be as above, and let $u \colon X \, \setminus \, \Gamma \to B$ be a function. Let $F \subset X \, \setminus \, \Gamma$ be a set, and let $b \in B$. We say that $u$ has \emph{non-tangential limit $b$ at $p_{0} \in \Gamma$ along $F$} if
\begin{equation}\label{form170} \mathop{\lim_{p \to p_{0}}}_{p \in F \cap V_{\Gamma}(p_{0},\theta)} u(p) = b \qquad \text{for all } \theta \in (0,1). \end{equation}
\end{definition}
We do not require that $p_{0}$ lies in the closure of $F \cap V_{\Gamma}(p_{0},\theta)$ for any $\theta \in (0,1)$. If this fails, then every $b \in B$ is a non-tangential limit of $u$ at $p_{0}$ along $F$. In practice, however, we are only interested in non-tangential limits in a scenario where $p_{0}$ lies in the closure of $F \cap V_{\Gamma}(p_{0},\theta_{0})$ for some $\theta_{0} \in (0,1)$, equivalently for all $0 < \theta \leq \theta_{0}$. In this situation, the non-tangential limit is unique (if it exists). We do not introduce separate notation for non-tangential limits, but we will always state explicitly (if not clear from the context) whether a certain limit should be understood in the non-tangential or classical sense.

If $\Gamma \subset \He$ happens to be an intrinsic $\lambda_{0}$-Lipschitz graph with axis $\mathbb{L}$, for some $\lambda_{0} \in (0,1)$, and $u$ is a $B$-valued function defined on (a subset of) $\He \, \setminus \, \Gamma$, we also define the \emph{conical maximal function with aperture $\lambda \in [\lambda_{0},1)$},
\begin{displaymath} \mathcal{C}_{\lambda}u(p) := \sup\{|u(q)| : q \in p \cdot C_{\mathbb{L}}(\lambda)\},  \end{displaymath} 
and the \emph{radial maximal function}
\begin{displaymath} \mathcal{N}_{\mathrm{rad}}u(p) := \sup\{|u(p \cdot v)| : v \in \mathbb{L} \, \setminus \, \{0\}\}. \end{displaymath}
Note that the "$\sup$" in the definition of $\mathcal{C}_{\lambda}$ is only taken over points $p \in \He \, \setminus \, \Gamma$ by the intrinsic Lipschitz assumption. Evidently
\begin{displaymath} \mathcal{N}_{\mathrm{rad}}u(p) \leq \mathcal{C}_{\lambda}u(p), \qquad p \in \Gamma, \, \lambda_{0} \leq \lambda < 1. \end{displaymath}
Furthermore, Lemma \ref{lemma12} implies that if $\Gamma \subset \He$ is an intrinsic Lipschitz graph, then there exist $\lambda \in (0,1)$ and $\theta > 0$ such that
\begin{equation}\label{radVsNt} \mathcal{N}_{\mathrm{rad}}u(p) \leq \mathcal{C}_{\lambda}u(p) \leq \mathcal{N}_{\theta}u(p), \qquad p \in \Gamma. \end{equation}
In other words, the radial maximal function is trivially dominated by the conical maximal function, and the conical maximal function is further dominated by the non-tangential maximal function if "$\lambda$" is close enough to $1$ and "$\theta$" is close enough to "$0$". 

We record the following standard Cotlar-type inequality:

\begin{lemma}\label{lemma-cotlar} Let $\sigma$ be an $k$-regular Borel measure on $(X,d)$, $k > 0$, and write $\Gamma := \spt \sigma$. Let $K \colon X \times X \, \setminus \, \{p = q\} \to \R^{n}$, $n \geq 1$, be a Borel function satisfying
\begin{equation}\label{skEstimates} |K(p,q)| \leq Cd(p,q)^{-k} \quad \text{and} \quad |K(p,q) - K(p',q)| \leq \frac{Cd(p,p')^{\eta}}{d(p,q)^{k + \eta}} \end{equation}
for all $p,p',q \in X$ with $d(p,p') \leq d(p,q)/2$, and for some constants $C \geq 1$ and $\eta > 0$. For $f \in L^{p}(\sigma)$, $1 < p < \infty$, define
\begin{displaymath} \mathcal{K}f(p) := \int K(p,q)f(q) \, d\sigma(q), \qquad p \in X \, \setminus \, \Gamma. \end{displaymath}
The integral is absolutely convergent for $p \notin \Gamma$. For $p \in \Gamma$, define also
\begin{displaymath} K^{\ast}f(p) := \sup_{\epsilon > 0} \left|\int_{\{q \in X : d(p,q) > \epsilon\}} K(p,q) f(q) \, d\sigma(q) \right|. \end{displaymath}
Then, for any $\theta > 0$, there exists a constant $A = A(K,\sigma,\theta) \geq 1$ such that
\begin{equation}\label{e:NT-cotlar} \mathcal{N}_{\theta}\mathcal{K}f(p_{0}) \leq A M_{\sigma}f(p_{0}) + K^{\ast}f(p_{0}), \qquad p_{0} \in \Gamma, \end{equation}
where $\calM_\sigma f$ is the centred Hardy-Littlewood maximal function of $f$, relative to the measure $\sigma$.
\end{lemma}

%Let $L$ be the Lipschitz constant of the function $\phi$ defining the flag domain $\Omega$. Then for any $\alpha< L^{-1}$, we can find a small $\delta>0$ so that for all $p_0 \in \partial \Omega$ we have
%\begin{align}\label{e:NT-cotlar}
%    \calN_\alpha^\delta \calK f (p_0) := \sup_{p \in p_0\cdot\calC \cap B(p_0, \delta)} |K f(p)|  \lesssim \calM_\sigma f(p_0) + (R^t)^* (\nu f) (p_0).
%\end{align}
%where $\calM_\sigma$ is the Hardy-Littlewood maximal function and $(R^t)^*$ is the maximal operator of the adjoint Riesz transform. 

\begin{proof} Fix $p \in V_{\Gamma}(p_{0},\theta)$, and set $r := d(p,p_{0})$. We start by writing
\begin{equation}\label{MaxNT10}
    |\mathcal{K} f(p)| \leq \int_{B(p_{0}, 2r)} |K(p,q)||f(q)| \, d \sigma(q) + \left| \int_{B(p_{0}, 2r)^{c}} K(p,q)f(q) \, d \sigma(q) \right| =: I_1 + I_2. \end{equation}
First, we control $I_1$ by the maximal function. Note that $d(p, q) \geq \dist(p,\Gamma) > \theta r$ for all $q \in \Gamma$, hence $|K(p,q)| \leq Cd(p, q)^{-3} \leq C\theta^{-3}r^{-3}$. We therefore obtain
\begin{align*}
    I_1 \lesssim_{K,\theta} \frac{1}{r^{3}} \int_{B(p_{0}, 2r)} |f(q)| \, d \sigma(q) \lesssim_{\sigma} \calM_\sigma f (0).
\end{align*}
We next consider $I_2$. First, observe that
\begin{displaymath}
    I_2 \leq \left|\int_{B(p_{0}, 2r)^{c}} [K(p,q) - K(p_{0},q)] f(q) \, d \sigma(q)\right| + K^{\ast}f(p_{0}). \end{displaymath}
So, it remains to estimate the first term on the RHS, denoted $I_{2}^{2}$. Evidently $d(p,p_{0}) = r \leq d(p_{0},q)/2$ for all $q \in B(p_{0},2r)^{c}$, so the H\"older estimate in \eqref{skEstimates} is applicable:
\begin{displaymath} I_{2}^{2} \leq C \sum_{j \geq 1} \int_{A(p_{0},2^{j}r,2^{j + 1}r)} \frac{r^{\eta}}{d(p_{0},q)^{3 + \eta}}f(q) \, d\sigma(q) \lesssim_{\sigma} M_{\sigma}f(p_{0}) \sum_{j \geq 1} \frac{r^{\eta}}{(2^{j}r)^{\eta}} \sim_{\eta} M_{\sigma}f(p_{0}). \end{displaymath}
Here $A(x,r,R) = B(x,R) \, \setminus \, \bar{B}(x,r)$. This concludes the proof of \eqref{e:NT-cotlar}. \end{proof}
We apply the lemma to the vectorial Riesz transform and its adjoint, namely
\begin{displaymath} \vec{\mathcal{R}}f(p) =  \int \nabla G(q^{-1} \cdot p)f(q) \, d\sigma(q) \quad \text{and} \quad \vec{\mathcal{R}}^{t}V(p) =  \int \nabla G(p^{-1} \cdot q)f(q) \, d\sigma(q). \end{displaymath} 
%where $\sigma$ is the horizontal perimeter measure of a flag domain. Note that the integrals above are absolutely convergent for all $f \in L^{p}(\sigma)$, $1 < p < \infty$, and for all $p \in \He \, \setminus \, \partial \Omega$. 

\begin{cor}\label{K-MaxNT} Let $\Omega \subset \He$ be a flag domain, $\sigma := |\partial \Omega|_{\He}$. Then, the operators 
\begin{displaymath} f \mapsto \mathcal{N}_{\theta}\vec{\mathcal{R}}f, \, f \mapsto \mathcal{N}_{\theta}\vec{\mathcal{R}}^{t}f, \, f \mapsto \mathcal{N}_{\mathrm{rad}}\vec{\mathcal{R}}f, \quad \text{and} \quad f \mapsto \mathcal{N}_{\mathrm{rad}}\vec{\mathcal{R}}^{t}f, \end{displaymath}
and $f \mapsto \mathcal{N}_{\theta}(\nabla \mathcal{S}f)$, and $f \mapsto \mathcal{N}_{\theta}(\mathcal{D} f)$, are bounded on $L^{p}(\sigma)$ for all $\theta \in (0,1)$ and $1 < p < \infty$.
\end{cor}
\begin{proof} The kernels $K(p,q) = \nabla G(q^{-1} \cdot p)$ and $K^{t}(p,q) = \nabla G(p^{-1} \cdot q)$ satisfy the assumptions of Lemma \ref{lemma-cotlar} with $k = 3$ and $\eta = \tfrac{1}{2}$ by \cite[Lemma 2.1]{CFO2}. Therefore, we infer from \eqref{e:NT-cotlar} that
\begin{displaymath} \mathcal{N}_{\theta}\vec{\mathcal{R}}f(p_{0}) \lesssim_{\sigma,\theta} M_{\sigma}f(p_{0}) + \vec{R}^{\ast}f(p_{0}), \qquad p_{0} \in \partial \Omega, \end{displaymath}
and a similar inequality holds for $\mathcal{N}_{\theta}\vec{\mathcal{R}}^{t}f$. Here $\vec{R}^{\ast}$ is the maximal vectorial Riesz transform, introduced in Theorem \ref{t:FO1}, which we know to be bounded on $L^{p}(\sigma)$. Since also $M_{\sigma}$ is bounded on $L^{p}(\sigma)$, the proof of is complete (recalling also inequality \eqref{radVsNt} between radial and non-tangential maximal functions associated with intrinsic Lipschitz graphs.) The operator $f \mapsto \mathcal{N}_{\theta}(\nabla \mathcal{S}f)$ is just another way of writing $f \mapsto \mathcal{N}_{\theta}(\vec{\mathcal{R}}f)$, and the boundedness of $f \mapsto \mathcal{N}_{\theta}\mathcal{D}f = \mathcal{N}_{\theta}\mathcal{R}^{t}(f\nu)$ follows from the boundedness of $f \mapsto \mathcal{N}_{\theta}\vec{\mathcal{R}}^{t}f$. \end{proof}

\section{Boundary behaviour of the Riesz transform}\label{s:jump}

In this section examine the non-tangential limits as $p \to p_{0} \in \partial \Omega$ of the operators
 \begin{displaymath} \mathcal{R}V(p) =  \int \langle \nabla G(q^{-1} \cdot p),V(q) \rangle \, d\sigma(q) \quad \text{and} \quad \mathcal{R}^{t}V(p) =  \int \langle \nabla G(p^{-1} \cdot q),V(q) \rangle \, d\sigma(q), \end{displaymath} 
when $\Omega \subset \He$ is a flag domain. The integrals above are absolutely convergent for all horizontal vector fields $V \in L^{p}(\sigma)$, $1 < p < \infty$, and $p \in \He \, \setminus \, \partial \Omega$. The main result is Theorem \ref{t:rJumps} below, but as useful consequences, we obtain variants of the classical \emph{jump relations} for the double layer potential and the normal derivative of the single layer potential. These consequences are collected in Corollary \ref{c:allJumps}.

\begin{thm}\label{t:rJumps} Let $\Omega \subset \He$ be a flag domain, $\sigma := |\partial \Omega|_{\He}$, let $1 < p < \infty$, and let $V = f\nu + g\tau \in L^{p}(\sigma)$ be a horizontal vector field. Then, for $\sigma$-almost all $p \in \partial \Omega$, the non-tangential limits of $\mathcal{R}V$ and $\mathcal{R}^{t}V$ along $\Omega$ and $\He \, \setminus \, \overline{\Omega}$ exist. Their values are given by
\begin{equation}\label{form166}
    \mathop{\lim_{q \to p}}_{q \in \Omega} \mathcal{R}V(q) = -\tfrac{1}{2}f(p) + RV(p) \quad \text{and} \quad \mathop{\lim_{q \to p}}_{q \in \He \, \setminus \, \overline{\Omega}} \mathcal{R}V(q) = \tfrac{1}{2}f(p) + RV(p), \end{equation}
    and
    \begin{equation}\label{form166a}
    \mathop{\lim_{q \to p}}_{q \in \Omega} \mathcal{R}^{t}V(q) = \tfrac{1}{2}f(p) + R^{t}V(p) \quad \text{and} \quad \mathop{\lim_{q \to p}}_{q \in \He \, \setminus \, \overline{\Omega}} \mathcal{R}^{t}V(q) = -\tfrac{1}{2}f(p) + R^{t}V(p). \end{equation}
\end{thm}

It is worth emphasising the asymmetric role of the functions $f$ and $g$ in the relations \eqref{form166}-\eqref{form166a}. This will be important in Corollary \ref{c:allJumps}.  We start by treating a special case:

\begin{lemma}\label{jump-K-smooth}
Theorem \ref{t:rJumps} holds for $V = f\nu + g\tau$, where $f,g \in C_c^\infty (\He)$. In fact, then the non-tangential limits \eqref{form166} exist at all weak tangent points $p \in \partial \Omega$.
\end{lemma}
\begin{proof}
Let $p \in \partial \Omega$ be a weak tangent point of $\partial \Omega$, so all the principal values $R^{t}(f\nu)(p)$, $R^{t}(g\tau)(p)$ and $R(f\nu)(p)$, and $R(g\tau)(p)$ exist by Corollaries \ref{c:pv-R-smooth}, \ref{c:pv-Htauf0}, and \ref{c:pv-Htauf}. The plan of the proof is to first treat $R^{t}(f\nu)$ and $R(f\nu)$, which are a little more difficult, and then $R^{t}(g\tau)$ and $R(g\tau)$. 

Fix $\theta > 0$, and consider a sequence  $\{p_j\}_{j \in \N} \subset \Omega \cap V_{\partial \Omega}(p,\theta)$ such that $p_j \to p$ as $j\to \infty$. (We will deal with the "exterior" case $\{p_{j}\}_{j \in \N} \subset \He \, \setminus \, \overline{\Omega}$ later). For each $j \in \N$, there exists $\epsilon_{j} \in (0,1)$ such that $B(p_j, \ve_{j}) \subset \Omega$. Moreover, there exists $R \geq 1$ such that $\spt(f) \subset B(p_{j}, R)$ for all $j \in \N$. Now set
\begin{align*}
   E_{j} = \left( \Omega \cap B(p_{j}, R) \right) \, \setminus \, B(p_j, \ve_{j}), \qquad j \in \N, 
\end{align*}
so
\begin{displaymath} \partial E_{j} = [\partial \Omega \cap B(p_{j},R)] \cup [\Omega \cap \partial B(p_{j},R)] \cup \partial B(p_{j},\epsilon_{j}). \end{displaymath}
The perimeter measure $|\partial E_{j}|_{\He}$ restricted to these three pieces equals $\sigma$, $\sigma_{p_{j},R}$, and $\sigma_{p_{j},\epsilon_{j}}$, respectively, using the locality of horizontal perimeter (see \cite{MR2678937}), as we already did in the proof of Lemma \ref{pv-K}. Therefore, using that $q \mapsto G(p_j^{-1} \cdot q)$ is $\bigtriangleup^{\flat}$-harmonic in $E_{j}$, and $q \mapsto \nabla G (p^{-1}_j \cdot q) \in C^{\infty}(\bar{E}_{j})$, we infer from the divergence theorem that
\begin{align*}
   0 = - \int_{E_{j}} \Lap^{\flat} G(p^{-1}_j \cdot q) \, d q &  = \int_{\partial E_{j}} \left< \nabla G(p_j^{-1} \cdot q), \nu_{E_{j}}(q) \right> d \sigma_{E_{j}}(q)\\
   & = \int_{\partial \Omega \cap B(p_{j},R)} \langle G(p_{j}^{-1} \cdot q),\nu(q) \rangle \, d\sigma(q)\\
   & \quad + \int_{\Omega \cap \partial B(p_{j},R)} \langle \nabla G(p_{j}^{-1} \cdot q),\nu_{p_{j},R}(q) \rangle \, d\sigma_{p,R}(q)\\
   & \quad - \int_{\partial B(p_{j},\epsilon_{j})} \langle \nabla G(p_{j}^{-1} \cdot q),\nu_{p_{j},\epsilon_{j}}(q) \rangle \, d\sigma_{p_{j},\epsilon_{j}}(q).
\end{align*}
Rearranging terms, and recalling from \eqref{kappa} that the last term above equals "$1$", we have
\begin{align}\label{jump-10}
     & \int_{\partial \Omega \cap B(p_{j}, R)} \langle \nabla G(p_j^{-1} \cdot q),\nu(q) \rangle \, d\sigma(q) \notag\\
    & =  1 -  \int_{\Omega \cap \partial B(p_{j},R)} \langle \nabla G(p_{j}^{-1} \cdot q),\nu_{p_{j},R}(q) \rangle \, d\sigma_{p,R}(q). 
\end{align}
Since $d(p_{j},q) \geq R/2$ for all $q \in \partial B(p,R)$ and $j \in \N$ large enough, we may let $j \to \infty$ on the RHS of \eqref{jump-10} to deduce that
\begin{equation}\label{form167} \lim_{j \to \infty}  \int_{\partial \Omega \cap B(p_{j},R)} \langle \nabla G(p_{j}^{-1} \cdot q),\nu(q) \rangle \, d\sigma(q) = 1 - \mathfrak{r}(\Omega,R,p), \end{equation}
where the constant $\mathfrak{r}(\Omega,R,p) \leq 1$ was defined in \eqref{K-pv10}. Recalling that $\Omega$ is a flag domain, we also infer from \eqref{form167} and Lemma \ref{l:right-left} (with $\psi := \mathbf{1}_{B(0,R)}$) that
\begin{equation}\label{form167a} \lim_{j \to \infty}  \int_{\partial \Omega \cap B(p_{j},R)} \langle \nabla G(q^{-1} \cdot p_{j}),\nu(q) \rangle \, d\sigma(q) = \mathfrak{r}(\Omega,R,p) - 1. \end{equation}
(The assumption $0 \notin \spt \psi$ in that lemma can be omitted here, since $p_{j} \notin \partial \Omega$.) Now we are prepared to compute the limits of $\mathcal{R}(f\nu)(p_{j})$ and $\mathcal{R}^{t}(f\nu)(p_{j})$ as $j \to \infty$. The proofs are virtually the same, so we only give full details for $\mathcal{R}^{t}(f\nu)(p_{j})$. Recall that $\spt(f) \subset B(p_{j},R)$ for all $j \in \N$. For $j \in \N$ fixed, we may therefore write
\begin{align*}
    \mathcal{R}^{t}(f\nu)(p_j) & =  \int_{B(p_{j}, R)} \langle\nabla G(p_j^{-1} \cdot q),\nu(q) \rangle \left( f(q) - f(p) \right) \, d\sigma(q) \\
    & \quad + f(p) \cdot  \int_{B(p_{j}, R)} \langle \nabla G(p_j^{-1} \cdot q), \nu(q) \rangle \, d\sigma(q) = I_1(p_j) + I_2(p_j). 
\end{align*}
Since $f \in C^{\infty}_{c}(\He)$, and
\begin{displaymath} d(p,q) \leq d(p,p_{j}) + d(p_{j},q) \leq \theta^{-1} \dist(p_{j},\partial \Omega) + d(p_{j},q) \leq (1 + \theta^{-1})d(p_{j},q) \end{displaymath}
for $p_{j} \in V_{\partial \Omega}(p,\theta)$, we have 
\begin{displaymath} |\nabla G(p_{j}^{-1} \cdot q)||f(q) - f(p)| \lesssim_{\theta} d(p,q)^{-3}|f(p) - f(q)| \lesssim d(p,q)^{-2}. \end{displaymath}
The RHS is $\sigma$-integrable for $q \in B(p_{j},R) \subset B(p,2R)$, so dominated convergence applies:
\begin{align}\label{form171}
   \lim_{j\to \infty} I_1(p_j) &  =   \int_{B(p,R)} \left< \nabla G(p^{-1} \cdot q), \nu(q) \right>  ( f(q) - f(p)) \, d \sigma(q) \\
   & = R^{t}(f\nu)(p) - f(p)\left( \tfrac{1}{2} - \mathfrak{r}(\Omega, R, p)\right). \notag
\end{align}
We used Lemma \ref{pv-K} in the last equation.
On the other hand, \eqref{form167} gives 
\begin{align*}
    \lim_{j \to \infty} I_2(p_j) = f(p)\left(1 - \mathfrak{r}(\Omega, R, p)\right).
\end{align*}
All in all we obtain 
\begin{displaymath}
    \lim_{j \to \infty} \mathcal{R}^{t}(f\nu)(p_j) = R^{t}(f\nu)(p) - f(p)\left(\tfrac{1}{2} - \mathfrak{r}(\Omega, R, p)\right) + f(p)(1-\mathfrak{r}(\Omega, R, p)) = \tfrac{1}{2}f(p) + R^{t}(f\nu)(p),
\end{displaymath}
as desired. Regarding $\mathcal{R}(f\nu)(p_{j})$, one needs to replace "$\nabla G(p_{j}^{-1} \cdot q)$" by "$\nabla G(q^{-1} \cdot p_{j})$" in the definitions of $I_{1}(p_{j})$ and $I_{2}(p_{j})$. For $I_{1}(p_{j})$, dominated convergence can be applied for the same reasons as above, and this time
\begin{displaymath} \lim_{j \to \infty} I_{1}(p_{j}) = R(f\nu)(p) - f(p) \cdot \pv \int_{B(p,R)} \langle \nabla G(q^{-1} \cdot p),\nu(q) \rangle \, d\sigma(q). \end{displaymath}
Recalling that $\Omega$ is a flag domain, one may then apply Lemma \ref{l:right-left} to turn "$q^{-1} \cdot p$" into "$p^{-1} \cdot q$" in the latter principal value, at the cost of changing the sign. After this one finds from Lemma \ref{pv-K} that
\begin{displaymath} \lim_{j \to \infty} I_{1}(p_{j}) = R(f\nu)(p) - f(p) \cdot \left(\mathfrak{r}(\Omega,R,p) - \tfrac{1}{2} \right). \end{displaymath}
The limit $\lim_{j \to \infty} I_{2}(p_{j})$ was already computed in \eqref{form167a}, and all in all we find that
\begin{displaymath} \lim_{j \to \infty} \mathcal{R}(f\nu)(p_{j}) = R(f\nu)(p) - f(p)\left(\mathfrak{r}(\Omega,R,p) - \tfrac{1}{2} \right) + f(p)(\mathfrak{r}(\Omega,R,p) - 1) = -\tfrac{1}{2}f(p) + R(f\nu)(p), \end{displaymath}
as desired.

We have now dealt with the "interior approach" for vector fields of the form $V = f\nu$. Before turning to $V = g\tau$, let us indicate the small changes needed for the "exterior approach". Let $\{p_{j}\}_{j \in \N} \subset \He \, \setminus \, \overline{\Omega} \cap V_{\partial \Omega}(p,\theta)$ be a sequence with $p_{j} \to p$ as $j \to \infty$. This time one defines $E_{j} := [B(p_{j},R) \, \setminus \, \overline{\Omega}] \, \setminus \, B(p_{j},\epsilon_{j}) \subset \He \, \setminus \, \overline{\Omega}$. This has the effect that
\begin{displaymath} \nu_{E_{j}}|_{\partial \Omega} = -\nu, \end{displaymath}
so \eqref{form167}-\eqref{form167a} are consequently replaced by
\begin{displaymath} \lim_{j \to \infty}  \int_{\partial \Omega \cap B(p_{j},R)} \langle \nabla G(p_{j}^{-1} \cdot q),\nu(q) \rangle \, d\sigma(q) = -1 + \mathfrak{r}(\He \, \setminus \, \overline{\Omega},R,p) = -\mathfrak{r}(\Omega,R,p) \end{displaymath}
and
\begin{displaymath} \lim_{j \to \infty}  \int_{\partial \Omega \cap B(p_{j},R)} \langle \nabla G(q^{-1} \cdot p_{j}),\nu(q) \rangle \, d\sigma(q) = \mathfrak{r}(\Omega,R,p). \end{displaymath}
The splitting of $\mathcal{D}f(p_{j})$ to the parts $I_{1}(p_{j})$ and $I_{2}(p_{j})$ is carried out as above. Eventually, 
\begin{displaymath} \lim_{j \to \infty} \mathcal{R}^{t}(f\nu)(p_j) = R^{t}(f\nu)(p) - f(p)(\tfrac{1}{2} - \mathfrak{r}(\Omega,R,p)) - f(p)\mathfrak{r}(\Omega,R,p) = -\tfrac{1}{2}f(p) + R^{t}(f\nu)(p), \end{displaymath} 
and
\begin{displaymath} \lim_{j \to \infty} \mathcal{R}(f\nu)(p_j) = R(f\nu)(p) - f(p)(\mathfrak{r}(\Omega,R,p) - \tfrac{1}{2}) + f(p)\mathfrak{r}(\Omega,R,p) = \tfrac{1}{2}f(p) + R^{t}(f\nu)(p), \end{displaymath} 
as claimed.

We have now proven \eqref{form166}-\eqref{form166a} for $V = f\nu$ with $f \in C^{\infty}_{c}(\He)$. Fortunately the case $V = g\tau$, with $g \in C^{\infty}_{c}(\He)$, is much simpler, and in particular the cases $\{p_{j}\}_{j \in \N} \subset \Omega$ and $\{p_{j}\}_{j \in \N} \subset \He \, \setminus \, \overline{\Omega}$ can be combined. So, as before, fix a weak tangent point $p \in \partial \Omega$, and $\theta > 0$, and pick a sequence $\{p_{j}\}_{j \in \N} \subset [\He \, \setminus \, \partial \Omega] \cap V_{\partial \Omega}(p,\theta)$ which tends to $p$ as $j \to \infty$. As before, take $R > 0$ so large that $\spt(g) \subset B(p_{j},R)$ for all $j \in \N$. Then, 
\begin{align*} \mathcal{R}^{t}(g\tau)(p_{j}) & =  \int_{B(p_{j},R)} \langle \nabla G(p_{j}^{-1} \cdot q),\tau(q) \rangle (g(q) - g(p)) \, d\sigma(q)\\
&\quad + g(p) \cdot  \int_{B(p_{j},R)} \langle \nabla G(p_{j}^{-1} \cdot q),\tau(q) \rangle \, d\sigma(q) =: I_{1}(p_{j}) + I_{2}(p_{j}). \end{align*}
The term $I_{2}(p_{j})$ vanishes identically by Lemma \ref{l:Ttau-pv} (where one may let $\epsilon \to 0$ since $p_{j} \notin \partial \Omega$). For the first term, we use the assumption $g \in C^{\infty}_{c}(\He)$ and dominated convergence exactly as in \eqref{form171} to deduce that
\begin{displaymath} \lim_{j \to \infty} I_{1}(p_{j}) = R^{t}(g\tau)(p) + g(p) \cdot \pv  \int_{B(p,R)} \langle \nabla G(p^{-1} \cdot q),\tau(q) \rangle \, d\sigma(q). \end{displaymath} 
The second term vanishes by another application of Lemma \ref{l:Ttau-pv}, and we have shown, as desired, that
\begin{displaymath} \lim_{j \to \infty} \mathcal{R}^{t}(g\tau)(p_{j}) = R^{t}(g\tau)(p). \end{displaymath}
The same formula holds if $\mathcal{R}^{t}$ and $R^{t}$ are replaced by $\mathcal{R}$ and $R$, respectively. Indeed, after decomposing $\mathcal{R}(g\tau)(p_{j}) = I_{1}(p_{j}) + I_{2}(p_{j})$ as above, one infers that $I_{2}(p_{j}) \equiv 0$ by first applying Lemma \ref{l:right-left}, and only then Lemma \ref{l:Ttau-pv}. By applying the same two lemmas once more, one also sees that $\lim_{j \to \infty} I_{1}(p_{j}) = R(g\tau)(p)$. This completes the proof of the lemma.  \end{proof}

We then record the extension to $V \in L^{p}(\sigma)$.

\begin{proof}[Proof of Theorem \ref{t:rJumps}] Fix a horizontal vector field $V \in L^{p}(\sigma)$, $1 < p < \infty$. Since $\{\nu(p),\tau(p)\}$ is an orthonormal basis for $\{X_{p},Y_{p}\}$ at $\sigma$ a.e. $p \in \partial \Omega$, we may rewrite
\begin{displaymath} V = \langle V,\nu \rangle \nu + \langle V,\tau \rangle \tau =: f\nu + g\tau \end{displaymath}
for some $f,g \in L^{p}(\sigma)$. Choose sequences $\{f_k\}_{k \in \N},\{g_{k}\}_{k \in \N} \subset C^\infty_c(\He)$ converging to $f$ and $g$, respectively, in $L^{p}(\sigma)$ and $\sigma$ a.e. Set $V_{k} := f_{k}\nu + g_{k}\tau$, so $V_{k} \to V$ in $L^{p}(\sigma)$. It follows from the boundedness of $R,R^{t}$ in $L^{p}(\sigma)$ that also $RV_{k} \to RV$ and $R^{t}V_{k} \to R^{t}V$ in $L^{p}(\sigma)$. Passing to further subsequences, we may assume that the convergence also occurs $\sigma$ a.e.

For the argument below, there is no difference between $\mathcal{R}$ and $\mathcal{R}^{t}$, so we only consider $\mathcal{R}$. Also, there is no conceptual difference between interior and exterior approaches, so we only consider the former. Fix $\theta > 0$, and infer from Lemma \ref{jump-K-smooth} that
\begin{displaymath} \lim_{k \to \infty} \lim_{q \to p} \mathcal{R}V_{k}(q) = \lim_{k \to \infty} -\tfrac{1}{2}f_{k}(p) + RV_{k}(p) = -\tfrac{1}{2}f(p) + RV(p) \end{displaymath}
for $\sigma$ a.e. $p \in \partial \Omega$. Here "$p \to q$" is an abbreviation for "$q \to p$ in $\Omega \cap V_{\partial \Omega}(p,\theta)$". Consequently, for any $\lambda > 0$ and $\sigma$ a.e. $p \in \partial \Omega$,
\begin{displaymath} \limsup_{q \to p} |\mathcal{R}V(q) + \tfrac{1}{2}f(p) - RV(p)| > \lambda \quad \Longrightarrow \quad \liminf_{k \to \infty} \mathcal{N}_{\theta}\mathcal{R}[V - V_{k}](p) > \lambda. \end{displaymath}
If the set of points $p \in \partial \Omega$ satisfying the LHS condition is denoted "$B_{\lambda}$", it follows that
\begin{displaymath} \sigma(B_{\lambda}) \leq \lambda^{-p} \liminf_{k \to \infty} \int \mathcal{N}_{\theta}\mathcal{R}[V - V_{k}]^{p} \, d\sigma \lesssim \lambda^{-p} \liminf_{k \to \infty} \|V - V_{k}\|_{L^{p}(\sigma)}^{p} = 0, \end{displaymath}
applying Corollary \ref{K-MaxNT} in the only non-trivial inequality. It follows that $\lim_{q \to p} \mathcal{R}V(q) = \tfrac{1}{2}f(p) - RV(p)$ for $\sigma$ a.e. $p \in \partial \Omega$, as claimed. \end{proof}

As a corollary of Theorem \ref{t:rJumps}, the following non-tangential limits exist for all $f \in L^{p}(\sigma)$, $1 < p < \infty$, and for $\sigma$ a.e. $p \in \partial \Omega$ (when $\Omega \subset \He$ is a flag domain):
\begin{displaymath} Z^{+}\mathcal{S}f(p) := \mathop{\lim_{q \to p}}_{q \in \Omega} Z \mathcal{S}f(q) = \mathop{\lim_{q \to p}}_{q \in \Omega} \mathcal{R}(fZ)(q) \quad \text{and} \quad Z^{-}\mathcal{S}f(p) := \mathop{\lim_{q \to p}}_{q \in \He \, \setminus \, \overline{\Omega}} Z \mathcal{S}f(q), \end{displaymath}
where $Z \in \{X,Y\}$, and
\begin{displaymath} \mathcal{D}^{+}f(p) := \mathop{\lim_{q \to p}}_{q \in \Omega} \mathcal{D}f(q) = \mathop{\lim_{q \to p}}_{q \in \Omega} \mathcal{R}^{t}(f\nu)(q) \quad \text{and} \quad \mathcal{D}^{-}f(p) := \mathop{\lim_{q \to p}}_{q \in \He \, \setminus \, \overline{\Omega}} \mathcal{D}f(q). \end{displaymath}
We define 
\begin{align}\label{nablapm}
\nabla^{\pm}\mathcal{S}f(p) := (X^{\pm}\mathcal{S}f(p),Y^{\pm}\mathcal{S}f(p)).
\end{align} The values of the limits above are easy to extract from Theorem \ref{t:rJumps}, and they can be expressed in terms of the principal values
\begin{displaymath} Df = R^{t}(f\nu) \quad \text{and} \quad D^{t}f = R(fX)\langle \nu,X \rangle + R(fY)\langle \nu,Y \rangle, \end{displaymath}
whose existence occupied Section \ref{s:principal}. We record the formulae for future reference:
\begin{cor}\label{c:allJumps} Let $\Omega \subset \He$ be a flag domain, and let $f \in L^{p}(\sigma)$ for some $1 < p < \infty$. Then, the following relations hold $\sigma$ a.e:
\begin{enumerate}
\item $\mathcal{D}^{\pm}f = (\pm\tfrac{1}{2}I + D)f$,
\item $\nabla_{\nu}^{\pm}\calS f := \langle \nabla^{\pm} \mathcal{S}f,\nu \rangle = (\mp\tfrac{1}{2}I + D^t)f$,
\end{enumerate}
and
\begin{equation}\label{noJump} \nabla_{\tau}Sf := \langle \nabla^{\pm} \mathcal{S}f,\tau \rangle = \pv  \int \langle \nabla G(q^{-1} \cdot),\tau(\cdot) \rangle f(q) \, d\sigma(q). \end{equation}
\end{cor}
Notably, even though $\nabla^{+}\mathcal{S}f(p) \neq \nabla^{-}\mathcal{S}f(p)$ in general, \eqref{noJump} states that the dot products with $\tau$ have the same value $\sigma$ a.e.

\begin{proof}[Proof of Corollary \ref{c:allJumps}] Part (1) is a restatement of \eqref{form166a}, applied to the horizontal vector field $V = f\nu \in L^{p}(\sigma)$. To prove (2) and \eqref{noJump}, we first decompose $fZ$ in the $\{\nu,\tau\}$-frame,
\begin{displaymath} Z \mathcal{S}f(q) = \mathcal{R}(fZ)(q) = \mathcal{R}(\langle \nu,Z \rangle f\nu + \langle \tau,Z \rangle f\tau)(q), \qquad Z \in \{X,Y\}, \, q \in \He \, \setminus \, \partial \Omega, \end{displaymath}
and then use \eqref{form166} (with "$f =\langle \nu,Z \rangle f$" and "$g = \langle \tau,Z \rangle f$") to deduce that
\begin{align*} Z^{\pm}\mathcal{S}f & = \mp \tfrac{1}{2}\langle \nu,Z \rangle f + R(fZ)  \end{align*} 
$\sigma$ a.e. on $\partial \Omega$. Consequently, writing $\nu = \langle \nu,X \rangle X + \langle \nu,Y \rangle Y$, we have 
\begin{align*} \langle \nabla^{\pm} \mathcal{S}f,\nu \rangle & = (X^{\pm}\mathcal{S}f)\langle \nu, X\rangle + (Y^{\pm}\mathcal{S}f)\langle \nu,Y \rangle\\
& = \mp \tfrac{1}{2}f \left[ \langle \nu,X \rangle^{2} + \langle \nu,Y \rangle^{2} \right] + \left[ R(fX)\langle \nu,X \rangle + R(fY)\langle \nu,Y \rangle \right] \\
& = \mp \tfrac{1}{2}f + D^{t}f.   \end{align*} 
This proves (2). Finally, \eqref{noJump} follows by a similar computation, decomposing instead $\tau = \langle \tau,X \rangle X + \langle \tau,Y \rangle Y$, and noting $\langle \nu,X \rangle \langle \tau,X \rangle + \langle \nu,Y \rangle \langle \tau,Y \rangle = \langle \nu,\tau \rangle \equiv 0$. Therefore,
\begin{displaymath} \langle \nabla^{\pm}\mathcal{S}f,\tau \rangle = R(fX)\langle \tau,X \rangle + R(fY)\langle \tau,Y \rangle, \end{displaymath}
which is another way of expressing the RHS of \eqref{noJump}. \end{proof}

\section{The divergence theorem and some corollaries}\label{s:div}

This section contains a few auxiliary results which easily follow from the horizontal divergence theorem in \cite{FSSC}. We recall that a measurable set $E \subset \He$ is a \emph{$\He$-Caccioppoli set} if the horizontal perimeter $\sigma = |\partial E|_{\He}$ is a locally finite measure.
\begin{thm}[Divergence theorem]\label{t:div} Let $E \subset \He$ be a $\He$-Caccioppoli set and let $Z \in C^{1}(\bar{E},H\He)$ be a horizontal vector field satisfying 
\begin{enumerate}
\item $\mathrm{div}^{\flat} Z \in L^{1}(E)$,
\item $Z \in L^{1}(\sigma) \cap L^{1}(E)$.
\end{enumerate}
Then the usual divergence theorem holds:
\begin{equation}\label{divThm} \int_{E} \mathrm{div}^{\flat} Z(p) \, dp = -\int \langle Z,\nu \rangle \, d\sigma. \end{equation}
\end{thm}
\begin{remark} (i) The notation $Z \in C^{1}(\bar{E})$ means that for every $p \in \bar{E}$, there is an open neighbourhood $U \ni p$ such that $Z \in C^{1}(U)$. This implies, by a standard partition-of-unity argument, the existence of a function $\varphi \in C^{\infty}(\He)$ with $\|\varphi\|_{L^{\infty}} \leq 1$ such that $\varphi \equiv 1$ on a neighbourhood of $\bar{E}$, and $\varphi Z \in C^{1}(\He)$.

(ii) The proof shows that the condition $Z \in L^{1}(E)$ could be replaced by the somewhat weaker hypothesis $R^{-1} \cdot \|Z\|_{L^{1}(E \cap [B(2R) \, \setminus B(R)])} \to 0$ as $R \to \infty$.
\end{remark}

\begin{proof}[Proof of Theorem \ref{t:div}] We just check that the obvious limiting procedure works. First, let $\varphi \in C^{\infty}(\He)$ be a function as mentioned in the previous remark. Then, let $\psi_{R} := \psi \circ \delta_{1/R}$ be a smooth cut-off at scale $R$, with $\mathbf{1}_{B(R)} \leq \psi_{R} \leq \mathbf{1}_{B(2R)}$. Note that 
\begin{equation}\label{form173} |\nabla \psi_{R}| \lesssim R^{-1} \cdot \mathbf{1}_{B(2R) \, \setminus \, B(R)}. \end{equation}
Then $\psi_{R}\varphi Z \in C_{c}^{1}(\He)$, so by definition of $\sigma$ and $\nu$ (see \cite[p. 482]{FSSC}), we have
\begin{displaymath} \int_{E} \mathrm{div}^{\flat} (\psi_{R}\varphi Z)(p) \, dp = -\int \psi_{R}\varphi \langle Z,\nu \rangle \, d\sigma. \end{displaymath}
We can now drop "$\varphi$" from both sides of the equation above, since $\varphi \equiv 1$ on a neighbourhood of $\bar{E} \supset \spt \sigma$. Since $Z \in L^{1}(\sigma)$, the RHS converges to the RHS of \eqref{divThm} as $R \to \infty$. For the LHS, we note that $\mathrm{div}^{\flat}(\psi_{R}Z) = \langle \nabla \psi_{R},Z \rangle + \psi_{R} \cdot \mathrm{div}^{\flat} Z$. The $E$-integral of the second term converges to the LHS of \eqref{divThm}, since $\mathrm{div}^{\flat} E \in L^{1}(E)$, whereas the integral of the first term satisfies
\begin{displaymath} \left| \int \langle \nabla_{\He}\psi_{R}(p),Z(p) \rangle \, dp \right| \lesssim R^{-1} \cdot \|Z\|_{L^{1}(E \cap B(2R) \, \setminus \, B(R))} \to 0 \end{displaymath}
by \eqref{form173}. This completes the proof.  \end{proof}

We record three corollaries. The first one is an integration-by-parts formula:
\begin{proposition} Let $E \subset \He$ be a $\He$-Caccioppoli set, write $\sigma := |\partial E|_{\He}$, let $Z \in \{X,Y\}$, and let $\varphi,\psi \in C^{1}(\bar{E})$ such that $(\varphi \psi)Z$ satisfies the hypotheses of Theorem \ref{t:div}. Then,
\begin{equation}\label{eq:iParts} \int_{E} Z(\varphi) \cdot \psi \, dp = -\int_{E} \varphi \cdot Z(\psi) - \int (\varphi \psi) \langle Z,\nu \rangle \, d\sigma. \end{equation}
\end{proposition}

\begin{proof} Note that
\begin{displaymath} \mathrm{div}^{\flat} [(\varphi \psi)Z] = Z(\varphi \psi) = Z(\varphi) \cdot \psi + \varphi \cdot Z(\psi).  \end{displaymath}
Consequently, by Theorem \ref{t:div}, we have
\begin{displaymath} \int (\varphi \psi) \langle Z,\nu \rangle \, d\sigma = - \int_{E} Z(\varphi) \cdot \psi \, dp  - \int_{E} \varphi \cdot Z(\psi) \, dp. \end{displaymath}
Re-arrange terms to complete the proof. \end{proof}

The second corollary is an analogue of the \emph{Rellich identity}:
 
 \begin{proposition}\label{p:rellich} Let $E \subset \He$ be a $\mathbb{H}$-Caccioppoli set, write $\sigma := |\partial E|_{\He}$, and assume that $u \in C^{2}(\bar{E})$ satisfies the following properties:
 \begin{enumerate}
 \item $\bigtriangleup^{\flat} u(p) = 0$ for all $p \in E$,
 \item $\nabla u \in L^{2}(\sigma) \cap L^{2}(E)$,
 \item $(Z_{1}u)(Z_{2}Z_{3}u) \in L^{1}(E)$ for all $Z_{1},Z_{2},Z_{3} \in \{X,Y\}$.
 \end{enumerate}
 Then,
\begin{equation}\label{rellich} \int |\nabla u|^{2} \langle X,\nu \rangle \, d\sigma = 2\int  Xu\langle \nabla  u,\nu \rangle \, d\sigma - 2 \int_{E} (Yu)(p)\, (Tu)(p) \, dp. \end{equation}
\end{proposition}

\begin{proof} We start with the following computation which works for $\varphi \in C^{2}(\R^{3})$:
\begin{align*} \tfrac{1}{2}X|\nabla \varphi|^{2} & = X\varphi X^{2}\varphi + Y\varphi XY\varphi\\
& = (2X\varphi X^{2}\varphi + Y\varphi YX\varphi + X\varphi Y^{2}\varphi)\\
&\quad + (Y\varphi XY\varphi - Y\varphi YX\varphi - X\varphi X^{2}\varphi - X\varphi Y^{2}\varphi)\\
& = \mathrm{div}^{\flat}(X\varphi \nabla \varphi) + Y\varphi[X,Y]\varphi - X\varphi\bigtriangleup^{\flat}\varphi\\
& = \mathrm{div}^{\flat}(X\varphi \nabla \varphi) + Y\varphi T\varphi - X\varphi \bigtriangleup^{\flat}\varphi. \end{align*} 
We then specialise to $\varphi = u \in C^{2}(\bar{E})$. By the $\bigtriangleup^{\flat}$-harmonicity assumption (1), and noting that $\mathrm{div}^{\flat}(\varphi X) = X\varphi$, the previous computation shows that
\begin{equation}\label{form7} -\mathrm{div}^{\flat}(|\nabla u|^{2}X)(p) = -2 \cdot \mathrm{div}^{\flat}(Xu\nabla u)(p) - 2Yu(p)Tu(p), \qquad p \in E. \end{equation}
We integrate this equation over $p \in E$, and then apply the divergence theorem, Theorem \ref{t:div}, to both vector fields $Z = |\nabla u|^{2}X$ and $Z = Xu\nabla u$. We note that
\begin{displaymath} \mathrm{div}^{\flat}(Xu\nabla u) = \langle \nabla Xu,\nabla u \rangle \in L^{1}(E) \quad \text{and} \quad YuTu = Yu[XY - YX]u \in L^{1}(E) \end{displaymath}
by the assumption (3), which implies the first condition ($\mathrm{div}^{\flat}Z \in L^{1}(E)$) of Theorem \ref{t:div}. The second condition ($Z \in L^{1}(\sigma) \in L^{1}(E)$) follows immediately from the present assumption (2). Now \eqref{rellich} is a consequence of \eqref{form7} and \eqref{divThm}. 
\end{proof}

The third corollary is a "Fubini theorem" for Lebesgue measure on the super- or sub-graph of an intrinsic Lipschitz function, recall the definition from Section \ref{s:ILG}. For a much more general result, but with slightly different hypotheses, see Montefalcone's work \cite[Theorem 2.2]{MR2165404}.

\begin{lemma} Let $\W := \{(0,y,t) : y,t \in \R\}$ and $\V := \{(x,0,0) : x \in \R\}$. Let $\phi \colon \W \to \V$ be an intrinsic Lipschitz function with intrinsic graph $\Gamma = \{w \cdot \phi(w) : w \in \W\} \subset \He$. Write $\Gamma^{+} := \{w \cdot v : v > \phi(v)\}$, and let $\sigma := |\partial \Gamma^{+}|_{\He}$. Let $g \colon \He \to \R$ be non-negative and $\mathcal{L}^{3}$ measurable, or $g \in C^{1}_{c}(\He)$. Then, $r \mapsto g(w \cdot (r,0,0))$ is $\mathcal{L}^{1}|_{[0,\infty]}$ measurable for $\sigma$ a.e. $w \in \Gamma$, and
\begin{displaymath} w \mapsto \int_{0}^{\infty} g(w \cdot (r,0,0)) \, dr \end{displaymath}
is $\sigma$ measurable, and
\begin{equation}\label{eq:Fubini} \int_{\Gamma^{+}} g(p) \, dp = \int_{\Gamma} \int_{0}^{\infty} g(w \cdot (r,0,0)) \, dr \, \langle X,\nu(w) \rangle \, d\sigma(w). \end{equation}
Here $\nu$ is the inward-pointing horizontal normal of $\Gamma^{+}$. \end{lemma}

\begin{proof} We start by proving \eqref{eq:Fubini} for $g \in C^{1}_{c}(\He)$, in which case the measurability statements are obvious. We define the auxiliary function $G \in C^{1}(\R^{3})$,
\begin{displaymath} G(p) := \int_{0}^{\infty} g(p \cdot (r,0,0)) \, dr. \end{displaymath}
This is well-defined, since $p \cdot (r,0,0) \in \Gamma^{+}$ for all $p \in \Gamma^{+}$ and $r \geq 0$. Moreover,
\begin{equation}\label{divG} \mathrm{div}^{\flat} (GX) = XG = -g. \end{equation}
The easiest way to see this is to define the path $\gamma_{p}(s) := p \cdot (s,0,0)$, and note that $Xf(p) = (f \circ \gamma_{p})'(0)$ for any $f \in C^{1}(\R^{3})$ (in particular $f = G$). On the other hand,
\begin{displaymath} (G \circ \gamma_{p})(s) = \int_{0}^{\infty} g(p \cdot (r + s,0,0)) \, dr = \int_{s}^{\infty} g(p \cdot (r,0,0)) \, dr,  \end{displaymath}
so $XG(p) = (G \circ \gamma_{p})'(0) = -g(p)$. Now we apply the divergence theorem, Theorem \ref{t:div}, to the $\He$-Caccioppoli set $E = \Gamma^{+}$ and the vector field $GX \in C^{1}(\bar{E},H\He)$. Note that the support of $G$ intersected with $\bar{E}$ is compact. Hence, there is no problem in verifying the hypotheses of Theorem \ref{t:div}, and, recalling \eqref{divG}, the conclusion is
\begin{displaymath} \int_{\Gamma^{+}} g(p) \, dp = \int_{\Gamma} G(w)\langle X,\nu(w) \rangle \, d\sigma(w) = \int_{\Gamma} \int_{0}^{\infty} g(w \cdot (r,0,0)) \, dr \, \langle X,\nu(w) \rangle \, d\sigma(w), \end{displaymath} 
as claimed. 

The extension to non-negative measurable functions follows the standard proof of Fubini's theorem, with one non-trivial step in the middle. One first argues that the measurability statements and the formula \eqref{eq:Fubini} remain valid for linear combinations and limits of monotone sequences of non-negative functions (for decreasing sequences, one adds the hypothesis that the first element is Lebesgue integrable). As a corollary, one obtains the theorem for open sets (their characteristic functions are increasing limits of $C^{1}_{c}$-functions), and then for bounded $G_{\delta}$-sets. After this, one is prepared to treat characteristic functions of bounded $\mathcal{L}^{3}$-null sets $E \subset \He$, since each of these is contained in a bounded $G_{\delta}$-set $B \subset \He$ with $\mathcal{L}^{3}(B) = 0$. Applying \eqref{eq:Fubini} to $g = \mathbf{1}_{B}$, one finds that
\begin{equation}\label{form155} 0 = \mathcal{L}^{3}(B) = \int_{\Gamma} \int_{0}^{\infty} \mathbf{1}_{B}(w \cdot (r,0,0)) \, dr \, \langle X,\nu(w) \rangle \, d\sigma(w). \end{equation}
Here is the non-trivial step: \cite[Theorem 1.6 and Corollary 4.2]{MR3168633} imply that $\langle X,\nu(w) \rangle > 0$ for $\sigma$ a.e. $w \in \Gamma$. From this and \eqref{form155}, it follows that $\mathbf{1}_{B}(w \cdot (r,0,0)) = 0$ for $\mathcal{L}^{1}$ a.e. $r \in [0,\infty]$ for $\sigma$ a.e. $w \in \Gamma$. Since $E \subset B$, one infers that $\mathbf{1}_{E}(w \cdot (r,0,0)) = 0$ for $\mathcal{L}^{1}$ a.e. $r \in [0,\infty]$ for $\sigma$ a.e. $w \in \Gamma$. Consequently $r \mapsto \mathbf{1}_{E}(w \cdot (r,0,0))$ is $\mathcal{L}^{1}|_{[0,\infty]}$ measurable and
\begin{displaymath} w \mapsto I_{E}(w) := \int_{0}^{\infty} \mathbf{1}_{E}(w \cdot (r,0,0)) \, dr = 0 \end{displaymath}
for $\sigma$ a.e. $w \in \Gamma$. Therefore $I_{E}$ is $\sigma$ measurable, and $\int I_{E}(w) \langle X,\nu(w) \rangle \, d\sigma(w) = 0$, which completes the case of bounded null sets. The remaining steps contain no surprises: first characteristic functions of bounded $\mathcal{L}^{3}$ measurable sets, then simple functions, and finally non-negative measurable functions. \end{proof}

%%%%%%%%%%%%%%%%%%%%%%%%%%%

\section{Vertical distributions and $T$-multipliers}\label{s:Fourier} In the arguments for the invertibility of the operators $\tfrac{1}{2}I \pm D$ and $\tfrac{1}{2}I \pm D^{t}$, we will use fractional differentiation operators, and other Fourier multipliers, in the $t$-variable. The purpose of this section is to develop a framework to treat these objects rigorously.

\begin{definition}[Vertically tempered distributions] Let $U \subset \R^{2}$ be open, and let $\Omega := U \times \R \subset \R^{3}$. Write $\calB(\Omega)$ for functions $\psi \in C^{\infty}(\Omega)$ such that $\|\psi\|_{L^{\infty}(\R^{3})} \leq 1$, and $\{z \in U : \psi_{z} \neq 0\}$ is compactly contained in $U$, where $\psi_{z}(t) := \psi(z,t)$. We define $\mathcal{S}'(\Omega)$ to be the collection of distributions $u \in \mathcal{D}'(\Omega)$ such that $\psi u \in \mathcal{S}'(\R^{3})$ for all $\psi \in \mathcal{B}(\Omega)$. \end{definition}

Here $\mathcal{D}'(\Omega)$ refers to distributions in $\Omega$, and $\mathcal{S}'(\R^{3})$ to tempered distributions in $\R^{3}$. The typical example of a vertically tempered distribution will be a smooth function on $\Omega = U \times \R$ which blows up near $\partial \Omega$, but lies in $L^{\infty}(K \times \R)$ for every compact set $K \subset U$.

\begin{definition}[Vertical Schwartz class] Let $\Omega := U \times \R$, as above. We say that $\varphi$ is a \emph{vertical Schwartz function in $\Omega$}, denoted $\mathcal{S}(\Omega)$, if $\varphi \in \mathcal{S}(\R^{3})$, and $\{z \in \R^{2} : \varphi_{z} \neq 0\}$ is compactly contained in $U$. \end{definition}

For $\varphi \in \mathcal{S}(\Omega)$, we define the \emph{vertical Fourier transform}
\begin{equation}\label{form52} \widehat{\varphi}(z,\tau) := \int_{\R} e^{-2\pi i t\tau} \varphi(z,t) \, dt, \qquad (z,\tau) \in \Omega. \end{equation}
We will not emphasise the "verticality" in the notation; this is the only sort of Fourier transform which will appear in a long while. Once the "full" Fourier transform eventually appears, we will make the distinction clear in the notation. Clearly $\widehat{\varphi} \in \mathcal{S}(\Omega)$ for every $\varphi \in \mathcal{S}(\Omega)$, and the vertical Fourier transform is invertible on $\mathcal{S}(\Omega)$: the inverse map is given by changing the sign in \eqref{form52}. 

\begin{definition}[Vertical Fourier transform of distributions]\label{def:vertFDist} Let $u \in \mathcal{S}'(\Omega)$. Then, we define the vertically tempered distribution $\hat{u} \in \mathcal{S}'(\Omega)$ by $\hat{u}(\varphi) := (\psi u)(\widehat{\varphi})$ for $\varphi \in \mathcal{S}(\Omega)$, where $\psi \in \mathcal{B}(\Omega)$ is an arbitrary function satisfying $\psi \widehat{\varphi} = \widehat{\varphi}$ (such $\psi \in \mathcal{B}(\Omega)$ exists, because $\widehat{\varphi} \in \mathcal{S}(\Omega)$, and the definition is evidently independent on the particular choice of $\psi$).  \end{definition}  

%\textcolor{cyan}{By definition of $\varphi \in \mathcal{S}(\Omega)$, we have $\spt \varphi \subset K \times \R$, where $K \subset U$ is compact. It follows that also $\spt \widehat{\varphi} \subset K \times \R$. Now, to find "$\psi$", start by choosing $\psi_{0} \in C^{\infty}_{c}(U) \subset C^{\infty}_{c}(\R^{2})$ with $\mathbf{1}_{K} \leq \psi_{0} \leq \mathbf{1}_{U}$. Then, the choice $\psi(z,t) := \psi_{0}(z)$ works.}

Since we will not need this fact, we leave it to the reader to check that $\hat{u} \in \mathcal{S}'(\Omega)$. Differentiation, and multiplication by polynomials, preserve the class of vertically tempered distributions, with the obvious definitions $\partial^{\beta}u(\varphi) = (-1)^{|\beta|}u(\partial^{\beta}\varphi)$ and $Pu(\varphi) := u(P\varphi)$. As usual, the vertical Fourier transform turns polynomials into derivatives and derivatives into polynomials. We record the formulae, but omit the proofs:
\begin{lemma}\label{lemma2} Let $u \in \mathcal{S}'(\Omega)$, and let $P(z)$ be a polynomial in the $z$-variables. Then,
\begin{enumerate}
\item $\widehat{P(z)u} = P(z)\hat{u}$ and $\widehat{\partial_{x}^{m}u} = \partial_{x}^{m}\hat{u}$ and $\widehat{\partial_{y}^{n}u} = \partial_{y}^{n}\hat{u}$, and
\item $\widehat{\partial_{t}u} = (2\pi i \tau)\hat{u}$, and $\widehat{(-2\pi it)u} = \partial_{t}\hat{u}$.
\end{enumerate}
\end{lemma} 

\subsubsection{Fourier multipliers in the $\tau$-variable}\label{tauMult} As before, let $U \subset \R^{2}$ be open, and let $\Omega := U \times \R$. For suitable distributions $u \in \mathcal{S}'(\Omega)$, and for suitable functions $\mathfrak{m} \in L^{1}_{\mathrm{loc}}(\R)$, the purpose of this section is to define the vertical Fourier multipliers $\mathfrak{m}(T)u$, and show that $Z(\mathfrak{m}(T)u) = \mathfrak{m}(T)(Zu)$ for all horizontal derivatives $Z$. Applications will include the functions $\mathfrak{m}(\tau) = |\tau|^{\alpha}$, and $\mathfrak{m}(\tau) = |\tau|/\tau$.

\begin{definition}[Class $\mathcal{G}(\Omega)$]\label{classG} Let $\mathcal{G}(\Omega)$ consist of those $u \in \mathcal{S}'(\Omega)$ such that $\widehat{\partial u} \in L^{1}(K \times \R)$ for all compact sets $K \subset U$, and for all (distributional) derivatives $\partial = \partial_{x}^{k}\partial_{y}^{m}\partial_{t}^{n}$.
 \end{definition}

We start by observing that $\mathcal{G}(\Omega)$, in fact, consists of smooth functions:
\begin{lemma}\label{lemma1} $\mathcal{G}(\Omega) \subset C^{\infty}(\Omega)$. \end{lemma}

\begin{proof} Fix $u \in \mathcal{G}(\Omega)$, let $k,m,n \in \N$, and let $\partial = \partial_{x}^{k}\partial_{y}^{m}\partial_{t}^{n}$. Note that the function $h \colon \Omega \to \C$,
\begin{displaymath} h(z,t) := \int_{\R} e^{2\pi i t\tau} \widehat{\partial u}(z,\tau) \, d\tau, \end{displaymath} 
satisfies $h \in L^{1}(K \times I)$ whenever $K \subset U$ and $I \subset \R$ are compact. Indeed,
\begin{displaymath} \int_{K \times I} |h(z,t)| \, dz \, dt \leq \int_{I} \int_{K} \int_{\R} |\widehat{\partial u}(z,\tau)| \, d\tau \, dz \, dt < \infty. \end{displaymath} 
Therefore, the following application of Fubini's theorem is legitimate for all $\varphi \in C^{\infty}_{c}(\Omega)$:
\begin{displaymath} \partial u(\varphi) =: \widehat{\partial u}(\widecheck{\varphi}) := \int_{\Omega} \widehat{\partial u}(z,\tau) \left( \int_{\R} e^{2\pi i t\tau} \varphi(z,t) \, dt \right) \, dz \, d\tau = \int_{\Omega} h(z,t) \cdot \varphi(z,t) \, dz \, dt. \end{displaymath}
This means that $\partial u = h \in L^{1}_{\mathrm{loc}}(\Omega)$ for all derivatives $\partial$. It now follows from the standard Sobolev embedding theorem, see \cite[Theorem 6, p. 2.70]{MR1625845}, that $u \in C^{\infty}(\Omega)$. \end{proof} 

\begin{lemma} Let $u \in \mathcal{G}(\Omega)$. Then any derivative of $u$ lies in $\mathcal{G}(\Omega)$. Also, if $P = P(z)$ is a polynomial in the $z$-variables, then $Pu \in \mathcal{G}(\Omega)$. \end{lemma}

\begin{proof} The first claim is clear from the definition. The second claim requires a little case-chase. One needs to check that $\widehat{\partial Pu} \in L^{1}(K \times \R)$ for all compact sets $K \subset U$ and all possible derivatives $\partial = \partial_{x}^{k}\partial_{y}^{m}\partial_{t}^{n}$. For example,
\begin{displaymath} \widehat{\partial_{x}Pu} = \partial_{x}(P\hat{u}) = (\partial_{x}P)\hat{u} + P(\partial_{x}\hat{u}) \in L^{1}(K \times \R). \end{displaymath}
A similar argument works more generally for $\partial_{x}^{k}$ and $\partial_{y}^{m}$. Finally,
\begin{displaymath} \widehat{\partial_{t}^{n}Pu} = P\widehat{\partial_{t}^{n}u} \in L^{1}(K \times \R), \end{displaymath}
since $u \in \mathcal{G}(\Omega)$, and $P$ does not depend on the $t$-variable.
 \end{proof}

%We start with $Pu$. By Lemma \ref{lemma2}, we have, for example,
%\begin{displaymath} (1 + |\tau|)^{k}\partial_{x}\widehat{Pu} = (1 + |\tau|^{k})\partial_{x}(P\hat{u}) = (1 + |\tau|^{k})[(\partial_{x}P)\hat{u} + P(\partial_{x}\hat{u})] \in L^{1}(K \times \R) \end{displaymath} 
%for every compact set $K \subset U$, by \eqref{form23}. A similar, but more notationally awkward, computation works with $\partial_{x}^{m}$ and $\partial_{y}^{n}$, $m,n > 1$.

%Regarding $\partial^{\beta}u$, there is nothing to do if $\partial^{\beta}$ only acts on the $z$-variables, since then $\widehat{\partial^{\beta}u} = \partial^{\beta}\hat{u}$ by Lemma \ref{lemma2}. Finally, if $\partial^{\beta} = T^{k}$, we have
%\begin{displaymath} (1 + |\tau|)^{k}\partial_{x}^{m}\partial_{y}^{n}\widehat{T^{m}u} = (1 + |\tau|^{m})(2\pi i \tau)^{k} \partial_{x}^{m}\partial_{y}^{n}\hat{u} \in L^{1}(K \times \R) \end{displaymath} 
%for all compact $K \subset U$, again using Lemma \ref{lemma2} and \eqref{form23}.

\begin{remark} How about multiplication by $t$? Here is a counterexample. Consider $u(z,t) = (1 + t^{2})^{-1/2} \in \mathcal{S}'(\R^{3})$. Then the vertical Fourier transform $\hat{u}$ is (up to some constants) the modified Bessel function $(z,\tau) \mapsto K_{0}(\tau)$, and one may verify that $u \in \mathcal{G}(\Omega)$. However, the Fourier transform of $tu$ would be the $\tau$-derivative of $K_{0}$, which has a non-integrable singularity over the origin.
\end{remark}

\begin{cor}\label{cor1} Let $u \in \mathcal{G}(\Omega)$, and let $Z$ be a horizontal derivative, that is, an arbitrary linear combination of concatenations of the $X$ and $Y$ vector fields. Then $Zu \in \mathcal{G}(\Omega)$.
\end{cor}

\begin{proof} One may re-write $Zu$ as a finite linear combination $Zu(z,t) = \sum_{\beta} P_{\beta}(z)\partial^{\beta}u(z,t)$, with $\partial^{\beta}$ referring to any derivative, and then apply the previous lemma. \end{proof}

We then define the relevant class of multipliers:
\begin{definition}[Class $\mathcal{M}$] A Lebesgue measurable function $\mathfrak{m} \colon \R^{3} \to \R$ lies in the class $\mathcal{M}$ if $\mathfrak{m}$ only depends on the third variable (this is denoted by writing "$\mathfrak{m}(\tau)$" in place of $\mathfrak{m}(z,\tau)$"), and there exists $k \in \N$ such that $|\mathfrak{m}(\tau)| \leq (1 + |\tau|)^{k}$ for a.e. $\tau \in \R$. \end{definition}

\begin{ex} Clearly $\tau \mapsto |\tau|^{\alpha} \in \mathcal{M}$ for all $\alpha \geq 0$, and $\tau \mapsto |\tau|/\tau \in \mathcal{M}$.  \end{ex}

For $u \in \mathcal{G}(\Omega)$ and $\mathfrak{m} \in \mathcal{M}$, we may then define $\mathfrak{m}(T)u$:

\begin{definition}[$\mathfrak{m}(T)u$] Let $u \in \mathcal{G}(\Omega)$ and $\mathfrak{m} \in \mathcal{M}$. We define
\begin{displaymath} (\mathfrak{m}(T)u)(z,t) := \int_{\R} e^{2\pi i t\tau} \mathfrak{m}(\tau)\hat{u}(z,\tau) \, d\tau, \qquad \text{for a.e. } (z,t) \in \Omega. \end{displaymath}
\end{definition}

\begin{remark} Note that if $u \in \mathcal{G}(\Omega)$, then $(2\pi i \tau)^{k}\hat{u} = \widehat{\partial_{t}^{n}u} \in L^{1}(K \times \R)$ for all $n \geq 0$ and for all compact $K \subset U$. In particular, 
\begin{displaymath} \iint_{K \times \R} |\hat{u}(z,\tau)| \, dz \, d\tau < \infty \quad \text{and} \quad \iint_{K \times \R} |\tau|^{k}|\hat{u}(z,\tau)| \, dz \, d\tau < \infty. \end{displaymath} 
Consequently, 
\begin{equation}\label{form23} \iint_{K \times \R} |\mathfrak{m}(\tau)||\hat{u}(z,\tau)| \, d\tau \, dz < \infty, \qquad u \in \mathcal{G}(\Omega), \end{equation}
for all compact sets $K \subset U$. It follows that the integral defining $\mathfrak{m}(T)u$ converges for a.e. $z \in U$ for all $t \in \R$, and in fact defines a function in $L^{1}_{\mathrm{loc}}(\Omega)$. \end{remark}

\begin{lemma} Let $u \in \mathcal{G}(\Omega)$, and $\mathfrak{m} \in \mathcal{M}$. Then $\mathfrak{m}(T)u \in \mathcal{G}(\Omega) \subset C^{\infty}(\Omega)$. \end{lemma}

\begin{proof} We first observe that $\mathfrak{m}(T)u \in \mathcal{S}'(\Omega)$. Indeed, if $\psi \in \mathcal{B}(\Omega)$ is fixed, and we write $K := \{z \in U : \psi_{z} \neq 0\}$ (a compact subset of $U$ by definition), then 
\begin{displaymath} |(\psi \cdot \mathfrak{m}(T)u)(\varphi)| \leq \int_{\R} \sup_{z \in \R^{2}} |\varphi(z,t)| \left( \iint_{K \times \R} |\mathfrak{m}(\tau)||\hat{u}(z,\tau)| \, d\tau \, dz \right) \, dt < \infty, \quad \varphi \in \mathcal{S}(\R^{3}), \end{displaymath} 
using \eqref{form23}, and noting that $\sup_{z \in \R^{2}} |\varphi(z,t)| \lesssim (1 + |t|)^{-2}$. This calculation shows that $\psi \cdot \mathfrak{m}(T)u \in \mathcal{S}'(\R^{3})$, and consequently $\mathfrak{m}(T)u \in \mathcal{S}'(\Omega)$. It is now easy to verify that the vertical Fourier transform of $\mathfrak{m}(T)u$ is the locally integrable function $\mathfrak{m} \cdot \hat{u}$:
\begin{equation}\label{form55} \widehat{\mathfrak{m}(T)u} = \mathfrak{m} \cdot \hat{u}. \end{equation}
With this information in hand, one may compute
\begin{displaymath} \widehat{\partial_{x}^{k}\partial_{y}^{k}\mathfrak{m}(T)u} = \partial_{x}^{k}\partial_{y}^{m}\widehat{\mathfrak{m}(T)u} = \partial_{x}^{k}\partial_{y}^{m}\mathfrak{m} \cdot \hat{u} =\ldots=\mathfrak{m} \cdot \widehat{\partial_{x}^{k}\partial_{y}^{m}u} \in L^{1}(K \times \R) \end{displaymath} 
for any compact $K \subset U$, applying \eqref{form23} to $\partial_{x}^{k}\partial_{y}^{m}u \in \mathcal{G}(\Omega)$, and noting that multiplication by $\mathfrak{m}$ commutes with $x$ and $y$ derivatives. Similarly,
\begin{displaymath} \widehat{\partial_{t}^{n}\mathfrak{m}(T)u} = [(2\pi i\tau)^{n} \cdot \mathfrak{m}] \cdot \hat{u} \in L^{1}(K \times \R), \end{displaymath}
this time applying \eqref{form23} to $\tau \mapsto (2\pi i \tau)^{n}\mathfrak{m}(\tau) \in \mathcal{M}$.
These computations show that $\mathfrak{m}(T)u \in \mathcal{G}(\Omega)$, as claimed, and the inclusion $\mathcal{G}(\Omega) \subset C^{\infty}(\Omega)$ was already Lemma \ref{lemma1}. \end{proof}

%We check that the distributional Fourier transform of $\mathfrak{m}(T)u$ equals the function $\mathfrak{m}\hat{u}$. First, by definition,
%\begin{displaymath} \widehat{\mathfrak{m}(T)u}(\varphi) := \int_{U} \int_{\R} (\mathfrak{m}(T)u)(z,\tau) \cdot \widehat{\varphi}(z,\tau) \, d\tau \, dz, \qquad \varphi \in \mathcal{S}(\Omega), \end{displaymath} 
%where the integral is absolutely convergent by \eqref{form54}. For a.e. $z \in U$, note that $\tau \mapsto (\mathfrak{m}(T)u)(z,\tau)$ is the inverse transform of the (finite) signed measure $\tau \mapsto \mathfrak{m}(\tau)\hat{u}(z,\tau)$. For these $z \in U$, we may apply the variant of Plancherel for signed measures and Schwartz functions on $\R$. This yields
%\begin{displaymath} \int_{U} \int_{\R} (\mathfrak{m}(T)u)(z,\tau) \cdot \widehat{\varphi}(z,\tau) \, d\tau \, dz = \int_{U} \int_{\R} \mathfrak{m}(\tau)\hat{u}(z,\tau) \cdot \varphi(z,\tau) \, d\tau \, dz, \end{displaymath}
%and this means by definition that $\widehat{\mathfrak{m}(T)u} = \mathfrak{m}\hat{u}$. Now, using the assumption that $\hat{u} \in \mathcal{G}(\Omega)$, and noting that $\mathfrak{m}$ only depends on $\tau$, we deduce that the distributional derivatives of $\mathfrak{m}\hat{u}$ in the $z$-variables are locally integrable functions in $\Omega$, and
%\begin{displaymath} (1 + |\tau|)^{k}\partial_{x}^{m}\partial_{y}^{n}[\mathfrak{m}\hat{u}] \in L^{1}(K \times \R) \end{displaymath}
%for all compact sets $K \subset U$. Since $\widehat{\mathfrak{m}(T)u} = \mathfrak{m}\hat{u}$, this means by definition that $\mathfrak{m}(T)u \in \mathcal{G}(\Omega)$. 

We then package the information above into a useful proposition:

\begin{proposition}\label{prop2} Let $\Omega = U \times \R$, with $U$ open, let $u \in \mathcal{G}(\Omega)$ as in Definition \ref{classG}, and let $\mathfrak{m} \in \mathcal{M}$. If $Z$ is any finite linear combination of concatenations of the $X$ and $Y$ vector fields (or their right invariant counterparts $\tilde{X}$ and $\tilde{Y}$), then 
\begin{displaymath} Zu,\mathfrak{m}(T)u \in \mathcal{G}(\Omega) \subset C^{\infty}(\Omega) \quad \text{and} \quad \mathfrak{m}(T)(Zu) = Z(\mathfrak{m}(T)u). \end{displaymath}
\end{proposition}

\begin{proof} It only remains to prove the last claim. As observed in the proof of Corollary \ref{classG}, the vector field $Z$ can be written in the form
\begin{displaymath} Z = \sum P_{\beta}(z)\partial^{\beta}, \end{displaymath}
where $\partial^{\beta} = \partial_{x}^{k}\partial_{y}^{m}\partial_{t}^{n}$ is an arbitrary derivative, and $P_{\beta}(z)$ is a polynomial in only the $z$-variables. To prove the formula $\mathfrak{m}(T)(Zu) = Z(\mathfrak{m}(T)u)$, it then suffices to consider the case of a single summand, that is, $Z = P(z)\partial_{x}^{k}\partial_{y}^{m}\partial_{t}^{n}$. Applying \eqref{form55} first to $Zu$, then using Lemma \ref{lemma2}, then using that multiplication by $\mathfrak{m}$ commutes with $P(z),(2\pi i\tau)^{n},\partial_{x},\partial_{y}$, and finally applying \eqref{form55} once again to $u$, one has
\begin{align*} \widehat{m(T)(Z u)} & = \mathfrak{m} \cdot \widehat{Z u}\\
& = \mathfrak{m} \cdot P(z)(2\pi i \tau)^{n}\partial_{x}^{k}\partial_{y}^{m}\hat{u}\\
& = P(z)(2\pi i\tau)^{n}\partial_{x}^{k}\partial_{y}^{m} \cdot \widehat{m(T)u}\\
& = \widehat{Z(\mathfrak{m}(T)u)}. \end{align*} 
This completes the proof. \end{proof}

The machinery above will mainly be applied to functions of the form $u = \mathcal{S}f$:
\begin{proposition}\label{prop1} Let $\Omega = U \times \R$ be a flag domain, and let $f \in C^{\infty}_{c}(\He)$. Then $\mathcal{S}f \in \mathcal{G}(\Omega)$. \end{proposition}

\begin{proof} In Section \ref{appA}, we compute the following explicit formula for the (distributional) Fourier transform of $\mathcal{S}f \in \mathcal{S}'(\Omega)$:
\begin{equation}\label{form70} \widehat{(\mathcal{S}f)}(z,\tau) = c\int_{\partial U} e^{\pi i \omega(z,w)\tau} K_{0}(\tfrac{\pi}{2} |\tau| |z - w|^{2})\hat{f}(w,\tau) \, d\calH_{E}^{1}(w), \qquad (z,\tau) \in \R^{3}. \end{equation} 
The function $K_{0}$ is the modified Bessel function of the second kind of index $0$. It has a mild singularity at $0$, and decays rapidly at infinity. From these observations, and similar facts about higher order derivatives of $K_{0}$, it will follow that $\mathcal{S}f \in \mathcal{G}(\Omega)$. The full proof is a little technical, and hence postponed to Corollary \ref{cor3}. \end{proof}

\begin{cor}\label{cor2} Let $\Omega$ be a flag domain, let $f \in C^{\infty}_{c}(\He)$, and let $\mathfrak{m} \in \mathcal{M}$. Then $\mathfrak{m}(T)(\mathcal{S}f)$ is $\bigtriangleup^{\flat}$-harmonic in $\Omega$. \end{cor}
\begin{proof} Since $\bigtriangleup^{\flat} = X^{2} + Y^{2}$, Proposition \ref{prop2} is applicable, and implies that $\bigtriangleup^{\flat}(\mathfrak{m}(T)\mathcal{S}f) = \mathfrak{m}(T)(\bigtriangleup^{\flat}\mathcal{S}f) = 0$, since $\bigtriangleup^{\flat}\mathcal{S}f = 0$ in $\Omega$.  \end{proof}

%%%%%%%%%%%%%%%%%%%%%%%%%%%%%%%%%%%%%%%%%

\section{The operators $\tfrac{1}{2}I \pm D^{t}$ are injective and have closed range}\label{s:injectivity}
We now begin the proof of Theorem \ref{main2}, stating that the operators $\tfrac{1}{2}I \pm D$ and $\tfrac{1}{2}I + D^{t}$ are invertible on $L^{2}(\sigma)$ whenever $\Omega$ is a flag domain and $\sigma = |\partial \Omega|_{\He} = \mathcal{H}^{2}_{E}|_{\partial \Omega}$. It follows from Theorem \ref{t:FO1} that both operators are bounded $L^{2}(\sigma) \to L^{2}(\sigma)$, so the remaining tasks are to establish injectivity and surjectivity. 

The strategy is the following. In Theorem \ref{t:injectivity}, we prove that $\|f\|_{L^{2}(\sigma)} \lesssim \|(\tfrac{1}{2}I\pm D^{t})f\|_{L^{2}(\sigma)}$ for all $f \in L^{2}(\sigma)$. This shows that $\tfrac{1}{2}\pm D^{t}$ is injective. Then, we recall a standard result in functional analysis, see \cite[Theorem 4.15]{MR1157815}: a bounded operator $T \colon L^{2}(\sigma) \to L^{2}(\sigma)$ is surjective if and only if the adjoint $T^{t}$ satisfies $\|T^{t}f\|_{L^{2}(\sigma)} \gtrsim \|f\|_{L^{2}(\sigma)}$ for all $f \in L^{2}(\sigma)$. Since $\tfrac{1}{2}I \pm D^{t}$ is the adjoint of $\tfrac{1}{2}I \pm D$, we infer that $\tfrac{1}{2}I \pm D$ is surjective. In Section \ref{s:surjectivity}, we prove separately that $\tfrac{1}{2}I \pm D^{t}$ is surjective. By applying  \cite[Theorem 4.15]{MR1157815} once more, this implies that $\tfrac{1}{2}I + D$ is injective, and completes the proof of Theorem \ref{main2}. In particular: to prove Theorem \ref{main2}, we only need to prove the invertibility of $\tfrac{1}{2}I \pm D^{t}$.

\begin{thm}\label{t:injectivity} Let $\Omega = \{(x,y,t) : x < A(y)\} \subset \R^{3}$ be a flag domain, and let $\sigma = |\partial \Omega|_{\He} = \calH_{E}^{2}|_{\partial \Omega}$. Then the operators $\tfrac{1}{2}I \pm D^{t}$ are injective and have closed range in $L^{2}(\sigma)$. In fact,
 \begin{equation}\label{form4} \|f\|_{L^{2}(\sigma)} \lesssim \min\{\|(\tfrac{1}{2}I - D^{t})f\|_{L^{2}(\sigma)},\|(\tfrac{1}{2}I + D^{t})f\|_{L^{2}(\sigma)}\}, \qquad f \in L^{2}(\sigma), \end{equation}
 where the implicit constant only depends on the Lipschitz constant of $A$.
 \end{thm}
  
\begin{remark} The injectivity of the operators $\tfrac{1}{2}I \pm D^{t}$ is immediately clear from \eqref{form4}. The closed range also follows easily. If, for example, $(\tfrac{1}{2}I + D^{t})f_{j} \to g$, then $\{(\tfrac{1}{2}I + D^{t})f_{j}\}_{j \in \N}$ is a bounded sequence in $L^{2}(\sigma)$, and hence also $\{f_{j}\}_{j \in \N}$ is bounded in $L^{2}(\sigma)$ by \eqref{form4}. Therefore, one may pick a subsequence converging weakly to some $f \in L^{2}(\sigma)$, and now it is easy to check (using duality) that $(\tfrac{1}{2}I + D^{t})f = g$. \end{remark}
 
Since $\tfrac{1}{2}I + D^{t}$ is a bounded operator on $L^{2}(\sigma)$, it suffices to establish \eqref{form4} for \emph{a priori} $f \in C^{\infty}_{c}(\He)$. Fix such a function $f$. We start by writing
 \begin{displaymath} \|f\|_{L^{2}(\sigma)} \leq \|(\tfrac{1}{2}I + D^{t})f\|_{L^{2}(\sigma)} + \|(-\tfrac{1}{2}I + D^{t})f\|_{L^{2}(\sigma)}, \end{displaymath}
by the triangle inequality. So, to establish \eqref{form4}, it will suffice to prove that
 \begin{align} \|(-\tfrac{1}{2}I + D^{t})f\|_{L^{2}(\sigma)} \lesssim \|(\tfrac{1}{2}I + D^{t})f\|_{L^{2}(\sigma)} & + \|f\|_{L^{2}(\sigma)}^{1/2}\|(\tfrac{1}{2}I + D^{t})f\|_{L^{2}(\sigma)}^{1/2} \notag\\
 &\label{form2} + \|f\|_{L^{2}(\sigma)}^{3/4}\|(\tfrac{1}{2}I + D^{t})f\|_{L^{2}(\sigma)}^{1/4}. \end{align}
 Formally, one should also prove the same inequality with the roles of "$+$" and "$-$" reversed, but this is clear by symmetry. To prove \eqref{form2}, we write $u := \mathcal{S}f$, and recall from the jump relations in Corollary \ref{c:allJumps} that
 \begin{equation}\label{form17} (\tfrac{1}{2}I + D^{t})f = \nabla_{\nu}^{-}u \quad \text{and} \quad  (-\tfrac{1}{2}I + D^{t})f = \nabla_{\nu}^{+}u. \end{equation} 
As in Corollary \ref{c:allJumps}, the notation $\nabla_{\nu}^{+}u(p)$ and $\nabla^{-}_{\nu}u(p)$ refer to the interior and exterior, respectively, non-tangential limits of $q \mapsto \langle \nabla u(q),\nu(p) \rangle$. In particular, they coincide $\sigma$ a.e. with the radial limits
 \begin{displaymath} \nabla_{\nu}^{+}u(p) = \lim_{r \to 0+} \langle \nabla u(p \cdot (-r,0,0)),\nu(p) \rangle \quad \text{and} \quad \nabla_{\nu}^{-}u(p) = \lim_{r \to 0+} \langle \nabla u(p \cdot (r,0,0)),\nu(p) \rangle. \end{displaymath}
We also recall from Corollary \ref{c:allJumps} that
 \begin{displaymath} \nabla_{\tau}u(p) := \mathop{\lim_{q \to p}}_{q \in \He \, \setminus \, \partial \Omega} \langle \nabla u(q),\tau(p) \rangle = \mathrm{p.v.} \int \langle \nabla G(q^{-1} \cdot p),\tau(p) \rangle f(q) \, d\sigma(q), \end{displaymath} 
for $\sigma$ a.e. $p \in \partial \Omega$, where $\tau(p)$ is the horizontal tangent vector at $p \in \partial \Omega$, introduced in Definition \ref{HTangent}. The proof of \eqref{form2} consists of two parts. First, we show that
 \begin{equation}\label{form3} \|\nabla^{+}_{\nu}u\|_{L^{2}(\sigma)} \lesssim \|\nabla_{\tau}u\|_{L^{2}(\sigma)} + \|f\|_{L^{2}(\sigma)}^{1/2}\|T^{1/2}u\|_{L^{2}(\sigma)}^{1/2},  \end{equation}
 (the interpretation of $\|T^{1/2}u\|_{L^{2}(\sigma)}$ will be clarified in \eqref{form56}), and then
 \begin{equation}\label{form5} \begin{cases} \|T^{1/2}u\|_{L^{2}(\sigma)} \lesssim \|f\|_{L^{2}(\sigma)}^{1/2}\|\nabla_{\nu}^{-}u\|_{L^{2}(\sigma)}^{1/2}. \\ \|\nabla_{\tau}u\|_{L^{2}(\sigma)} \lesssim \|\nabla_{\nu}^{-}u\|_{L^{2}(\sigma)} +  \|f\|_{L^{2}(\sigma)}^{1/2}\|\nabla_{\nu}^{-}u\|_{L^{2}(\sigma)}^{1/2}. \end{cases} \end{equation}
 Combining the estimates \eqref{form3}-\eqref{form5}, we see that
 \begin{displaymath} \|\nabla^{+}_{\nu}u\|_{L^{2}(\sigma)} \lesssim \|\nabla_{\nu}^{-}u\|_{L^{2}(\sigma)} +  \|f\|_{L^{2}(\sigma)}^{1/2}\|\nabla_{\nu}^{-}u\|_{L^{2}(\sigma)}^{1/2} + \|f\|^{3/4}_{L^{2}(\sigma)}\|\nabla_{\nu}^{-}u\|_{L^{2}(\sigma)}^{1/4}, \end{displaymath}
 which is equivalent to \eqref{form2} by \eqref{form17}.
 
 \subsubsection{Proof of \eqref{form3}} Recall that $\Omega = U \times \R$, where $U = \{(x,y) \in \R^{2} : x < A(y)\}$, and $A \colon \R \to \R$ is Lipschitz. We begin by introducing a sequence of auxiliary domains $\overline{\Omega}_{j} \subset \Omega$, which are, in fact, just translates of $\Omega$. For $j \in \N$, we define the map $\Phi_{j} \colon \He \to \He$ as the right translation $\Phi_{j}(p) = p \cdot (-2^{-j},0,0)$. Then, we set
 \begin{displaymath} \Omega_{j} := \Phi_{j}(\Omega). \end{displaymath} 
 We list some basic properties of the domains $\Omega_{j}$ and the maps $\Phi_{j}$:
 \begin{itemize}
 \item[(a)] $\Omega_{j}$ is also a flag domain, indeed $\Omega_{j} = \{(x,y,t) : x < A(y) - 2^{-j}\}$, so in particular the horizontal normals $\nu_{j}$ and tangents $\tau_{j}$ are well-defined (see Section \ref{s:normals}).
 \item[(b)] The following relation holds between the (coordinates of the) horizontal normals and tangents of $\partial \Omega$ and $\partial \Omega_{j}$: 
 \begin{equation}\label{form86} \nu_{j}(\Phi_{j}(p)) = \nu(p) \quad \text{and} \quad \tau_{j}(\Phi_{j}(p)) = \tau(p) \qquad \text{for $\sigma$ a.e. $p \in \partial \Omega$.}\end{equation}
 \item[(c)] The $\Phi_{j}$-push-forward of $\sigma$ coincides with $\sigma_{j} = \mathcal{H}^{2}|_{\partial \Omega_{j}}$, that is, $\sigma_{j} = \Phi_{j\sharp}\sigma$.
 \end{itemize}
Properties (b) and (c) use essentially the flag domain assumption, which guarantees that left- and right-invariant horizontal perimeter measures coincide (with Euclidean $\calH^{2}$-measure), and the coordinates of the left and right horizontal normals also coincide (as vectors in $S^{1}$, recall Proposition \ref{prop3}). With this in mind, (b) and (c) follow by noting that the right translation $\Phi_{j}$ preserves right-invariant objects, which in the flag domain setting happen to coincide with their left-invariant counterparts.

\begin{remark} Why do we use the right translation $\Phi_{j}$ here? The reason is that $\nabla u(\Phi_{j}(p))$ is dominated by radial maximal function $\mathcal{N}_{\mathrm{rad}}(\nabla u)(p)$ defined in Section \ref{s:nt}, and the same would not (necessarily) be true of $u((-2^{-j},0,0) \cdot p)$.  \end{remark}

 Armed with the properties (a)-(c), we start from the LHS of \eqref{form3}. The first equation below is by dominated convergence, noting that $|\nabla u(\Phi_{j}(p))| \leq \mathcal{N}_{\mathrm{rad}}(\nabla u)(p)$ for $j \geq 1$:
 \begin{align} \|\nabla_{\nu}^{+}u\|_{L^{2}(\sigma)}^{2} & = \lim_{j \to \infty} \int |\langle \nabla  u(\Phi_{j}(p)), \nu(p) \rangle|^{2} \, d\sigma(p) \notag\\
 & \stackrel{\textup{(b)}}{=} \lim_{j \to \infty} \int |\langle \nabla u(\Phi_{j}(p)),\nu_{j}(\Phi_{j}(p)) \rangle|^{2} \, d\sigma(p) \notag\\
 &\label{form13} \stackrel{\textup{(c)}} = \lim_{j \to \infty} \int |\langle \nabla u(p),\nu_{j}(p) \rangle|^{2} \, d\sigma_{j}(p) = \lim_{j \to \infty} \|\langle \nabla u,\nu_{j} \rangle\|_{L^{2}(\sigma_{j})}^{2}. \end{align} 
There are two key reasons to work in the domains $\Omega_{j}$ instead of $\Omega$: first, $u \in C^{\infty}(\overline{\Omega}_{j})$, so the divergence theorem, Theorem \ref{t:div} and its corollaries in Section \ref{s:div} becomes applicable. Second, functions of the form $Z(T^{\alpha}u)$ are well-defined and smooth in $\overline{\Omega}_{j}$.

We start by applying the Rellich identity, Proposition \ref{p:rellich}, to the function $u = \mathcal{S}f$, so let us check the conditions (1)-(3). The harmonicity of $u$ in $\Omega_{j}$ is clear. Second, we have $\nabla u \in L^{2}(\sigma_{j}) \cap L^{2}(\Omega_{j})$, recalling that $\nabla u(p) = \int (\nabla G)(q^{-1} \cdot p)f(q) \, d\sigma(q)$, where $f \in C^{\infty}_{c}(\He)$. This implies that $\nabla u$ is bounded on $\overline{\Omega}_{j}$ with $|\nabla u(p)| \lesssim_{f,j} \min\{1,\|p\|^{-3}\}$ for $p \in \overline{\Omega}_{j}$. Finally, if $Z_{1},Z_{2},Z_{3} \in \{X,Y\}$, we note that $Z_{2}Z_{3}u$ is bounded on $\overline{\Omega}_{j}$ with $|(Z_{1}Z_{2}u)(p)| \lesssim \|p\|^{-4}$ for $p \in \overline{\Omega}_{j}$. Consequently 
 \begin{displaymath} |(Z_{1}u)(p)(Z_{2}Z_{3}u)(p)| \lesssim_{f,j} \min\{1,\|p\|^{-7}\}, \qquad p \in \overline{\Omega}_{j}, \end{displaymath}
 so certainly $(Z_{1}u)(Z_{2}Z_{3}u) \in L^{1}(\Omega_{j})$.
 
 With the conditions verified, we apply \eqref{rellich} to deduce that
 \begin{equation}\label{form9} \int |\nabla u|^{2} \langle X,\nu_{j} \rangle \, d\sigma_{j} = 2\int  Xu\langle \nabla u,\nu_{j} \rangle \, d\sigma_{j} - 2 \int_{\Omega_{j}} Yu \cdot Tu \, dp. \end{equation}
To benefit from this information, we expand $|\nabla u|^{2}$ in the $\{\nu_{j},\tau_{j}\}$-basis as
\begin{equation}\label{form8} |\nabla u(p)|^{2} = \langle \nabla u(p),\nu_{j}(p) \rangle^{2} + \langle \nabla u(p),\tau_{j}(p) \rangle^{2}.   \end{equation} 
The equation \eqref{form8} follows from the fact that $\{\tau_{j}(p),\nu_{j}(p)\}$ is an orthonormal basis for $\spa\{X_{p},Y_{p}\}$, so $\nabla u  = \langle \nabla u,\nu_{j} \rangle\nu_{j} + \langle \nabla u,\tau_{j} \rangle \tau_{j}$. We may also use this expansion to write
\begin{equation}\label{form10} Xu = \langle \nabla u,X \rangle = \langle \nabla u, \nu_{j} \rangle\langle X,\nu_{j} \rangle + \langle \nabla u,\tau_{j} \rangle\langle X,\tau_{j} \rangle. \end{equation}
Now, combining \eqref{form9}, \eqref{form8}, and \eqref{form10}, we obtain the identity
\begin{align*} \int & \Big( |\langle \nabla u,\nu_{j} \rangle|^{2} + |\langle \nabla u,\tau_{j} \rangle|^{2} \Big)\langle X,\nu_{j} \rangle \, d\sigma_{j} = -2 \int_{\Omega_{j}} Yu \cdot Tu \,dp \\
& \quad + 2\int \Big( |\langle \nabla u,\nu_{j} \rangle|^{2} \langle X,\nu_{j} \rangle + \langle \nabla u, \tau_{j} \rangle \langle X,\tau_{j} \rangle\langle \nabla u,\nu_{j} \rangle \,  \Big) \, d\sigma_{j}. \end{align*}
Rearranging terms, we obtain
\begin{align} \int |\langle \nabla u,\nu_{j} \rangle|^{2} \langle X,\nu_{j} \rangle \, d\sigma_{j} & = \int |\langle \nabla u,\tau_{j} \rangle|^{2}\langle X,\nu_{j} \rangle \, d\sigma_{j} \notag\\
&\label{form11} \qquad - 2 \int \langle \nabla u,\nu_{j} \rangle \langle \nabla u, \tau_{j} \rangle \langle X,\tau_{j} \rangle \, d\sigma_{j}\\
& \qquad + 2 \int_{\Omega_{j}} Yu \cdot Tu \,dp. \notag \end{align}
To proceed, we remark that the dot product $p \mapsto \langle X,\nu_{j}(p) \rangle$ has constant negative sign $\sigma_{j}$ a.e. on $\partial \Omega_{j}$ and is bounded from below in absolute value: indeed $\langle X,\nu_{j} \rangle = (1,0) \cdot (n_{1},n_{2})$, where
\begin{displaymath} (n_{1}(y,A(y) - 2^{-j}),n_{2}(y,A(y) - 2^{-j})) = \left(\frac{-1}{\sqrt{1 + |A'(y)|^{2}}},\frac{A'(y)}{\sqrt{1 + |A'(y)|^{2}}} \right), \quad y \in \R. \end{displaymath}
Therefore, taking absolute values of \eqref{form11} (and using Cauchy-Schwarz if the middle term happens to dominate), we obtain 
\begin{equation}\label{form12} \|\langle \nabla u,\nu_{j} \rangle\|_{L^{2}(\sigma_{j})}^{2} \lesssim \|\langle \nabla u,\tau_{j}\rangle\|_{L^{2}(\sigma_{j})}^{2} + \left|\int_{\Omega_{j}} Yu \cdot Tu \, dp \right|. \end{equation}
%\textcolor{cyan}{Here are the missing details for the middle term. Consider if they should be added. Assume that among the three terms in \eqref{form11}, the middle one has largest absolute value. Then, using that $-\langle X,\nu_{j} \rangle \sim 1$,
%\begin{align*} \int |\langle \nabla u,\nu_{j} \rangle|^{2} \, d\sigma_{j} & \sim \left| \int |\langle \nabla u,\nu_{j} \rangle|^{2} \langle X,\nu_{j} \rangle \, d\sigma_{j} \right| \lesssim \int |\langle \nabla u,\nu_{j} \rangle||\langle \nabla u,\tau_{j} \rangle| \, d\sigma_{j}\\
%& \leq  \left( \int |\langle \nabla u,\nu_{j} \rangle|^{2} \, d\sigma_{j} \right)^{1/2} \left( \int |\langle \nabla u,\tau_{j} \rangle|^{2} \, d\sigma_{j} \right)^{1/2}. \end{align*}
%Now divide by the first factor on the RHS to end up with \eqref{form12}.}
For later reference, we note that the roles of $\tau_{j}$ and $\nu_{j}$ could be reversed, by the same argument:
\begin{equation}\label{form66} \|\langle \nabla u,\tau_{j} \rangle\|_{L^{2}(\sigma_{j})}^{2} \lesssim \|\langle \nabla u,\nu_{j}\rangle\|_{L^{2}(\sigma_{j})}^{2} + \left|\int_{\Omega_{j}} Yu \cdot Tu \, dp \right|. \end{equation}
Comparing \eqref{form12} with \eqref{form13}, we have shown that
\begin{equation}\label{form14} \|\nabla_{\nu}^{+}u\|_{L^{2}(\sigma)}^{2} \lesssim \liminf_{j \to \infty} \left[ \|\langle \nabla u,\tau_{j}\rangle\|_{L^{2}(\sigma_{j})}^{2} + \left|\int_{\Omega_{j}} Yu \cdot Tu \, dp \right| \right]. \end{equation} 
For the first term on the RHS in \eqref{form14}, the limit exists and equals $\|\nabla_{\tau}u\|_{L^{2}(\sigma)}$. In fact,
\begin{displaymath} \|\langle \nabla u,\tau_{j}\rangle\|_{L^{2}(\sigma_{j})}^{2} = \int_{\partial \Omega} |\langle \nabla u(\Phi_{j}(p)),\tau(p) \rangle|^{2} \, d\sigma(p),   \end{displaymath} 
by the properties (b) and (c) of the domains $\Omega_{j}$. Furthermore, 
\begin{displaymath} |\nabla u(\Phi_{j}(p))| = |\nabla u(p \cdot (-2^{-j},0,0))| \leq \mathcal{N}_{\mathrm{rad}}(\nabla u)(p), \qquad p \in \partial \Omega. \end{displaymath}
Since $\mathcal{N}_{\mathrm{rad}}(\nabla u) \in L^{2}(\sigma)$ by Corollary \ref{K-MaxNT}, it follows from dominated convergence that
\begin{equation}\label{form18} \lim_{j \to \infty} \|\langle \nabla u,\tau_{j}\rangle\|_{L^{2}(\sigma_{j})}^{2} = \|\nabla_{\tau}u\|_{L^{2}(\sigma)}^{2}. \end{equation} 
We then consider the second term on the RHS of \eqref{form14}. First a remark about notation:
\begin{remark} Next, we will use the formalism of vertical Fourier multipliers introduced in Section \ref{tauMult}. We will only need $\mathfrak{m}_{\alpha}(\tau) = |\tau|^{\alpha}$, with $\alpha \geq 0$, and $\mathfrak{m}_{\sgn}(\tau) = |\tau|/\tau$, so we assign special notation for these two types of multiplier:
\begin{displaymath} T^{\alpha}u := \mathfrak{m}_{\alpha}(T)u \quad \text{and} \quad Hu := \mathfrak{m}_{\sgn}(T)u, \qquad u \in \mathcal{G}(\Omega). \end{displaymath}
The letter "$H$" stands for "Hilbert transform" (in the vertical direction). It may be worth noting that $T^{1}u \neq Tu$, in fact $T^{1}u = THu$. Propositions \ref{prop2} and \ref{prop1} imply that $T^{\alpha}u,Hu \in C^{\infty}(\Omega) \cap \mathcal{G}(\Omega)$, and $Z(T^{\alpha}u) = T^{\alpha}(Zu)$ in $\Omega$ for horizontal derivatives $Z$.
\end{remark}

 Here is the estimate for the second term on the RHS of \eqref{form14}:
\begin{lemma}\label{lemma4} With the previous notation,
\begin{displaymath}\left| \int_{\Omega_{j}} Yu \cdot Tu \, dp \right| \lesssim \|T^{1/2}u\|_{L^{2}(\sigma_{j})}\|\langle \nabla u,\nu_{j} \rangle\|_{L^{2}(\sigma_{j})}^{1/2}\|\nabla u\|_{L^{2}(\sigma_{j})}^{1/2}, \quad j \in \N. \end{displaymath}
\end{lemma}

\begin{proof}[Proof of Lemma \ref{lemma4}] Let us start by mentioning that $Yu \cdot Tu \in L^{1}(\Omega_{j})$, since $Yu$ and $Tu$ are bounded functions on $\Omega_{j}$ with $|Yu(p)| \lesssim \|p\|^{-3}$ and $|Tu(p)| \lesssim \|p\|^{-4}$ for $p \in \Omega_{j}$. So, it is legitimate to use Fubini's theorem and write
\begin{equation}\label{form19} \int_{\Omega_{j}} Yu \cdot Tu \, dp = \int_{U_{j}} \int_{\R} Yu(z,t) \cdot Tu(z,t) \, dt \, dz, \end{equation} 
where $U_{j} = \{(x,y) : x < A(y) - 2^{-j}\}$. It is also clear that $t \mapsto Yu(z,t) \in L^{1}(\R)$ and $t \mapsto Tu(z,t) \in L^{1}(\R)$ for $z \in U_{j}$ fixed, so the Fourier transforms in the $t$-variable exist in the classical sense, and Plancherel's theorem may be applied:
\begin{equation}\label{form20} \int_{\R} Yu(z,t)Tu(z,t) \, dt \lesssim \left( \int_{\R} |\tau|^{1/2} |\widehat{Yu}(z,\tau)|^{2} \, d\tau \right)^{1/2} \cdot \left(\int_{\R} |\tau|^{3/2}|\hat{u}(z,\tau)|^{2}  \, d\tau\right)^{1/2} \end{equation}
for $z \in U_{j}$. Here, the notation $\widehat{\varphi}(z,\tau)$ refers to the Fourier transform of the function $t \mapsto \varphi(z,t)$ evaluated at $\tau$. Now, plugging \eqref{form20} to the RHS of \eqref{form19}, and using Cauchy-Schwarz, we find that
\begin{align*} \left| \int_{\Omega_{j}} Yu \cdot Tu \, dp \right| & \lesssim \left(\int_{U_{j}} \int_{\R} |\tau|^{1/2}|\widehat{\nabla u}(z,\tau)|^{2} \, d\tau \, dz \right)^{1/2}\\
& \qquad \times \left(\int_{U_{j}} \int_{\R} |\tau|^{3/2}|\hat{u}(z,\tau)|^{2} \, d\tau \, dz \right)^{1/2} =: I_{1} \times I_{2}. \end{align*}
We will show that
\begin{equation}\label{form21} I_{1} \leq \|T^{1/2}u\|_{L^{2}(\sigma_{j})}^{1/2}\|\langle \nabla u,\nu_{j} \rangle\|_{L^{2}(\sigma_{j})}^{1/2} \end{equation}
and
\begin{equation}\label{form37} I_{2} \lesssim \|T^{1/2}u\|_{L^{2}(\sigma_{j})}^{1/2}\|\nabla u\|_{L^{2}(\sigma_{j})}^{1/2}, \end{equation}
which combined prove the lemma. Note that \eqref{form21} is slightly sharper than \eqref{form37}. We begin with \eqref{form21}, first applying Plancherel in the $\tau$-variable to find that $I_{1}$ equals the $L^{2}(\Omega_{j})$-norm of $T^{1/4}(\nabla u)$. Then, we apply Propositions \ref{prop2} and \ref{prop1} to deduce that $T^{1/4}(\nabla u) = \nabla (T^{1/4}u)$ pointwise in $\Omega$, so in fact
\begin{displaymath} I_{1} = \left(\int_{\Omega_{j}} |\nabla (T^{1/4}u)(p)|^{2} \, dp \right)^{1/2}. \end{displaymath} 
From Corollary \ref{cor2}, we know that $\bigtriangleup^{\flat}(T^{1/4}u)(p) = 0$ for $p \in \Omega$, hence
\begin{equation}\label{form49} |\nabla (T^{1/4}u)(p)|^{2} = \mathrm{div}^{\flat}[(T^{1/4}u) \cdot \nabla (T^{1/4}u)](p), \qquad p \in \Omega. \end{equation}
With this in mind, we define the vector field
\begin{displaymath} Z := (T^{1/4}u) \cdot \nabla (T^{1/4}u). \end{displaymath} 
To apply the divergence theorem, Theorem \ref{t:div}, we need to know the following facts:
\begin{itemize}
\item $\mathrm{div}^{\flat} Z \in L^{1}(\Omega_{j})$, because of the relation \eqref{form49}, and because $\nabla (T^{1/4}u) \in L^{2}(\Omega_{j})$. This is a special case of Corollary \ref{cor3}.
\item $Z \in L^{1}(\Omega_{j}) \cap L^{1}(\sigma_{j})$ by first using Cauchy-Schwarz, and then using Corollary \ref{cor3} to deduce that that $T^{1/4}u \in L^{2}(\mathfrak{M})$ and $\nabla (T^{1/4}u) \in L^{2}(\mathfrak{M})$ with either
\begin{displaymath} \mathfrak{M} = \mathcal{L}^{2}|_{U_{j}} \times \mathcal{L}^{1} = \mathcal{L}^{3}|_{\Omega_{j}} \quad \text{or} \quad \mathfrak{M} = c\mathcal{H}^{1}|_{\partial U_{j}} \times \mathcal{L}^{1} = \sigma_{j}. \end{displaymath}
\end{itemize}
After these preliminaries, Theorem \ref{t:div} shows that
\begin{align}\label{form50} \int_{\Omega_{j}} |\nabla (T^{1/4}u)(p)|^{2} \, dp & = \int (T^{1/4}u) \cdot \langle \nabla (T^{1/4}u),\nu_{j} \rangle \, d\sigma_{j}\\
& = c\int_{\partial U_{j}} \int_{\R} (T^{1/4}u)(z,t)[T^{1/4}(\nabla u)(z,t) \cdot n_{j}(z)] \, dt  \, d\calH^{1}(z). \notag \end{align} 
To arrive at the second line, we used again Proposition \ref{prop2} to infer that $\nabla (T^{1/4}u)(p) = T^{1/4}(\nabla u)(p)$ for $p \in \partial \Omega_{j} \subset \Omega$. Then we recalled from \eqref{form45} that $\nu_{j}(z,t) = n_{j}(z)$, and also used that $\sigma_{j} = c\calH^{1}|_{\partial U_{j}} \times \mathcal{L}^{1}$. Next, fix $z \in \partial U_{j}$, and note that the dot product with $n_{j}(z)$ contributes only a multiplication by a constant relative to the $t$-integration. Hence, we may use Plancherel's theorem twice to deduce that
\begin{align*} \int_{\R} (T^{1/4}u)(z,t)[T^{1/4}(\nabla u)(z,t) \cdot n_{j}(z)] \, dt & = \int_{\R} (|\tau|^{1/2}\hat{u}(z,\tau))[\widehat{\nabla u}(z,\tau) \cdot n_{j}(z)] \, d\tau\\
& = \int_{\R} (T^{1/2}u(z,t))\langle \nabla u(z,t),\nu_{j}(z) \rangle \, dt. \end{align*} 
Integrating over $z \in \partial U_{j}$, and using Cauchy-Schwarz, \eqref{form21} follows.

Next, we prove the estimate \eqref{form37} for the factor $I_{2}$, repeated below:
\begin{displaymath} I_{2} = \left(\int_{U_{j}} \int_{\R} |\tau|^{3/2}|\hat{u}(z,\tau)|^{2} \, d\tau \, dz \right)^{1/2} \lesssim \|T^{1/2}u\|_{L^{2}(\sigma_{j})}^{1/2}\|\nabla u\|_{L^{2}(\sigma_{j})}^{1/2}. \end{displaymath} 
We start by fixing $z \in U_{j}$ and re-writing the inner integral as 
\begin{displaymath} \int_{\R} |\tau|^{1/2}\hat{u}(z,\tau)[\tau \cdot \tfrac{|\tau|}{\tau} \hat{u}(z,\tau)] \, d\tau = (2\pi i)\int_{\R} (T^{1/2}u)(z,t) \cdot T(Hu)(z,t) \, dt, \end{displaymath} 
using Plancherel. Next, we expand $T = XY - YX$, and estimate separately the terms
\begin{equation}\label{form172} I_{21} := \int_{\Omega_{j}} (T^{1/2}u)(p) \cdot XY(Hu)(p) \, dp \quad \text{and} \quad I_{22} := \int_{\Omega_{j}} (T^{1/2}u)(p) \cdot YX(Hu)(p) \, dp. \end{equation}
Formally, we use the integration-by-parts formula \eqref{eq:iParts} to the vector field $Z = X$ and the functions
\begin{displaymath} \varphi = T^{1/2}u \in C^{\infty}(\overline{\Omega}_{j}) \quad \text{and} \quad \psi = YHu \in C^{\infty}(\overline{\Omega}_{j}) \end{displaymath}
 as follows:
\begin{equation}\label{form51} I_{21}  = -\int_{\Omega_{j}} (XT^{1/2}u)(p) \cdot (YHu)(p) \, dp - \int (T^{1/2}u)(YHu)\langle X,\nu_{j} \rangle \, d\sigma_{j}.\end{equation}
The justifications of using \eqref{eq:iParts} are very similar to the justifications needed to apply the divergence theorem at \eqref{form50}. The relevant facts are
\begin{displaymath} (X\varphi)\psi, \varphi X\psi \in L^{1}(\Omega_{j}) \quad \text{and} \quad \varphi \psi \in L^{1}(\sigma_{j}) \cap L^{1}(\Omega_{j}), \end{displaymath}
and each can be verified by applying Cauchy-Schwarz and Corollary \ref{cor3} (also note that $YHu = HYu$ and $Yu$ have the same norms in $L^{2}(\Omega_{j})$ and $L^{2}(\sigma_{j})$ by Plancherel).

With \eqref{form51} in hand, we may then estimate
\begin{equation}\label{form53} I_{21} \leq \left|\int_{\Omega_{j}} (XT^{1/2}u)(p) \cdot (YHu)(p) \, dp \right| + \|T^{1/2}u\|_{L^{2}(\sigma_{j})}\|\nabla u\|_{L^{2}(\sigma_{j})}. \end{equation}
To evaluate the first term on the RHS, we write $XT^{1/2}u = T^{1/2}Xu$ and $YHu = HYu$, using Proposition \ref{prop2}, and then apply Plancherel in the $t$-variable:
\begin{align*} \left| \int_{\Omega_{j}} (XT^{1/2}u)(p) \cdot (YHu)(p) \, dp \right| & \leq \int_{U_{j}} \int_{\R} |\tau|^{1/2} |\widehat{Xu}(z,\tau)| \cdot |\widehat{Yu}(z,\tau)| \, d\tau \, dz\\
& \leq \int_{U_{j}}\int_{\R} |\tau|^{1/2}|\widehat{\nabla u}(z,\tau)|^{2} \, d\tau \, dz.   \end{align*} 
Now, it remains to note that the RHS is the square of $I_{1}$, and we already obtained a good estimate for this factor in \eqref{form21}. It follows from this, and \eqref{form53} that
\begin{displaymath} I_{21} \lesssim \|T^{1/2}u\|_{L^{2}(\sigma_{j})}\|\nabla u\|_{L^{2}(\sigma_{j})}. \end{displaymath} 
The term $I_{22}$ from \eqref{form172} can be handled in the same way, just swapping the roles of $X$ and $Y$ in the calculations above. So, we have established \eqref{form37}, and hence completed the proof of Lemma \ref{lemma4}. \end{proof}
 
Recalling \eqref{form14} and \eqref{form18}, and noting also that
\begin{equation}\label{form65} \|\nabla u\|_{L^{2}(\sigma_{j})} = \int |\nabla u(\Phi_{j}(p))|^{2} \, d\sigma(p) \leq \int |\mathcal{N}_{\mathrm{rad}}(\nabla u)(p)|^{2} \, d\sigma(p) \lesssim \|f\|_{L^{2}(\sigma)} \end{equation}
for all $j \geq 1$, we may now deduce from Lemma \ref{lemma4} that
\begin{equation}\label{form22} \|\nabla_{\nu}^{+}u\|_{L^{2}(\sigma)}^{2} \lesssim \|\nabla_{\tau}u\|_{L^{2}(\sigma)} + \|f\|_{L^{2}(\sigma)}^{1/2} \cdot \liminf_{j \to \infty} \|T^{1/2}u\|_{L^{2}(\sigma_{j})}^{1/2}. \end{equation}
To proceed further, we need to know the limit on the RHS:
\begin{lemma}\ Let $f \in C_{c}^{\infty}(\He)$, and let $u := \mathcal{S}f$. Then,
\begin{equation}\label{form56} \lim_{j \to \infty} \|T^{1/2}u\|_{L^{2}(\sigma_{j})}^{2} = \int_{\partial \Omega} |\tau||\hat{u}(z,\tau)|^{2} \, d\sigma(z,\tau) =: \|T^{1/2}u\|_{L^{2}(\sigma)}^{2}. \end{equation}
\end{lemma}

\begin{remark} Although it is not strictly needed here, it may be clarifying to remark that $f \mapsto T^{1/2}u$ defines an operator bounded on $L^{2}(\sigma)$. This follows from Theorem \ref{t:bdd} and Proposition \ref{prop5}. We also recall here that $\sigma = \mathcal{H}^{2}_{E}|_{\partial \Omega} = c\mathcal{H}^{1}_{E}|_{\partial U} \times \mathcal{L}^{1}$, where $U = \{(A(y),y) : y \in \R\}$. Also, the proof below will show that the RHS of \eqref{form56} is finite. \end{remark}

\begin{proof} The RHS of \eqref{form56} is an integral of the vertical distributional Fourier transform of $u = \mathcal{S}f$, which is a tempered distribution in $\He$, and not only $\Omega$. As will be justified in Lemma \ref{l:Fourier1}, this Fourier transform turns out to be a function in $C^{\infty}(\He \, \setminus \, \{\tau = 0\})$, so its $\sigma$-integral makes sense. Let us also record the formula for $\hat{u}$ here:
\begin{displaymath} \widehat{u}(z,\tau) = c\int_{\partial U} e^{\pi i \omega(z,w)\tau} K_{0}(\tfrac{\pi}{2} |\tau| |z - w|^{2})\hat{f}(w,\tau) \, d\calH^{1}(w), \qquad \tau \in \R \, \setminus \, \{0\}, \end{displaymath}
where $\partial U = \{(A(y),y) : y \in \R\}$. Since $\hat{u} \in C^{\infty}(\He \, \setminus \, \{\tau = 0\})$, we have
\begin{displaymath} \int_{\partial \Omega} |\tau||\hat{u}(z,\tau)|^{2} \, d\sigma(z,\tau) = \int_{\partial \Omega} \lim_{j \to \infty} |\tau||\hat{u}(z - (2^{-j},0),\tau)|^{2} \, d\sigma(z,\tau). \end{displaymath} 
On the other hand, recalling that $\sigma_{j} = \calH^{2}|_{\partial \Omega_{j}}$, $\sigma = \calH^{2}|_{\partial \Omega}$, and $\partial \Omega_{j} = \partial \Omega - (2^{-j},0,0)$, we have, using Plancherel in the $t$-variable,
\begin{displaymath} \lim_{j \to \infty} \|T^{1/2}u\|_{L^{2}(\sigma_{j})}^{2} = \lim_{j \to \infty} \int_{\partial \Omega} |\tau||\hat{u}(z - (2^{-j},0),\tau)|^{2} \, d\sigma(z,\tau). \end{displaymath} 
So, proving \eqref{form56} boils down to justifying the use of the dominated convergence theorem, that is, verifying that
\begin{displaymath} \mathfrak{u}(z,\tau) := \sup_{0 \leq |\epsilon| \leq 1} |\tau||\hat{u}(z + \epsilon,\tau)|^{2} \in L^{1}(\sigma). \end{displaymath}
To accomplish this, we employ the decay estimates for $\hat{u}$ established in Lemma \ref{lemma3}, and which we copy here for the reader's convenience: for all $0 < \beta < \tfrac{1}{2}$, we have
\begin{displaymath} |\hat{u}(z,\tau)| \lesssim_{f} \begin{cases} \|\hat{f}(\cdot,\tau)\|_{L^{\infty}(\R^{2})}|\tau|^{-\beta}, & \text{for } (z,\tau) \in \He, \, \tau \neq 0, \\ e^{-\tfrac{\pi}{2}|\tau|\dist(z,K_{f})^{2}}, & \text{for } |\tau| \geq \dist(z,K_{f})^{-2}, \end{cases}  \end{displaymath}
where $K_{f} := \{w \in \R^{2} : f_{w} \neq \emptyset\}$. Since $K_{f}$ is a bounded set by the assumption $f \in C^{\infty}_{c}(\He)$, we assume with no loss of generality that $K_{f} \subset B(0,\tfrac{1}{2}) \subset \R^{2}$. We estimate the $L^{1}(\sigma)$-norm of $\mathfrak{u}$ as follows:
\begin{equation}\label{form42} \int \mathfrak{u} \, d\sigma = \int_{\R}\int_{\Gamma \cap B(2)} \mathfrak{u} \, d\sigma + \sum_{k \geq 1} \int_{\R} \int_{\Gamma \cap A_{k}} \mathfrak{u} \, d\sigma =: I_{0} + \sum_{k \geq 1} I_{k}. \end{equation} 
Here $B(r) := B(0,r) \subset \R^{2}$ and $A_{k} := B(2^{k +1}) \, \setminus \, \bar{B}(2^{k})$ for $k \geq 1$. The term $I_{0}$ is easy:
\begin{displaymath} I_{0} \lesssim_{f} \int_{\R} \|\hat{f}(\cdot,\tau)\|_{L^{\infty}(\R^{2})}^{2}|\tau|^{1 - 2\beta} \, d\tau < \infty, \end{displaymath} 
since $|\hat{f}(z,\tau)| \lesssim (1 + |\tau|)^{-2}$ by our assumption $f \in C^{\infty}_{c}(\He)$. Next, to estimate the series in \eqref{form42}, fix $k \geq 1$, and write further
\begin{align*} I_{k} & \lesssim_{f} \int_{\{|\tau| \leq 2^{-k + 3}\}} |\tau| \int_{\Gamma \cap A_{k}} |\tau|^{-\beta} \, d\sigma(z,\tau)\\
& \quad + \int_{\{|\tau| \geq 2^{-k + 3}\}} |\tau| \int_{\Gamma \cap A_{k}} \sup_{0 \leq |\epsilon| \leq 1} |\hat{u}(z + \epsilon,\tau)|^{2} \, d\sigma(z,\tau). \end{align*} 
The point of this splitting is that if $z \in A_{k}$, $k \geq 1$, and $|\tau| \geq 2^{-k + 3}$, then also
\begin{equation}\label{form40} |\tau| \geq 2^{2(2 - k)} \geq \sup_{0 \leq |\epsilon| \leq 1} \dist(z + \epsilon,K_{f})^{-2}, \end{equation}
which is needed for the second decay bound on $\hat{u}(z,\tau)$. With implicit constants depending on the Lipschitz constant of $\Gamma$, we first note that
\begin{displaymath}  \int_{\{|\tau| \leq 2^{-k + 3}\}} |\tau| \int_{\Gamma \cap A_{k}} |\tau|^{-\beta} \, d\sigma(z,\tau) \leq \calH^{1}(\Gamma \cap B(2^{k + 1})) \int_{\{|\tau| \leq 2^{-k + 3}\}} |\tau|^{1-\beta} \, d\tau \lesssim 2^{-(1 - \beta)k}, \end{displaymath}
which is certainly summable in $k$ whenever $0 < \beta < \tfrac{1}{2}$. To deal with the second term in the bound for $I_{k}$, we apply the second decay bound for $\hat{u}$, which is available by \eqref{form40}:
\begin{displaymath} \int_{\{|\tau| \geq 2^{-k + 3}\}} |\tau| \int_{\Gamma \cap A_{k}} \sup_{0 \leq |\epsilon| \leq 1} |\hat{u}(z + \epsilon,\tau)|^{2} \, d\sigma(z,\tau) \lesssim_{f} 2^{k} \cdot \int_{\{|\tau| \geq 2^{-k + 3}\}} |\tau| e^{-\tfrac{\pi}{2}|\tau|2^{2(k - 2)}} \, d\tau. \end{displaymath}
By making a change-of-variable $\tau \mapsto \xi \cdot 2^{-2k}$, it is easy to see that the terms on the RHS are summable over $k \geq 1$. This concludes the verification that $\mathfrak{u} \in L^{1}(\sigma)$, and proves the lemma. \end{proof}
Combining the previous lemma with \eqref{form22}, we have now shown that
\begin{displaymath}\|\nabla_{\nu}^{+}u\|_{L^{2}(\sigma)} \lesssim \|\nabla_{\tau}u\|_{L^{2}(\sigma)} + \|f\|_{L^{2}(\sigma)}^{1/2}\|T^{1/2}u\|_{L^{2}(\sigma)}^{1/2}, \end{displaymath}
as claimed in \eqref{form3}. 

\subsubsection{Proof of \eqref{form5}} Recall the claims:
\begin{equation}\label{form59} \begin{cases} \|T^{1/2}u\|_{L^{2}(\sigma)} \lesssim \|f\|_{L^{2}(\sigma)}^{1/2}\|\nabla_{\nu}^{-}u\|_{L^{2}(\sigma)}^{1/2}. \\ \|\nabla_{\tau}u\|_{L^{2}(\sigma)} \lesssim \|\nabla_{\nu}^{-}u\|_{L^{2}(\sigma)} + \|f\|_{L^{2}(\sigma)}^{1/2}\|\nabla_{\nu}^{-}u\|_{L^{2}(\sigma)}^{1/2}. \end{cases} \end{equation}
In the previous section, we extensively used the approximating domains $\Omega_{j} = U_{j} \times \R = \{(x,y,t) : x < A(y) - 2^{-j}\}$, which were images of $\Omega$ under the map $\Phi_{j}(p) = p \cdot (-2^{-j},0,0)$. In this section, we redefine the notation as follows:
\begin{displaymath} \Omega_{j} := U_{j} \times \R = \{(x,y,t) : x > A(y) + 2^{-j}\} \quad \text{and} \quad \Phi_{j}(p) = p \cdot (2^{-j},0,0). \end{displaymath}
We also denote the inward-pointing horizontal normal of $\Omega_{j}$ by $\nu_{j}$. In particular, the computations from \eqref{form13}, \eqref{form18}, and \eqref{form56} now show that
\begin{equation}\label{form60} \|\nabla_{\nu}^{-}u\|_{L^{2}(\sigma)} = \lim_{j \to \infty} \|\langle \nabla u,\nu_{j} \rangle\|_{L^{2}(\sigma_{j})}, \end{equation}
\begin{equation}\label{form61} \|\nabla_{\tau}u\|_{L^{2}(\sigma)} = \lim_{j \to \infty} \|\langle \nabla u,\tau_{j} \rangle\|_{L^{2}(\sigma_{j})}, \end{equation}
and
\begin{equation}\label{form62} \|T^{1/2}u\|_{L^{2}(\sigma)} = \lim_{j \to \infty} \|T^{1/2}u\|_{L^{2}(\sigma_{j})}. \end{equation}
We then begin proving the first part of \eqref{form59}. Motivated by \eqref{form62}, we start with the following lemma:
\begin{lemma} For $j \in \N$, we have
\begin{displaymath} \|T^{1/2}u\|_{L^{2}(\sigma_{j})} \lesssim \|\langle \nabla u,\nu_{j} \rangle\|_{L^{2}(\sigma_{j})}^{1/2} \cdot \|\nabla u\|_{L^{2}(\sigma_{j})}^{1/2}. \end{displaymath}
\end{lemma}

\begin{proof} We, again, use the observation that $\langle X,\nu_{j} \rangle$ has constant sign and stays bounded away from $0$ in absolute value. Therefore, we may estimate, and then use the divergence theorem (Theorem \ref{t:div}), as follows:
\begin{equation}\label{form63} \|T^{1/2}u\|_{L^{2}(\sigma_{j})}^{2} \lesssim \int \langle |T^{1/2}u|^{2}X,\nu_{j} \rangle \, d\sigma_{j} = 2 \left| \int_{\Omega_{j}} X(T^{1/2}u)(p) \cdot (T^{1/2}u)(p) \, dp \right|. \end{equation} 
The justification for the use of the divergence theorem is very similar to the one seen at \eqref{form50}, so we will not repeat the details; what needs to be checked is that
\begin{displaymath} T^{1/2}u \in L^{2}(\Omega_{j}) \cap L^{2}(\sigma_{j}) \quad \text{and} \quad X(T^{1/2}u) \cdot T^{1/2}u \in L^{1}(\Omega_{j}), \end{displaymath}
these facts can be inferred from Corollary \ref{cor3}. We then proceed with \eqref{form63} with the familiar tricks of commuting $X(T^{\alpha}u) = T^{\alpha}(Xu)$, and moving one-quarter $T$-derivative across (by applying Plancherel in the $t$-variable):
\begin{equation}\label{form64} \eqref{form63} \leq \|\nabla (T^{1/4}u)\|_{L^{2}(\Omega_{j})} \cdot \|T^{3/4}u\|_{L^{2}(\Omega_{j})}. \end{equation} 
But we have seen the factors on the RHS before, within the proof of Lemma \ref{lemma4}; they are exactly the factors $I_{1}$ and $I_{2}$, except that the domain $\Omega_{j}$ is now contained in $\He \, \setminus \, \overline{\Omega}$. Therefore, the (proof of the) estimates \eqref{form21}-\eqref{form37} show that
\begin{displaymath} \eqref{form64} \lesssim \|T^{1/2}u\|_{L^{2}(\sigma_{j})}\|\langle \nabla u,\nu_{j} \rangle\|_{L^{2}(\sigma_{j})}^{1/2}\|\nabla u\|_{L^{2}(\sigma_{j})}^{1/2}. \end{displaymath}
Plugging this estimate back into \eqref{form63}, and dividing by $\|T^{1/2}u\|_{L^{2}(\sigma_{j})}$ completes the proof of the lemma. \end{proof}
Combining the lemma with \eqref{form60}, \eqref{form62}, and $\|\nabla u\|_{L^{2}(\sigma_{j})} \lesssim \|f\|_{L^{2}(\sigma)}$ (as observed in \eqref{form65}), we arrive at the first part of \eqref{form59}. To prove the second part, we apply a corollary of the Rellich identity, which we already recorded in \eqref{form66}:
\begin{displaymath} \|\langle \nabla u,\tau_{j} \rangle\|_{L^{2}(\sigma_{j})}^{2} \lesssim \|\langle \nabla u,\nu_{j}\rangle\|_{L^{2}(\sigma_{j})}^{2} + \left|\int_{\Omega_{j}} Yu \cdot Tu \, dp \right|, \qquad j \geq 1. \end{displaymath}
Consequently, applying \eqref{form60}-\eqref{form61}, we arrive at
\begin{displaymath} \|\nabla_{\tau}u\|_{L^{2}(\sigma)}^{2} \lesssim \|\nabla_{\nu}^{-}u\|_{L^{2}(\sigma)}^{2} + \liminf_{j \to \infty} \left| \int_{\Omega_{j}} Yu \cdot Tu \, dp \right|, \end{displaymath}
and finally an application of Lemma \ref{lemma4}, \eqref{form60}, and \eqref{form62} shows that
\begin{displaymath} \liminf_{j \to \infty} \left| \int_{\Omega_{j}} Yu \cdot Tu \, dp \right| \lesssim \|T^{1/2}u\|_{L^{2}(\sigma)}\|\nabla_{\nu}^{-}u\|_{L^{2}(\sigma)}^{1/2}\|f\|_{L^{2}(\sigma)}^{1/2}. \end{displaymath}
To complete the proof the second part of \eqref{form59}, it remains to estimate the factor $\|T^{1/2}u\|_{L^{2}(\sigma)}$ by the first part of \eqref{form59}.

\subsubsection*{Summary} We have now completed the proofs of the estimates \eqref{form3}-\eqref{form5}. As explained right after these estimates were first claimed, they imply \eqref{form2}, and hence Theorem \ref{t:injectivity}.

%%%%%%%%%%%%%%%%%%%%%%%%%%%%%%%%%%%%%%%%

\section{The operators $\tfrac{1}{2}I \pm D^{t}$ are surjective}\label{s:surjectivity}

If $\Omega = \{(x,y,t) : x < A(y)\}$ is a flag domain, we denote by $D^{t}_{A}$ the adjoint double layer potential associated to $\partial \Omega$, or more precisely $\sigma := |\partial \Omega|_{\He}$. In this section, we prove that the operators $\tfrac{1}{2}I \pm D^{t}_{A}$ are surjective for all Lipschitz functions $A \colon \R \to \R$, and we begin by treating the case $A \equiv 0$. Then, $\partial \Omega = \W = \{(0,y,t) : y,t \in \R\}$.
\begin{proposition}\label{p:invert} The operators $\tfrac{1}{2}I \pm D^{t}_{0}$ are invertible on $L^{2}(\W)$.
\end{proposition} 
\begin{proof} The inward horizontal normal $\{(x,y,t) : x < 0\}$ is the constant vector $(-1,0)$ (in the $\{X,Y\}$-basis). Therefore,
\begin{displaymath} D^{t}_{0}f(p) = \mathrm{p.v. } \int -XG(q^{-1} \cdot p)f(q) \, d\sigma(q).  \end{displaymath} 
Recall that $G(z,t) = c\|(z,t)\|^{-2} = c(|z|^{4} + t^{2})^{-1/2}$ for some positive constant $c > 0$. It follows by explicit computation that for $(x,y,t) \in \W \, \setminus \, \{0\}$, we have
\begin{displaymath} XG(x,y,t) = -\frac{8cyt}{\|(x,y,t)\|^{6}} = -\frac{8c(-y)(-t)}{\|(-x,-y,-t)\|^{6}} = XG(-x,-y,-t), \end{displaymath}
that is, the restriction of $XG$ to $\W$ is even. Therefore, using again that $\nu \equiv (-1,0)$,
\begin{displaymath} D^{t}_{0}f(p) = \mathrm{p.v. } \int -XG(p^{-1} \cdot q)f(q) \, d\sigma(q) = \mathrm{p.v. } \int \langle \nabla G(p^{-1} \cdot q),\nu(q) \rangle f(q) \, d\sigma(q) = D_{0}f(p), \end{displaymath}
and hence $\tfrac{1}{2}I \pm D_{0}^{t} = \tfrac{1}{2}I\pm D_{0}$. We already verified in the previous section that the operators $\tfrac{1}{2}I \pm D_{0}$ are surjective on $L^{2}(\sigma)$ (recall the discussion just above Theorem \ref{t:injectivity}). It follows that the operators $\tfrac{1}{2}I \pm D_{0}^{t}$ are surjective, and they are injective by Theorem \ref{t:injectivity}.  \end{proof}

To show that the operators $\tfrac{1}{2}I \pm D_{A}^{t}$ are invertible for a general Lipschitz function $A \colon \R \to \R$, we employ a standard tool, the \emph{method of continuity}. It is based on the following well-known lemma, whose proof we include for completeness:
\begin{lemma}\label{l:kenig} Let $B$ be a Banach space, and let $V$ be a normed space. Let $\{T_{s}\}_{s \in [0,1]} \subset \mathcal{B}(B,V)$ be a parametrised family of bounded linear operators satisfying
\begin{itemize}
\item[(a)] $\|T_{s}f\|_{V} \geq c\|f\|_{B}$ for all $f \in B$ and for some $c > 0$ independent of $s$,
\item[(b)] $s \mapsto T_{s}$ is continuous $[0,1] \to \mathcal{B}(B,V)$,
\item[(c)] $T_{0}$ is surjective.
\end{itemize}
Then $T_{1}$ is surjective $B \to V$.
\end{lemma}

\begin{proof} Note that $s \mapsto T_{s}$ is uniformly continuous on the compact set $[0,1]$, so there exists $\delta > 0$ such that 
\begin{displaymath} |s - t| < \delta \quad \Longrightarrow \quad \|T_{s} - T_{s}\|_{B \to V} < \tfrac{c}{2}. \end{displaymath} 
Split the interval $[0,1]$ into consecutive subintervals $I_{1},\ldots,I_{N}$ of length $< \delta/2$. We claim that if $T_{s}$ is surjective for some $s \in I_{j}$, then $T_{t}$ is surjective for all $t \in I_{j + 1}$. Since $T_{0}$ is surjective, this will imply that $T_{1}$ is surjective by induction.

To check the claim, fix $s \in I_{j}$ such that $T_{s}$ is surjective. Clearly $T_{s}$ is injective by (a), so $T_{s}^{-1} \colon V \to B$ exists, is linear, and $\|T_{s}^{-1}\|_{V \to B} \leq c^{-1}$. Fix $t \in I_{j + 1}$ (so $|s - t| < \delta$) and $g \in V$. Consider the map $\Psi_{g} \colon B \to B$,
\begin{displaymath} \Psi_{g}(f) := T_{s}^{-1}g + T_{s}^{-1}(T_{s} - T_{t})f, \qquad f \in B. \end{displaymath}
Since $\|T_{s} - T_{t}\|_{B \to V} < c/2$, we see that $\Phi_{g}$ is a contraction:
\begin{displaymath} \|\Psi_{g}f_{1} - \Psi_{g}f_{2}\|_{B} = \|T_{s}^{-1}(T_{s} - T_{t})(f_{1} - f_{2})\|_{B}  \leq \tfrac{1}{2}  \|f_{1} - f_{2}\|, \qquad f_{1},f_{2} \in B. \end{displaymath}
It follows from the Banach fixed point theroem that there exists $f \in B$ with $\Psi_{g}(f) = f$. This is equivalent to $T_{t}f = g$, so $T_{t}$ is surjective. \end{proof}

To apply the lemma, we define the Lipschitz maps $A_{s} \colon \R \to \R$ for all $s \in \R$ (not just $s \in [0,1]$) by $A_{s}(y) := sA(y)$, and the associated flag domains $\Omega_{s} := \{(x,y,t) : x < A_{s}(y)\}$. We abbreviate $D_{s}^{t} := D^{t}_{A_{s}}$. Lemma \ref{l:kenig} is not directly applicable to the operators $D^{t}_{s}$, as they are not defined on a common space "$B$". To fix this, we need to use the parametrisations of $\partial \Omega_{s}$ by $\R^{2}$, which we already encountered in Lemma \ref{lemma5}. Recall from Lemma \ref{lemma5} the map $\Gamma_{s} \colon \R^{2} \to \partial \Omega_{s}$ defined by
\begin{equation}\label{eq:Gammas} \Gamma_{s}(y,t) := \left(A_{s}(y),y,t - \tfrac{y}{2}A_{s}(y) + \int_{0}^{y} A_{s}(r) \, dr \right), \qquad (y,t) \in \R^{2}, \, s \in \R. \end{equation}
We then define the auxiliary operators $R_{s,\epsilon}f(w) := \langle \vec{R}_{s,\epsilon}f(w),\nu_{s}(\Gamma_{s}(w)) \rangle$, where
\begin{displaymath} \vec{R}_{s,\epsilon}f(w) := \int (\nabla G)_{\epsilon}(\Gamma_{s}(v)^{-1} \cdot \Gamma_{s}(w))J_{s}(v)f(v) \, dv, \qquad w \in \R^{2}, \, f \in L^{2}(\R^{2}).\end{displaymath}  
In the definition above,
\begin{itemize}
\item $\nu_{s}$ refers to the inward pointing horizontal normal of $\Omega_{s}$,
\item $(\nabla G)_{\epsilon} = \psi_{\epsilon}(p)\nabla G(p)$ is a smooth radial truncation of the kernel $\nabla G$, and
\item $J_{s}(y,t) := \sqrt{1 + A_{s}'(y)^{2}}$ for the "Jacobian" of the parametrisation $\Gamma_{s}$.
\end{itemize}
As we have seen in Proposition \ref{prop10}, principal values are not affected by the particular choice of (radial) truncations, so
\begin{displaymath} D_{s}^{t}g(p) = \lim_{\epsilon \to 0} \int \langle (\nabla G)_{\epsilon}(q^{-1} \cdot p),\nu_{s}(p)\rangle g(q) \, d\sigma_{s}(q), \qquad g \in L^{2}(\sigma_{s}), \end{displaymath}
for $\sigma_{s}$ a.e. $p \in \partial \Omega_{s}$. Consequently, using also the area formula \eqref{form16}, we may infer that the limit
\begin{equation}\label{form74} R_{s}f(w) := \lim_{\epsilon \to 0} R_{s,\epsilon}f(w) = D_{s}^{t}(f \circ \Gamma_{s}^{-1})(\Gamma_{s}(w)) \end{equation}
exists for all $f \in L^{2}(\R^{2})$ and for a.e. $w \in \R^{2}$. Also, noting that $J_{s} \sim_{A} 1$ for all $0 \leq s \leq 1$, it is clear from \eqref{form74}, and the area formula \eqref{form16}, that $R_{s}$ defines a bounded operator on $L^{2}(\R^{2})$, and $\tfrac{1}{2}I \pm D_{s}^{t}$ is invertible on $L^{2}(\sigma_{s})$ if and only if 
\begin{equation}\label{form184} T_{s} := \tfrac{1}{2}I \pm R_{s} \end{equation}
is invertible on $L^{2}(\R^{2})$ (indeed $T_{s}f(w) = (\tfrac{1}{2}I \pm D_{t})(f \circ \Gamma_{s}^{-1})(\Gamma_{s}(w))$ for a.e. $w \in \R^{2}$).

So, it remains to verify the conditions (a)-(c) of Lemma \ref{l:kenig} for the operators $T_{s}$. Theorem \ref{t:injectivity} applied to the flag domain $\Omega_{s}$ shows that (a) is satisfied, since 
\begin{displaymath} \int |T_{s}f(w)|^{2} \, dw = \int \left|(\tfrac{1}{2}I \pm D_{s}^{t})(f \circ \Gamma_{s}^{-1})(\Gamma_{s}(w))\right|^{2} \, dw \gtrsim_{A} \|f \circ \Gamma_{s}^{-1}\|_{L^{2}(\sigma_{s})}^{2} \sim_{A} \|f\|_{2}^{2} \end{displaymath}
for $f \in L^{2}(\R^{2})$. The invertibility of $T_{0}$ follows from Lemma \ref{p:invert}, since $\Gamma_{0} = \mathrm{Id}$ (under the identification $\R^{2} \cong \W$), and $T_{0} = \tfrac{1}{2}I \pm D^{t}_{0}$. So, it remains to verify condition (b), and we will prove \emph{a fortiori} that 
\begin{equation}\label{form190} \|T_{r} - T_{s}\|_{L^{2}(\R^{2}) \to L^{2}(\R^{2})} \lesssim_{A} |r - s|, \qquad r,s \in [0,1]. \end{equation}
This will follow if we can show that $\|R_{r,\epsilon} - R_{s,\epsilon}\|_{L^{2}(\R^{2}) \to L^{2}(\R^{2})} \lesssim_{A} |r - s|$ with bounds independent of $\epsilon > 0$. To show this, fix $f \in C^{\infty}_{c}(\R^{2})$, $w \in \R^{2}$, $\epsilon > 0$, and expand
\begin{align}\label{form80} R_{r,\epsilon}f(w) - R_{s,\epsilon}f(w) = & \left\langle \vec{R}_{s,\epsilon}f(w) - \vec{R}_{r,\epsilon}f(w), \nu_{s}(\Gamma_{s}(w)) \right\rangle\\
&\label{form81} \quad + \left\langle \vec{R}_{r,\epsilon}f(w), [\nu_{s}(\Gamma_{s}(w)) - \nu_{r}(\Gamma_{r}(w))] \right \rangle. \end{align} 
Writing $w := (y,t)$, we have $\Gamma_{s}(w) = (A_{s}(y),y,\tau)$ for some $\tau \in \R$ which will not affect our calculations. Since $\Omega_{r}$ and $\Omega_{s}$ are both flag domains, we can apply the formula \eqref{flagNormal} to find their horizontal normals, and the result is $\left| \nu_{s}(\Gamma_{s}(w)) - \nu_{r}(\Gamma_{r}(w)) \right| = $
\begin{displaymath} \left| \left(\tfrac{-1}{\sqrt{1 + s^{2}A'(y)^{2}}},\tfrac{sA'(y)}{\sqrt{1 + s^{2}A'(y)^{2}}} \right) -  \left(\tfrac{-1}{\sqrt{1 + r^{2}A'(y)^{2}}},\tfrac{rA'(y)}{\sqrt{1 + r^{2}A'(y)^{2}}} \right)\right| \lesssim_{A} |r - s|, \quad r,s \in [0,1]. \end{displaymath}
Therefore,
\begin{displaymath} \|\eqref{form81}\|_{2} \lesssim_{A} |r - s| \|f\|_{2}, \end{displaymath}
also using that $\|\vec{R}_{s,\epsilon}\|_{L^{2}(\R^{2}) \to L^{2}(\R^{2})} \leq C$ for some $C > 0$ independent of $s \in [0,1]$ or $\epsilon > 0$ (this follows by another application of the area formula \eqref{form16}, and the $L^{2}$-boundedness of the maximal Riesz transform on boundaries of flag domains, recall Theorem \ref{t:FO1}).

To show that also $\|\eqref{form80}\|_{2} \lesssim |r - s| \cdot \|f\|_{2}$, we need to introduce one more "cross term":
\begin{align}\label{form187} \vec{R}_{r,\epsilon}f(w) - \vec{R}_{s,\epsilon}f(w) & = \int (\nabla G)_{\epsilon}(\Gamma_{s}(v)^{-1} \cdot \Gamma_{s}(w))[J_{r}(v) - J_{s}(v)]f(v) \, dv\\
&\label{form186} + \int [(\nabla G)_{\epsilon}(\Gamma_{r}(v)^{-1} \cdot \Gamma_{r}(w)) - (\nabla G)_{\epsilon}(\Gamma_{s}(v)^{-1} \cdot \Gamma_{s}(w))]g_{r}(v) \, dv, \end{align}
where $g_{r} := J_{r}f$. Relying again on the $L^{2}$-boundedness of $\vec{R}_{s,\epsilon}$, we have
\begin{displaymath} \|\eqref{form187}\|_{2} \lesssim_{A} \|(J_{r} - J_{s})f\|_{2} \lesssim_{A} |r - s| \|f\|_{2}, \qquad r,s \in [0,1], \end{displaymath} 
using also that $\|J_{r} - J_{s}\|_{\infty} \lesssim_{A} |r - s|$ (recall $J_{s}(y,t) = \sqrt{1 + sA'(y)}$). To treat (the most difficult) term \eqref{form186}, we will eventually apply the following result:
\begin{thm}\label{t:sDerOp} Let $K \in C^{\infty}(\He)$ be a horizontally odd kernel satisfying
\begin{equation}\label{kernelConstants} |\nabla^{n}K(p)| \leq C_{n}\|p\|^{-3 - n}, \qquad p \in \He, \, n \geq 0. \end{equation}
Let $\Gamma_{s} \colon \R^{2} \to \partial \Omega_{s}$, $s \in \R$, be the maps from \eqref{eq:Gammas}. For every $s \in \R$, define the kernel $K_{s} \colon \R^{2} \times \R^{2} \to \C$,
\begin{displaymath} K_{s}(w,v) := \partial_{s}K(\Gamma_{s}(v)^{-1} \cdot \Gamma_{s}(w)), \qquad v,w \in \R^{2}, \end{displaymath}
Then, there exists a constant $C \geq 1$, depending only on the kernel constants "$C_{n}$" in \eqref{kernelConstants}, and the Lipschitz constant of $A$ such that 
\begin{displaymath} \int_{\R^{2}} \left| \int_{\R^{2}} K_{s}(w,v)f(v) \, dv \right|^{2} \, dw \leq C\|f\|_{2}^{2}, \qquad f \in L^{2}(\R^{2}), \, 0 \leq s \leq 1. \end{displaymath}
\end{thm}

Before proving the Theorem \ref{t:sDerOp}, let us apply it to treat the term \eqref{form186}. Theorem \ref{t:sDerOp}, applied with $K := (\nabla G)_{\epsilon} \in C^{\infty}(\He)$, shows that the operator
\begin{displaymath} \vec{\mathcal{R}}_{s,\epsilon}g(w) := \int \partial_{s}(\nabla G)_{\epsilon}(\Gamma_{s}^{-1}(v) \cdot \Gamma_{s}(w))g(v) \, dv, \qquad g \in L^{2}(\R^{2}), \end{displaymath}
satisfies $\|\vec{\mathcal{R}}_{s,\epsilon}g\|_{2} \lesssim \|g\|_{2}$ with implicit constant independent of $s \in [0,1]$ and $\epsilon > 0$. Moreover, for $g \in C^{\infty}_{c}(\R^{2})$ and $\epsilon > 0$ fixed, it is easy to check that
\begin{equation}\label{form87} \frac{\vec{R}_{s,\epsilon}g - \vec{R}_{r,\epsilon}g}{s - r} \rightharpoonup \vec{\mathcal{R}}_{s,\epsilon}g \end{equation}
as $r \to s$ (the harpoon refers to weak* convergence). We will give some details after the proof of Theorem \ref{t:sDerOp}, see Remark \ref{rem2}; by then we will have an explicit expression for $\vec{\mathcal{R}}_{s,\epsilon}g$, which will be helpful for verifying \eqref{form87}. Taking \eqref{form87} for granted, we find that if $\varphi \in C^{\infty}_{c}(\R^{2})$ is a fixed test function with $\|\varphi\|_{2} = 1$, and we define the auxiliary path $\gamma_{g,\varphi} \colon [0,1] \to \R^{2}$ by
\begin{displaymath} \gamma_{g,\varphi}(s) := \int (\vec{R}_{s,\epsilon}g)(w)\varphi(w) \, dw, \qquad s \in [0,1], \end{displaymath}
then
\begin{displaymath} \limsup_{r \to s} \frac{|\gamma_{g,\varphi}(s) - \gamma_{g,\varphi}(r)|}{|s - r|} \leq \left| \int (\vec{\mathcal{R}}_{s,\epsilon}g)(w)\varphi(w) \, dw \right| \lesssim \|g\|_{2}. \end{displaymath} 
It follows by a standard argument, which we record in Lemma \ref{l:metric} below, that $\gamma_{g,\varphi}$ is a Lipschitz map $[0,1] \to \R^{2}$ with Lipschitz constant $\lesssim \|g\|_{2}$. This implies that the operator norm of $\vec{R}_{s,\epsilon} - \vec{R}_{r,\epsilon}$ on $L^{2}(\R^{2})$ is bounded from above by $\lesssim |r - s|$. Indeed, we have just shown that
\begin{displaymath} \|[\vec{R}_{s,\epsilon} - \vec{R}_{r,\epsilon}]g\|_{2} = \sup_{\|\varphi\|_{2} = 1} |\gamma_{g,\varphi}(s) - \gamma_{g,\varphi}(r)| \lesssim |s - r| \|g\|_{2} \end{displaymath}
for $g \in C^{\infty}_{c}(\R^{2})$, and the bound extends to all $g \in L^{2}(\R^{2})$. In particular, it may be applied with $g = g_{r}$:
\begin{displaymath} \|\eqref{form186}\|_{2} \lesssim \|[\vec{R}_{s,\epsilon} - \vec{R}_{r,\epsilon}]g_{r}\|_{2} \lesssim |r - s| \|g_{r}\|_{2} \lesssim_{A} |r - s| \|f\|_{2} \end{displaymath}
for all $0 \leq r,s \leq 1$. This completes the verification of the hypotheses of Lemma \ref{l:kenig}, and the lemma then shows that the operators $\tfrac{1}{2}I \pm R_{1}$ and $\tfrac{1}{2}I \pm D^{t}_{1} = \tfrac{1}{2}I \pm D^{t}_{A}$ are invertible. Before moving to the proof of Theorem \ref{t:sDerOp}, let us record the "local to global Lipschitz" lemma used right above. We claim no originality for this argument, see for example \cite[Lemma 2.3]{MR2369866} for a very similar result.

\begin{lemma}\label{l:metric} Let $(X,d)$ be a metric space, and let $\gamma \colon [0,1] \to X$ be a map satisfying 
\begin{displaymath} \limsup_{r \to s} \frac{d(\gamma(s),\gamma(r))}{|r - s|} \leq C, \qquad 0 \leq s \leq 1. \end{displaymath}
Then $\gamma$ is $C$-Lipschitz.
\end{lemma}

\begin{proof} Fix $\epsilon > 0$, $0 \leq r_{0} < s_{0} \leq 1$, and define
\begin{displaymath} \widetilde{E}_{i} := \{s \in [r_{0},s_{0}] : d(\gamma(s),\gamma(r)) < (C + \epsilon)|r - s| \text{ for all } r \in B(s,1/i)\}, \quad i \geq 1. \end{displaymath}
Then the sets $\widetilde{E}_{i}$, $i \geq 1$, cover $[r_{0},s_{0}]$, and so do the disjoint sets
\begin{displaymath} E_{i} := \widetilde{E}_{i} \setminus \, \bigcup_{j < i} \widetilde{E}_{j}, \qquad i \geq 1. \end{displaymath}
If the sets $E_{i}$ are further decomposed into disjoint subsets $E_{ij}$ of diameter $< 1/i$, then it is clear from $E_{ij} \subset \widetilde{E}_{i}$ that $\gamma|_{E_{ij}}$ is $(C + \epsilon)$-Lipschitz for all $i,j \in \N$. It follows that
\begin{displaymath} d(\gamma(r_{0}),\gamma(s_{0})) \leq \calH^{1}(\gamma([r_{0},s_{0}])) \leq \sum_{i,j} \calH^{1}(\gamma(E_{ij})) \leq (C + \epsilon)\sum_{i,j} \calH^{1}(E_{ij}) = (C + \epsilon)|s_{0} - r_{0}|. \end{displaymath}
Letting $\epsilon \to 0$ completes the proof. \end{proof}

We then prove Theorem \ref{t:sDerOp}, which follows rather easily from \cite[Theorem 7.8]{2019arXiv191103223F}. We copy the statement of \cite[Theorem 7.8]{2019arXiv191103223F} here for the reader's convenience:
\begin{thm}\label{t:FO} Let $A,B \colon \R \to \R$ be Lipschitz functions, and let $K \in C^{\infty}(\He)$ be a horizontally odd kernel satisfying \eqref{kernelConstants}. Let
\begin{displaymath} \Gamma(y,t) := \left(B(y),y,t -\tfrac{y}{2}B(y) + \int_{0}^{y} B(r) \, dr \right), \qquad (y,t) \in \R^{2}, \end{displaymath}
and write $K_{\Gamma}(w,v) := K(\Gamma(v)^{-1} \cdot \Gamma(w))$. Then, the kernels $K_{\Gamma}D_{A,1}$ and $K_{\Gamma}D_{A,2}$ define operators bounded on $L^{2}(\R^{2})$, where
\begin{equation}\label{form82} D_{A,1}(w,v) := \frac{A(x) - A(y)}{x - y} \quad \text{and} \quad D_{A,2}(w,v) := \int_{x}^{y} \frac{A(x) + A(y) - 2A(\theta)}{2(x - y)^{2}} \, d\theta \end{equation} 
for $w = (x,s)$ and $v = (y,s)$. The norms $\|K_{\Gamma}D_{A,j}\|_{L^{2}(\R^{2}) \to L^{2}(\R^{2})}$ only depend on the Lipschitz constants of $A,B$ and the kernel constants \eqref{kernelConstants}. \end{thm} 

We can then prove Theorem \ref{t:sDerOp}:

\begin{proof}[Proof of Theorem \ref{t:sDerOp}] We need to compute an explicit expression for the kernel $K_{s}(w,v) = \partial_{s}K(\Gamma_{s}(v)^{-1} \cdot \Gamma_{s}(w))$. To this end, start by computing that
\begin{displaymath} \Gamma_{s}(y,r)^{-1} \cdot \Gamma_{s}(x,t) = \left(A_{s}(x) - A_{s}(y), x - y, t - r + \int_{x}^{y} \left[\frac{A_{s}(x) + A_{s}(y) - 2A_{s}(\theta)}{2} \right] \, d
\theta \right). \end{displaymath} 
Recalling that $A_{s} = sA$, writing $w := (x,t)$ and $v := (y,r)$, and applying the chain rule, this yields
\begin{equation}\label{form88} K_{s}(w,v) = K_{1,s}(w,v) \cdot D_{A,1}(w,v) + K_{2,s}(w,v) \cdot D_{A,2}(w,v), \end{equation} 
where
\begin{displaymath} K_{1}(x,y,t) = y\partial_{x}K(x,y,t) \quad \text{and} \quad K_{2}(x,y,t) = y^{2}\partial_{t}K(x,y,t), \end{displaymath} 
$K_{j,s}(w,v) := K_{j}(\Gamma_{s}(v)^{-1} \cdot \Gamma_{s}(w))$, and the factors $D_{A,j}$ are as in \eqref{form82}. The horizontal oddness of the kernel $K$ implies that the kernels $K_{1}$ and $K_{2}$ are both horizontally odd. Also, applying the chain rule, recalling that $K$ satisfies \eqref{kernelConstants}, and and using
\begin{displaymath} \partial_{t} = XY - YX \quad \text{and} \quad \partial_{x} = X + (y/2)(XY - YX), \end{displaymath}
it is easy to verify that $K_{1},K_{2}$ also satisfy \eqref{kernelConstants}, with possibly different constants. Consequently, Theorem \ref{t:FO} applied to both kernels $K_{1}$ and $K_{2}$, and with "$B = A_{s}$", now concludes the proof of Theorem \ref{t:sDerOp}. \end{proof}

\begin{remark}\label{rem2} We close the section by giving some details for the claim
\begin{displaymath} \frac{\vec{R}_{s,\epsilon}g - \vec{R}_{r,\epsilon}g}{s - r} \rightharpoonup \vec{\mathcal{R}}_{s,\epsilon}g \end{displaymath}
whenever $g \in C^{\infty}_{c}(\R^{2})$, which appeared in \eqref{form87}. Fix $s \neq r$, $w \in \R$, abbreviate $K := (\nabla G)_{\epsilon}$, and write
\begin{align*} & \frac{\vec{R}_{s,\epsilon}g(w) - \vec{R}_{r,\epsilon}g(w)}{s - r} - \vec{\mathcal{R}}_{s,\epsilon}g(w)\\
&\qquad = \tfrac{1}{s - r}\int_{r}^{s} \int_{\R^{2}} \left[ \partial_{\xi} K(\Gamma_{\xi}^{-1}(v) \cdot \Gamma_{\xi}(w)) - \partial_{s} K(\Gamma_{s}^{-1}(v) \cdot \Gamma_{s}(w)) \right]g(v) \, dv \, d\xi. \end{align*} 
Now, recall the explicit expression \eqref{form88} for the kernels $\partial_{s}K(\Gamma_{s}^{-1}(v) \cdot \Gamma_{s}(w))$, and note that the factors $D_{A,1}(w,v)$ and $D_{A,2}(w,v)$ are uniformly bounded, and do not depend on $s$ or $\xi$. Therefore, the difference of the kernels in square brackets is bounded (in absolute value) by the sum of two differences of the form
\begin{equation}\label{form135} |\mathcal{K}(\Gamma_{\xi}(v)^{-1} \cdot \Gamma_{\xi}(w)) - \mathcal{K}(\Gamma_{s}(v)^{-1} \cdot \Gamma_{s}(w))|, \end{equation}
where $\mathcal{K} \in \{y\partial_{x}K,y^{2}\partial_{t}K\}$ is another kernel satisfying the $3$-dimensional kernel estimates \eqref{kernelConstants}. Moreover, since $K$ was supported away from $0$, the same is true for $\mathcal{K}$. So, it remains to show that
\begin{displaymath} \lim_{\xi \to s} \iint |\mathcal{K}(\Gamma_{\xi}(v)^{-1} \cdot \Gamma_{\xi}(w)) - \mathcal{K}(\Gamma_{s}(v)^{-1} \cdot \Gamma_{s}(w))||g(v)||\varphi(w)| \, dv \, dw = 0 \end{displaymath}
for any $\varphi \in C^{\infty}_{c}(\R^{2})$. Since also $g \in C^{\infty}_{c}(\R^{2})$, and $0 \notin \spt \mathcal{K}$, this follows easily from the fact that the differences in \eqref{form135} tend to zero for all $(v,w) \in \R^{2} \times \R^{2}$ as $\xi \to s$. \end{remark}

%%%%%%%%%%%%%%%%%%%%%%%%%%%%%%%%%%%

\section{Invertibility of the single layer potential}\label{s:singleInvert}

In this section, we prove Theorem \ref{main4}, which characterises the set of boundary values $\{(\mathcal{S}f)|_{\partial \Omega} : f \in L^{2}(\sigma)\}$ as a homogeneous Sobolev space defined on $\partial \Omega$. Before getting to the details, we verify that the boundary values exist in the non-tangential sense:
\begin{proposition}\label{prop11} Let $\Omega \subset \He$ be an open set with locally finite $\He$-perimeter, $\sigma := |\partial \Omega|_{\He}$, and assume that $\sigma$ is upper $3$-regular. Let $f \in L^{p}(\sigma)$ with $1 \leq p < 3$. Then, $\mathcal{S}f$ has non-tangential limit equal to
\begin{equation}\label{form178} Sf(p_{0}) := \int G(q^{-1} \cdot p_{0})f(q) \, d\sigma(q) \end{equation} 
at $p_{0} \in \partial \Omega$ along $\He \, \setminus \, \partial \Omega$ whenever the integral above is absolutely convergent. This is true $\sigma$ a.e., and $Sf \in L^{p}(\sigma) + L^{\infty}(\sigma)$. \end{proposition}

\begin{proof} To see that the integral defining $Sf(p)$ is absolutely convergent for $\sigma$ a.e. $p \in \partial \Omega$, decompose
\begin{displaymath} Sf(p) := \int G_{\mathrm{loc}}(q^{-1} \cdot p)f(q) \, d\sigma(q) + \int G_{\mathrm{glo}}(q^{-1} \cdot p)f(q) \, d\sigma(q) =: S_{\mathrm{loc}}f(p) + S_{\mathrm{glo}}f(p), \end{displaymath}
where $G_{\mathrm{loc}} := G\mathbf{1}_{B(1)}$, and $G_{\mathrm{glo}} := G - G_{\mathrm{loc}}$ (both non-negative kernels). Since $G_{\mathrm{loc}} \in L^{1}(\sigma)$ by the upper $3$-regularity of $\sigma$, Schur's test implies
\begin{displaymath} \int \left( \int G_{\mathrm{loc}}(q^{-1} \cdot p)|f(q)| \, d\sigma(q) \right)^{p} \, d\sigma(p) \lesssim \|f\|_{L^{p}(\sigma)}^{p}, \qquad 1 \leq p < \infty. \end{displaymath}
In particular, the integral defining $S_{\mathrm{loc}}f$ is absolutely convergent $\sigma$ a.e. and $S_{\mathrm{loc}}f \in L^{p}(\sigma)$. On the other hand $G_{\mathrm{glo}} \in L^{q}(\sigma)$ for all $q > \tfrac{3}{2}$, which implies by H\"older's inequality, and $f \in L^{p}(\sigma)$ with $1 \leq p < 3$, that the integral defining $S_{\mathrm{glo}}f(p)$ converges absolutely for all $p \in \partial \Omega$, and $S_{\mathrm{glo}}f \in L^{\infty}(\sigma)$. Hence, $Sf$ converges absolutely whenever $S_{\mathrm{loc}}f$ does.

Now, fix $\theta \in (0,1)$, and let $p_{0} \in \partial \Omega$ be a point where $Sf(p_{0})$ converges absolutely, and which lies in the closure of $[\He \, \setminus \, \partial \Omega] \cap V_{\partial \Omega}(p_{0},\theta)$. Fix $p \in [\He \, \setminus \, \partial \Omega] \cap V_{\partial \Omega}(p_{0},\theta)$, and use the triangle inequality to find that $d(p_{0},q) \leq (1 + \theta^{-1})d(p,q)$ for all $q \in \partial \Omega$. Recalling that $G(p) \sim \|p\|^{-2}$, it follows that $G(q^{-1} \cdot p) \lesssim_{\theta} G(q^{-1} \cdot p_{0})$ for all $q \in \partial \Omega$, hence $q \mapsto G(q^{-1} \cdot p_{0})|f(q)|$ is a $\sigma$-integrable dominating function for $q \mapsto G(q^{-1} \cdot p)|f(q)|$, with constants independent of $p \in [\He \, \setminus \, \partial \Omega] \cap V_{\partial \Omega}(p_{0},\theta)$. By dominated convergence,
\begin{displaymath} \lim_{p \to p_{0}} \mathcal{S}f(p) = \lim_{p \to p_{0}} \int G(p^{-1} \cdot q)f(q) \, d\sigma(q) = \int G(p_{0}^{-1} \cdot q)f(q) \, d\sigma(q) = Sf(p_{0}), \end{displaymath} 
where the limit naturally refers to non-tangential approach. This concludes the proof. \end{proof} 

It may be worth emphasising that the small technicalities in defining $Sf(p)$ are absent if, say, $f \in L^{\infty}(\sigma) \cap L^{1}(\sigma)$. This generality will suffice for most of this section. The homogeneous Sobolev space $L^{2}_{1,1/2}(\sigma)$ will be defined as a "$\Gamma$-lift" of the following non-isotropic homogeneous Sobolev space on $\R^{2}$:

\begin{definition}\label{def:L2112} The homogeneous Sobolev space $L^{2}_{1,1/2}(\R^{2})$ consists of those tempered distributions $T \in \mathcal{S}'(\R^{2})$ such that $\widehat{T} \in L^{1}_{\mathrm{loc}}(\R^{2})$, and
\begin{displaymath} \|T\|_{1,1/2} := \left( \iint \|(\xi,\tau)\|^{2} \cdot |\widehat{T}(\xi,\tau)|^{2} \, d\xi \, d\tau \right)^{1/2}. \end{displaymath}
Here $\|(\xi,\tau)\| := |\xi| + \sqrt{|\tau|}$, and $\widehat{T}$ refers to the full Fourier transform in $\R^{2}$. \end{definition}
\begin{remark} For a clear treatment of homogeneous Sobolev spaces, see \cite[\S 1.3]{MR2768550}. Here is what we need: using that $\|(\xi,\tau)\|^{-2} \in L^{1}(B(0,1))$, we have the inclusion $L^{2}_{1,1/2}(\R^{2}) \subset L^{2}(\R^{2}) + L^{\infty}(\R^{2})$, and $(L^{2}_{1,1/2}(\R^{2}),\|\cdot\|_{1,1/2})$ is, in fact, a Hilbert space. In particular, $\|\cdot\|_{1,1/2}$ is a proper norm in $L^{2}_{1,1/2}(\R^{2})$ (this uses the \emph{a priori} assumption $\widehat{T} \in L^{1}_{\mathrm{loc}}(\R^{2})$, which excludes Dirac masses at $0$). \end{remark}

Since $L^{2}_{1,1/2}(\R^{2})$ is a Hilbert space, we have the following useful property:
\begin{lemma}\label{lemma7} Let $\{T_{j}\}_{j \in \N} \subset L^{2}_{1,1/2}(\R^{2})$ be a bounded sequence. Then there exists $T \in L^{2}_{1,1/2}(\R^{2})$ such that $T_{j} \rightharpoonup T$ and
\begin{equation}\label{form120} \|T\|_{1,1/2} \leq \liminf_{j \to \infty} \|T_{j}\|_{1,1/2}. \end{equation}
Here the symbol "$\rightharpoonup$" may refer to either weak convergence w.r.t. the inner product in $L^{1,1/2}(\R^{2})$, or the convergence of tempered distributions (i.e. both types of convergence hold simultaneously).
\end{lemma}

\begin{proof} Every bounded sequence in a Hilbert space has a weakly convergent subsequence, so there exists an element $T \in L^{2}_{1,1/2}(\R^{2})$ satisfying \eqref{form120} such that
\begin{displaymath} \lim_{j \to \infty} \iint \|(\xi,\tau)\|^{2} \cdot \widehat{T}_{j}(\xi,\tau) \cdot \overline{\widehat{T}}'(\xi,\tau) \, d\xi \, d\tau = \iint \|(\xi,\tau)\|^{2} \cdot \widehat{T}(\xi,\tau) \cdot \overline{\widehat{T}}'(\xi,\tau) \, d\xi \, d\tau \end{displaymath}
for all $T' \in L^{2}_{1,1/2}(\R^{2})$. It follows that
\begin{equation}\label{form145} \lim_{j \to \infty} \iint \widehat{T}_{j}(\xi,\tau) \cdot \varphi(\xi,\tau) \, d\xi \, d\tau = \iint \widehat{T}(\xi,\tau) \cdot \varphi(\xi,\tau) \, d\xi \, d\tau, \qquad \varphi \in \mathcal{S}(\R^{2}), \end{equation}
because we can factorise
\begin{displaymath} \varphi(\xi,\tau) = \|(\xi,\tau)\|^{2} \cdot \frac{\varphi(\xi,\tau)}{\|(\xi,\tau)\|^{2}} =: \|(\xi,\tau)\|^{2} \cdot H, \qquad (\xi,\tau) \neq 0, \end{displaymath}
and note that the inverse Fourier transform of $H \in L^{1}(\R^{2})$ is an element of $L^{2}_{1,1/2}(\R^{2})$. Now \eqref{form145} means that $\widehat{T}_{j} \rightharpoonup \widehat{T}$ in the sense of distributions, and hence $T_{j} \rightharpoonup T$. \end{proof}

We proceed to define the space $L^{2}_{1,1/2}(\sigma)$ via the familiar map $\Gamma_{A} \colon \R^{2} \to \partial \Omega$,
\begin{equation}\label{form93} \Gamma_{A}(y,t) := \left(A(y),y,t - \tfrac{y}{2}A(y) + \int_{0}^{y} A(r) \, dr \right), \qquad (y,t) \in \R^{2}. \end{equation}

\begin{definition} The space $L^{2}_{1,1/2}(\sigma)$ consists of those functions $f \in L^{2}(\sigma) + L^{\infty}(\sigma)$ such that $f \circ \Gamma_{A} \in L^{2}_{1,1/2}(\R^{2})$. We define the norm $\|f\|_{1,1/2,\sigma} := \|f \circ \Gamma_{A}\|_{1,1/2}$. \end{definition}

The area formula \eqref{form16} for the map $\Gamma_{A}$ is used here implicitly to make sure that $f \circ \Gamma_{A} \in L^{2}(\R^{2}) + L^{\infty}(\R^{2}) \subset \mathcal{S}'(\R^{2})$. The main result in this section is the following:
\begin{thm}\label{sInvert1} The operator $S$ is bounded and invertible $L^{2}(\sigma) \to L^{2}_{1,1/2}(\sigma)$. In other words, whenever $f \in L^{2}(\sigma)$, then the $\sigma$ a.e. convergent integral $Sf$ from \eqref{form178} defines an element of $L^{2}_{1,1/2}(\sigma)$ with $\|Sf\|_{L^{2}_{1,1/2}(\sigma)} \lesssim \|f\|_{L^{2}(\sigma)}$. Moreover, every $g \in L^{2}_{1,1/2}(\sigma)$ can be represented as $g = Sf$ for some $f \in L^{2}(\sigma)$ with $\|f\|_{L^{2}(\sigma)} \sim \|g\|_{1,1/2,\sigma}$. \end{thm}

Theorem \ref{sInvert1} and Proposition \ref{prop11} combined imply that if $g \in L^{2}_{1,1/2}(\sigma)$, then $g$ coincides with the non-tangential limits of $\mathcal{S}f$ for some $f \in L^{2}(\sigma)$. This completes the proof of Theorem \ref{main5}.
\begin{ex}\label{SLPEx} Let $\Omega \subset \He$ be a flag domain, and let $p \in \Omega$. Then, $q \mapsto g_{p} := G(q^{-1} \cdot p) \in L^{2}_{1,1/2}(\sigma)$ by Lemma \ref{lemma6} (see Remark \ref{r:SLP} for details). Consequently, Theorem \ref{main5} implies that there exists $f_{p} \in L^{2}(\sigma)$ with the property that $g_{p}$ coincides $\sigma$ a.e. with the non-tangential limits of $u_{p} := \mathcal{S}f_{p}$ on $\partial \Omega$. Further, Corollary \ref{K-MaxNT} implies that 
\begin{displaymath} \|\mathcal{N}_{\theta}(\nabla u_{p})\|_{L^{2}(\sigma)} \lesssim \|f_{p}\|_{L^{2}(\sigma)} \sim \|g_{p}\|_{1,1/2}, \qquad \theta \in (0,1). \end{displaymath}
However, this approach does not seem to give $\|\mathcal{N}_{\theta}u_{p}\|_{L^{2}(\sigma)} \lesssim \|g_{p}\|_{L^{2}(\sigma)}$, although clearly $g_{p} \in L^{2}(\sigma)$. For this reason, it is preferable to find the harmonic function $u_{p}$, instead, as a double layer potential of some $f_{p} \in \mathbf{L}^{2}_{1,1/2}(\sigma)$, as explained in Example \ref{ex:Green}. \end{ex}

For technical reasons (similar to those encountered in Section \ref{s:surjectivity}), we will not prove Theorem \ref{sInvert1} directly, but rather the following variant: 

\begin{thm}\label{sInvert2} Let $A \colon \R \to \R$ be Lipschitz, and $f \in L^{2}(\R^{2})$. Then, the integral
\begin{equation}\label{eq:SA} S_{A}f(w) := \int G(\Gamma_{A}(v)^{-1} \cdot \Gamma_{A}(w))J_{A}(v)f(v) \, dv \end{equation}
converges absolutely for Lebesgue a.e. $w \in \R^{2}$, and $S_{A}f \in L^{2}_{1,1/2}(\R^{2})$ with $\|S_{A}f\|_{1,1/2} \lesssim \|f\|_{2}$. Moreover, the map $S_{A} \colon L^{2}(\R^{2}) \to L^{2}_{1,1/2}(\R^{2})$ is invertible.
\end{thm}

To motivate the definition of "$S_{A}f"$, we record that
\begin{equation}\label{form144} (Sf) \circ \Gamma_{A} = S_{A}(f \circ \Gamma_{A}), \qquad f \in L^{2}(\sigma). \end{equation} 
This follows easily from the area formula \eqref{form16}.

\begin{remark}\label{rem3} Let us a briefly remark that Theorem \ref{sInvert1} follows from Theorem \ref{sInvert2}. Once Theorem \ref{sInvert2} is known, we have can use the formula \eqref{form144} to deduce that
\begin{displaymath} \|Sf\|_{1,1/2,\sigma} := \|(Sf) \circ \Gamma_{A}\|_{1,1/2} = \|S_{A}(f \circ \Gamma_{A})\|_{1,1/2} \sim_{A} \|f \circ \Gamma_{A}\|_{2} \sim \|f\|_{L^{2}(\sigma)}. \end{displaymath}
The invertibility of $S$ also follows from the invertibility of $S_{A}$: if $h \in L^{2}_{1,1/2}(\sigma)$, then $h \circ \Gamma_{A} \in L^{2}_{1,1/2}(\R^{2})$ by definition, so $h \circ \Gamma_{A} = S_{A}f$ for some $f \in L^{2}(\R^{2})$ by the invertibility of $S_{A}$. Now, one defines $g := f \circ \Gamma_{A}^{-1} \in L^{2}(\sigma)$ and uses \eqref{form144} to show that $Sg = h$.
\end{remark}

\begin{remark}\label{rem4} We make some initial remarks on the proof of Theorem \ref{sInvert2}. The kernel $G_{A}(w,v) := G(\Gamma_{A}(v)^{-1} \cdot \Gamma_{A}(w))$ is a $2$-dimensional standard kernel on $\R^{2}$ equipped with the parabolic metric $d_{\mathrm{par}}((y,t),(y',t')) = |y - y'| + \sqrt{|t - t'|}$. This follows from the observations that $G(p) \sim \|p\|^{-2}$ is a $2$-dimensional standard kernel in $\He$, and $\Gamma_{A} \colon (\R^{2},d_{\mathrm{par}}) \to (\partial \Omega,d)$ is bilipschitz, see \cite[Lemma 7.5]{2019arXiv191103223F}. In particular, since Lebesgue measure on $\R^{2}$ is $3$-regular in $(\R^{2},d_{\mathrm{par}})$, we have $G_{A} \in L^{1}_{\mathrm{loc}}(\R^{2} \times \R^{2})$. This implies, following the local-global decomposition argument in Proposition \ref{prop11}, that the integral in \eqref{eq:SA} converges absolutely Lebesgue a.e. (for $f \in L^{2}(\R^{2})$), and
\begin{displaymath} S_{A}f \in L^{2}(\R^{2}) + L^{\infty}(\R^{2}) \subset \mathcal{S}'(\R^{2}). \end{displaymath}
To prove that \emph{a fortiori} $S_{A}f \in L^{2}_{1,1/2}(\R^{2})$ with $\|S_{A}f\|_{1,1/2} \lesssim \|f\|_{2}$, it suffices to show the same under the additional assumption $f \in C^{\infty}_{c}(\R^{2})$. Indeed, it is easy to check that if $\{\varphi_{j}\}_{j \in \N} \subset C^{\infty}_{c}(\R^{2})$ converge to $f \in L^{2}(\R^{2})$ in $L^{2}(\R^{2})$, then $S_{A}\varphi_{j} \rightharpoonup S_{A}f$ as tempered distributions. Then, assuming that the boundedness of $S_{A}$ on $C^{\infty}_{c}(\R^{2})$ is already known, Lemma \ref{lemma7} implies that
\begin{displaymath} \|S_{A}f\|_{1,1/2} \leq \liminf_{j \to \infty} \|S_{A}\varphi_{j}\|_{1,1/2} \lesssim \liminf_{j \to \infty} \|\varphi_{j}\|_{2} = \|f\|_{2}. \end{displaymath}
So, we may concentrate on proving the boundedness of $S_{A}$ on $C^{\infty}_{c}(\R^{2})$. The invertibility will, again, be established by the method of continuity, see Section \ref{s:checklist} for more details.  \end{remark}

\subsection{Some preliminary lemmas} The main purpose of this section is to quantify the intuition that the space $L^{2}_{1,1/2}(\sigma)$ consists of functions on $\partial \Omega$ with one tangential "$\tau$-derivative" in $L^{2}(\sigma)$, and additionally a $\tfrac{1}{2}$-order $t$-derivative in $L^{2}(\sigma)$. See Lemma \ref{lemma6} for a more precise statement. The proofs in this subsection are quite routine, so we suggest that the reader just skims through the statements and moves on to Section \ref{s:caseA0}. We start with a technical observation:
\begin{lemma}\label{lemma8} Let $f \in L^{2}(\R^{2})$, and let $f_{y}(t) := f(y,t)$ for $y \in \R$. Then,
\begin{displaymath} \iint \rho(\tau) \cdot |\widehat{f_{y}}(\tau)|^{2} \, d\tau \, dy = \iint \rho(\tau) \cdot |\hat{f}(\xi,\tau)|^{2} \, d\xi \, d\tau \end{displaymath}
for any non-negative $\rho \in L^{\infty}_{\mathrm{loc}}(\R)$. \end{lemma}

Note: the value of both integrals may be infinity. 

\begin{proof} Define $g_{\tau}(y) := \widehat{f_{y}}(\tau)$. For $f \in \mathcal{S}(\R^{2})$, we have the pointwise relation
\begin{displaymath} \widehat{g_{\tau}}(\xi) = \int e^{-2\pi i y\xi} g_{\tau}(y) \, dy = \iint e^{-2\pi i (\xi,\tau) \cdot (y,t)} f(y,t) \, dy \, dt = \hat{f}(\xi,\tau), \quad (\xi,\tau) \in \R^{2}. \end{displaymath}
Consequently, we infer from Fubini's theorem and Plancherel that
\begin{align*} \iint \rho(\tau) \cdot |\widehat{f_{y}}(\tau)|^{2} \, d\tau \, dy = \int \rho(\tau) \int |\widehat{f_{y}}(\tau)|^{2} \, dy \, d\tau & = \int \rho(\tau) \int |\widehat{g_{\tau}}(\xi)|^{2} \, d\xi \, d\tau\\
& = \iint \rho(\tau) \cdot |\hat{f}(\xi,\tau)|^{2} \, d\xi \, d\tau, \end{align*}
as desired. Now, if $f \in L^{2}(\R^{2})$ is arbitrary, choose a sequence $\{\varphi^{i}\}_{i \in \N} \subset C^{\infty}_{c}(\R^{2})$ such that $\varphi^{i} \to f$ in $L^{2}$. Since $f_{y} \in L^{2}(\R)$ for a.e. $y \in \R$, we infer from Plancherel that
\begin{displaymath} \lim_{i \to \infty} \iint \rho(y,\tau) \cdot |\widehat{\varphi^{i}_{y}}(\tau) - \widehat{f_{y}}(\tau)|^{2} \, d\tau \, dy \leq \|\rho\|_{L^{\infty}(\R^{2})} \cdot \lim_{i \to \infty} \|\varphi^{i} - f\|_{L^{2}(\R^{2})}^{2} = 0 \end{displaymath} 
for any weight $\rho \in L^{\infty}(\R^{2})$. Finally, for non-negative $\rho \in L^{\infty}_{\mathrm{loc}}(\R)$, we define $\rho_{r}(\tau) := |\tau| \cdot \mathbf{1}_{|\tau| \leq r} \in L^{\infty}(\R)$ for $r > 0$. Then, by dominated convergence and the previous steps,
\begin{align*} \iint \rho(\tau) \cdot |\widehat{f_{y}}(\tau)|^{2} \, d\tau \, dy & = \lim_{r \to \infty} \iint \rho_{r}(\tau) \cdot |\widehat{f_{y}}(\tau)|^{2} \, d\tau \, dy\\
& = \lim_{r \to \infty} \lim_{i \to \infty} \iint \rho_{r}(\tau) \cdot |\widehat{\varphi_{y}^{i}}(\tau)|^{2} \, d\tau \, dy\\
& = \lim_{r \to \infty} \lim_{i \to \infty} \iint \rho_{r}(\tau) \cdot |\widehat{\varphi^{i}}(\xi,\tau)|^{2} \, d\xi \, d\tau\\
& = \lim_{r \to \infty} \iint \rho_{r}(\tau) \cdot |\hat{f}(\xi,\tau)|^{2} \, d\xi \, d\tau = \iint \rho(\tau) \cdot |\hat{f}(\xi,\tau)|^{2} \, d\xi \, d\tau.  \end{align*} 
This concludes the proof. \end{proof}

We now give a sufficient condition for $\varphi \in L^{2}_{1,1/2}(\sigma)$. We make a few remarks to clarify the statement. Recall that if $\Omega \subset \He$ is a flag domain, then the perimeter measure $\sigma = |\partial \Omega|_{\He}$ has a product form $\sigma = c\calH^{1}_{E}|_{\mathrm{graph}(A)} \times \mathcal{L}^{1}$. In particular, if $\varphi \in L^{2}(\sigma)$, then $t \mapsto \varphi(z,t) \in L^{2}(\R)$ for $\calH^{1}_{E}$ a.e. $z \in \mathrm{graph}(A)$, and hence $\tau \mapsto \widehat{\varphi}(z,\tau) \in L^{2}(\R)$ for the same $z \in \mathrm{graph}(A)$. Here $\widehat{\varphi}(z,\tau)$ refers to the Fourier transform of $t \mapsto \varphi(z,t)$ evaluated at $\tau$. If $\varphi \in C^{1}(\partial \Omega)$ (i.e. $\varphi$ is continuously differentiable in a neighbourhood of $\partial \Omega$), we write $\nabla_{\tau}\varphi(p) := \langle \nabla \varphi(p),\tau(p) \rangle$, where "$\tau$" is the tangential vector field from Definition \ref{HTangent}.

\begin{lemma}\label{lemma6} Let $\Omega = \{(x,y,t) : x < A(y)\} \subset \He$ be a flag domain, $\sigma := |\partial \Omega|_{\He}$, and let $\varphi \in C^{1}(\partial \Omega) \cap L^{2}(\sigma)$ be such that $\nabla_{\tau} \varphi \in L^{2}(\sigma)$, and 
\begin{equation}\label{form180} \|T^{1/2}\varphi\|_{L^{2}(\sigma)}^{2} := \iint |\tau| \cdot |\widehat{\varphi}(z,\tau)|^{2} \, d\sigma(z,\tau) < \infty. \end{equation}
Then, $\varphi \in L^{2}_{1,1/2}(\sigma)$, and $\|\varphi\|_{1,1/2,\sigma} \sim \|\nabla_{\tau}\varphi\|_{L^{2}(\sigma)} + \|T^{1/2}\varphi\|_{L^{2}(\sigma)}$. \end{lemma}

\begin{remark}\label{r:SLP} Following up on Example \ref{SLPEx}, we check that the hypotheses of Lemma \ref{lemma6} are satisfied for $\varphi_{p}(q) := \Gamma(q^{-1} \cdot p)$, where $p \in \Omega$. Clearly $\varphi \in C^{1}(\He \, \setminus \, \{p\})$, and $\nabla \varphi_{p} \in L^{2}(\sigma)$ follows from $|\nabla \varphi_{p}(q)| \lesssim_{p} \min\{1,\|q\|^{-3}\}$, $q \in \partial \Omega$. The condition \eqref{form180} follows from an explicit expression for the vertical Fourier transform of $G$, computed in \eqref{form68}. Writing $p = (z_{0},t_{0})$ and $q = (z,t)$, one finds from this formula that
\begin{displaymath} |\widehat{\varphi_{p}}(z,\tau)| = \tfrac{1}{2}K_{0}(\tfrac{\pi}{2}|z - z_{0}|^{2}|\tau|), \qquad (z,\tau) \in \partial \Omega. \end{displaymath}
The (modified Bessel) function $K_{0}$ has a logarithmic singularity at $0$, and decays exponentially at $\infty$, see \eqref{K0asymp1}-\eqref{K0asymp2}. In particular, $(z,\tau) \mapsto |\tau| \cdot |\widehat{\varphi_{p}}(z,\tau)|^{2} \in L^{1}(\sigma)$. \end{remark}

\begin{proof}[Proof of Lemma \ref{lemma6}] Write $\varphi_{A} := \varphi \circ \Gamma_{A} \in L^{2}(\R^{2})$. By definition of the $\|\cdot\|_{1,1/2,\sigma}$-norm, and Lemma \ref{lemma8}, 
\begin{displaymath} \|\varphi_{A}\|_{1,1/2,\sigma}^{2} \sim \iint |\partial_{y} \varphi_{A}(y,t)|^{2} \, dy \, dt + \iint |\tau| \cdot |\widehat{(\varphi_{A})_{y}}(\tau)|^{2} \, dy \, d\tau =: I_{1} + I_{2}. \end{displaymath}
Here $\partial_{y} \varphi_{A}(y,t)$ formally refers to the distributional $y$-derivative of $\varphi_{A}$, but $y \mapsto \varphi_{A}(y,t)$ is locally Lipschitz under the assumptions: indeed, $\varphi_{A}(\cdot,t) = \varphi \circ \gamma_{t}$ is the composition of the $C^{1}$-function $\varphi$ with the horizontal (Lipschitz) curve $\gamma_{t} := \Gamma_{A}(\cdot,t)$. We already computed in \eqref{form94} that $\dot{\gamma}_{t}(y) = \tau(\gamma_{t}(y)) \cdot J_{A}(y,t)$, so
\begin{equation}\label{form193} \partial_{y}\varphi_{A}(y,t) = (\varphi \circ \gamma_{t})'(y) = \langle \nabla \varphi(\gamma_{t}(y)),\dot{\gamma}_{t}(y)) \rangle = \nabla_{\tau}\varphi(\Gamma_{A}(y,t)) \cdot J_{A}(y,t)  \end{equation}
for a.e. $y \in \R$. Consequently, by the area formula \eqref{form16}, we have $I_{1} \sim_{\mathrm{Lip}(A)} \|\nabla_{\tau}\varphi\|_{L^{2}(\sigma)}^{2}$. To treat $I_{2}$, we recall from the definition of $\Gamma_{A}$ that
\begin{displaymath} (\varphi_{A})_{y}(t) = \varphi \left(A(y),y,t + I_{A}(y) \right) \quad \text{with} \quad I_{A}(y) := -\tfrac{y}{2}A(y) + \int_{0}^{y} A(r) \, dr. \end{displaymath}
Consequently, for all $y \in \R$ such that $t \mapsto \varphi(A(y),y,t + I_{A}(y)) \in L^{2}(\R)$, we have
\begin{equation}\label{form139} \widehat{(\varphi_{A})_{y}}(\tau) = e^{2\pi i \tau I_{A}(y)}\widehat{\varphi}(A(y),y,\tau) \end{equation}
as $L^{2}$-functions, and in particular $|\widehat{(\varphi_{A})_{y}}| = |\widehat{\varphi}(A(y),y,\cdot)|$. Recalling once more that $\sigma \sim \calH^{1}_{E}|_{\mathrm{graph}(A)} \times \mathcal{L}^{1}$, and using the (Euclidean) area formula, it follows that
\begin{equation}\label{form179} I_{2} = \iint |\tau| \cdot |\widehat{\varphi}(A(y),y,\tau)|^{2} \, dy \, d\tau \sim_{\mathrm{Lip}(A)} \iint |\tau| \cdot |\widehat{\varphi}(z,\tau)|^{2} \, d\sigma(z,\tau), \end{equation}
as claimed.  \end{proof}

The next proposition is almost a consequence of the lemma above, since $S\varphi \in L^{2}(\sigma)$ for $\varphi \in C^{\infty}_{c}(\He)$. However, a little care is needed with the details, since $S\varphi \notin C^{1}(\partial \Omega)$.

\begin{proposition}\label{prop5} Let $\varphi \in C^{\infty}_{c}(\He)$. Then, $S\varphi \in L^{2}_{1,1/2}(\sigma)$, and
\begin{equation}\label{form140} \|S\varphi\|_{1,1/2,\sigma} \sim_{\mathrm{Lip}(A)} \|\nabla_{\tau}S\varphi\|_{L^{2}(\sigma)} + \|T^{1/2}\mathcal{S}\varphi\|_{L^{2}(\sigma)}, \end{equation} 
where $\nabla_{\tau}S\varphi$ refers to the principal value defined in \eqref{noJump} (or \eqref{form71} below), and
\begin{displaymath} \|T^{1/2}\mathcal{S}\varphi\|_{L^{2}(\sigma)}^{2} := \int |\tau| \cdot |\widehat{\mathcal{S}\varphi}(z,\tau)|^{2} \, d\sigma(z,\tau). \end{displaymath}
Here $\widehat{\mathcal{S}\varphi}$ is the vertical distributional Fourier transform of the single layer potential $\mathcal{S}\varphi$, as in Definition \ref{def:vertFDist}. Lemma \ref{l:Fourier1} verifies that $\widehat{\mathcal{S}\varphi} \in C^{\infty}(\He \, \setminus \, \{\tau = 0\})$, so the $\sigma$-integral makes sense.  \end{proposition}

Recall that the finiteness of $\|T^{1/2}\mathcal{S}\varphi\|_{L^{2}(\sigma)}$ was already established during the proof of \eqref{form56}, and $\|\nabla_{\tau}S\varphi\|_{L^{2}(\sigma)} \lesssim \|\varphi\|_{L^{2}(\sigma)} < \infty$ by Theorem \ref{t:FO1}. So, the relation \eqref{form140} implies the finiteness of $\|S\varphi\|_{1,1/2,\sigma}$.

\begin{proof} It is easy to check that $S\varphi \in L^{2}(\sigma)$, so also $(S\varphi) \circ \Gamma_{A} \in L^{2}(\R^{2})$ by the area formula \eqref{form16}. Moreover, we recall from \eqref{form144} that $(S\varphi) \circ \Gamma_{A} = S_{A}\varphi_{A}$
where $\varphi_{A} := \varphi \circ \Gamma_{A}$. So, $S_{A}\varphi_{A} \in L^{2}(\R^{2}) \subset \mathcal{S}'(\R^{2})$. To prove \eqref{form140}, we need to show that
\begin{equation}\label{form134} \|S_{A}\varphi_{A}\|_{1,1/2} \sim_{\mathrm{Lip}(A)} \|\nabla_{\tau}S\varphi\|_{L^{2}(\sigma)} + \|T^{1/2}\mathcal{S}\varphi\|_{L^{2}(\sigma)}. \end{equation}
Recall that $\nabla_{\tau}S\varphi$ is the $\sigma$ a.e. defined principal value
\begin{equation}\label{form71} \nabla_{\tau}S\varphi(p) = \lim_{\epsilon \to 0} \int \langle \nabla G_{\epsilon}(q^{-1} \cdot p),\tau(p) \rangle\varphi(q) \, d\sigma(q), \end{equation}
where $G_{\epsilon}$ is a smooth radially symmetric truncation of $G$; by Proposition \ref{prop10}, the particular choice of truncation does not affect the principal value. 

Since $S_{A}\varphi_{A} \in L^{2}(\R^{2})$, we may use Lemma \ref{lemma8}:
\begin{align} \label{form174}  \|S_{A}\varphi_{A}\|_{1,1/2}^{2} & \sim \|\partial_{y}S_{A}\varphi_{A}\|_{2}^{2} + \iint |\tau| \cdot |\widehat{(S_{A}\varphi_{A})_{y}}(\tau)|^{2} \, d\tau \, dy =: I_{1} + I_{2}. \end{align}
Here $\partial_{y}S_{A}\varphi_{A}$ refers to the distributional $y$-derivative, and $\widehat{(S_{A}\varphi_{A})_{y}}$ refers, for a.e. $y$, to the Fourier transform of the $L^{2}(\R)$-function $t \mapsto (S_{A}\varphi_{A})(y,t)$. The main task is to show that $I_{1} \sim \|\nabla_{\tau}S\varphi\|_{L^{2}(\sigma)}^{2}$. We also need to know that $I_{2} \sim \|T^{1/2}\mathcal{S}\varphi\|_{L^{2}(\sigma)}^{2}$, but this can be done exactly as in the proof of Lemma \ref{lemma6}. Indeed $S_{A}\varphi_{A} = (S\varphi) \circ \Gamma_{A} = (S\varphi)_{A}$ in the lemma's notation. Hence \eqref{form139} shows $|\widehat{(S_{A}\varphi_{A})_{y}}| = |\widehat{\mathcal{S}\varphi}(A(y),y,\cdot)|$ for a.e. $y \in \R$, and then a repetition of \eqref{form179} gives the claim.

To show that $I_{1} \sim \|\nabla_{\tau}S\varphi\|_{L^{2}(\sigma)}^{2}$, we wish to relate the distributional $y$-derivative $\partial_{y} S_{A}\varphi_{A}$ to the principal value $\nabla_{\tau}S\varphi$. We claim that
\begin{equation}\label{form141} \int \nabla_{\tau}S\varphi(\Gamma_{A}(w))J_{A}(w) \cdot \psi(w) \, dw = -\int (S_{A}\varphi_{A})(w) \cdot \partial_{y} \psi(w) \, dw, \quad \psi \in C_{c}^{\infty}(\R^{2}). \end{equation}
Abbreviate $\Gamma := \Gamma_{A}$ and $J := J_{A}$, and let $w := (y,t) \in \R^{2}$ be a point such that the principal value $\nabla_{\tau}S(\Gamma(w))$ exists, and
\begin{displaymath} \partial_{y} [y \mapsto G_{\epsilon}(\Gamma(v)^{-1} \cdot \Gamma(w))] = \langle \nabla G_{\epsilon}(\Gamma(v)^{-1} \cdot \Gamma(w)),\tau(\Gamma(w)) \rangle J(w), \qquad v \in \R^{2}. \end{displaymath}
Both properties holds at a.e. $w \in \R^{2}$ (recall \eqref{form193}). Then, using the area formula \eqref{form16}, we see that
\begin{align*} \nabla_{\tau}(S\varphi)(\Gamma(w))J(w) & = \lim_{\epsilon \to 0} \int \langle \nabla G_{\epsilon}(q^{-1} \cdot \Gamma(w)),\tau(\Gamma(w)) \rangle J(w) \varphi(q) \, d\sigma(q)\\
& = \lim_{\epsilon \to 0} \int \langle \nabla G_{\epsilon}(\Gamma(v)^{-1} \cdot \Gamma(w)),\tau(\Gamma(w)) \rangle J(w) \varphi(\Gamma(v))J(v) \, dv\\
& = \lim_{\epsilon \to 0} \partial_{y} \int G_{\epsilon}(\Gamma(v)^{-1} \cdot \Gamma(w))\varphi(\Gamma(v))J(v) \, dv. \end{align*}
Consequently,
\begin{align*} \int & [\nabla_{\tau}S(\Gamma(w))J(w)] \cdot \psi(w) \, dw\\
& =  \lim_{\epsilon \to 0} \int \left[ \int \langle \nabla G_{\epsilon}(q^{-1} \cdot \Gamma(w)),\tau(\Gamma(w)) \rangle J(w) \varphi(q) \, d\sigma(q) \right] \, \psi(w) \, dw\\
& = - \lim_{\epsilon \to 0} \int \left[ \int G_{\epsilon}(\Gamma(v)^{-1} \cdot \Gamma(w))\varphi(\Gamma(v))J(v) \, dv \right] \cdot \partial_{y} \psi(w) \, dw = - \int (S_{A}\varphi_{A})(w) \cdot \partial_{y} \psi(w) \, dw, \end{align*} 
as claimed. Taking the limit outside the integral on the second line can be justified by noting that
\begin{displaymath} w \mapsto \sup_{\epsilon > 0} \left| \int \langle \nabla G_{\epsilon}(q^{-1} \cdot \Gamma(w)),\tau(\Gamma(w)) \rangle J(w) \varphi(q) \, d\sigma(q) \right| \end{displaymath}
is dominated by the maximal Riesz transform of $\varphi$ evaluated at $\Gamma(w)$, and so defines an $L^{2}$-function by Theorem \ref{t:FO1}. The second interchange of limits and integration is more straightforward, since $S_{A}\varphi_{A} \in L^{\infty}(\R^{2})$ and $\psi \in C^{\infty}_{c}(\R^{2})$.

With \eqref{form141} in hand, and using that $w \mapsto \nabla_{\tau}S\varphi(\Gamma(w))J(w) \in L^{2}(\R^{2})$ by another application of the area formula, we also see that $\partial_{y} S_{A}\varphi_{A} \in L^{2}(\R^{2})$, and
\begin{displaymath} I_{1} = \|\partial_{y} S_{A}\varphi_{A}\|_{L^{2}(\R^{2})} \sim_{\mathrm{Lip}(A)} \|\nabla_{\tau}S\varphi\|_{L^{2}(\sigma)} < \infty. \end{displaymath}
This completes the proof of the proposition.  \end{proof}

\subsection{The case $A \equiv 0$}\label{s:caseA0} We now treat the case $A \equiv 0$ of Theorem \ref{sInvert2}. This special case should also follow from Jerison's result \cite[Theorem (6.1)]{MR639800}, which works in $\He^{n}$ for all $n \geq 1$, and in particular implies that the inverse of $S_{0}$ is a homogeneous distribution which is smooth away from the origin. Since Jerison's proof is rather involved (the problem is harder in higher dimensions), we repeat the details for the special case $n = 1$.

\begin{proposition}\label{prop4} The map $S_{0}$ is invertible $L^{2}(\R^{2}) \to L^{2}_{1,1/2}(\R^{2})$.\end{proposition}

\begin{proof} Since $G(\Gamma_{0}(y',t')^{-1} \cdot \Gamma_{0}(y,t)) = c((y - y')^{4} + 16(t - t')^{2})^{-1/2}$, the operator $S_{0}$ is simply a convolution operator on $\R^{2}$ with kernel $G(y,t) = c(y^{4} + 16t^{2})^{-1/2}$. Noting that 
\begin{displaymath} \widehat{S\varphi}(\xi,\tau) = \widehat{\varphi}(\xi,\tau)\widehat{G}(\xi,\tau), \end{displaymath}
the invertibility of $S \colon L^{2}(\R^{2}) \to L^{2}_{1,1/2}(\R^{2})$ will follow once we manage to show that
\begin{equation}\label{form91} \widehat{G}(\xi,\tau) \sim \frac{1}{\|(\xi,\tau)\|}.  \end{equation}
To prove \eqref{form91}, we infer from the homogeneity of $G$, that
\begin{displaymath} \widehat{G}(r\xi,r^{2}\tau) = \tfrac{1}{r}\widehat{G}(\xi,\tau), \qquad r > 0. \end{displaymath} 
This (with $r = \|(\xi,\tau)\|^{-1}$) will prove \eqref{form91} once we manage to show that
\begin{equation}\label{form92} 0 < \min_{\|(\xi,\tau)\| = 1} \widehat{G}(\xi,\tau) \leq \max_{\|(\xi,\tau)\| = 1} \widehat{G}(\xi,\tau) < \infty. \end{equation}
In \eqref{form68}, we will show that the (distributional) Fourier transform of $t \mapsto G(y,t)$, $y \neq 0$, is given by $\tau \mapsto \tfrac{c}{2}K_{0}(\tfrac{\pi}{2}|\tau||y|^{2})$, where $K_{0}$ is the modified Bessel function of the first kind of order zero, namely
\begin{displaymath} K_{0}(s) = \int_{0}^{\infty} e^{-s \cosh r} \, dr. \end{displaymath}
The function $K_{0}$ has a logarithmic singularity at $0$ and exponential decay, see \eqref{K0asymp1}-\eqref{K0asymp2}, so $y \mapsto K_{0}(\tfrac{\pi}{2}|\tau||y|^{2}) \in L^{1}(\R)$ with norm $\lesssim |\tau|^{-1/2}$. Using this, it is not difficult to justify that the full (distributional) Fourier transform of $G$ is given by the function
\begin{displaymath} \widehat{G}(\xi,\tau) = \tfrac{c}{2} \int_{\R} e^{-2\pi i y\xi} K_{0}(\tfrac{\pi}{2}|\tau||y|^{2}) \, dy = \tfrac{c}{2} \int_{0}^{\infty} \left[ \int_{\R} e^{-2\pi i y\xi} e^{-\tfrac{\pi}{2}|\tau||y|^{2} \cosh r} \, dy \right] \, dr, \quad \tau \neq 0. \end{displaymath}
But the integral in brackets is the Fourier transform of a Gaussian:
\begin{displaymath} \int_{\R} e^{-2\pi i y\xi} e^{-\tfrac{\pi}{2}|\tau||y|^{2} \cosh r} \, dy = \exp\left(\frac{-2\pi\xi^{2}}{|\tau|\cosh r} \right) \frac{\sqrt{2}}{\sqrt{|\tau| \cosh r}},\end{displaymath} 
and hence
\begin{displaymath} \widehat{G}(\xi,\tau) = \frac{c}{\sqrt{2}} \int_{0}^{\infty}  \exp\left(\frac{-2\pi\xi^{2}}{|\tau|\cosh r} \right) \frac{1}{\sqrt{|\tau| \cosh r}} \, dr, \qquad \tau \neq 0. \end{displaymath}
Given this formula, it is evident that \eqref{form92} holds on all compact subsets of $\{(\xi,\tau) : \|(\xi,\tau)\| = 1 \text{ and } \tau \neq 0\}$. So, to complete the proof, it suffices to show that $\widehat{G}(\xi,\tau)$ stays in a compact subset of $(0,\infty)$ as $(\xi,\tau) \to (1,0)$. Fix $\tau \neq 0$, and change variables $|\tau| \cosh r \mapsto u$, so 
\begin{displaymath} du = |\tau| \sinh r \, dr = \sqrt{\tau^{2} \cosh^{2}(r) - \tau^{2}} \, dr = \sqrt{u^{2} - \tau^{2}} \, dr, \end{displaymath}
and consequently (noting that $|\tau| \cosh 0 = |\tau|$)
\begin{displaymath} \widehat{G}(\xi,\tau) = \frac{c}{\sqrt{2}} \int_{|\tau|}^{\infty} \frac{\exp\left(\tfrac{-2\pi \xi^{2}}{u} \right)}{\sqrt{u}} \frac{du}{\sqrt{u^{2} - \tau^{2}}}. \end{displaymath}
Now, split the integration into the pieces where $|\tau| < u \leq 2|\tau|$ and $u \geq 2|\tau|$. It is straightforward to see that the first piece tends to $0$ as $\tau \to 0$, while
\begin{displaymath}\mathbf{1}_{\{u \geq 2|\tau|\}}(u)\frac{\exp\left(\tfrac{-2\pi \xi^{2}}{u} \right)}{\sqrt{u}} \frac{1}{\sqrt{u^{2} - \tau^{2}}} \leq \sqrt{2} \frac{\exp\left(\tfrac{-2\pi \xi^{2}}{u} \right)}{u^{3/2}}. \end{displaymath} 
So, the dominated convergence theorem yields
\begin{displaymath} \widehat{G}(\xi,\tau) \to \frac{c}{\sqrt{2}} \int_{0}^{\infty}  \frac{\exp\left(\tfrac{-2\pi \xi^{2}}{u} \right)}{u^{3/2}} \, du \quad \text{as $\tau \to 0$}. \end{displaymath}
Evidently the quantity on the right stays in a compact subset of $(0,\infty)$ for $|\xi| \sim 1$, which suffices for \eqref{form92}. For those interested, Maple tells that the value of the integral above is $\frac{c}{2}$. This completes the proof of the proposition. \end{proof}

\subsection{The checklist}\label{s:checklist} To prove the invertibility of $S_{A}$ for a general Lipschitz $A \colon \R \to \R$, we again resort to the method of continuity. Copying the notation from Section \ref{s:surjectivity}, we consider the family of Lipschitz functions $\{A_{s}\}_{s \in \R}$ given by $A_{s}(y) = sA(y)$. We abbreviate $\Gamma_{s}(y,t) := \Gamma_{A_{s}}(y,t)$, $G_{s}(w,v) := G(\Gamma_{s}(v)^{-1} \cdot \Gamma_{s}(w))$, $J_{s}(v) := J_{A_{s}}(v)$, and
\begin{equation}\label{op:Ss} S_{s}\varphi(w) := S_{A_{s}}\varphi(w) = \int G_{s}(w,v)J_{s}(v)\varphi(v) \, dv, \qquad s \in \R, \, \varphi \in C^{\infty}_{c}(\R^{2}). \end{equation}
We have already seen in Proposition \ref{prop4} that $S_{0} \colon L^{2}(\R^{2}) \to L^{2}_{1,1/2}(\R^{2})$ is invertible, and our goal will be to show that $S_{1} \colon L^{2}(\R^{2}) \to L^{2}_{1,1/2}(\R^{2})$ is also invertible. Recalling Lemma \ref{l:kenig}, here are the things that we need to check:
\begin{enumerate}
\item $\|S_{s}\|_{L^{2}(\R^{2}) \to L^{2}_{1,1/2}(\R^{2})} \leq C$ and $\|S_{s} - S_{t}\|_{L^{2}(\R^{2}) \to L^{2}_{1,1/2}(\R^{2})} \leq C|s - t|$,
\item $\|S_{s}f\|_{1,1/2} \geq c\|f\|_{2}$ for all $f \in L^{2}(\R^{2})$.
\end{enumerate}
Here $c,C > 0$ are constants which should be uniform (depend only only $\mathrm{Lip}(A)$) when $s,t$ range in a fixed compact subinterval of $\R$. 

\subsection{Verifying (1)} The property (1) will be deduced from the following general result, whose proof we postpone until the end of the section:
 \begin{thm}\label{t:bdd} Let $K \in C^{\infty}(\He \, \setminus \, \{0\})$ be a horizontally even (i.e. $K(z,t) = K(-z,t)$) kernel satisfying
\begin{equation}\label{kernelConstants2} \quad |\nabla^{n} K(p)| \leq C_{n}\|p\|^{-2 - n}, \qquad p \in \He \, \setminus \, \{0\}, \, n \geq 0. \end{equation}
Then, if $B \colon \R \to \R$ is a Lipschitz function, the kernel $K_{B}(w,v) := K(\Gamma_{B}(v)^{-1} \cdot \Gamma_{B}(w))$ defines a bounded operator $L^{2}(\R^{2}) \to L^{2}_{1,1/2}(\R^{2})$. This, in particular, applies when $K = G$.

Assume, in addition, that 
\begin{equation}\label{kernelConstants3} \sup_{z \in \R^{2} \, \setminus \, \{0\}} \int |K(z,t)| \, dt \leq C_{\star} < \infty, \end{equation}
let $A \colon \R \to \R$ be another Lipschitz function, and write
\begin{equation}\label{form116} D_{A,1}(w,v) := \frac{A(x) - A(y)}{x - y} \quad \text{and} \quad D_{A,2}(w,v) := \int_{x}^{y} \frac{A(x) + A(y) - 2A(\theta)}{2(x - y)^{2}} \, d\theta. \end{equation} 
Then, the kernels $K_{B}D_{A,1}$ and $K_{B}D_{A,2}$ define bounded operators $L^{2}(\R^{2}) \to L^{2}_{1,1/2}(\R^{2})$. The norms $\|K_{B}\|_{L^{2}(\R^{2}) \to L^{2}_{1,1/2}(\R^{2})}$ and $\|K_{B}D_{A,j}\|_{L^{2}(\R^{2}) \to L^{2}_{1,1/2}(\R^{2})}$ in the statements above only depend on the Lipschitz constants of $A,B$, and the kernel constants \eqref{kernelConstants2}-\eqref{kernelConstants3}.  \end{thm}

The fist part of (1) on our checklist is an immediate consequence of the (first part of the) theorem, applied with $K = G$ and $B = A_{s}$, since $\mathrm{Lip}(A_{s}) \lesssim_{I} \mathrm{Lip}(A)$ when $s$ ranges in a compact interval $I \subset \R$. We next turn to the second part, namely verifying that 
\begin{equation}\label{form191} \|S_{s} - S_{t}\|_{L^{2}(\R^{2}) \to L^{2}_{1,1/2}(\R^{2})} \leq C|s - t|, \qquad s,t \in [0,1]. \end{equation}
As a first reduction, note that
\begin{equation}\label{form194} S_{s}\varphi - S_{t}\varphi = S_{s}'([J_{s} - J_{t}]\varphi) + [S_{s}'(J_{t}\varphi) - S_{t}'(J_{t}\varphi)], \qquad \varphi \in C^{\infty}_{c}(\R^{2}), \end{equation}
where
\begin{displaymath} S_{r}'\psi(w) := \int G_{r}(w,v)\psi(v) \, dv. \end{displaymath}
Since Theorem \ref{t:bdd} implies that $S_{s}'$ is bounded $L^{2}(\R^{2}) \to L^{2}_{1,1/2}(\R^{2})$, the first term in \eqref{form194} satisfies $\|S_{s}'([J_{s} - J_{t}]\varphi)\|_{1,1/2} \lesssim \|J_{s} - J_{t}\|_{\infty} \cdot \|\varphi\|_{2} \lesssim |s - t| \cdot \|\varphi\|_{2}$. Therefore, to prove \eqref{form191}, it suffices to verify that
\begin{equation}\label{form191a} \|S_{s}' - S_{t}'\|_{L^{2}(\R^{2}) \to L^{2}_{1,1/2}(\R^{2})} \leq C|s - t|, \qquad s,t \in [0,1]. \end{equation}

As in Section \ref{s:surjectivity}, we plan to use Lemma \ref{l:metric} to deduce bounds for $\|S_{s}'f - S_{t}'f\|_{L^{2}_{1,1/2}(\R^{2})}$, for $f \in C_{c}^{\infty}(\R^{2})$. For $s \in \R$, we define the operator
\begin{displaymath} (\partial S)_{s}f := \int \partial_{s}G_{s}(w,v)f(v) \, dv, \qquad f \in C_{c}^{\infty}(\R^{2}). \end{displaymath} 
To apply Lemma \ref{l:metric}, the main task is to show that $(\partial S)_{s}$ defines a bounded operator $L^{2}(\R^{2}) \to L^{2}_{1,1/2}(\R^{2})$. Repeating the calculation in \eqref{form88}, we find that
\begin{equation}\label{form121} \partial_{s}G_{s}(w,v) = G_{1,s}(w,v) \cdot D_{A,1}(w,v) + G_{2,s}(w,v) \cdot D_{A,2}(w,v), \end{equation} 
where $D_{A,1}(w,v),D_{A,2}(w,v)$ are $L^{\infty}$-factors as in \eqref{form116},
\begin{equation}\label{form118a} G_{1}(x,y,t) := y\partial_{x}G(x,y,t) \quad \text{and} \quad G_{2}(x,y,t) := y^{2}\partial_{t}G(x,y,t), \end{equation}
and of course $G_{j,s}(w,v) = G_{j}(\Gamma_{s}(v)^{-1} \cdot \Gamma_{s}(w))$. So, to prove that $(\partial S)_{s} \colon L^{2}(\R^{2}) \to L^{2}_{1,1/2}(\R^{2})$ is bounded, it suffices to show the boundedness of the operators with kernels
\begin{displaymath} G_{s,1}(w,v) \cdot D_{A,1}(w,v) \quad \text{and} \quad G_{2,s}(w,v) \cdot D_{A,2}(w,v) \end{displaymath}
separately. The necessary pieces of information will be that
\begin{itemize}
\item $G_{1}$ and $G_{2}$ are horizontally even kernels satisfying $|\nabla^{n}G_{j}(p)| \lesssim_{n} \|p\|^{-2 - n}$ for all $p \in \He \, \setminus \, \{0\}$ and $n \geq 0$ (this is immediate from \eqref{form118a}),
\item we have the uniform bounds
\begin{equation}\label{form118} \int |G_{1}(z,t)| \, dt \leq C_{1} \quad \text{and} \quad \int |G_{2}(z,t)| \, dt \leq C_{2}, \qquad z \in \R^{2} \, \setminus \, \{0\}. \end{equation}
\end{itemize}
Perhaps surprisingly, the property \eqref{form118} is not implied by the $2$-dimensional decay assumption in the first point, since it fails for $G(z,t) = c(|z|^{4} + 16t^{2})^{-1/2}$. However, one may explicitly compute that
\begin{displaymath} G_{1}(x,y,t) = \frac{-2xy|z|^{2}}{\sqrt{|z|^{4} + 16t^2}^{3}} \quad \text{and} \quad G_{2}(x,y,t) = \frac{-16y^2t}{\sqrt{|z|^{4} + 16t^2}^{3}}. \end{displaymath}
From these expressions, one sees that $t \mapsto G_{j}(z,t) \in L^{1}(\R)$ for every $z \in \R^{2} \, \setminus \, \{0\}$, and homogeneity considerations yield the uniform bounds required in \eqref{form118}. For example, changing variables $t \mapsto |z|^{2}u$ and estimating $y^{2} \leq |z|^{2}$, we have
\begin{displaymath} \int |G_{2}(z,t)| \, dt \leq 16|z|^{6} \int_{\R} \frac{u \, du}{\sqrt{|z|^{4} + 16|z|^{4}u^{2}}^{3}} = 16 \int_{\R} \frac{u \, du}{\sqrt{1 + 16u^{2}}^{3}} \lesssim 1. \end{displaymath} 
With these properties of the kernels $G_{j}$ in hand, the $L^{2}(\R^{2}) \to L^{2}_{1,1/2}(\R^{2})$ boundedness of $(\partial S)_{s}$ follows from the second part of Theorem \ref{t:bdd}.

%\begin{remark} The first assumption in \eqref{kernelConstants2} may look odd, and it is notably not satisfied by the arguably most basic $2$-dimensional kernel $K = G$, since
%\begin{displaymath} \int |G(z,t)| \, dt = \int \frac{dt}{\sqrt{|z|^{4} + 16t^{2}}} = \infty, \qquad z \in \R^{2}. \end{displaymath} 
%In fact, we do not know if Theorem \ref{t:bdd} holds even in the simplest case $B \equiv 0$ for the kernel $K = G$. More concretely: even though Proposition \ref{prop4} showed that the kernel $((y,t),(y',t')) \mapsto ((y - y')^{4} + 16(t - t')^{2})^{-1/2}$ induces a bounded operator $L^{2}(\R^{2}) \to L^{2}_{1,1/2}(\R^{2})$, we do not know if the same is true for kernels of the form
%\begin{displaymath} ((y,t),(y',t')) \mapsto \frac{1}{\sqrt{(y - y')^{4} + 16(t - t')^{2}}}\frac{A(y) - A(y')}{y - y'},  \end{displaymath} 
%when $A \colon \R \to \R$ is a general Lipschitz function. Fortunately, however, we only need to apply the conclusion of Theorem \ref{t:bdd} to the kernels $y \partial_{x}G$ and $y^{2}\partial_{t}G$, which decay a little faster than $G$ in the $t$-variable.
%\end{remark}

So, we know that the operators $(\partial S)_{s}$ are bounded $L^{2}(\R^{2}) \to L^{2}_{1,1/2}(\R^{2})$, but how does this imply \eqref{form191a}? We introduce the doubly truncated operators 
\begin{displaymath} S_{\epsilon,R,s}'f(w) := \int (\psi_{\epsilon,R}G)(\Gamma_{s}(v)^{-1} \cdot \Gamma_{s}(w))f(v) \, dv, \qquad f \in C^{\infty}_{c}(\R^{2}) \end{displaymath}
where $\psi_{\epsilon,R}$ is a smooth radially symmetric cut-off function with 
\begin{displaymath} \psi_{\epsilon,R}|_{B(\epsilon/2)} \equiv 0, \quad \psi_{\epsilon,R}|_{B(R) \, \setminus B(\epsilon)} \equiv 1, \quad \text{and} \quad \psi_{\epsilon,R}|_{B(2R)^{c}} \equiv 0. \end{displaymath}
We note that $S_{\epsilon,R,s}'f$ is continuous and compactly supported for $f \in C_{c}^{\infty}(\R^{2})$, so its Fourier transform can be computed in the classical sense. For $f \in C^{\infty}_{c}(\R^{2})$ fixed, moreover
\begin{displaymath} [S_{\epsilon,R,s}' - S_{\epsilon,R,t}']f \rightharpoonup [S_{s}' - S_{t}']f \end{displaymath}
as tempered distributions as $\epsilon \to 0$ and $R \to \infty$. So, by Lemma \ref{lemma7}, it will suffice to show that $\|S_{\epsilon,R,s}'f - S_{\epsilon,R,t}'f\|_{1,1/2} \leq C|s - t| \cdot \|f\|_{2}$ with $C > 0$ independent of $\epsilon,R$. Further, by the density of $C^{\infty}_{c}(\R^{2})$ in $L^{2}(\|(\xi,\tau)\|^{2} \, d\xi \, d\tau)$, it will suffice to show that
\begin{equation}\label{form136} \iint \varphi(\xi,\tau) \cdot \|(\xi,\tau)\|^{2} \cdot (\widehat{S_{\epsilon,R,s}'f}(\xi,\tau) - \widehat{S_{\epsilon,R,t}'f}(\xi,\tau)) \, d\xi \, d\tau \leq C|s - t| \cdot \|f\|_{2} \end{equation}
uniformly for all $\varphi \in C^{\infty}_{c}(\R^{2})$ with $\iint \|(\xi,\tau)\|^{2} |\varphi(\xi,\tau)|^{2} \, d\xi \, d\tau = 1$. We write
\begin{displaymath} \gamma(s) := \iint \varphi(\xi,\tau) \cdot \|(\xi,\tau)\|^{2} \cdot \widehat{S_{\epsilon,R,s}'f}(\xi,\tau) \, d\xi \, d\tau, \qquad s \in \R. \end{displaymath}
By Lemma \ref{l:metric}, the inequality \eqref{form136} will follow once we manage to establish that $\limsup_{t \to s} |\gamma(s) - \gamma(t)|/|s - t| \leq C\|f\|_{2}$. This statement will follow from
\begin{equation}\label{form119} \lim_{t \to s} \left|\frac{\gamma(t) - \gamma(s)}{t - s} - \int \varphi(\xi,\tau) \cdot \|(\xi,\tau)\|^{2} \cdot \widehat{(\partial S)_{\epsilon,R,s}f}(\xi,\tau) \, d\xi \, d\tau \right| = 0, \end{equation}
since the operator
\begin{displaymath} (\partial S)_{\epsilon,R,s}f(w) = \int \partial_{s}(\psi_{\epsilon,R}G)(\Gamma_{s}(v)^{-1} \cdot \Gamma_{s}(w))f(v) \, dv \end{displaymath}
is bounded $L^{2}(\R^{2}) \to L^{2}_{1,1/2}(\R^{2})$ as a consequence of Theorem \ref{t:bdd} (the smooth truncation of the kernel $G$ does not affect the conclusion). So, it remains to prove \eqref{form119}. We will suppress "$\epsilon$" and "$R$" from the notation. Note, then, that the difference of the operators appearing  in \eqref{form119} can be written as
\begin{displaymath} Z_{s,t}(w) := \tfrac{S_{s}'f(w) - S_{t}'f(w)}{s- t} - (\partial S)_{s}f(w) = \frac{1}{s - t} \int_{t}^{s} \int_{\R^{2}} \left[ \partial_{\xi}G_{\xi}(w,v) - \partial_{s}G_{s}(w,v) \right]f(v) \, dv \, d\xi. \end{displaymath}
After computing the derivatives $\partial_{\xi}G_{\xi}$ and $\partial_{s}G_{s}$ explicitly as in \eqref{form121}, one finds that $Z_{s,t}(w) \to 0$ for every $w \in \R^{2}$ as $t \to s$. Also, for $0 < \epsilon < R < \infty$ fixed, the functions $Z_{s,t}$, $s,t \in [0,1]$, are uniformly compactly supported and bounded, so their Fourier transforms are uniformly in $L^{\infty}(\R^{2})$, and converge pointwise to $0$ as $t \to s$. Therefore, \eqref{form119} is a consequence of the dominated convergence theorem for any fixed $\varphi \in C^{\infty}_{c}(\R^{2})$. This concludes the proof of \eqref{form191a}, that is, $\|S_{s}' - S_{t}'\|_{L^{2}(\R^{2}) \to L^{2}_{1,1/2}(\R^{2})} \leq C|s - t|$.

\subsection{Verifying (2)}\label{s:checklist2} Before turning to the proof of Theorem \ref{t:bdd}, we establish the point (2) on our checklist in Section \ref{s:checklist}, namely that $\|S_{s}f\|_{1,1/2} \geq c\|f\|_{2}$ for all $s \in [0,1]$ and $f \in L^{2}(\R^{2})$. Since we already know that $S_{s} \colon L^{2}(\R^{2}) \to L^{2}_{1,1/2}(\R^{2})$ is bounded, it suffices to verify the lower bound for all "$f$" in a dense subset of $L^{2}(\R^{2})$. Functions of the form $\varphi_{s} := \varphi \circ \Gamma_{s}$ with $\varphi \in C^{\infty}_{c}(\He)$ are a comfortable choice, because during the proof of Proposition \ref{prop5} (more precisely \eqref{form134}) we showed that
\begin{equation}\label{form138} \|S_{s}\varphi_{s}\|_{1,1/2} \sim_{\mathrm{Lip}(A)} \|\nabla_{\tau}S\varphi\|_{L^{2}(\sigma_{s})} + \|T^{1/2}\mathcal{S}\varphi\|_{L^{2}(\sigma_{s})}, \qquad s \in [0,1]. \end{equation}
Since $\|\varphi_{s}\|_{2} \sim \|\varphi\|_{L^{2}(\sigma_{s})}$ by the area formula \eqref{form16}, we also have
\begin{displaymath} \|\varphi_{s}\|_{2} \lesssim \|(\tfrac{1}{2}I + D_{s}^{t})\varphi\|_{L^{2}(\sigma_{s})} + \|(-\tfrac{1}{2}I + D_{s}^{t})\varphi\|_{L^{2}(\sigma_{s})} = \|\nabla_{\nu}^{-}\mathcal{S}\varphi\|_{L^{2}(\sigma_{s})} + \|\nabla_{\nu}^{+}\mathcal{S}\varphi\|_{L^{2}(\sigma_{s})}. \end{displaymath} 
Here we used the jump relations in Corollary \ref{c:allJumps} for the interior and exterior normal derivatives of $\mathcal{S}\varphi$. In \eqref{form3}, we already established that
 \begin{displaymath} \max\{\|\nabla^{-}_{\nu}\mathcal{S}\varphi\|_{L^{2}(\sigma_{s})},\|\nabla^{+}_{\nu}\mathcal{S}\varphi\|_{L^{2}(\sigma_{s})}\} \lesssim \|\nabla_{\tau}S\varphi\|_{L^{2}(\sigma_{s})} + \|\varphi\|_{L^{2}(\sigma_{s})}^{1/2}\|T^{1/2}\mathcal{S}\varphi\|_{L^{2}(\sigma_{s})}^{1/2}.  \end{displaymath}
 Now the inequality $\|\varphi_{s}\|_{2} \lesssim_{\mathrm{Lip}(A)} \|S_{s}\varphi_{s}\|_{1,1/2}$, for $s \in [0,1]$, is an immediate consequence of $\|\varphi\|_{L^{2}(\sigma_{s})} \sim \|\varphi_{s}\|_{2}$ and \eqref{form138}. This completes the proof of (2) on our checklist.

\subsection{Proof of Theorem \ref{t:bdd}} To prove the invertibility of $S_{A} \colon L^{2}(\R^{2}) \to L^{2}_{1,1/2}(\R^{2})$, and hence (by Remark \ref{rem3}) the invertibility of $S \colon L^{2}(\sigma) \to L^{2}_{1,1/2}(\sigma)$, it only remains to prove Theorem \ref{t:bdd}. Fix $\varphi \in C^{\infty}_{c}(\R^{2})$, and write 
\begin{displaymath} T_{j}\varphi(w) := \int K_{B}(w,v)D_{A,j}(w,v)\varphi(v) \, dv, \qquad j \in \{1,2\}. \end{displaymath}
The integral is absolutely convergent since the kernel $K_{B}$ is locally integrable, following the argument in Remark \ref{rem4}. If $K^{\epsilon,R}(p) = \psi_{\epsilon,R}(p)K(p)$ is any (smooth double) truncation of $K$, it then follows from dominated convergence that
\begin{displaymath} T^{\epsilon,R}_{j}\varphi(w) := \int K_{B}^{\epsilon,R}(w,v)D_{A,j}(w,v)\varphi(v) \, dv \to T_{j}\varphi(w), \qquad \text{a.e. } w \in \R^{2}, \end{displaymath}
which implies that $T_{j}^{\epsilon,R}\varphi \rightharpoonup T_{j}\varphi$ in the sense of distributions as $\epsilon \to 0$ and $R \to \infty$. So, if we manage to show that $\|T_{j}^{\epsilon,R}\varphi\|_{1,1/2} \lesssim_{A,B,C_{n},C_{\star}} \|\varphi\|_{2}$, it will follow from Lemma \ref{lemma7} that also $T_{j}\varphi \in L^{2}_{1,1/2}(\R^{2})$, and
\begin{displaymath} \|T_{j}\varphi\|_{1,1/2} \leq \sup_{0 < \epsilon < R < \infty} \|T_{j}^{\epsilon,R}\varphi\|_{1,1/2} \lesssim_{A,B,C_{n},C_{\star}} \|\varphi\|_{2}. \end{displaymath}
Therefore, we may assume that $\spt K$ is a compact subset of $\He \, \setminus \, \{0\}$. This will be (implicitly) used to justify applications of Fubini's theorem and differentiations under the integral sign. Also, if both $\spt K \subset \He \, \setminus \, \{0\}$ and $\spt \varphi \subset \R^{2}$ are compact, then $T_{j}\varphi$ is a bounded compactly supported continuous function.

Under this \emph{a priori} assumption, we will show that
\begin{equation}\label{form96} \iint |\partial_{y}T_{j}\varphi(y,t)|^{2} \, dy \, dt = \iint |\partial_{y}\widehat{T_{j}\varphi}(y,\tau)|^{2} \, d\tau \, dy \lesssim_{A,B,C_{n},C_{\star}} \|\varphi\|_{2}^{2} \end{equation}
and
\begin{equation}\label{form97} \quad \iint |\tau| \cdot |\widehat{T_{j}\varphi}(y,\tau)|^{2} \, d\tau \, dy \lesssim_{A,C_{1},C_{2}} \|\varphi\|_{2}^{2}. \end{equation} 
Here $\widehat{T_{j}\varphi}$ refers to the partial Fourier transform ot $T_{j}\varphi$ in the $t$-variable, i.e. $\widehat{T_{j}\varphi}(y,\tau) := \widehat{(T_{j}\varphi)_{y}}(\tau)$. Note that these estimates imply $\|T_{j}\varphi\|_{1,1/2} \lesssim \|\varphi\|_{2}$ by Lemma \ref{lemma8}.

For as long as possible, we work under the (kernel decay) assumption \eqref{kernelConstants2} alone. Also, to make the proof a little shorter, we will combine some arguments regarding the two terms in \eqref{form96}-\eqref{form97}. 
We recall that $\Gamma_{B}(y,t) = \gamma(y) \cdot (0,0,t)$, where $\gamma(y) = (B(y),y,-\tfrac{y}{2}B(y) + \int_{0}^{y} B(r) \, dr)$. With this notation, we may calculate
\begin{align} \widehat{T_{j}\varphi}(y,\tau) & = \int e^{-2\pi i t\tau} \iint K(\Gamma_{B}(y',t')^{-1} \cdot \Gamma_{B}(y,t))D_{A,j}(y,y')\varphi(y',t') \, dy' \, dt' \, dt \notag\\
& = \int e^{-2\pi i t\tau} \iint K(\gamma(y')^{-1} \cdot \gamma(y) \cdot (0,0,t - t'))D_{A,j}(y,y')\varphi(y',t') \, dy' \, dt' \, dt \notag\\
& = \iint \varphi(y',t')D_{A,j}(y,y') \left[ \int e^{-2\pi i t\tau} K(\gamma(y')^{-1} \cdot \gamma(y) \cdot (0,0,t - t')) \, dt  \right] \, dt' \, dy' \notag\\
& = \iint \varphi(y',t')D_{A,j}(y,y') \left[ \int e^{-2\pi i (u + t')\tau} K(\gamma(y')^{-1} \cdot \gamma(y) \cdot (0,0,u)) \, du \right] \, dt' \, dy' \notag\\
&\label{form123} = \int \widehat{\varphi}(y',\tau)D_{A,j}(y,y') \left[ \int e^{-2\pi i u\tau} K(\gamma(y')^{-1} \cdot \gamma(y) \cdot (0,0,u)) \, du \right] \, dy',  \end{align}
where $\widehat{\varphi}(y',\tau) = \widehat{\varphi_{y}}(\tau)$ is the Fourier transform of $t \mapsto \varphi(y,s)$ at $\tau$. Motivated by final line above, we define the horizontally even kernels
\begin{equation}\label{eq:smallk} k_{K,\tau}(p) := \int e^{-2\pi i u\tau} K(p \cdot (0,0,u)) \, du, \qquad \tau \in \R, \end{equation} 
so that 
\begin{equation}\label{form124} \int e^{-2\pi i u\tau} K(\gamma(y')^{-1} \cdot \gamma(y) \cdot (0,0,u)) \, du = k_{K,\tau}(\gamma(y')^{-1} \cdot \gamma(y)). \end{equation}
It turns out that the $\tau$-dependence leads to technical issues, so we wish to reduce it. To this end, for $r > 0$, define the auxiliary kernels $K_{r}(z,t) := r^{-2}K(r^{-1}z,r^{-2}t)$, and note that $K_{r}$ satisfies the same kernel estimates \eqref{kernelConstants2}-\eqref{kernelConstants3} as $K$ does, and with constants independent of $r > 0$. With this notation, plus writing $\rho := \sqrt{|\tau|}$ and changing variables from $u$ to $u/|\tau|$, we have
\begin{equation}\label{form125} k_{K,\tau}(p) = \int e^{-2\pi i u \sgn(\tau)} (K_{\rho})(\delta_{\rho}(p) \cdot (0,0,u)) \, d\theta = k_{K_{\rho},\sgn(\tau)}(\delta_{\rho}(p)).\end{equation}
The right hand side refers to a kernel of the same type as $k_{K,\tau}$ above, with $\tau = 1$, and with $K$ replaced by the "renormalised" kernel $K_{\rho}$. Combining \eqref{form123}-\eqref{form125}, we find that
\begin{displaymath} \widehat{T_{j}\varphi}(y,\tau) = \int \widehat{\varphi}(y',\tau)D_{A,j}(y,y')k_{K_{\rho},\sgn(\tau)}(\delta_{\rho}(\gamma(y')^{-1} \cdot \gamma(y))) \, dy', \end{displaymath}
and consequently (changing variables $y \mapsto x/\rho$ and $y' \mapsto x'/\rho$),
\begin{displaymath} \eqref{form96} = \int \frac{1}{\rho} \int \left| \partial_{x} \int \widehat{\varphi}\left(\tfrac{x'}{\rho},\tau \right)D_{A,j}\left(\tfrac{x}{\rho},\tfrac{x'}{\rho} \right)k_{K_{\rho},\sgn(\tau)}(\delta_{\rho}(\gamma(x')^{-1} \cdot \gamma(x))) \, dx' \, dx'\right|^{2} \, dx \, d\tau. \end{displaymath}
We remark that, repeating the arguments above verbatim, one sees that \eqref{form97} also equals the displayed formula above, except that the "$\partial_{x}$" should be removed. This expression looks unwieldy, but it can be substantially simplified. First, we fix $\tau \neq 0$ and define $\psi_{\tau}(x) := \widehat{\varphi}\left(x/\rho,\tau \right)$. We then claim, for $\tau \neq 0$ fixed, that
\begin{equation}\label{form98} \int \left| \partial_{x} \int \psi_{\tau}(x')D_{A,j}\left(\tfrac{x}{\rho},\tfrac{x'}{\rho} \right)k_{\rho^{-1}K,\sgn(\tau)}(\delta_{\rho}(\gamma(x')^{-1} \cdot \gamma(x))) \, dx' \, dy'\right|^{2} \, dx \lesssim \|\psi_{\tau}\|_{2}^{2}, \end{equation}
where the implicit constants may depend on $A,B,C_{n},C_{\star}$. Once \eqref{form98} has been verified, we may infer that
\begin{displaymath} \eqref{form96} \lesssim_{A,B,C_{n},C_{\star}} \int \frac{1}{\rho} \int |\psi_{\tau}(x)|^{2} \, dx \, d\tau = \iint |\widehat{\varphi}(x,\tau)|^{2} \, dx \, d\tau = \|\varphi\|_{2}^{2}, \end{displaymath}
as desired. To deal with \eqref{form97}, the same argument as above shows that it suffices to prove the following estimate:
\begin{equation}\label{form108} \int \left| \int \psi_{\tau}(x')D_{A,j}\left(\tfrac{x}{\rho},\tfrac{x'}{\rho} \right)k_{\rho^{-1}K,\sgn(\tau)}(\delta_{\rho}(\gamma(x')^{-1} \cdot \gamma(x))) \, dx' \, dy'\right|^{2} \, dx \lesssim \|\psi_{\tau}\|_{2}^{2}. \end{equation}
The only difference between \eqref{form96} and \eqref{form108} is the presence, or lack, of "$\partial_{x}$". Heuristically, this makes \eqref{form98} into a rather delicate statement regarding singular integrals, whereas \eqref{form108} can be, eventually, handled by maximal function estimates.

To prove \eqref{form98}-\eqref{form108}, we abbreviate $\psi_{\tau} =: \psi$ and $k := k_{K_{\rho},\sgn(\tau)}$. It might look like we are completely sweeping the $\tau$-dependence under the carpet, but all we will use is that $k$ is a kernel of the form $k_{K,1}$ or $k_{K,-1}$ (as in \eqref{eq:smallk}) associated to some horizontally even kernel $K$ satisfying \eqref{kernelConstants2} (and \eqref{kernelConstants3} in the second part of the theorem).

 By direct computation, we moreover observe that
\begin{displaymath} D_{A,j}\left(\tfrac{x}{\rho},\tfrac{x'}{\rho} \right) = D_{A_{\rho},j}(x,x'), \qquad j \in \{1,2\}, \, x,x' \in \R, \end{displaymath}
where $A_{\rho}(x) = \rho \cdot A(x/\rho)$, and
\begin{displaymath} \delta_{\rho}(\gamma(x')^{-1} \cdot \gamma(x)) = \gamma_{\rho}(x')^{-1} \cdot \gamma_{\rho}(x), \qquad x,x' \in \R, \end{displaymath}
where 
\begin{displaymath} \gamma_{\rho}(x) = \left(B_{\rho}(x),x,-\tfrac{x}{2}B_{\rho}(x) + \int_{0}^{x} B_{\rho}(r) \, dr \right). \end{displaymath}
Since $A_{\rho},B_{\rho}$ have the same Lipschitz constants as $A$ and $B$, and this is all the information our bounds will depend on, we may further abbreviate $D_{A_{\rho},j} := D_{A,j}$ and $\gamma_{\rho} := \gamma$. 

With this notation, \eqref{form98} and \eqref{form108} are equivalent to
\begin{equation}\label{form126} \int \left| \partial_{x} \int \psi(x')D_{A,j}(x,x')k(\gamma(x')^{-1} \cdot \gamma(x)) \, dx' \right|^{2} \, dx \lesssim_{A,B,C_{n},C_{\star}} \|\psi\|_{2}^{2} \end{equation}
and
\begin{equation}\label{form127} \int \left|\int \psi(x')D_{A,j}(x,x')k(\gamma(x')^{-1} \cdot \gamma(x)) \, dx' \right|^{2} \, dx \lesssim_{A,C_{1},C_{2}} \|\psi\|_{2}^{2}, \end{equation}
respectively.

To make further progress, we examine the kernel $k$ more closely. We claim that $k$ is a horizontally even kernel (this is immediate from the definition) satisfying 
\begin{equation}\label{form102} |\nabla^{n}k(z,t)| \lesssim C_{n}|z|^{-n}, \qquad (z,t) \in \He \, \setminus \, \{|z| = 0\} \end{equation} 
for all $n \geq 1$, and
\begin{equation}\label{form128} |k(z,t)| \lesssim (C_{1} + C_{2}) \min\{\ln^{+} \tfrac{1}{|z|} + 1,|z|^{-2}\}. \end{equation}
If \eqref{kernelConstants3} is assumed, then \eqref{form128} can be further upgraded to
\begin{equation}\label{form129} |k(z,t)| \lesssim \min\{C_{\star},C_{2}|z|^{-2}\}. \end{equation}
In fact, the decay for $|z| \gg 1$ could be much improved, but we will not need this.

The cases $n \geq 1$ are completely straightforward, using the uniform kernel estimates \eqref{kernelConstants2} for $K = K_{\rho}$:
\begin{align} |\nabla^{n}k(z,t)| \leq \int |\nabla^{n}K(z,t + u)| \, du & \leq C_{n}\int \frac{du}{\|(z,t + u)\|^{n + 2}} \notag\\
&\label{form110} \stackrel{u \mapsto |z|^{2}v}{=} C_{n} \int \frac{|z|^{2} \, dv}{(|z|^{4} + 16|z|^{4}v^{2})^{(n + 2)/4}}\\
& = \frac{C_{n}}{|z|^{n}} \int \frac{dv}{(1 + 16v^{2})^{(n + 2)/4}}. \notag\end{align} 
The final integral is uniformly bounded for $n \geq 1$, and this gives \eqref{form102} for $n \geq 1$. We note that, as a corollary, also 
\begin{equation}\label{form130} |\partial_{y}k(z,t)|,|\partial_{x}k(z,t)| \lesssim (C_{1} + C_{2})|z|^{-1}, \end{equation}
using that $\partial_{x} = X + (y/2)\partial_{t} = X + (y/2)(XY - YX)$ and $\partial_{y} = Y - (x/2)\partial_{t}$.

To handle the case $n = 0$, we first integrate by parts:
\begin{equation}\label{form131} |k(z,t)| \leq \frac{1}{2 \pi} \int |\partial_{t}K(z,t + u)| \, du \leq \frac{1}{2\pi} \int |\nabla^{2}K(z,t + u)| \, du \lesssim \frac{C_{2}}{|z|^{2}}.\end{equation}
This gives the decay at infinity claimed in \eqref{form128}-\eqref{form129}. The uniform bound claimed in \eqref{form129} is immediate from the definition of the constant $C_{\star}$, so it remains to verify the logarithmic blow-up at $0$ stated in \eqref{form128}. We already know from \eqref{form131} that if $|z_{0}| = 1$, then $|k(z_{0},t)| \lesssim C_{2}$, and on the other hand $|\nabla_{z}k(z,t)| \lesssim (C_{1} + C_{2})|z|^{-1}$ by \eqref{form130}. Therefore, if $z = |z| z_{0} \in B(1) \, \setminus \, \{0\}$ with $0 < |z| < 1$, we may write
\begin{displaymath} |k(z,t) - k(z_{0},t)| \leq \int_{|z|}^{1} |\nabla_{z}k(rz_{0},t)| \, dr \lesssim (C_{1} + C_{2}) \int_{|z|}^{1} \frac{dr}{r} = (C_{1} + C_{2}) \ln \frac{1}{|z|}.\end{displaymath}
This completes the proof of the estimate \eqref{form128}.

The proof of \eqref{form127} is now rather immediate: we simply put absolute values inside and first estimate $\|D_{A,j}\|_{L^{\infty}(\R \times \R)} \lesssim \mathrm{Lip}(A)$. Then, we note that the $z$-component of $\gamma(x')^{-1} \cdot \gamma(x)$ is at least $|x - x'|$, so 
\begin{displaymath} \eqref{form127} \lesssim_{A,C_{1},C_{2}} \int \left( \int |\psi(x')| \min\{\ln^{+} \tfrac{1}{|x - x'|} + 1,|x - x'|^{-2}\} \, dx' \right)^{2} \, dx. \end{displaymath}
But $x \mapsto \min\{\ln^{+} |x|^{-1} + 1, |x|^{-2}\} \in L^{1}(\R)$, so the RHS is dominated by $\|\psi\|_{2}^{2}$ by Young's convolution inequality. This completes the proof of \eqref{form127}, and hence \eqref{form97}. Note that this estimate did not involve the constant $C_{\star}$, so it holds as long as $K$ satisfies \eqref{kernelConstants2}.

Now we move to the more delicate estimate \eqref{form126}, where the first task will be to calculate the derivative
\begin{align}\label{form99} \partial_{x}\left[ k(\gamma(x')^{-1} \cdot \gamma(x))D_{A,j}(x,x') \right] & = [\partial_{x}k(\gamma(x')^{-1} \cdot \gamma(x))] \cdot D_{A,j}(x,x')\\
&\label{form100} \qquad + k(\gamma(x')^{-1} \cdot \gamma(x)) \cdot [\partial_{x}D_{A,j}(x,x')]. \end{align}
Note that if $D_{A,j}(x,x') \equiv 1$ (this corresponds to the first part of the theorem, where we only consider the kernel $K_{B}$, and not $K_{B}D_{A,j}$), the term on line \eqref{form100} vanishes identically. And, as we will see immediately below, handling the term \eqref{form100} is the only place in the proof where the stronger "$C_{\star}$"-kernel assumption \eqref{kernelConstants3} is required.

We note that the derivative $\partial_{x} D_{A_{j}}(x,x')$ is a linear combination of
\begin{displaymath} \frac{A'(x)}{x - x'}, \quad \frac{D_{A,1}(x,x')}{(x - x')} \quad \text{and} \quad \frac{D_{A,2}(x,x')}{x - x'} \end{displaymath} 
for a.e. $x \in \R$ (where in fact $A = A_{\rho}$). This means that once the term coming from \eqref{form100} is placed into the LHS of \eqref{form126}, we will face the sum of expressions of the form
\begin{equation}\label{form101} \int \left| \int \psi(x') \frac{k(\gamma(x')^{-1} \cdot \gamma(x))}{2(x - x')}D_{A}(x,x')D_{A,j}(x,x') \, dx' \right|^{2} \, dx, \end{equation}
where $D_{A}(x,x') \in \{A'(x),D_{A,1}(x,x'),D_{A,2}(x,x')\}$. Recalling from \eqref{form129} that $k \in C^{\infty}(\He \, \setminus \, \{|z| = 0\})$ is a horizontally even $L^{\infty}$-function, the kernel
\begin{displaymath} (x,x') \mapsto \frac{k(\gamma(x')^{-1} \cdot \gamma(x))}{2(x - x')}D_{A}(x,x')D_{A,j}(x,x') \end{displaymath}
in \eqref{form101} is essentially of the type treated by \cite[Theorem 4.3]{2019arXiv191103223F}; we recall the result, and what "essentially" means in Remark \ref{rem1}. This theorem then states that 
\begin{equation}\label{form106} \eqref{form101} \lesssim_{A,B,C_{n},C_{\star}} \|\psi\|_{2}^{2}. \end{equation}
Of course, it is critical here that the Lipschitz constants of $A_{\rho}$ and $B_{\rho}$ (hiding in the definition of $\gamma = \gamma_{\rho}$) are no larger than those of $A$ and $B$, respectively.

Next, we consider the term coming from \eqref{form99}. We claim that this term can be handled without the "$C_{\star}$" kernel assumption. Unwrapping the definitions (in particular, recall that $k = k_{K_{\rho},\sgn(\tau)}$ is a kernel of the type defined in \eqref{eq:smallk}), and differentiating under the integral sign, we find $\partial_{x}k(\gamma(x')^{-1} \cdot \gamma(x)) =$
\begin{displaymath} \int e^{-2\pi i u \sgn(\tau)} \partial_{x} \left[x \mapsto  K\left(B(x) - B(x'),x - x',u + \int_{x}^{x'} \tfrac{B(x) + B(x') - 2B(\theta)}{2} \, d\theta \right) \right] \, du, \end{displaymath}
where in fact $\gamma = \gamma_{\rho}$, $K = K_{\rho}$, and $B = B_{\rho}$. A direct calculation shows that the RHS is a linear combination of terms of the form
\begin{align} D_{B}(x,x') & \int e^{-2\pi i u\sgn(\tau)} \mathcal{K}\left(B(x) - B(x'),x - x',u + \int_{x}^{x'} \tfrac{B(x) + B(x') - 2B(\theta)}{2} \, d\theta \right) \, du \notag\\
&\label{form104} = D_{B}(x,x') \int e^{-2\pi i u\sgn(\tau)} \mathcal{K}(\gamma(x')^{-1} \cdot \gamma(x) \cdot (0,0,u)) \, du, \end{align}
where $\mathcal{K} \in \{\partial_{x}K,\partial_{y}K,y\partial_{t}K\}$ and 
\begin{displaymath} D_{B}(x,x') \in \left\{\mathrm{const},B'(x),D_{B,1}(x,x')\right\}. \end{displaymath}
Since $K$ was assumed horizontally even, every $\mathcal{K}$, as above, is horizontally odd, and it follows from the estimates in \eqref{kernelConstants2} that
\begin{displaymath} |\mathcal{K}(p)| \lesssim C_{n}\|p\|^{-3 - n}, \qquad p \in \He \, \setminus \, \{0\}, \, n \geq 0. \end{displaymath} 
The kernels appearing in \eqref{form104}, namely
\begin{displaymath} \kappa(p) := \int e^{-2\pi i u\sgn(\tau)} \mathcal{K}(p \cdot (0,0,u)) \, du, \end{displaymath}
are therefore also horizontally odd, and satisfy $|\nabla^{n}\kappa(z,t)| \lesssim C_{n}|z|^{-1 - n}$ for $(z,t) \in \He$ with $z \neq 0$, and $n \geq 0$. These estimates follow by repeating the computations in \eqref{form110}, and the "$C_{\star}$"-constant is not required. Heuristically, the point is that when $\mathcal{K}$ is a $3$-dimensional kernel to begin with, then $\kappa$ turns out to be a $1$-dimensional kernel. However, in the computations surrounding \eqref{form110}, the kernel $K$ was only $2$-dimensional, and the "$0$-dimensionality" of $k$ is a more delicate point.

Now, plugging in the contribution from \eqref{form99} to \eqref{form98}, and rewriting $\partial_{x}k_{\tau}(\gamma(x')^{-1} \cdot \gamma(x))$ as we have done in \eqref{form104}, we find terms of the form
\begin{equation}\label{form105} \int \left| \int \psi(y')\kappa(\gamma(y')^{-1} \cdot \gamma(y))D_{A,j}(y,y')D_{B}(y,y') \, dy \right|^{2} \, dy. \end{equation} 
As before, \cite[Theorem 4.3]{2019arXiv191103223F} applied to the kernels
\begin{displaymath} (x,x') \mapsto \kappa(\gamma(x')^{-1} \cdot \gamma(x))D_{A,2}(x,x')D_{B}(x,x') \end{displaymath} 
shows that
\begin{displaymath} \eqref{form105} \lesssim \|\psi\|_{2}^{2}. \end{displaymath}
(See Remark \ref{rem1} for similar provisos as we needed to reach \eqref{form106}.) We have now established \eqref{form126}, and hence completed the proof of the theorem. 

We close with a brief summary. The theorem claimed that $K_{B}(w,v)$ induces a bounded operator $L^{2}(\R^{2}) \to L^{2}_{1,1/2}(\R^{2})$ only assuming the kernel estimates \eqref{kernelConstants2}. This worked, because (i) the easier estimate \eqref{form97} (even for the kernels $K_{B}D_{A,j}(w,v)$) did not require the kernels $k$ to be bounded (the logarithmic singularity at $0$ was manageable), and (b) in the more difficult estimate \eqref{form96}, the only term requiring the boundedness of $k$ (namely \eqref{form100}) vanished. For a general Lipschitz function $A \colon \R \to \R$, however, our proof needed the kernel $k$ to be an $L^{\infty}$-function, and this is simply not true without the additional kernel assumption \eqref{kernelConstants3}. For example, if $K = G$, one finds that $|k(z,t)| \sim K_{0}(\tfrac{\pi}{2} |z|^{2})$, where $K_{0}$ is the modified Bessel function of the first kind of index zero (this function appears prominently in Section \ref{appA}). The function $K_{0}$ indeed has a logarithmic singularity at the $0$, which shows that the estimate \eqref{form128} is fairly optimal near the origin.

\begin{remark}\label{rem1} The proof of Theorem \ref{t:bdd} made two references to \cite[Theorem 4.3]{2019arXiv191103223F}, which states the following. If $k \in C^{\infty}(\He \, \setminus \, \{|z| = 0\})$ is a horizontally odd kernel satisfying $|\nabla^{n}k(z,t)| \lesssim_{n} |z|^{-1 - n}$ for $z \in \R^{2} \, \setminus \, \{0\}$, and $A,B \colon \R \to \R$ are Lipschitz functions, then the kernels $K_{B}D_{A,1}$ and $K_{B}D_{A,2}$ yield operators bounded on $L^{2}(\R)$, where $D_{A,j}$ are the same $L^{\infty}$-factors which appear in the statement of Theorem \ref{t:bdd}, and
\begin{displaymath} K_{B}(y,y') = k\left(B(y) - B(y'),y - y',\int_{y}^{y'} \frac{B(y) + B(y') - 2B(\theta)}{2} \, d\theta \right). \end{displaymath}
In \eqref{form101}, we needed two mild extensions of this result. First, in place of $k$ above, we had an an even "$0$-dimensional" kernel $k \in C^{\infty}(\He \, \setminus \, \{|z| = 0\})$ satisfying $|k(z,t)| \lesssim_{n} |z|^{-n}$, and $K_{B}$ is replaced by
\begin{displaymath} \widetilde{K}_{B}(y,y') = \frac{1}{y - y'} \cdot k\left(B(y) - B(y'),y - y',\int_{y}^{y'} \frac{B(y) + B(y') - 2B(\theta)}{2} \, d\theta \right). \end{displaymath}
Second, we needed to know that $\widetilde{K}_{B}D_{A_{1}}D_{A_{2}}$ yields an operator bounded on $L^{2}(\R)$ with two distinct factors $D_{A_{1}}, D_{A_{2}}$ of the familiar kind. The extension to two $D_{A,j}$-factors was also needed in the second application of \cite[Theorem 4.3]{2019arXiv191103223F}, right after \eqref{form105}.

The extension to two -- or more -- factors $D_{A,j}$ requires only notational changes in the proof of \cite[Theorem 4.3]{2019arXiv191103223F}, and this is even mentioned in \cite[Remark 4.5]{2019arXiv191103223F}. In fact, the presence of a single factor causes small technicalities, explained in \cite[Remark 4.58]{2019arXiv191103223F}, but any number of further factors can be dealt with in the same way.

How about the difference between the kernels $K_{B}$ and $\widetilde{K}_{B}$? The proof presented for $K_{B}$ in \cite{2019arXiv191103223F} only requires one change: as the argument begins right after the statement of \cite[Theorem 4.3]{2019arXiv191103223F}, one needs to replace the auxiliary function
\begin{displaymath} "\kappa(u;\theta_{1},\theta_{2}) := \chi(\theta_{1},\theta_{2})k(u,u \cdot (2L\theta_{1}),u^{2} \cdot (4L^{2}\theta_{2}))" \end{displaymath} 
by the function
\begin{displaymath}  "\kappa(u;\theta_{1},\theta_{2}) = \frac{\chi(\theta_{1},\theta_{2})}{u}k(u,u \cdot (2L\theta_{1}),u^{2} \cdot (4L^{2}\theta_{2}))". \end{displaymath}
After this change, the argument works verbatim. The horizontal evenness of $k$ implies that $\kappa$ is an odd function of $u$. Consequently, also using the smoothness of $k$ away from $\{|z| = 0\}$, the kernels "$\kappa_{\mathbf{n}}$" appearing shortly afterwards are odd $1$-dimensional standard kernels, whose kernel constants decay rapidly as $|\mathbf{n}| \to \infty$. The properties of the kernels in \cite{2019arXiv191103223F} are only needed to derive these facts, and they remain true in our setting. \end{remark}

%%%%%%%%%%%%%%%%%%%%%%%%%%%%%%%%%

\section{Continuity up to the boundary}

Our main results, Theorems \ref{main3} and \ref{main5}, yield solutions to the Dirichlet problem \eqref{dirichlet} which converge to the boundary values "$g$" non-tangentially $\sigma$ a.e. The main question in this section is: can the a.e. non-tangential convergence be upgraded to classical convergence (everywhere), assuming that $g \in C(\partial \Omega)$? Besides being a natural question, a positive answer would, and will, be useful in the proof of Theorem \ref{main6}, concerning the uniqueness in the Dirichlet problem. Theorem \ref{t:regularity} answers the question affirmatively if $\partial \Omega$ is smooth. The Lipschitz case remains open. Fortunately, the smooth case will be sufficient for the application in Theorem \ref{main6}, see Section \ref{s:uniqueness}.

Recall that if $\Omega = \{(x,y,t) : x < A(y)\}$ is a flag domain, and $u \colon \Omega \to \C$ is a function, typically $\bigtriangleup^{\flat}$-harmonic, then the radial limit of $u$ at $p \in \partial \Omega$ (along $\Omega = \mathrm{dom}(u)$) is
\begin{displaymath} u_{\mathrm{rad}}(p) := \lim_{t \to 0+} u(p \cdot (-t,0,0)), \end{displaymath}
whenever the limit exists. We also recall from Section \ref{s:nt} the radial maximal function
\begin{displaymath} (\mathcal{N}_{\mathrm{rad}}u)(p) := \sup_{r > 0} |u(p \cdot (-r,0,0))|. \end{displaymath}
\begin{thm}\label{t:regularity} Let $A \in C^{\infty}(\R)$, $\Omega := \{(x,y,t) : x < A(y)\}$, and $\sigma := |\partial \Omega|_{\He}$. Let $g \in C(\partial \Omega)$, and let $u \colon \Omega \to \C$ be a $\bigtriangleup^{\flat}$-harmonic function such that $u_{\mathrm{rad}}(p) = g(p)$ for $\sigma$ a.e. $p \in \partial \Omega$.  If $\mathcal{N}_{\mathrm{rad}}u \in L^{1}_{\mathrm{loc}}(\sigma)$, then $u \in C(\overline{\Omega})$. \end{thm}
The proof of Theorem \ref{t:regularity} relies on the existence and boundary regularity of Green's function in bounded domains in $D \subset \He$. We briefly recall the definition. Fix $p \in D$, and let $u = u_{D,p}$ be the \emph{Perron-Wiener-Brelot} solution (see \cite[\S6.7]{BLU}) to the Dirichlet problem
\begin{equation}\label{eq:greenFunction} \begin{cases} \bigtriangleup^{\flat}u = 0 \text{ in } D, \\ u|_{\partial D} = G(p^{-1} \cdot q). \end{cases} \end{equation} 
As always, $G$ is the fundamental solution of $-\bigtriangleup^{\flat}$. The function $G_{D}(p,q) := G(p^{-1} \cdot q) - u_{D,p}(q)$ is known as \emph{Green's function for $D$}. As explained in \cite[\S 9.2]{BLU}, in particular \cite[Theorem 9.2.4]{BLU}, Green's function $G_{D}(p,\cdot)$ equals $G(p^{-1} \cdot)$ minus the greatest $\bigtriangleup^{\flat}$-harmonic minorant of $G(p^{-1} \cdot)$ in $D$; in particular $0 \leq G_{D}(p,q) \leq G(p^{-1} \cdot q)$ for all $p,q \in D$. Also, \cite[Proposition 9.2.10]{BLU} states that Green's function is symmetric:
\begin{equation}\label{greenSymmetry} G_{D}(p,q) = G_{D}(q,p), \qquad p,q \in D. \end{equation}
Green's function will be used via the following classical representation formula:
\begin{lemma}\label{l:greenRep} Let $D \subset \He$ be an open set with locally finite $\He$-perimeter, and let $u \in C^{2}_{c}(D)$. Then,
\begin{displaymath} u(p) = -\int_{D} G_{D}(p,q)(\bigtriangleup^{\flat}u)(q) \, dq, \qquad p \in D. \end{displaymath} 
\end{lemma}

\begin{proof} We give a slightly formal argument; see \cite[p. 34]{MR1625845} for full details. Fix $p \in D$, and define the horizontal vector field $V \in C^{1}_{c}(D)$ by
\begin{displaymath} V(q) := u(q) \nabla [q \mapsto G_{D}(p,q)] - G_{D}(p,q) \nabla u(q), \qquad q \in D \, \setminus \, \{p\}. \end{displaymath}
Then $\mathrm{div}^{\flat} V(q) = -u(q) \cdot \delta_{p} - G_{D}(p,q)(\bigtriangleup^{\flat} u)(q)$, so the divergence theorem in \cite{FSSC} yields
\begin{displaymath} -u(p) - \int_{D} G_{D}(p,q)(\bigtriangleup^{\flat} u)(q) \, dq = \int \mathrm{div}^{\flat} V(q) \, dq  = -\int_{\partial D} \langle V,\nu \rangle \, d|\partial D|_{\He} = 0, \end{displaymath} 
since $V \in C^{1}_{c}(D)$. Rearranging terms yields the claim.  \end{proof}

Before turning to further properties of Green's functions, we record another useful consequence of the divergence theorem:

\begin{lemma}\label{lemma11} Let $D \subset \He$ be an open set with locally finite $\He$-perimeter, let $u,v \in C^{1}(D)$, and let $\eta \in C^{2}_{c}(D)$. Then,
\begin{displaymath} \int u\langle \nabla v,\nabla \eta \rangle \, dp = -\int \langle \nabla u,\nabla  \eta \rangle v \, dp - \int_{D} u v(\bigtriangleup^{\flat} \eta) \, dp. \end{displaymath} \end{lemma}

\begin{proof} Consider the horizontal vector field $V := uv \nabla \eta \in C^{1}_{c}(D)$, and note that
\begin{displaymath} \mathrm{div}^{\flat} V = u\langle \nabla v,\nabla \eta \rangle + v\langle \nabla u,\nabla \eta \rangle + uv \bigtriangleup^{\flat}\eta. \end{displaymath}
The claim follows from
\begin{displaymath} \int \mathrm{div}^{\flat} V(p) \, dp = - \int_{\partial D} \langle V,\nu \rangle \, d\sigma = 0 \end{displaymath}
by rearranging terms.  \end{proof}

Before the next lemma, we recall the definition of \emph{regular points for the Dirichlet problem}:
\begin{definition} Let $D \subset \He$ be a bounded domain. A point $p_{0} \in \partial D$ is a \emph{regular point for the Dirichlet problem}, denoted $p_{0} \in [\partial D]_{\mathrm{reg}}$, if the Perron-Wiener-Brelot solution of any Dirichlet problem in $D$ with continuous boundary data is continuous at $p_{0}$.
\end{definition}
As a special case of the definition, and the symmetry property \eqref{greenSymmetry}, one has
\begin{equation}\label{form175} \lim_{p \to p_{0}} G_{D}(p,q) = \lim_{p \to p_{0}} G_{D}(q,p) = 0, \qquad p_{0} \in [\partial D]_{\mathrm{reg}}. \end{equation} 
\begin{lemma}\label{l:schauder} Let $D \subset \He$ be a bounded domain, and let $G_{D}$ be Green's function for $D$. Let $Z_{q} = X_{i_{1}}\cdots X_{i_{s}}$ be a horizontal derivative acting on the variable "$q$". Then
\begin{displaymath} \lim_{p \to p_{0}} Z_{q}G_{D}(p,q) = 0 \end{displaymath}
for all $q \in D$, and for all $p_{0} \in [\partial D]_{\mathrm{reg}}$.
\end{lemma}

So, if we define $Z_{q}G_{D}(\cdot,q)|_{\partial D} \equiv 0$, then $Z_{q}G_{D}(\cdot,q) \in C(D \cup [\partial D]_{\mathrm{reg}})$ for all $q \in D$. 

\begin{proof}[Proof of Lemma \ref{l:schauder}] We apply a sub-elliptic Schauder estimate due to Capogna, Garofalo, and Nhieu \cite[Corollary 5.3]{MR1500279}: if $D \subset \He$ is a bounded open set, $u \colon D \to [0,\infty)$ is a non-negative $\bigtriangleup^{\flat}$-harmonic function, and $q \in D$ and $0 < r < r(D)$ are such that $B(q,Cr) \subset D$ for a sufficiently large absolute constant $C \geq 1$, then $|Zu(q)| \lesssim_{D,s} r^{-s} u(q)$. Here $Z = X_{i_{1}}\cdots X_{i_{s}}$ with $X_{i_{j}} \in \{X,Y\}$.

We apply the inequality to the non-negative $\bigtriangleup^{\flat}$-harmonic function $q \mapsto G_{D}(p,q)$, $q \in D \, \setminus \, \{p\}$. Fix $r_{0} > 0$ so small that $B(q,Cr_{0}) \subset D$, and take $p$ so close to $p_{0}$ that also $B(q,Cr_{0}) \subset D \, \setminus \{p\}$. We then infer from the Schauder estimate that
\begin{equation}\label{form151} |Z_{q}G_{D}(p,q)| \lesssim_{D,r_{0},s} G_{D}(p,q).  \end{equation}
Since the implicit constant in \eqref{form151} is independent of $p$ (for $p$ close enough to $p_{0}$), and $G_{D}(p,q) \to 0$ as $p \to p_{0}$ for $p_{0} \in [\partial D]_{\mathrm{reg}}$, recalling \eqref{form175}, we see that also $Z_{q}G_{D}(p,q) \to 0$ as $p \to p_{0}$. \end{proof}

\begin{remark}\label{rem:reg} If $\Omega \subset \He$ is a smooth flag domain, $D \subset \He$ is a bounded domain, and $\partial D \cap B(p_{0},r) = \partial \Omega \cap B(p_{0},r)$ for some $p_{0} \in \partial D \cap \partial \Omega$ and $r > 0$, then $p_{0} \in [\partial D]_{\mathrm{reg}}$. In particular $G_{D}(p,q) \to 0$ as $p \to p_{0}$. This is because the point $p \in \partial D$ then satisfies the \emph{outer (Kor\'anyi) ball condition}, and the regularity can be established via the classical method of barrier functions. The statement is \cite[Ex. 9, \S7]{BLU}, and the use of barrier functions is explained earlier in that section, see also \cite[Theorem 6.10.4]{BLU}.
\end{remark}

We then prove Theorem \ref{t:regularity}.

\begin{proof}[Proof of Theorem \ref{t:regularity}] Fix $p_{0} \in \partial \Omega$, let $V_{0} \ni p$ be any bounded open neighbourhood, and write $D := V_{0} \cap \Omega$. We will need the following estimates for Green's function $G_{D}$ (this is where the proof breaks for general flag domains): there is a radius $r_{0} > 0$ such that
\begin{equation}\label{form153} G_{D}(p,\cdot) \in \mathrm{Lip}(\bar{D} \cap B(p_{0},3r_{0}) \, \setminus \, B(p,\epsilon)), \qquad p \in D \cap B(p_{0},r_{0}), \, \epsilon > 0, \end{equation}
(of course we define $G_{D}(p,\cdot)|_{\partial D} \equiv 0$), and
\begin{equation}\label{form152} |\nabla_{q}G_{D}(p,q)| \lesssim_{D,r_{0}} 1, \qquad p \in D \cap B(p_{0},r_{0}), \, q \in D \cap A(p_{0},2r_{0},3r_{0}). \end{equation}
Here $A(p_{0},2r_{0}3r_{0}) = B(p_{0},3r_{0}) \, \setminus \, \bar{B}(p_{0},2r_{0})$, and the radius $r_{0} > 0$ is chosen (at least) so small that 
\begin{displaymath} \partial D \cap B(p_{0},10r_{0}) = \partial \Omega \cap B(p_{0},10r_{0}).  \end{displaymath}
In particular, $\partial D \cap B(p_{0},10r_{0}) \subset [\partial D]_{\mathrm{reg}}$ by Remark \ref{rem:reg}. The property \eqref{form153} will only be needed to deduce that $G_{D}(p,q) \lesssim_{p} \dist(q,\partial D)$, for $q \in \bar{D} \cap B(p_{0},3r_{0})$ away from $p$, and we are not interested in the $p$-dependence of the implicit constant. The property \eqref{form152} is more delicate: both $p$ and $q$ are allowed to approach $\partial D$ within $D \cap B(p_{0},3r_{0})$ without $|\nabla_{q}G_{D}(p,q)|$ blowing up, as long $d(p,q)$ stays bounded from below. 

There are at least two "well-known" ways to justify \eqref{form153}-\eqref{form152} for some $r_{0} > 0$ small enough (any $r_{0} > 0$ is fine for us). The first one is to choose $D$ so that it satisfies the \emph{uniform outer ball condition} (UOBC) of Lanconelli-Uguzzoni \cite{MR1620876} and Capogna-Garofalo-Nhieu \cite{MR1890994}. This means that every point $p \in \partial D$ has an associated Kor\'anyi ball $B_{p} \subset \He \, \setminus \, D$ with $p \in \partial B_{p}$, with a uniform lower bound on $\diam(B_{p})$. In this case, properties \eqref{form153}-\eqref{form152} hold by \cite[Theorem 3.6]{MR1620876} and \cite[Theorems 4.1-4.2]{MR1890994}.

The UOBC is quite strict, and it may not be comfortable to check, in full detail, that $D$ can be chosen in this manner. Instead, we suggest looking at the proof of \cite[Theorem 4.2]{MR1890994} and noting that the estimates are local: if the UOBC holds for $\partial D \cap B(p_{0},10r_{0}) = \partial \Omega \cap B(p_{0},10r_{0})$, then the bounds for $G_{D}(p,q)$ and $\nabla_{q}G_{D}(p,q)$ hold in $D \cap B(p_{0},3r_{0})$, in the sense specified in \eqref{form153}-\eqref{form152}. Justifying the outer ball condition (locally) for $\partial \Omega$ is easy: instead of finding, directly, Kor\'anyi balls tangent to to $\partial \Omega$, note that the projections of such balls to the $xy$-plane are standard Euclidean discs, and the projection of $\partial \Omega$ is the graph "$\Gamma$" of a $C^{\infty}(\R)$-function. Then, choose Euclidean discs tangent to $\Gamma$ (with radii bounded from below by local upper bounds on the curvature of $\Gamma$), and lift them back to $\He$ to find large Kor\'anyi balls in $\He \, \setminus \, \Omega$ tangent to $\partial \Omega$.

An alternative strategy to prove \eqref{form153}-\eqref{form152} is to apply a result of Jerison, \cite[Theorem (7.1)]{MR639800}, which states that the derivatives of $G_{D}(p,\cdot)$ of all orders are continuous up to $\partial D$ at non-characteristic points (in particular near $\partial D \cap B(p_{0},10r_{0})$). To obtain the quantitative bound \eqref{form152}, one needs to look inside the proof, and use the estimate \cite[(7.2)]{MR639800}, which in our terminlory states the following:
\begin{equation}\label{eq:Jerison} \|\phi \cdot G_{D}(p,\cdot)\|_{\Gamma_{\beta + 2}(\bar{D})} \lesssim_{D,\phi,\psi} \|\psi \cdot G_{D}(p,\cdot)\|_{L^{2}(D)}, \qquad \beta > 0,\footnote{The RHS of Jerison's estimate also involves certain functions "$f$" and "$g$", but they are identically zero in our situation; in fact $f = \bigtriangleup^{\flat} G_{D}(p,\cdot)$ and $g = G_{D}(p,\cdot)|_{\partial D}$.}  \end{equation}
where $\phi,\psi \in C^{\infty}_{c}(\He)$ satisfy $p \notin \spt \psi$, $(\spt \psi) \cap \partial D$ is non-characteristic, and $\psi|_{\spt \phi} \equiv 1$. To obtain \eqref{form152}, one has to choose $\phi \approx \psi \approx \mathbf{1}_{A(p_{0},2r_{0},3r_{0})}$. In English, \eqref{eq:Jerison} then states that all horizontal derivatives of $\phi \cdot G_{D}(p,\cdot)$ are continuous on $\bar{D}$, and their H\"older norms are controlled by $\|\psi \cdot G_{D}(p,\cdot)\|_{2}$. But $G_{D}(p,q) \lesssim d(p,q)^{-2}$, so 
\begin{displaymath} \|\psi \cdot G_{D}(p,\cdot)\|_{2} \lesssim_{r_{0} }\|q \mapsto \psi(q) \cdot d(p,q)^{-2}\|_{\infty} \lesssim_{r_{0}} 1 \end{displaymath}
whenever $p \in D \cap B(p_{0},r_{0})$, and taking $\psi \approx \mathbf{1}_{A(p_{0},2r_{0},3r_{0})}$ as stated above. Therefore \eqref{eq:Jerison} yields \eqref{form152} (and similar bounds for derivatives of all orders, if desired).

After these preliminaries on $G_{D}$, we may now begin the proof of Theorem \ref{t:regularity} in earnest. Recall that $u$ was a $\bigtriangleup^{\flat}$-harmonic function in $\Omega$ with radial limits $\sigma$ a.e. equal to $g \in C(\partial \Omega)$. We plan to prove that $u \in C(\bar{\Omega} \cap B(p_{0},r_{0}))$ by showing that $u - v \in C(\bar{\Omega} \cap B(p_{0},r_{0}))$ for a certain function $v \in C(\bar{D} \cap B(p_{0},10r_{0}))$, which we now introduce: let $v \colon D \to \R$ be the Perron-Wiener-Brelot solution to the following Dirichlet problem in $D$:
\begin{displaymath} \begin{cases} \bigtriangleup^{\flat} v = 0 & \text{in } D, \\ v|_{\partial D} = \tilde{g}. \end{cases} \end{displaymath}
Here $\tilde{g} \equiv g$ in $B(p_{0},10r_{0})$, and we do not care about the values of $\tilde{g}$ in other parts of $\partial D$, as long as $\tilde{g} \in C(\partial D)$. We remark that Perron-Brelot-Wiener solutions enjoy a maximum principle: $\|v\|_{L^{\infty}(D)} \leq \|\tilde{g}\|_{L^{\infty}(\partial \Omega)} < \infty$. Since $\partial D \cap B(p_{0},10r_{0})$ consists of regular points for the Dirichlet problem (recall that this part of $\partial D$ is non-characteristic), we also have
\begin{displaymath} v \in C(\bar{D} \cap B(p_{0},10r_{0})). \end{displaymath}
So, it remains to show that $u - v \in C(\bar{D} \cap B(p_{0},r_{0}))$. This means showing that $u - v$ has classical limits equal to $0$ on $\partial D \cap B(p_{0},r_{0})$.

We define some auxiliary cut-off functions. The first one is any $\psi \in C_{c}^{\infty}(B(p_{0},3r_{0}))$ with the properties that $0 \leq \psi \leq 1$ and $\psi(p) = 1$ for all $p \in B(p_{0},2r_{0})$. To define further auxiliary functions, we first recall the interior domains
\begin{displaymath} \Omega_{j} := \{(x,y,t) : x < A(y) - 2^{-j}\}. \end{displaymath}
Then, we define a family of functions $\{\varphi_{j}\}_{j \in \N} \subset C^{\infty}(\Omega)$ with the following properties: 
\begin{itemize}
\item $\mathbf{1}_{\Omega_{j}} \leq \varphi_{j} \leq \mathbf{1}_{\Omega_{j + 1}}$,
\item $\|\nabla \varphi_{j}\|_{\infty} \lesssim 2^{j}$ and $\|\nabla^{2}\varphi_{j}\|_{\infty} \lesssim 2^{2j}$ (these estimates only need to hold in $B(p_{0},10r_{0})$).
\end{itemize}
We also set $\eta_{j} := \varphi_{j}\psi \in C^{\infty}_{c}(D)$. Clearly $\eta_{j} \to \psi \cdot \mathbf{1}_{\Omega}$ as $j \to \infty$, and $\eta_{j}$ satisfies the same derivative estimates as $\varphi_{j}$, up to multiplicative constants depending on $r_{0}$.

Fix $p \in D \cap B(p_{0},r_{0})$, write $w := u - v$ and $w_{j} := \eta_{j}w \in C^{\infty}_{c}(D)$, and apply Lemma \ref{l:greenRep} to represent $w_{j}$ in terms of Green's function:
\begin{equation}\label{form156} u(p) - v(p) = \lim_{j \to \infty} w_{j}(p) = -\lim_{j \to \infty} \int_{D} G_{D}(p,q) (\bigtriangleup^{\flat} w_{j})(q) \, dq =: -\lim_{j \to \infty} I_{j}(p). \end{equation} 
We record that 
\begin{displaymath} (\bigtriangleup^{\flat}w_{j})(q) =  (\bigtriangleup^{\flat}\eta_{j})(q)w(q) + 2\langle \nabla w(q),\nabla \eta_{j}(q) \rangle, \qquad q \in D, \end{displaymath}
so
\begin{displaymath} I_{j}(p) = \int_{D} G_{D}(p,q)(\bigtriangleup^{\flat}\eta_{j})(q)w(q) \, dq + 2 \int_{D}\langle \nabla w(q),\nabla \eta_{j}(q) \rangle G_{D}(p,q) \, dq. \end{displaymath}
We denote the first term on the RHS by $I^{1}_{j}(p)$. Regarding the second term, we note that $\eta_{j} = \varphi_{j}\psi$ is constant on a neighbourhood of $p \in D \cap B(p_{0},r_{0})$ for all $j \in \N$ large enough (since $\psi$ is), so the integrand vanishes on a neighbourhood of $p$. Therefore Lemma \ref{lemma11} is applicable with "$u := G(p,\cdot)$", "$\eta := \eta_{j}$", and "$v := w$", despite the singularity of $G_{D}(p,q)$ at $p$, and shows that
\begin{displaymath} \int_{D}\langle \nabla w(q),\nabla \eta_{j}(q) \rangle G_{D}(p,q) \, dq = -\int_{D}\langle \nabla_{q}G_{D}(p,q),\nabla \eta_{j}(q) \rangle w(q) \, dq - I_{j}^{1}(p). \end{displaymath}
We denote the integral on the RHS by $I^{2}_{j}(p)$, and so we have shown that
\begin{displaymath} u(p) - v(p) = \lim_{j \to \infty} \left[ I_{j}^{1}(p) + 2I_{j}^{2}(p) \right], \qquad p \in D \cap B(p_{0},r_{0}), \end{displaymath}
assuming that the limit exists. This will be clarified below. Sending $j \to \infty$ formally, and recalling that $\eta_{j} \to \psi \cdot \mathbf{1}_{\Omega}$, one expects (and we will soon prove) that
\begin{equation}\label{form147} I^{1}(p) := \lim_{j \to \infty} I_{j}^{1}(p) = \int_{D \cap A(p_{0},2r_{0},3r_{0})} G_{D}(p,q)w(q)(\bigtriangleup^{\flat}\psi)(q) \, dq \end{equation}
and
\begin{equation}\label{form148} I^{2}(p) := \lim_{j \to \infty} I_{j}^{2}(p) = \int_{D \cap A(p_{0},2r_{0},3r_{0})} \langle \nabla_{q}G_{D}(p,q),\nabla \psi(q) \rangle w(q) \, dq \end{equation} 
for all $p \in D \cap B(p_{0},r_{0})$. We already took, in \eqref{form147}-\eqref{form148}, into account that the function $\psi$ is non-constant only in the annulus $A(p_{0},2r_{0},3r_{0})$. The integrals on the right hand sides of \eqref{form147}-\eqref{form148} are absolutely convergent for all $p \in D \cap B(p_{0},r_{0})$ by virtue of the following estimates:
\begin{itemize}
\item $|G_{D}(p,q)| \lesssim d(p,q)^{-2} \lesssim_{r_{0}} 1$ for $p \in D \cap B(p_{0},r_{0})$ and $q \in D \cap A(p_{0},2r_{0},3r_{0})$, since Green's function is always dominated by the fundamental solution,
\item $|\nabla_{q}G_{D}(p,q)| \lesssim_{D,r_{0}} 1$ for $p \in D \cap B(p_{0},r_{0})$ and $q \in D \cap A(p_{0},2r_{0},3r_{0})$ by \eqref{form152},
\item $w \in L^{1}(D)$. To see this, we apply the Fubini formula \eqref{eq:Fubini} with "$\Gamma = \partial \Omega$" and "$g = w\mathbf{1}_{D}$". We define the (bounded) set $[\partial \Omega]_{D} := \{\xi \in \partial \Omega : (\xi \cdot (-r,0,0)) \in D \text{ for some } r > 0\}$, and then \eqref{eq:Fubini} implies
\begin{displaymath} \int_{D} |w(q)| \, dq \leq \int_{[\partial \Omega]_{D}} \int_{0}^{\infty} (|w|\mathbf{1}_{D})(\xi \cdot (-r,0,0)) \, dr \, d\sigma(\xi) \lesssim_{\diam(D)} \int_{[\partial \Omega]_{D}} \mathcal{N}_{\mathrm{rad}}w \, d\sigma < \infty,  \end{displaymath} 
using the main assumption $\mathcal{N}_{\mathrm{rad}}u \in L^{1}_{\mathrm{loc}}(\partial \Omega)$, and also that $\|v\|_{L^{\infty}(D)} < \infty$, which followed from the maximum principle for Perron-Wiener-Brelot solutions.
\end{itemize}

Before establishing \eqref{form147}-\eqref{form148}, let us use them to conclude the proof of the theorem. We claim that
\begin{equation}\label{form154} \lim_{D \ni p \to p_{0}} I^{1}(p) = 0 = \lim_{D \ni p \to p_{0}} I^{2}(p), \qquad p_{0} \in \partial D \cap B(p_{0},r_{0}), \end{equation} 
where the limits are classical. Since $u - v = w = I^{1} + 2I^{2}$ in $D \cap B(p_{0},r_{0})$ by \eqref{form156}-\eqref{form148}, this will show that $u - v$ has vanishing classical limits on $\partial D \cap B(p_{0},r_{0})$, as desired. To prove \eqref{form154}, we first recall from Lemma \ref{l:schauder} that
\begin{displaymath} \lim_{D \ni p \to p_{0}} G_{D}(p,q) = 0 = \lim_{D \ni p \to p_{0}} \nabla_{q}G_{D}(p,q), \qquad q \in D, \, p_{0} \in \bar{D} \cap B(p_{0},r_{0}), \end{displaymath}
because $\partial D \cap B(p_{0},10r_{0}) \subset [\partial D]_{\mathrm{reg}}$. Moreover, since $w \in L^{1}(D)$, and the other factors in the integrands in \eqref{form147}-\eqref{form148} are uniformly bounded for $p \in D \cap B(p_{0},r_{0})$ and $q \in D \cap A(p_{0},2r_{0},3r_{0})$, dominated convergence applies and proves \eqref{form154}.

It remains to prove \eqref{form147}-\eqref{form148}.

\subsubsection{Proof of \eqref{form147}} Recalling that $\eta_{j}|_{\Omega_{j}} = \psi$, and that $\psi$ is a constant outside the annulus $A(p_{0},2r_{0},3r_{0})$, we may write
\begin{align*} I_{j}^{1}(p) & = \int_{D \, \setminus \, \Omega_{j}} G_{D}(p,q)(\bigtriangleup^{\flat}\eta_{j})(q)w(q) \, dq + \int_{D \cap \Omega_{j}} G_{D}(p,q)(\bigtriangleup^{\flat}\psi)(q)w(q) \, dq\\
& =: I_{j}^{1,0}(p) + \int_{D \cap \Omega_{j} \cap A(p_{0},2r_{0},3r_{0})} G_{D}(p,q)w(q)(\bigtriangleup^{\flat}\psi)(q) \, dq, \qquad p \in B(p_{0},r_{0}). \end{align*} 
The second term evidently converges to the RHS of \eqref{form147} as $j \to \infty$, since $G_{D}(p,\cdot) \in C(\bar{D} \cap A(p_{0},2r_{0},3r_{0}))$, and $\psi \in C^{\infty}_{c}(\He)$, and $w \in L^{1}(D)$. In other words,
\begin{displaymath} \lim_{j \to \infty} I_{j}^{1}(p) = \lim_{j \to \infty} I^{1,0}_{j}(p) + \int_{D \cap A(p_{0},2r_{0},3r_{0})} G_{D}(p,q)(\bigtriangleup^{\flat}\psi)(q)w(q) \, dq. \end{displaymath}
So, to prove \eqref{form147}, it remains to show that 
\begin{displaymath} \lim_{j \to \infty} |I_{j}^{1,0}(p)| \lesssim \limsup_{j \to \infty} 2^{2j} \int_{D \cap B(p_{0},3r_{0}) \, \setminus \, \Omega_{j}} |G_{D}(p,q)||w(q)| \, dq = 0. \end{displaymath}
To see this, we use that that $G_{D}(p,\cdot) \in \mathrm{Lip}(\bar{D} \cap B(p_{0},3r_{0}) \, \setminus \, \Omega_{j})$ by \eqref{form153} (for all $j \in \N$ so large that $p \in \Omega_{j - 1}$), and $G_{D}(p,q) = 0$ for all $q \in \partial D$. It follows that $|G_{D}(p,q)| \lesssim_{p} 2^{-j}$ for all $q \in D \cap B(p_{0},3r_{0}) \, \setminus \, \Omega_{j}$, and hence
\begin{align}\label{form149} |I_{j}^{1,0}(p)| & \lesssim_{p} 2^{j} \int_{D \cap B(p_{0},3r_{0}) \, \setminus \, \Omega_{j}} |w(q)| \, dq\\
& \leq \int_{[\partial \Omega]_{D \cap B(p_{0},3r_{0})}} 2^{j} \int_{0}^{\infty} (|w|\mathbf{1}_{D \, \setminus \, \Omega_{j}})(\xi \cdot (-r,0,0)) \, dr \, d\sigma(\xi) \notag\\
& =: \int_{[\partial \Omega]_{D \cap B(p_{0},3r_{0})}} \mathcal{J}_{j}(\xi) \, d\sigma(\xi). \notag\end{align}
Here we used, again, the Fubini formula \eqref{eq:Fubini}, and denoted
\begin{displaymath} [\partial \Omega]_{D \cap B(p_{0},3r_{0})} := \{\xi \in \partial D : (\xi \cdot (-r,0,0)) \in D \cap B(p_{0},3r_{0}) \text{ for some } r > 0\}. \end{displaymath} Note that $[\partial \Omega]_{D \cap B(p_{0},3r_{0})} \subset \partial D \cap B(p_{0},10r_{0})$, so $w = u - v$ has vanishing radial limits $\sigma$ a.e. on $[\partial \Omega]_{D \cap B(p_{0},3r_{0})}$. Since moreover
\begin{displaymath} \mathcal{L}^{1}(\{r > 0 : (\xi \cdot (-r,0,0)) \in D \, \setminus \, \Omega_{j}\}) \lesssim_{D} 2^{-j}, \qquad \xi \in [\partial \Omega]_{D \cap B(p_{0},3r_{0})}, \end{displaymath}
we infer that $\mathcal{J}_{j}(\xi) \to 0$ for $\sigma$ a.e. $\xi \in [\partial \Omega]_{D \cap B(p_{0},3r_{0})}$ as $j \to \infty$. Moreover, 
\begin{displaymath} \mathcal{J}_{j}(\xi) \lesssim_{D} (\mathcal{N}_{\mathrm{rad}}w)(\xi) \in L^{1}(\partial D \cap B(p_{0},10r_{0})), \end{displaymath}
so dominated convergence applies and shows that $|I_{j}^{1,0}(p)| \to 0$ as $j \to \infty$.

\subsubsection{Proof of \eqref{form148}} We start with a splitting familiar from the previous case, writing
\begin{displaymath} I_{j}^{2}(p) =  I_{j}^{2,0}(p) + \int_{D \cap \Omega_{j}}\langle \nabla_{q} G_{D}(p,q),\nabla \psi(q) \rangle w(q) \, dq, \qquad p \in D \cap B(p_{0},r_{0}), \end{displaymath} 
where
\begin{displaymath} I_{j}^{2,0}(p) = \int_{D \, \setminus \, \Omega_{j}} \langle \nabla_{q} G_{D}(p,q),\nabla \eta_{j}(q) \rangle w(q) \, dq. \end{displaymath}
Again, the plan is to show that $I_{j}^{2,0}(p) \to 0$ as $j \to \infty$. Evidently the limit of the second term is the desired RHS of \eqref{form148}, since the integrand is in $L^{1}(D)$ by the argument under \eqref{form148}. To prove that $I_{j}^{2,0}(p) \to 0$, we use the bounds $|\nabla_{q}G_{D}(p,q)| \lesssim_{D,r_{0}} 1$ and $\|\nabla \eta_{j}\|_{L^{\infty}} \lesssim_{r_{0}} 2^{j}\mathbf{1}_{B(p_{0},3r_{0})}$ to deduce that
\begin{displaymath} |I_{j}^{2,0}(p)| \lesssim_{D,p,r_{0}} 2^{j} \int_{D \cap B(p_{0},3r_{0}) \, \setminus \, \Omega_{j}} |w(q)| \, dq. \end{displaymath}
From this point on, one may follow the argument starting at \eqref{form149}. The proof of Theorem \ref{t:regularity} is complete. \end{proof}

%%%%%%%%%%%%%%%%%%%%%%%%%%%%%%%%%

\section{Uniqueness in the Dirichlet problem}\label{s:uniqueness}

In this section, we prove Theorem \ref{main6}, which we repeat here:
\begin{thm} Let $\Omega = \{(x,y,t) : x < A(y)\}$ be a flag domain, $\sigma := |\partial \Omega|_{\He}$. Let $u \in C^{\infty}(\Omega)$ be a $\bigtriangleup^{\flat}$-harmonic function with $\mathcal{N}_{\mathrm{rad}}u \in L^{2}(\sigma)$, whose radial limits vanish $\sigma$ a.e. on $\partial \Omega$. Then $u$ is the zero function.
\end{thm}

\begin{proof} We follow \cite[Section 6]{MR1418902} very closely. Write $L := \mathrm{Lip}(A)$. Let $\psi \in C_{c}^{\infty}(\R)$ be a standard bump function, that is, $0 \leq \psi \lesssim \mathbf{1}_{[-1,1]}$ and $\int \psi = 1$. For $\delta > 0$, write $\psi_{\delta}(x) := \delta^{-1} \psi(x/\delta)$, and define the $L$-Lipschitz maps 
\begin{displaymath} A_{j} := A \ast \psi_{c2^{-j}} - 2^{-j} \in C^{\infty}(\R), \qquad j \in \N. \end{displaymath}
Here $c \sim 1/L$ has been chosen so small that $\|A \ast \psi_{c2^{-j}} - A\|_{\infty} < 2^{-j}$, which implies that $A_{j} < A$. Let $\Omega_{j} := \{(x,y,t) : x < A_{j}(y)\} \subset \Omega$ be the flag domain bounded by $A_{j}$, and write $\sigma_{j} := |\partial \Omega_{j}|_{\He}$. In Section \ref{s:injectivity}, where the smoothness of $A_{j}$ played no role, we had the convenient relation $\sigma_{j} = \Phi_{j\sharp}\psi$, where $\Phi_{j}$ was simply right translation by the constant vector $(-2^{-j},0,0)$. We need a similar relationship here, but the map $\Phi_{j}$ is necessarily a bit more complicated, since $\partial \Omega_{j}$ is not just a translate of $\partial \Omega$. Let $\mathbb{L} := \{(x,0,0) : x \in \R\}$. For $p = (A(y),y,t) \in \partial \Omega$, let $\Phi_{j}(p) \in \partial \Omega_{j}$ be the unique point determined by
\begin{displaymath} \Phi_{j}(p) \in \partial \Omega_{j} \cap (p \cdot \mathbb{L}). \end{displaymath}
To see that the point is unique, consider the projection $\pi(x,y,t) = (x,y)$, and note that $\pi(p \cdot \mathbb{L}) \subset \R^{2}$ is the line parallel to the $x$-axis containing $(A(y),y)$. But since $\pi(\partial \Omega_{j}) = \{(A_{j}(y),y) : y \in \R\}$ is also a graph over the $y$-axis, the intersection $\pi(\partial \Omega_{j}) \cap \pi(p \cdot \mathbb{L})$ contains only the point $(A_{j}(y),y)$, and so $\partial \Omega_{j} \cap (p \cdot \mathbb{L})$ contains exactly one point of the form $(A_{j}(y),y,\tau(t,y))$ for some $\tau(t,y) \in \R$.

It moreover turns out that $\tau(t,y) = \tau(0,y) + t$, so $\Phi_{j}$ actually has the form
\begin{displaymath} \Phi_{j}(A(y),y,t) = (A_{j}(y),y,\tau(y) + t), \qquad (y,t) \in \R^{2}, \end{displaymath} 
for some (continuous) $\tau \colon \R \to \R$. The reason is that if $p_{1},p_{2} \in \partial \Omega$ only differ in the $t$-coordinate, thus $p_{1} = (x,y,t_{1})$ and $p_{2} = (x,y,t_{2})$, then the lines $p_{1} \cdot \mathbb{L}$ and $p_{2} \cdot \mathbb{L}$ are vertical translates of each other: 
\begin{displaymath} p_{2} \cdot \mathbb{L} = [p_{1} \cdot (0,0,t_{2} - t_{1})] \cdot \mathbb{L} = (p_{1} \cdot \mathbb{L}) \cdot (0,0,t_{2} - t_{1}). \end{displaymath}
This implies that if $\tau(y) \in \R$ is determined by $\Phi_{j}(A(y),y,0) = (A_{j}(y),y,\tau(y))$, then $\Phi_{j}(A(y),y,t) = (A_{j}(y),y,\tau(y) + t)$ for all $t \in \R$, as desired.

Now we can show that the measures $\Phi_{j\sharp}\sigma$ and $\sigma_{j}$ are mutually absolutely continuous with Radon-Nikodym derivative bounded from below and above by constants depending only on $L$. Indeed, if $f \in C_{c}(\partial \Omega_{j})$, $f \geq 0$, then
\begin{align*} \int f(z) \, d\Phi_{j\sharp}\sigma(z) & = \int (f \circ \Phi_{j})(z) \, d\sigma(z)\\
& = \int \int (f \circ \Phi_{j})(A(y),y,t) \sqrt{1 + A'(y)^{2}} \, dy\, dt\\
& = \int \int f(A_{j}(y),y,\tau(y) + t) \sqrt{1 + A'(y)^{2}} \, dt \, dy\\
& = \int \int f(A_{j}(y),y,t) \sqrt{1 + A'(y)^{2}} \, dt \, dy \sim_{L} \int f(z) \, d\sigma_{j}(z)  \end{align*} 
by the (standard) area formula and Fubini's theorem. We also remark that since $\Phi_{j}(p) \in p \cdot \mathbb{L}$ for all $p \in \partial \Omega$, we have $|u(\Phi_{j}(p))| \leq (\mathcal{N}_{\mathrm{rad}}u)(p)$ for all $u \colon \Omega \to \R$ and $p \in \partial \Omega$.

For $j \in \N$ fixed, we let $u_{j} \in C^{\infty}(\Omega_{j})$ be the double layer potential solution to the Dirichlet problem
\begin{displaymath} \begin{cases} \bigtriangleup^{\flat} u_{j} = 0, \\ (u_{j})|_{\partial \Omega_{j}} = h|_{\partial \Omega_{j}}. \end{cases} \end{displaymath} 
This exists by Theorem \ref{main3}, since
\begin{equation}\label{form90} \int |u(p)|^{2} \, d\sigma_{j}(p) \sim \int |u(\Phi_{j}(p))|^{2} \, d\sigma(p) \leq \int |\mathcal{N}_{\mathrm{rad}}u(p)|^{2} \, d\sigma(p) < \infty. \end{equation}
So, $u_{j} = \mathcal{D}f_{j}$ for some $f_{j} \in L^{2}(\sigma_{j})$ with $\|f_{j}\|_{L^{2}(\sigma_{j})} \lesssim \|u\|_{L^{2}(\sigma_{j})}$. By Theorem \ref{main3}, the radial limits of $u_{j}$ on $\partial \Omega_{j}$ agree $\sigma_{j}$ a.e. with the function $u|_{\partial \Omega_{j}} \in C(\partial \Omega_{j})$, and $\|\mathcal{N}_{j}u_{j}\|_{L^{2}(\sigma_{j})} < \infty$, where $\mathcal{N}_{j}$ is the radial maximal function associated with $\partial \Omega_{j}$. In particular, since $A_{j} \in C^{\infty}(\R)$, Theorem \ref{t:regularity} implies
\begin{equation}\label{form157} u_{j} \in C(\overline{\Omega_{j}}) \quad \text{and} \quad u_{j}|_{\partial \Omega_{j}} \equiv u|_{\partial \Omega_{j}}.  \end{equation}
The crux of the following proof will be to show that
\begin{equation}\label{form89} u|_{\Omega_{j}} \equiv (u_{j})|_{\Omega_{j}}, \qquad j \in \N. \end{equation}
Once \eqref{form89} has been established, the rest is easy. Using \eqref{form89} and Corollary \ref{K-MaxNT}, we have
\begin{displaymath} \|\mathcal{N}_{j}u\|_{L^{2}(\sigma_{j})} = \|\mathcal{N}_{j}u_{j}\|_{L^{2}(\sigma_{j})} = \|\mathcal{N}_{j}(\mathcal{D}f_{j})\|_{L^{2}(\sigma_{j})} \lesssim \|f_{j}\|_{L^{2}(\sigma_{j})} \lesssim \|u\|_{L^{2}(\sigma_{j})}. \end{displaymath} 
The right hand side tends to zero as $j \to \infty$ by \eqref{form90}, dominated convergence, and since $u(\Phi_{j}(p)) \to 0$ for $\sigma$ a.e. $p \in \partial \Omega$. Consequently,
\begin{displaymath} \int \mathcal{N}_{\mathrm{rad}}u(p)^{2} \, d\sigma(p) \lesssim \liminf_{j \to \infty} \int \mathcal{N}_{j}u(\Phi_{j}(p))^{2} \, d\sigma(p) = \liminf_{j \to \infty} \|\mathcal{N}_{j}u\|_{L^{2}(\sigma_{j})} = 0, \end{displaymath}
so $\mathcal{N}_{\mathrm{rad}}u(p) = 0$ for $\sigma$ a.e. $p \in \partial \Omega$. In other words, $u$ vanishes on $\sigma$ a.e. line of the form $p \cdot \mathbb{L}$. The union of these lines is dense in $\Omega$, so $u \equiv 0$ by continuity.

We turn to the proof of \eqref{form89}. Write $g := u_{j} - u$, so $g \in C(\overline{\Omega_{j}})$ with $g|_{\partial \Omega_{j}} \equiv 0$ by \eqref{form157}, and $\|\mathcal{N}_{j}g\|_{L^{2}(\sigma)} < \infty$. We claim that $g|_{\Omega_{j}} \equiv 0$. Since $g|_{\partial \Omega_{j}} \equiv 0$, we may extend $g$ continuously to $\He$ by setting $g|_{\He \, \setminus \overline{\Omega_{j}}} \equiv 0$. Then, we fix $p_{0} \in \partial \Omega_{j}$ and a radius $R > 0$, and consider the Perron-Wiener-Brelot solution to the boundary value problem
\begin{displaymath} \begin{cases} \bigtriangleup^{\flat} w = 0 & \text{in } B(p_{0},R), \\ w|_{\partial B(p_{0},R)} = |g|. \end{cases}  \end{displaymath} 
By \cite[Corollary 10]{MR871992}, the boundary of every domain $D \subset \He$ satisfying the (uniform) outer ball condition is regular for the Dirichlet problem, $[\partial D]_{\mathrm{reg}} = \partial D$. Since $B(p_{0},R)$ satisfies the outer ball condition by \cite[Theorem 2.16]{MR2500489} or \cite[Remark 3.5]{MR1620876}, we have $[\partial B(p_{0},R)]_{\mathrm{reg}} = \partial B(p_{0},R)$, and hence $w \in C(\bar{B}(p_{0},R))$. Evidently $w \geq 0$ in $\bar{B}(p_{0},R)$ by the maximum principle \cite[5.13.4]{BLU}, so since $g|_{\partial \Omega_{j}} \equiv 0$, we see that 
\begin{displaymath} |g||_{\partial [\Omega_{j} \cap B(p_{0},R)]} \leq w|_{\partial [\Omega_{j} \cap B(p_{0},R)]}. \end{displaymath}
In particular $w - g \geq 0$ and $w + g \geq 0$ on $\partial [\Omega_{j} \cap B(p_{0},R)]$, so also $w - g \geq 0$ and $w + g \geq 0$ in $\Omega_{j} \cap B(p_{0},R)$ by another application of the maximum principle. Consequently
\begin{equation}\label{form158} |g||_{\Omega_{j} \cap B(p_{0},R)} \leq w|_{\Omega_{j} \cap B(p_{0},R)}. \end{equation}
It turns out that $w$ is easier to estimate from above than $g$, since $w$ has an explicit representation via the Poisson kernel of $B(p_{0},R)$. Indeed, \cite[Theorem 4.13]{MR1890994} states that
\begin{equation}\label{form159} w(p) = \int_{\partial B(p_{0},R)} w(q)P_{R}(p,q) \, d\sigma_{R}(q) = \int_{\partial B(p_{0},R)} |g(q)|P_{R}(p,q) \, d\sigma_{R}(q), \end{equation}
where $\sigma_{R} := c \cdot \mathcal{S}^{3}|_{\partial B(p_{0},R)}$ for an appropriate absolute constant $c > 0$, and 
\begin{displaymath} P_{R}(p,q) = \langle \nabla_{q} G_{R}(p,q),\nu_{R}(q) \rangle, \qquad p \in B(p_{0},R), \, q \in \partial B(p_{0},R), \end{displaymath}
is the Poisson kernel for $B(p_{0},R)$ (see \cite[(4.10)]{MR1890994}).\footnote{The Poisson kernel defined in \cite[(4.10)]{MR1890994} is normalised in a different way, but also the measure "$\sigma$" appearing in \cite[Theorem 4.13]{MR1890994} is the Euclidean surface measure, rather than the $3$-dimensional spherical measure relative to the metric $d$ as in \eqref{form159}. The relation between the two measures is given by \cite[Corollary 7.7]{FSSC}, and once this relation is plugged into \cite[Theorem 4.13]{MR1890994}, the formula \eqref{form159} appears as stated.}  Here $G_{p_{0},R}(p,\cdot)$ is Green's function for the ball $B(p_{0},R)$ with pole at $p \in B(p_{0},R)$, and $\nu_{R}$ is the inward-pointing horizontal normal of $B(p_{0},R)$. From \cite[Theorem 4.2]{MR1890994}, we infer that
\begin{displaymath} |P_{R}(p,q)| \leq |\nabla_{q}G_{p_{0},R}(p,q)| \lesssim d(p,q)^{-3} \qquad \text{for } p \in B(p_{0},R) \text{ and $\sigma_{R}$ a.e. } q \in \partial B(p_{0},R), \end{displaymath}
in fact for all $q \in \partial B(p_{0},R)$ apart from the two characteristic points of $\partial B(p_{0},R)$. The implicit constants above are independent or $R > 0$ and $p_{0} \in \He$, because it is easy to check from the definition (and uniqueness) of Green's function that
\begin{displaymath} G_{p_{0},R}(p,q) = R^{-2}G_{0,1}(\delta_{1/R}(p_{0}^{-1} \cdot p),\delta_{1/R}(p_{0}^{-1} \cdot q)), \qquad p,q \in B(p_{0},R), \, p \neq q. \end{displaymath}
Combining \eqref{form158} and \eqref{form159}, using Cauchy-Schwarz and noting that $\sigma_{R}(\partial B(p_{0},R)) \sim R^{3}$, we may now deduce for $p \in B(p_{0},R/2)$ that
\begin{displaymath} |g(p)|^{2} \leq w(p)^{2} \lesssim \int_{\partial B(p_{0},R)} \frac{|g(q)|^{2}}{d(p,q)^{3}} \, d\sigma_{R}(q) \lesssim R^{-3} \int_{\partial B(p_{0},R) \cap \Omega_{j}} |g(q)|^{2} \, d\sigma_{R}(q).  \end{displaymath} 
We also took into account that $g \equiv 0$ in $\He \, \setminus \overline{\Omega_{j}}$.  Note that the LHS does not depend on the "free parameter" $R > 0$, so we may take an average over $R \in [R_{0},2R_{0}]$, to obtain the following estimate for all $R_{0} > 0$ and $p \in B(p_{0},R_{0}/2)$:
\begin{equation}\label{form160} |g(p)|^{2} \lesssim R_{0}^{-4} \int_{R_{0}}^{2R_{0}} \int_{\partial B(p_{0},R) \cap \Omega_{j}} |g(q)|^{2} \, d\sigma_{R}(q) \, dR. \end{equation} 
To estimate the RHS, we will apply the coarea inequality (also known as Eilenberg's inequality)
\begin{displaymath} \int_{\R} \int_{f^{-1}(\{r\})} \rho(q) \, d\calS^{3}(q) \, dr \lesssim \mathrm{Lip}(f) \int \rho(q) \, dq, \end{displaymath}
which is valid for all Lipschitz functions $f \colon \He \to \R$ and non-negative Borel functions $\rho \colon \He \to [0,\infty]$. See \cite{MR7643} for the original reference, or for example \cite[Proposition 3.15]{MR2039660} and \cite[Theorem 1.1]{2020arXiv200600419E} for more recent sources. Applying this inequality with the choices $f(p) := \dist(p,p_{0})$ and $\rho(q) := |g(q)|^{2}\mathbf{1}_{A(p_{0},R_{0},2R_{0}) \cap \Omega_{j}}(q)$, we infer that 
\begin{displaymath} \eqref{form160} \lesssim R_{0}^{-4} \int _{A(p_{0}R_{0},2R_{0}) \cap \Omega_{j}} |g(q)|^{2} \, dq \leq R_{0}^{-4} \int_{B(p_{0},2R_{0}) \cap \Omega_{j}} |g(q)|^{2} \, dq.   \end{displaymath} 
To conclude the proof from here, we wish to relate the RHS to the radial maximal function of $g$, so we apply the Fubini-type formula \eqref{eq:Fubini} to the domain $\Omega_{j}$ bounded by the intrinsic Lipschitz graph $\Gamma = \partial \Omega_{j}$:
\begin{align*} R_{0}^{-4}\int_{B(p_{0},2R_{0}) \cap \Omega_{j}} |g(q)|^{2} \, dq & \leq R_{0}^{-4} \int_{\partial \Omega_{j}} \int_{0}^{\infty} |(\mathbf{1}_{B(p_{0},2R_{0})}g)(w \cdot (-r,0,0))|^{2} \, dr \, d\sigma_{j}(w)\\
& \lesssim R_{0}^{-3} \int_{\partial \Omega_{j}} (\mathcal{N}_{j}g)(w)^{2} \, d\sigma_{j}(w). \end{align*} 
Here we used that $\mathcal{L}^{1}(\{r > 0 : w \cdot (-r,0,0) \in B(p_{0},2R_{0})\}) \lesssim R_{0}$ for all $w \in \partial \Omega_{j}$. Since $\mathcal{N}_{j}g \in L^{2}(\sigma_{j})$, the RHS above evidently tends to $0$ as $R_{0} \to \infty$. Recalling \eqref{form160}, which is valid for all $R_{0} > 2d(p,p_{0})$, we infer that $g(p) = 0$. This completes the proof of \eqref{form89}, and hence the proof of the theorem.  \end{proof}

%%%%%%%%%%%%%%%%%%%%%%%%%%%%%%%

\section{Invertibility of $\tfrac{1}{2}I + D$ on $\mathbf{L}^{2}_{1,1/2}(\sigma)$}\label{s:Invert3}

The purpose of this section is to prove Theorem \ref{main7}, which we recall here:
\begin{thm}\label{t:main7a} Let $\Omega = \{(x,y,t) : x < A(y)\} \subset \He$ be a flag domain, $\sigma := |\partial \Omega|_{\He}$, let $g \in \mathbf{L}^{2}_{1,1/2}(\sigma)$, and let $u$ be the double layer potential solution to the Dirichlet problem \eqref{dirichlet}, as given by Theorem \ref{main3}. Then, $\|\mathcal{N}_{\theta}(\nabla u)\|_{L^{2}(\sigma)} \lesssim \|g\|_{\mathbf{L}^{2}_{1,1/2}(\sigma)}$ for all $\theta \in (0,1)$.
\end{thm}

Recall that $\mathbf{L}^{2}_{1,1/2}(\sigma)$ is the inhomogeneous Sobolev space
\begin{displaymath} \mathbf{L}_{1,1/2}^{2}(\sigma) := L^{2}(\sigma) \cap L^{2}_{1,1/2}(\sigma) \end{displaymath}
equipped with the norm $\|f\|_{\mathbf{L}^{2}_{1,1/2}(\sigma)} := \|f\|_{L^{2}(\sigma)} + \|f\|_{1,1/2,\sigma}$. The norm $\|\cdot\|_{1,1/2,\sigma}$ was introduced in Definition \ref{def:L2112}. The crux of the section will be to show the following:
\begin{thm}\label{t:Invert} The operator $\tfrac{1}{2}I + D$ is bounded and invertible on $\mathbf{L}^{2}_{1,1/2}(\sigma)$. \end{thm}
Before showing how Theorem \ref{t:main7a} follows, we prove a useful formula:

\begin{proposition}\label{prop12} Let $\Omega \subset \He$ be a flag domain, $\sigma := |\partial \Omega|_{\He}$, and let $f \in L^{p}(\sigma)$, $1 < p < 3$, be such that also $Sf \in L^{p}(\sigma)$. Then,
\begin{equation}\label{SanD} \mathcal{D}(Sf)(p) =  -\mathcal{S}(\tfrac{1}{2}I + D^{t})f(p), \qquad p \in \Omega. \end{equation}
\end{proposition}

\begin{proof} We first consider the case where $f \in B(\sigma)$, where $B(\sigma)$ stands for bounded and compactly supported $\sigma$-measurable functions. Note that if $f \in B(\sigma)$, then $|Sf(q)| \lesssim \min\{1,\|q\|^{-2}\}$ for $q \in \partial \Omega$, and hence $Sf \in L^{r}(\sigma)$ for all $r > \tfrac{3}{2}$. In particular, $\mathcal{D}(Sf)$ is pointwise defined on $\Omega$. 

Fix $f \in B(\sigma)$ and $p \in \Omega$, write $G_{p}(q) = G(p^{-1} \cdot q)$, and consider the vector field
\begin{displaymath}
    V_p := G \nabla \calS f - \calS f  \nabla G_{p} \in C^{\infty}(\Omega). 
\end{displaymath}
We note that $\mathrm{div}^{\flat} V_{p} = \mathcal{S}f \cdot \delta_{p}$, recalling that $\mathcal{S}f$ is $\bigtriangleup^{\flat}$-harmonic in $\Omega$, and $G$ is the fundamental solution of $-\bigtriangleup^{\flat}$. Therefore, a formal application of the divergence theorem, Theorem \ref{t:div}, yields
\begin{align}\label{form181} \mathcal{S}f(p) = \int_{\Omega} \mathrm{div}^{\flat} V_{p}(q) \, dq & = -\int_{\partial \Omega} \langle \nabla^{+} \mathcal{S}f(q),\nu(q) \rangle G(p^{-1} \cdot q) \, d\sigma(q)\\
&\label{form192}\qquad - \int_{\partial \Omega} \langle \nabla G_{p},\nu(q) \rangle Sf(q) \, d\sigma(q).\end{align} 
Recall that $\nabla^\pm \mathcal{S}f$ was defined in \eqref{nablapm}.
There are two challenges in making this computation rigorous. First, $V_{p} \notin C^{1}(\overline{\Omega})$. To fix this, one uses the auxiliary (flag) domains $\overline{\Omega}_{j} \subset \Omega$ familiar from Section \ref{s:injectivity}, more precisely \eqref{form86}. We also recall the maps $\Phi_{j} \colon \partial \Omega \to \partial \Omega_{j}$ with the convenient properties that $\nu_{j}(\Phi_{j}(p)) = \nu(p)$ and $\Phi_{j\sharp}\sigma = \sigma_{j}$. Then $V_{p} \in C^{1}(\overline{\Omega}_{j})$. For $f \in B(\sigma)$, it is also straightforward to justify the integrability conditions in Theorem \ref{t:div}, noting that $|V_{p}(q)| \lesssim \min\{1,\|q\|^{-5}\}$ for $q \notin B(p,\epsilon)$, and since also $V_{p} \in L^{1}(B(p,\epsilon))$. So, \eqref{form181} holds with "$\Omega_{j}$" in place of "$\Omega$" (for all $j \in \N$ so large that $p \in \Omega_{j}$). Finally, to obtain \eqref{form181} as stated, one uses dominated convergence and the $\sigma$ a.e. existence of the non-tangential limits 
\begin{equation}\label{form182} \langle \nabla^{+} \mathcal{S}f(q),\nu(q) \rangle = (-\tfrac{1}{2}I + D^{t})f(q) \quad \text{and} \quad Sf(q), \qquad q \in \partial \Omega, \end{equation}
recall Theorem \ref{c:allJumps} and Proposition \ref{prop11}. To give some details for the term on line \eqref{form181}, one uses the relations $\sigma_{j} = \Phi_{j\sharp}\sigma$ and $\nu_{j}(q) = \nu(\Phi_{j}(q))$ to write the term on line \eqref{form181} as
\begin{displaymath} \int_{\partial \Omega_{j}} \langle \nabla \mathcal{S}f(q),\nu_{j}(q) \rangle G(p^{-1} \cdot q) \, d\sigma_{j}(q) = \int_{\partial \Omega} \langle \nabla \mathcal{S}f(\Phi_{j}(q)),\nu(q) \rangle G(p^{-1} \cdot \Phi_{j}(q)) \, d\sigma(q), \end{displaymath}
and then estimates $|\langle \nabla \mathcal{S}f(q),\nu_{j}(q) \rangle| \leq \mathcal{N}_{\mathrm{rad}}(\nabla \mathcal{S}f)(q)$ and $G(p^{-1} \cdot \Phi_{j}(q)) \lesssim d(p,q)^{-2}$, uniformly for all $j \in \N$ so large that $d(q,\Phi_{j}(q)) < d(p,q)/2$. Since $f \in B(\sigma) \subset L^{p}(\sigma)$ for all $1 < p < \infty$, one e.g. has $\mathcal{N}_{\mathrm{rad}}(\nabla \mathcal{S}f) \in L^{2}(\sigma)$ by Corollary \ref{K-MaxNT}, and then it follows from H\"older's inequality that $q \mapsto \mathcal{N}_{\mathrm{rad}}(\nabla \mathcal{S}f)d(p,q)^{-2} \in L^{1}(\sigma)$. Dominated convergence for the term on line \eqref{form192} is even easier to justify, using $|\mathcal{S}f(q)| \lesssim d(p,q)^{-2}$.

Now that \eqref{form181} has been established, one re-writes it as
\begin{displaymath} \mathcal{S}f(p) = -\mathcal{S}(-\tfrac{1}{2}I + D^{t})f(p) - \mathcal{D}(Sf)(p), \end{displaymath}
recalling the definition of "$\mathcal{D}$", and also using the jump relation recorded in \eqref{form182}. Rerranging terms yields \eqref{SanD} for $f \in B(\sigma)$.

To handle the general case, it suffices to consider non-negative $f \in L^{p}(\sigma)$ with $Sf \in L^{p}(\sigma)$, by the linearity of \eqref{SanD}. For such $f$, define $f_{j} := f \cdot \mathbf{1}_{E_{j}} \in B(\sigma)$, $j \in \N$, where $E_{j} := B(0,j) \cap \{p \in \partial \Omega : f(p) \leq j\}$. Then \eqref{SanD} holds for each $f_{j}$. Moreover, $\|f_{j} - f\|_{L^{p}(\sigma)} \to 0$, and also $\|Sf_{j} - Sf\|_{L^{p}(\sigma)} \to 0$ for the reasons that $0 \leq Sf_{j} \leq Sf$, and $Sf_{j}(p) \nearrow Sf(p)$ for $\sigma$ a.e. $p \in \partial \Omega$. Moreover, since $\tfrac{1}{2}I + D^{t}$ is bounded on $L^{p}(\sigma)$, we have $g_{j} := (\tfrac{1}{2}I + D^{t})f_{j} \to (\tfrac{1}{2}I + D^{t})f =: g$ in $L^{p}(\sigma)$. It follows easily from H\"older's inequality, and the assumption $1 < p < 3$ (note that $\mathcal{S}g_{j},\mathcal{S}g$ are only well-defined when $p < 3$), that
\begin{displaymath} \mathcal{D}(Sf_{j})(p) \to \mathcal{D}(Sf)(p) \quad \text{and} \quad \mathcal{S}g_{j}(p) \to \mathcal{S}g(p), \qquad p \in \Omega, \end{displaymath}
which gives \eqref{SanD} for $f$. \end{proof}

The formula \eqref{SanD} will be used within the proof of Theorem \ref{t:Invert}, but it is also needed to infer Theorem \ref{t:main7a} from Theorem \ref{t:Invert}:

\begin{proof}[Proof of Theorem \ref{t:main7a} assuming Theorem \ref{t:Invert}] The solution "$u$" has the form $\mathcal{D}h$, where $(\tfrac{1}{2}I + D)h = g$. Since $\tfrac{1}{2}I + D$ is invertible on $\mathbf{L}^{2}_{1,1/2}(\sigma)$, we have $\|h\|_{\mathbf{L}^{2}_{1,1/2}(\sigma)} \lesssim \|g\|_{\mathbf{L}^{2}_{1,1/2}(\sigma)} < \infty$. In particular, $h \in L^{2}(\sigma)$ can be expressed, by Theorem \ref{main4}, as the single layer potential "$Sf$" of some function $f \in L^{2}(\sigma)$ satisfying $\|f\|_{L^{2}(\sigma)} \lesssim \|h\|_{1,1/2,\sigma} \lesssim \|g\|_{\mathbf{L}^{2}_{1,1/2}(\sigma)}$. Now, we may use the formula \eqref{SanD} to write
\begin{displaymath} u = \mathcal{D}h = \mathcal{D}(Sf) = -\mathcal{S}(\tfrac{1}{2}I + D^{t})f. \end{displaymath}
It finally it follows from Corollary \ref{K-MaxNT}, and the $L^{2}(\sigma)$-boundedness of $D^{t}$, that
\begin{displaymath} \|\mathcal{N}_{\theta}(\nabla u)\|_{L^{2}(\sigma)} = \left\|\mathcal{N}_{\theta}\left[\nabla \mathcal{S}(\tfrac{1}{2}I + D^{t})f \right]\right\|_{L^{2}(\sigma)} \lesssim \|(\tfrac{1}{2}I + D^{t})f\|_{L^{2}(\sigma)} \lesssim \|g\|_{\mathbf{L}^{2}_{1,1/2}(\sigma)}. \end{displaymath}
This completes the proof of Theorem \ref{t:main7a} (and hence Theorem \ref{main7}). \end{proof}

We then move to other corollaries of \eqref{SanD}, which will be used to prove Theorem \ref{t:Invert}.

\begin{cor} Let $\Omega \subset \He$ be a flag domain, $\sigma := |\partial \Omega|_{\He}$, and $1 < p < 3$. Let $f \in L^{p}(\sigma)$ be such that $Sf \in L^{p}(\sigma)$. Then
\begin{equation}\label{form183} (\tfrac{1}{2}I + D)(Sf)(p) = -S(\tfrac{1}{2}I + D^{t})f(p) \end{equation}
for $\sigma$ a.e. $p \in \partial \Omega$.\end{cor}

\begin{proof} Take non-tangential limits on both sides of the equation \eqref{SanD}, and use the jump relation for $\mathcal{D}^{+}$ recorded in Corollary \ref{c:allJumps}. This proves \eqref{form183}. \end{proof}

\begin{cor} Let $\Omega = \{(x,y,t) : x < A(y)\}$ be a flag domain, $\sigma := |\partial \Omega|_{\He}$. Then, the operator $\tfrac{1}{2}I + D$ is bounded on $\mathbf{L}^{2}_{1,1/2}(\sigma)$, and in fact
\begin{equation}\label{form185} \|(\tfrac{1}{2}I + D)f\|_{\mathbf{L}^{2}_{1,1/2}(\sigma)} \sim \|f\|_{\mathbf{L}^{2}_{1,1/2}(\sigma)}, \qquad f \in \mathbf{L}^{2}_{1,1/2}(\sigma). \end{equation}
The implicit constant only depends on the Lipschitz constant of $A$. \end{cor}

\begin{proof} Let $f \in \mathbf{L}^{2}_{1,1/2}(\sigma)$. Since $f \in L^{2}_{1,1/2}(\sigma)$, there exists by Theorem \ref{main4} a function $g \in L^{2}(\sigma)$ such that $Sg = f \in L^{2}(\sigma)$, and $\|g\|_{L^{2}(\sigma)} \sim \|f\|_{1,1/2,\sigma}$. Consequently, by \eqref{form183},
\begin{align*} \|(\tfrac{1}{2}I + D)f\|_{1,1/2,\sigma} & = \|(\tfrac{1}{2}I + D)(Sg)\|_{1,1/2,\sigma}\\
& = \|S(\tfrac{1}{2}I + D^{t})g\|_{1,1/2,\sigma}\\
& \sim \|(\tfrac{1}{2}I + D^{t})g\|_{L^{2}(\sigma)} \sim \|f\|_{1,1/2,\sigma}, \end{align*}
using also the boundedness and invertibility of the operators $\tfrac{1}{2}I + D^{t} \colon L^{2}(\sigma) \to L^{2}(\sigma)$ and $S \colon L^{2}(\sigma) \to L^{2}_{1,1/2}(\sigma)$, recall Theorem \ref{main2} and Theorem \ref{main4}. Since $\tfrac{1}{2}I + D$ is also invertible on $L^{2}(\sigma)$ by Theorem \ref{main2}, we have now proven \eqref{form185}. \end{proof}

It remains to prove the surjectivity of $\tfrac{1}{2}I + D$ on $\mathbf{L}^{2}_{1,1/2}(\sigma)$. As usual, we first dispose of the special case $A \equiv 0$:
\begin{proposition}\label{p:Invert2} The operator $\tfrac{1}{2}I + D$ is invertible on $\mathbf{L}^{2}_{1,1/2}(\W)$.
\end{proposition}

\begin{proof} On the (Abelian) vertical subgroup $(\W,\cdot) \cong (\R^{2},+)$, the operator $\tfrac{1}{2}I + D$ is a bounded operator on $L^{2}(\R^{2})$ which commutes with Euclidean translations. Therefore, by \cite[Theorem 2.5.10]{MR3243734}, the operator $\tfrac{1}{2}I + D$ can be expressed as convolution with a tempered distribution $u \in \mathcal{S}'(\R^{2})$ with $\hat{u} \in L^{\infty}(\R^{2})$. By Theorem \ref{main2}, the same assertions are true about the inverse $(\tfrac{1}{2}I + D)^{-1}$, which implies that also $\hat{u}^{-1} \in L^{\infty}(\R^{2})$. Since the norm of $f \in \mathbf{L}^{2}_{1,1/2}(\R^{2})$ can be expressed as
\begin{displaymath} \|f\|_{\mathbf{L}^{2}_{1,1/2}(\R^{2})}^2 = \iint (1 + \|(\xi,\tau)\|^{2})|\hat{f}(\xi,\tau)|^{2} \, d\xi \, d\tau, \end{displaymath}
recall Definition \ref{def:L2112}, it is now evident that Fourier multiplication by $\hat{u}$ defines a bounded invertible operator on $\mathbf{L}^{2}_{1,1/2}(\R^{2})$. \end{proof}

To prove the surjectivity of $\tfrac{1}{2}I + D$ on $\mathbf{L}^{2}_{1,1/2}(\sigma)$ in the general case, we aim to apply, once more, the method of continuity, Lemma \ref{l:kenig}. For this purpose, we recall the graph maps $\Gamma_{s} \colon \R^{2} \to \partial \Omega_{s} = \{(sA(y),y,t) : y,t \in \R\}$ used several times in the paper, see e.g. \eqref{eq:Gammas}, and define the auxiliary operators
\begin{itemize}
\item $R_{s}f := D(f \circ \Gamma_{s}^{-1}) \circ \Gamma_{s}$ and $R_{s}^{t}f := D^{t}(f \circ \Gamma_{s}^{-1}) \circ \Gamma_{s}$,
\item $T_{s} := \tfrac{1}{2}I + R_{s}$ and $T_{s}^{t} := \tfrac{1}{2}I + R_{s}^{t}$,
\item $S_{s}f := S(f \circ \Gamma_{s}^{-1}) \circ \Gamma_{s}$ (these operators are the same as in \eqref{op:Ss}).
\end{itemize}
There is a slight notational inconsistency regarding the operators $T_{s}$ and $T_{s}^{t}$. In Section \ref{s:surjectivity}, see \eqref{form74}-\eqref{form184}, we studied "$R_{s}$" and "$T_{s}$", which in the current notation are "$R_{s}^{t}$" and "$T_{s}^{t}$". This is important to keep in mind, because the properties of the operators $R_{s}^{t}$ and $T_{s}^{t}$ established in Section \ref{s:surjectivity} will be used, below, to infer similar properties for $R_{s}$ and $T_{s}$.

To prove that $\tfrac{1}{2}I + D$ is surjective on $\mathbf{L}^{2}_{1,1/2}(\sigma)$, it suffices to show that $T_{s}$ is surjective on $\mathbf{L}^{2}_{1,1/2}(\R^{2})$ (because $\mathbf{L}^{2}_{1,1/2}(\sigma) = \{f \in L^{2}(\sigma) : f \circ \Gamma_{s} \in \mathbf{L}^{2}_{1,1/2}(\R^{2})\}$). We have already established that the operators $T_{s}$ are bounded on $\mathbf{L}^{2}_{1,1/2}(\R^{2})$ and satisfy $\|T_{s}f\|_{\mathbf{L}^{2}_{1,1/2}(\R^{2})} \sim \|f\|_{\mathbf{L}^{2}_{1,1/2}(\R^{2})}$ by \eqref{form185}. Also, $T_{0}$ is invertible on $\mathbf{L}^{2}_{1,1/2}(\R^{2})$ by Proposition \ref{p:Invert2}. Therefore, by Lemma \ref{l:kenig}, it remains to verify that
\begin{equation}\label{form188} \|T_{r} - T_{s}\|_{\mathbf{L}^{2}_{1,1/2}(\R^{2}) \to \mathbf{L}^{2}_{1,1/2}(\R^{2})} \lesssim |r - s|, \qquad r,s \in [0,1]. \end{equation}
As usual, the implicit constant will only depend on $\mathrm{Lip}(A)$. Note that
\begin{displaymath} \|T_{r} - T_{s}\|_{L^{2}(\sigma) \to L^{2}(\sigma)} = \|T_{r}^{t} - T_{s}^{t}\|_{L^{2}(\sigma) \to L^{2}(\sigma)} \lesssim |r - s|, \qquad r,s \in [0,1], \end{displaymath}
by \eqref{form190}, so \eqref{form188} will follow once we manage to show that
\begin{equation}\label{form189} \|T_{r}\varphi - T_{s}\varphi\|_{1,1/2} \lesssim |r - s| \|\varphi\|_{1,1/2}, \qquad r,s \in [0,1], \, \varphi \in C^{\infty}_{c}(\R^{2}). \end{equation}
To prove \eqref{form189}, fix $\varphi \in C^{\infty}_{c}(\R^{2})$, and choose (using Theorem \ref{sInvert2}) functions $f_{r},f_{s} \in L^{2}(\R^{2})$ with the properties 
\begin{itemize}
\item $S_{r}f_{r} = \varphi = S_{s}f_{s}$,
\item $\|f_{r}\|_{2} \sim \|\varphi\|_{1,1/2} \sim \|f_{s}\|_{2}$.
\end{itemize}
Now, applying \eqref{form183} to the function $f_{r} \circ \Gamma_{r}^{-1} \in L^{2}(\sigma_{r})$, whose single layer potential also satisfies $S(f_{r} \circ \Gamma_{r}) = (S_{r}f_{r}) \circ \Gamma_{r}^{-1} = \varphi \circ \Gamma_{r}^{-1} \in L^{2}(\sigma_{r})$, we write
\begin{align*} T_{r}\varphi & = (\tfrac{1}{2}I + D)(\varphi \circ \Gamma_{r}^{-1}) \circ \Gamma_{r}\\
& = (\tfrac{1}{2}I + D)((S_{r}f_{r}) \circ \Gamma_{r}^{-1}) \circ \Gamma_{r}\\
& = (\tfrac{1}{2}I + D)(S(f_{r} \circ \Gamma_{r}^{-1})) \circ \Gamma_{r}\\
& = -S[(\tfrac{1}{2}I + D^{t})(f_{r} \circ \Gamma_{r}^{-1})] \circ \Gamma_{r}\\
& = -S_{r}[((\tfrac{1}{2}I + D^{t})(f_{r} \circ \Gamma_{r}^{-1}) \circ \Gamma_{r}] = -S_{r}[T_{r}^{t}(f_{r})], \qquad r \in \R. \end{align*} 
Consequently, introducing cross terms and using the estimate $\|S_{r} - S_{s}\|_{L^{2}(\R^{2}) \to L^{2}_{1,1/2}(\R^{2})} \lesssim |r - s|$ (established in \eqref{form191}), we have
\begin{align*} \|T_{r}\varphi - T_{s}\varphi\|_{1,1/2} & = \|S_{r}[T_{r}^{t}(f_{r})] - S_{s}[T_{s}^{t}(f_{s})]\|_{1,1/2}\\
& \leq \|S_{r}[T_{r}^{t}(f_{r}) - T_{s}^{t}(f_{s})]\|_{1,1/2} + \|(S_{r} - S_{s})[T_{s}^{t}(f_{s})]\|_{1,1/2}\\
& \lesssim \|[T_{r}^{t}(f_{r}) - T_{s}^{t}(f_{s})]\|_{2} + |r - s|\|T^{t}_{s}(f_{s})\|_{2}. \end{align*} 
Since $T^{t}_{s}$ is bounded on $L^{2}(\R^{2})$, the final term is bounded from above by $|r - s|\|f_{s}\|_{2} \sim |r - s|\|\varphi\|_{1,1/2}$, as desired. To bound the penultimate term, we introduce more cross terms:
\begin{displaymath} \|[T_{r}^{t}(f_{r}) - T_{s}^{t}(f_{s})]\|_{2} \leq \|T_{r}^{t}(f_{r} - f_{s})]\|_{2} + \|(T_{r}^{t} - T_{s}^{t})f_{s}\|_{2}. \end{displaymath} 
The last term is bounded from above by $|r - s|\|f_{s}\|_{2} \sim |r - s|\|\varphi\|_{1,1/2}$ by \eqref{form190}, while the penultimate term is bounded from above by
\begin{align*} \|f_{r} - f_{s}\|_{2} & = \|(S_{r}^{-1}S_{s} - I)f_{s}\|_{2}\\
& \sim \|S_{r}[(S_{r}^{-1}S_{s} - I)f_{s}]\|_{1,1/2}\\
& = \|(S_{s} - S_{r})f_{s}\|_{1,1/2} \lesssim |r - s|\|f_{s}\|_{2} \sim |r - s|\|\varphi\|_{1,1/2}.  \end{align*}
Above, we first used the invertibility of $S_{r} \colon L^{2}(\R^{2}) \to L^{2}_{1,1/2}(\R^{2})$, and finally the estimate \eqref{form191}. This concludes the proof of \eqref{form189}, and hence the proof of Theorem \ref{t:Invert}.

\section{Vertical Fourier transform of the single layer potential}\label{appA}

The purpose of this section is to study the vertical distributional Fourier transform (as in Definition \ref{def:vertFDist}) of the single layer potential $\mathcal{S}f$ in flag domains, and somewhat more general settings. In particular, we will prove Proposition \ref{prop1}. We start by deriving a formula for the distributional vertical Fourier transform of $\mathcal{S}f$:

\begin{lemma}\label{l:Fourier1} Let $\sigma$ be a measure on $\He$ of the form $\sigma = \mu \times \mathcal{L}^{1}$, where $\mu$ is a Radon measure on $\R^{2}$ satisfying
\begin{equation}\label{frostman} \mu(B(z,r)) \leq Cr^{\nu}, \qquad z \in \R^{2}, \, r > 0, \end{equation}
for some $\nu > 0$ and $C \geq 1$. Let $f \in C_{c}^{\infty}(\He)$, and write
\begin{displaymath} \mathcal{S}_{\sigma}f(p) := \int G(q^{-1} \cdot p)f(q) \, d\sigma(q), \qquad p \in \He, \end{displaymath}
where $G(p) := \|p\|^{-2}$. Then $\mathcal{S}_{\sigma}f \in L^{\infty}(\He) \cap C(\He)$, in particular $\mathcal{S}_{\sigma}f \in \mathcal{S}'(\He)$, and the vertical distributional Fourier transform of $\mathcal{S}_{\sigma}f$ (as in Definition \ref{def:vertFDist}) is the function
\begin{equation}\label{form25} \widehat{(\mathcal{S}_{\sigma}f)}(z,\tau) = \tfrac{1}{2}\int e^{\pi i \omega(z,w)\tau} K_{0}(\tfrac{\pi}{2} |\tau| |z - w|^{2})\hat{f}(w,\tau) \, d\mu(w), \qquad \tau \in \R \, \setminus \, \{0\}. \end{equation} 
The integral in \eqref{form25} converges absolutely for every $(z,\tau) \in \He$ with $\tau \neq 0$ and defines a function in $C(\He \, \setminus \, \{\tau = 0\})$. In \eqref{form25}, $K_{0}$ is the \emph{the modified Bessel function of the second kind of order $0$}, namely
\begin{displaymath} K_{0}(s) = \int_{0}^{\infty} e^{-s \cosh r} \, dr, \qquad s > 0, \end{displaymath}
and $\omega(w,z) = xy' - yx'$ for $w = (x,y),z = (x',y') \in \R^{2}$.
\end{lemma}

\begin{remark} In this section, as in the lemma above, we define $G(p) := \|p\|^{-2}$ for $p \in \He \, \setminus \, \{0\}$. This "$G$" differs from the fundamental solution of $-\bigtriangleup^{\flat}$ by a positive multiplicative constant (and we prefer to avoid writing that constant in every formula). \end{remark}

\begin{remark} We justify that $\mathcal{S}_{\sigma}f \in L^{\infty}(\He)$. The reason is that $f \, d\sigma$ is a compactly supported measure with $(f \, d\sigma)(B(p,r)) \lesssim_{f} r^{2 + \nu}$ for all $p \in \He$ and $r > 0$. Now, just note that $G(q^{-1} \cdot p) = d(p,q)^{-2}$ and split the integral defining $\mathcal{S}_{\sigma}f$ into dyadic annuli.  \end{remark}

\begin{remark} To prove Lemma \ref{l:Fourier1}, and understand the integral in \eqref{form25}, we need to know a few basic facts about $K_{0}$. First, $K_{0}$ is the distributional Fourier transform of the function $g(t) \mapsto (t^{2} + 1)^{-1/2}$, or more precisely
\begin{equation}\label{form31} \hat{g}(\tau) = 2K_{0}(2\pi |\tau|), \qquad \tau \neq 0. \end{equation}
Based on \eqref{form31}, one may (formally) compute the distributional vertical Fourier transform of $G$:
\begin{align} \widehat{G}(z,\tau) = \int_{\R} \frac{e^{-2\pi i t\tau}}{\sqrt{|z|^{4} + 16t^{2}}} \, dt & = \frac{|z|^{2}}{4} \int_{\R} \frac{e^{-2\pi i (|z|^{2}u)\tau/4}}{\sqrt{|z|^{4} + (u|z|^{2})^{2}}} \, du \notag\\
&\label{form68} = \tfrac{1}{4} \hat{g}(|z|^{2}\tau/4) = \tfrac{1}{2} K_{0}(\tfrac{\pi}{2}|\tau||z|^{2}), \quad |\tau||z| \neq 0. \end{align} 
The asymptotic behaviour of the function $K_{0}$ is well-understood, both near the origin, and at infinity. Near the origin, $K_{0}$ has a logarithmic singularity, see \cite[(A.3)]{MR3229830}, but for us it suffices to know that
\begin{equation}\label{K0asymp1} |K_{0}(s)| \lesssim_{\beta} s^{-\beta}, \qquad 0 < s \leq 1, \, \beta > 0. \end{equation}
At infinity, $K_{0}$ decays exponentially, see \cite[(A.5)]{MR3229830}:
\begin{equation}\label{K0asymp2} |K_{0}(s)| \lesssim e^{-s}, \qquad s \gtrsim 1. \end{equation}
\end{remark}
We then prove Lemma \ref{l:Fourier1}. Deriving the formula \eqref{form25} formally is quite straightforward, but justifying the computation properly takes some technical work.

\begin{proof}[Proof of Lemma \ref{l:Fourier1}] We need to show that
\begin{displaymath} \int_{\He} \mathcal{S}_{\sigma}f(p) \cdot \widehat{\varphi}(p) \, dp = \int_{\He} H(p) \cdot \varphi(p) \, dp, \qquad \varphi \in C^{\infty}_{c}(\He), \end{displaymath}
where $H$ is the function appearing in \eqref{form25}, and $\widehat{\varphi}$ is the vertical Fourier transform of $\varphi$, recall \eqref{form52}. We will accomplish this by writing $\rho := \varphi_{z} \in C_{c}^{\infty}(\R)$ for $z \in \R^{2}$ fixed, and showing that \begin{equation}\label{form67} \int_{\R} \mathcal{S}_{\sigma}f(z,t) \cdot \hat{\rho}(t) \, dt = \int_{\R} H(z,\tau) \cdot \rho(\tau) \, d\tau. \end{equation}
To make sense of these integrals, let us remark that $\tau \mapsto H(z,\tau) \in L^{1}(\R)$ for every $z \in \R^{2}$ fixed. Indeed, applying the local growth bound \eqref{K0asymp1} with $2\beta < \nu$, we have
\begin{displaymath} \int_{\R} \int |K_{0}(\tfrac{\pi}{2}|\tau||z - w|^{2})|\hat{f}(w,\tau)| \, d\mu w \, d\tau \lesssim \int_{\{w : f_{w} \neq 0\}} |z - w|^{-2\beta} \int |\tau|^{-\beta}|\hat{f}(w,\tau)| \, d\tau \, d\mu w < \infty. \end{displaymath}
Now, to prove \eqref{form67}, let $\{\varphi^{\epsilon}\}_{\epsilon > 0}$ be a standard (even, compactly supported) approximate identity in $\R$. Then $\widehat{\varphi^{\epsilon}} \in \mathcal{S}(\R)$ for $\epsilon > 0$ fixed, and $\widehat{\varphi^{\epsilon}} \to 1$ uniformly on compact sets as $\epsilon \to 0$. Also $t \mapsto (\mathcal{S}_{\sigma}f)(z,t)\widehat{\varphi^{\epsilon}}(t) \in L^{1}(\R)$ for $\epsilon > 0$ fixed. Consequently,
\begin{align} \int_{\R} (\mathcal{S}_{\sigma}f)(z,t) \cdot \hat{\rho}(t) \, dt & = \lim_{\epsilon \to 0} \int_{\R} (\mathcal{S}_{\sigma}f)(z,t)\widehat{\varphi^{\epsilon}}(t) \cdot \hat{\rho}(t) \, dt \notag\\
&\label{form21a} = \lim_{\epsilon \to 0} \int_{\R} \rho(\tau) \left[ \int_{\R} e^{-2\pi i t\tau} (\mathcal{S}_{\sigma}f)(z,t)\widehat{\varphi^{\epsilon}}(t) \, dt \right] \, d\tau. \end{align}
We next expand the expression inside the square brackets:
\begin{align} \int_{\R} & e^{-2\pi i t\tau} (\mathcal{S}_{\sigma}f)(z,t)\widehat{\varphi^{\epsilon}}(t) \, dt \notag\\
& = \int_{\R} e^{-2\pi it\tau} \widehat{\varphi^{\epsilon}}(t) \iint G((w,s)^{-1} \cdot (z,t))f(w,s) \, ds \, d\mu w \, dt \notag\\
& = \int_{\R} e^{-2\pi i t\tau} \widehat{\varphi^{\epsilon}}(t) \iint G(z -w,t - s - \tfrac{1}{2}\omega(w,z))f(w,s) \, ds \, d\mu w \, dt \notag \\
& = \iint f(w,s) \int_{\R} e^{-2\pi i t\tau} G(z - w,t - s - \tfrac{1}{2}\omega(w,z))\widehat{\varphi^{\epsilon}}(t) \, dt \, ds \, d\mu w \notag \\
& = \iint f(w,s) \int_{\R} e^{-2\pi i(u + s + \tfrac{1}{2}\omega(w,z))\tau} G(z - w,u)\widehat{\varphi^{\epsilon}}(u + s + \tfrac{1}{2}\omega(w,z)) \, du \, ds \, d\mu w \notag\\
&= \int e^{\pi i \omega(z,w)\tau} \int_{\R} e^{-2\pi i s\tau} f(w,s) \int_{\R} e^{-2\pi i u\tau}G(z - w,u)\widehat{\varphi^{\epsilon}}(u + s + \tfrac{1}{2}\omega(w,z)) \, du \, ds \, d\mu w \notag\\
&\label{form17a}  =: \int e^{\pi i \omega(z,w)\tau} \int_{\R} e^{-2\pi i s\tau} f(w,s) \cdot \Phi^{\epsilon}((w,z),(s,\tau)) \, ds \, d\mu w, \end{align}
where 
\begin{align*} \Phi^{\epsilon}((w,z),(s,\tau)) & = \int_{\R} e^{-2\pi i u\tau}G(z - w,u)\widehat{\varphi^{\epsilon}}(u + s + \tfrac{1}{2}\omega(w,z)) \, du\\
& =: \int_{\R} e^{-2\pi i u\tau}G_{z - w}(u)\psi^{\epsilon}_{s,z,w}(u) \, du = \widehat{G_{z - w}} \ast \widehat{\psi^{\epsilon}_{s,z,w}}(\tau), \end{align*}
where
\begin{displaymath} G_{z - w}(u) := G(z - w,u) \quad \text{and} \quad \psi^{\epsilon}_{s,z,w}(u) := \widehat{\varphi^{\epsilon}}(u + s + \tfrac{1}{2}\omega(w,z)). \end{displaymath}
We know the Fourier transforms of both $G_{z - w}$ and $\psi_{s,z,w}$:
\begin{displaymath} \widehat{G_{z - w}}(\xi) \stackrel{\eqref{form68}}{=} \tfrac{1}{2}K_{0}(\tfrac{\pi}{2} |\xi||z - w|^{2}) \quad \text{and} \quad \widehat{\psi^{\epsilon}_{s,z,w}}(\xi) = e^{2\pi i (s + \tfrac{1}{2}\omega(w,z))\xi}\varphi^{\epsilon}(\xi). \end{displaymath} 
Consequently,
\begin{equation}\label{form19a} \Phi^{\epsilon}((w,z),(s,\tau)) = \tfrac{1}{2}\int_{\R} e^{2\pi i(s + \tfrac{1}{2}\omega(z,w))(\tau - \xi)} K_{0}(\tfrac{\pi}{2} |\xi||z - w|^{2})\varphi^{\epsilon}(\tau - \xi) \, d\xi. \end{equation}
For $z,w,s,\tau$ fixed with $|\tau||z - w| \neq 0$, recalling that $\{\varphi^{\epsilon}\}_{\epsilon > 0}$ is an approximate identity, it is clear that 
\begin{equation}\label{form69} \Phi^{\epsilon}((w,z),(s,\tau)) \to \tfrac{1}{2}K_{0}(\tfrac{\pi}{2} |\tau||z - w|^{2}) \quad \text{as} \quad \epsilon \to 0. \end{equation}
However, for the purposes of applying dominated convergence, we need some estimates for $\Phi^{\epsilon}((w,z),(s,\tau))$ which do not depend on $\epsilon$. To obtain them, we recall from \eqref{K0asymp1} that $|K_{0}(x)| \lesssim_{\beta} |x|^{-\beta}$ for any $\beta > 0$,
and we also note that $\varphi^{\epsilon} \lesssim \epsilon^{-1}\mathbf{1}_{[-\epsilon,\epsilon]}$. Then, we fix $w \neq z$, $\tau \neq 0$, and $0 < \beta < 1$, and consider the following two cases:
\begin{itemize}
\item First, assume that $0 < \epsilon \leq |\tau|/2$. Then, using the expression \eqref{form19a},
\begin{displaymath} |\Phi^{\epsilon}((z,w),(s,\tau))| \lesssim \frac{1}{\epsilon} \int_{\tau - \epsilon}^{\tau + \epsilon} (|\xi||z - w|^{2})^{-\beta} \, d\xi \lesssim |\tau|^{-\beta}|z - w|^{-2\beta}. \end{displaymath}
\item Second, assume that $\epsilon > |\tau|/2$. Then, by similar reasoning,
\begin{displaymath} |\Phi^{\epsilon}((z,w),(s,\tau))| \lesssim \frac{1}{\epsilon} \int_{-3\epsilon}^{3\epsilon} (|\xi||z - w|^{2})^{-\beta} \, d\xi \lesssim \epsilon^{-\beta} |z - w|^{-2\beta} \lesssim |\tau|^{-\beta}|z - w|^{-2\beta}. \end{displaymath}
\end{itemize}
The two estimates agree, so we have established that
\begin{displaymath} \sup_{\epsilon > 0} |\Phi^{\epsilon}((z,w),(s,\tau))| \lesssim_{\beta} |\tau|^{-\beta}|z - w|^{-2\beta}, \qquad \tau \neq 0, \, z \neq w, \end{displaymath}
for any $0 < \beta < 1$. We now plug \eqref{form17a} back into \eqref{form21a}, and put absolute values inside to find that
\begin{align*} \int_{\R} & |\rho(\tau)| \int \int_{\R} |f(w,s)| \cdot |\Phi^{\epsilon}((w,z),(s,\tau))| \, ds \, d\mu w \, d\tau\\
& \lesssim_{f} \int_{\R} |\tau|^{-\beta}|\rho(\tau)| \int |z - w|^{-2\beta} \, d\mu w \, d\tau \lesssim_{\beta,\rho} \infty,  \end{align*}  
as long as $0 < 2\beta < \min\{\nu,2\}$, where $\nu > 0$ was the growth exponent of $\mu$, recall \eqref{frostman}. This estimate justifies the use of the dominated convergence theorem, and \eqref{form69} yields
\begin{align*} \int_{\R} (\mathcal{S}_{\sigma}f)(z,t) \cdot \hat{\rho}(t) \, dt & = \lim_{\epsilon \to 0} \int_{\R} \rho(\tau) \left[ \int_{\R} e^{-2\pi i t\tau} (\mathcal{S}_{\sigma}f)(z,t)\widehat{\varphi^{\epsilon}}(t) \, dt \right] \, d\tau\\
& = \int_{\R} \rho(\tau) \left[ \tfrac{1}{2}\int e^{\pi i \omega(z,w)\tau}\int_{\R} e^{-2\pi i s\tau} f(w,s) K_{0}(\tfrac{\pi}{2} |\tau||z - w|^{2}) \, ds \, d\mu w \right] \, d\tau\\
& = \int_{\R} \rho(\tau) \left[ \tfrac{1}{2}\int e^{\pi i \omega(z,w)\tau} K_{0}(\tfrac{\pi}{2}|\tau||z - w|^{2})\hat{f}(w,\tau) \, d\mu w \right] \, d\tau. \end{align*} 
This concludes the proof of the lemma. \end{proof}

Our next goal will be to show that $\mathcal{S}_{\sigma}f \in \mathcal{G}(\Omega)$ (recall the "good" class of vertical distributions from Definition \ref{classG}), where $\Omega = (\spt \mu)^{c} \times \R$. To this end, we will need to understand the $\partial_{x}$ and $\partial_{y}$ derivatives of $\widehat{\mathcal{S}_{\sigma}f}$. Taking a look at the formula \eqref{form25}, this entails understanding the derivatives of $K_{0}$. We record some well-known facts:
\begin{lemma}\label{besselDerivative} We have $K_{0}'(s) = -K_{1}(s)$, and 
\begin{equation}\label{form29} K_{1}'(s) = -K_{0}(s) - \tfrac{1}{s}K_{1}(s), \qquad s > 0. \end{equation}
Here $K_{1}$ is the modified Bessel function of the second kind of index $1$, defined by
\begin{equation}\label{form26} K_{1}(s) = \int_{0}^{\infty} (\cosh r)e^{-s \cosh r} \, dr. \end{equation}
\end{lemma}
\begin{proof} The relation $K_{0}'(s) = -K_{1}(s)$ (for $s > 0$) is immediate from the definitions. For \eqref{form29}, see \cite[(A.15)]{MR3229830}. \end{proof}

\begin{remark} We record the asymptotic behaviour of $K_{1}$ near the origin, and at infinity: 
\begin{equation}\label{K1asymp1} |K_{1}(s)| \lesssim s^{-1}, \qquad s > 0, \end{equation}
and
\begin{equation}\label{K1asymp2} |K_{1}(s)| \lesssim e^{-s}, \qquad s \gtrsim 1. \end{equation}
For a reference, ee \cite[(A.3)-(A.5)]{MR3229830}.
\end{remark}
An iteration of Lemma \ref{besselDerivative} allows us to compute derivatives of $K_{0}$ of all orders in terms of the first two functions $K_{0}$ and $K_{1}$. In particular, we have the following bounds:
\begin{proposition} Let $m \geq 1$. Then there is a constant $C_{m} \geq 1$ such that
\begin{equation}\label{form27} |K_{0}^{(m)}(s)| \leq C_{m}s^{-m}, \qquad s > 0, \end{equation}
and
\begin{equation}\label{form28} |K_{0}^{(m)}(s)| \leq C_{m}e^{-s}, \qquad s \geq 1. \end{equation}
\end{proposition}
\begin{proof} We first claim inductively that $K_{1}^{(m)}$ has the form 
\begin{equation}\label{form30} K_{1}^{(m)}(s) = \sum_{j = 0}^{m} \frac{a(j,m)}{s^{j}}K_{0}(s) + \sum_{j = 0}^{m} \frac{b(j,m)}{s^{j}} K_{1}(s), \end{equation} 
where $a(j,m),b(j,m) \in \Z$ for $0 \leq j \leq m$. The case $m = 0$ is clear. For $m \geq 1$, we use induction, the relation $K_{0}' = -K_{1}$, and also the equation \eqref{form29}:
\begin{align*} K_{1}^{(m)}(s) & \stackrel{\eqref{form29}}{=} -K_{0}^{(m - 1)}(s) - \partial_{s}^{(m - 1)} \tfrac{1}{s}K_{1}(s)\\
& = K_{1}^{(m - 2)}(s) - \sum_{j = 0}^{m - 1} \binom{m - 1}{j} (\partial_{s}^{j} \tfrac{1}{s}) K_{1}^{(m - 1 - j)}(s)\\
& = K_{1}^{(m - 2)}(s) - \sum_{j = 0}^{m - 1} \frac{c(j,m)}{s^{j + 1}} K_{1}^{(m - 1 - j)}(s),  \end{align*} 
where $c(j,m) = (-1)^{j}j!\binom{m - 1}{j}$. Applying the induction hypothesis, the expression on the last line can be re-written as
\begin{displaymath} \sum_{j = 0}^{m + 1} \frac{a(j,m + 1)}{s^{j}} K_{0}(s) + \sum_{j = 0}^{m + 1} \frac{b(j,m + 1)}{s^{j}}K_{1}(s). \end{displaymath}
This completes the proof of \eqref{form30}. Now, we infer from the relation $K_{0}'(s) = -K_{1}(s)$ that
\begin{displaymath} |K_{0}^{(m)}(s)| \lesssim_{m} \sum_{j = 0}^{m - 1} \frac{|K_{0}(s)|}{s^{j}} + \sum_{j = 0}^{m - 1} \frac{|K_{1}(s)|}{s^{j}}, \qquad m \geq 1. \end{displaymath}
Using the asymptotics \eqref{K0asymp1}-\eqref{K0asymp2} of $K_{0}$ (with $\beta = 1$) and \eqref{K1asymp1}-\eqref{K1asymp2} of $K_{1}$ completes the proof of the lemma.
\end{proof}

As corollaries, we obtain decay estimates for $\partial_{x}^{m}\partial_{y}^{n}\widehat{Sf}$:

\begin{lemma}\label{lemma3} Let $\mu$ be a Radon measure on $\R^{2}$ satisfying the growth estimate \eqref{frostman} for some $\nu > 0$, and let $f \in C^{\infty}_{c}(\He)$. As always, let $\sigma := \mu \times \mathcal{L}^{1}$ and $u := \mathcal{S}_{\sigma}f \in \mathcal{S}'(\He)$, and write 
\begin{displaymath} K_{f} := \{w \in \spt \mu : f_{w} \neq 0\} \subset \spt \mu, \end{displaymath}
where $f_{w}(t) = f(w,t)$. Then we have the following decay estimates for $\hat{u}(z,\tau)$ whenever $2\beta < \nu$:
\begin{equation}\label{uDecay1} |\hat{u}(z,\tau)| \lesssim_{f} \begin{cases} \|\hat{f}(\cdot,\tau)\|_{L^{\infty}(\R^{2})}|\tau|^{-\beta}, & \text{for } (z,\tau) \in \He, \, \tau \neq 0, \\ e^{-\tfrac{\pi}{2}|\tau|\dist(z,K_{f})^{2}}, & \text{for } |\tau| \geq \dist(z,K_{f})^{-2}. \end{cases}  \end{equation}
Moreover, if $\epsilon > 0$, we also have the following decay estimates for $\partial_{z}^{n}\hat{u}(z,\tau)$, where $\partial^{n}_{z} = \partial_{x}^{m}\partial_{y}^{n - m}$, and which are valid under the \emph{a priori} assumption $\dist(z,K_{f}) \geq \epsilon$:
\begin{equation}\label{uDecay2} |\partial_{z}^{n}\hat{u}(z,\tau)| \lesssim_{\epsilon,f,n} \begin{cases} 1, \\ (|\tau|\dist(z,K_{f}))^{n}e^{-\tfrac{\pi}{2}|\tau|\dist(z,K_{f})^{2}}, & \text{for } |\tau| \geq \dist(z,K_{f})^{-2}, \end{cases}  \end{equation}
Rough summary: $\hat{u}$ and all of its $z$-derivatives decay rapidly in the region $\{(z,\tau) : \dist(z,K_{f}) \geq \epsilon \text{ and } |\tau| \geq \dist(z,K_{f})^{-2}\}$.
\end{lemma}
\begin{proof} We start with the less technical estimate \eqref{uDecay1}. We simply need to use the asymptotics \eqref{K0asymp1}-\eqref{K0asymp2} to find that 
\begin{equation}\label{form34} |K_{0}(\tfrac{\pi}{2}|\tau||z - w|^{2})| \lesssim_{\beta} \begin{cases} (|\tau||z - w|^{2})^{-\beta}, & \text{for } |\tau||z - w| \neq 0,\\  e^{-\tfrac{\pi}{2}|\tau||z - w|^{2}}, & \text{for } |\tau||z - w|^{2} \geq 1. \end{cases}  \end{equation}
This bound is valid for any $\beta > 0$, and
\begin{displaymath} |\hat{u}(z,\tau)| \leq \int_{K_{f}} |K_{0}(\tfrac{\pi}{2}|\tau||z - w|^{2})||\hat{f}(w,\tau)| \, d\mu w \lesssim |\tau|^{-\beta}\|\hat{f}(\cdot,\tau)\|_{\infty}\int_{K_{f}} |z - w|^{-2\beta} \, d\mu w. \end{displaymath} 
Note that the integration may be restricted to $K_{f}$, since $\tau \mapsto \hat{f}(w,\tau) \equiv 0$ whenever $w \notin K_{f}$. If $2\beta < \nu$ (the growth exponent of $\mu$), the integral on the RHS converges, and hence we have the first part of \eqref{uDecay1}. To prove the second bound, fix $z \in \R^{2} \, \setminus \, K_{f}$ and $|\tau| \geq \dist(z,K_{f})^{-2}$. Then in particular $|\tau||z - w|^{2} \geq 1$ for all $w \in K_{f}$, and we may employ the second bound in \eqref{form34}:
\begin{displaymath} |\hat{u}(z,\tau)| \lesssim \int_{K_{f}} e^{-\tfrac{\pi}{2}|\tau||z - w|^{2}}|\hat{f}(w,\tau)| \, d\mu w \lesssim_{f} e^{-\tfrac{\pi}{2}|\tau|\dist(z,K_{f})^{2}},\end{displaymath}
as desired. This completes the proof of \eqref{uDecay1}.

We then move to the proof of \eqref{uDecay1}. We only treat the case $\partial_{z}^{n} = \partial_{x}^{n}$; the general case is very similar, but adds substantial notational inconvenience. Fix $w = (w_{1},w_{2}) \in K_{f}$ and $z = (x,y) \in \R^{2}$ with $\dist(z,K_{f}) \geq \epsilon$. From Fa\`a di Bruno's formula, and noting that $\partial_{x}^{j}(z \mapsto |z - w|^{2}) \equiv 0$ for $j \geq 3$, we deduce for $n \geq 1$ that
\begin{displaymath} \partial_{x}^{n} K_{0}(\tfrac{\pi}{2}|\tau||z - w|^{2}) = \sum_{m_{1} + 2m_{2} = n} \frac{\pi^{m_{1} + m_{2}}n!}{2m_{1}!m_{2}!}K_{0}^{(m_{1} + m_{2})}(\tfrac{\pi}{2}|\tau||z - w|^{2})|\tau|^{m_{1} + m_{2}}(x - w_{1})^{m_{1}} \end{displaymath}  
for $\tau \neq 0$. For $m_{1} + 2m_{2} = n$ fixed, we may derive from the asymptotics \eqref{form27}-\eqref{form28} the following bounds for the individual terms in the sum above:
\begin{align} |K_{0}^{(m_{1} + m_{2})} & (\tfrac{\pi}{2}|\tau||z - w|^{2})| |\tau|^{m_{1} + m_{2}} |x - w_{1}|^{m_{1}} &\notag \\
&\label{form24} \lesssim_{n} |z - w|^{-2(m_{1} + m_{2})}|x - w_{1}|^{m_{1}} \lesssim_{\epsilon,n} 1, \qquad \tau \neq 0, \end{align} 
and on the other hand
\begin{align} |K_{0}^{(m_{1} + m_{2})} & (\tfrac{\pi}{2}|\tau||z - w|^{2})| |\tau|^{m_{1} + m_{2}} |x - w_{1}|^{m_{1}} \notag\\
&\label{form32} \lesssim_{\epsilon,n} (|\tau||z - w|^{2})^{n} e^{-\tfrac{\pi}{2}|\tau||z - w|^{2}}, \qquad |\tau||z - w|^{2} \geq 1. \end{align}
Combining \eqref{form24}-\eqref{form32}, and summing over all indices $m_{1} + 2m_{2} = n$, we find that
\begin{equation}\label{form33} |\partial_{x}^{n}K_{0}(\tfrac{\pi}{2}|\tau||z - w|^{2})| \lesssim_{\epsilon,n} \begin{cases} 1, & \text{for } \tau \neq 0, \\  (|\tau||z - w|^{2})^{n} e^{-\tfrac{\pi}{2}|\tau||z - w|^{2}}, & \text{for } |\tau||z - w|^{2} \geq 1, \end{cases}  \end{equation}
whenever $n \geq 1$, $w \in K_{f}$, and and $\dist(z,K_{f}) \geq \epsilon$. A point worth mentioning is that $\tau \mapsto \partial_{x}^{n}K_{0}(\tfrac{\pi}{2}|\tau||z - w|^{2})$ has no singularity at the origin when $n \geq 1$ -- unlike $\tau \mapsto K_{0}(\tfrac{\pi}{2}|\tau||z - w|^{2})$ with a logarithmic singularity.

To estimate $\partial_{x}^{n}\hat{u}(z,\tau)$, we first differentiate (formally) under the integral sign:
\begin{equation}\label{form36} \partial_{x}^{n}\hat{u}(z,\tau) = \int_{K_{f}} \partial_{x}^{n}[z \mapsto e^{\pi i \omega(z,w)\tau}K_{0}(\tfrac{\pi}{2}|\tau||z - w|^{2})] \hat{f}(w,\tau) \, d\mu w, \quad \tau \neq 0. \end{equation}
Recalling that $\omega(z,w) = (xw_{2} - yw_{1})$, we record for $w \in K_{f}$ and $\dist(z,K_{f}) \geq \epsilon$ that
\begin{displaymath} |\partial_{x}^{m} e^{-\pi i \omega(z,w)\tau}| \lesssim (|\tau||w|)^{m} \lesssim_{f} |\tau|^{m} \lesssim_{\epsilon} (|\tau||z - w|^{2})^{m}, \qquad m \geq 0. \end{displaymath}
Therefore, using the Leibniz rule and \eqref{form33}, we have
\begin{align} |\partial_{x}^{n} & [z \mapsto e^{\pi i \omega(z,w)\tau}K_{0}(\tfrac{\pi}{2}|\tau||z - w|^{2})]| \notag\\
&\label{form35} \lesssim_{\epsilon,f,n} \begin{cases} 1, & \text{for } \tau \neq 0, \\  (|\tau||z - w|^{2})^{n} e^{-\tfrac{\pi}{2}|\tau||z - w|^{2}}, & \text{for } |\tau||z - w|^{2} \geq 1, \end{cases}  \end{align}
for $n \geq 1$, $w \in K_{f}$, and $\dist(z,K_{f}) \geq \epsilon$. The estimate \eqref{form35} clearly justifies the exchange of differentiation and integration at \eqref{form36}. With \eqref{form35} in hand, the proof of \eqref{uDecay2} follows just like the proof of \eqref{uDecay1} followed from \eqref{form34}, also using that $\sup\{|z - w| : w \in K_{f}\} \leq \diam(K_{f}) + \dist(z,K_{f}) \lesssim_{\epsilon} \dist(z,K_{f})$.
\end{proof}

We can finally prove (a shaper version of) Proposition \ref{prop1}. We also study the integrability properties of $T^{\alpha}(Sf)$ and $\nabla T^{\alpha}(Sf)$ w.r.t. measures supported well inside $\Omega$. 
\begin{cor}\label{cor3} Let $\sigma = \mu \times \mathcal{L}^{1}$, $f \in C^{\infty}_{c}(\He)$, and $K_{f} = \{w \in \spt \mu : f_{w} \neq 0\}$, as in the previous lemma. Write
\begin{displaymath} \Omega_{f} := U_{f} \times \R := K_{f}^{c} \times \R. \end{displaymath}
Then $u = Sf \in \mathcal{G}(\Omega_{f})$ (recall Definition \ref{classG}, or see the argument below). In particular, $u \in \mathcal{G}(\Omega)$ with $\Omega := (\spt \mu)^{c} \times \R$.

Moreover, let $\alpha,\epsilon > 0$, let 
\begin{displaymath} U_{\epsilon} := U_{f,\epsilon} := \{z \in \R^{2} : \dist(z,K_{f}) \geq \epsilon\}, \end{displaymath}
and let $\mathfrak{M}$ be a measure on $\Omega_{\epsilon} = U_{\epsilon} \times \R$ of the form $\mathfrak{M} := \eta \times \mathcal{L}^{1}$, where $\eta$ is a Radon measure on $\R^{2}$ with $\spt \eta \subset U_{\epsilon}$, and
\begin{equation}\label{eta} \eta(B(x,R)) \leq R^{\theta}, \qquad R \geq 1, \end{equation} 
for some $\theta \in [0,2]$. Then, $T^{\alpha}u \in L^{2}(\mathfrak{M})$ and $\nabla (T^{\alpha}u) \in L^{2}(\mathfrak{M})$.
\end{cor}

\begin{remark} The only applications of $\mathfrak{M}$ in this paper will be measures of the form $\mathcal{L}^{3}|_{\Omega_{\epsilon}} = \mathcal{L}^{2}|_{U_{\epsilon}} \times \mathcal{L}^{1}$ and $\calH^{1}|_{\Gamma} \times \mathcal{L}^{1}$, where $\Gamma \subset U_{\epsilon}$ is a Lipschitz graph. As another remark, the assumption $\alpha > 0$ is necessary, at least for the claim "$T^{\alpha}u \in L^{2}(\Omega_{\epsilon})$". Indeed $|u(p)| \sim_{\epsilon} \|p\|^{-2}$ for $p \in \Omega_{\epsilon}$, so $u \notin L^{2}(\Omega_{\epsilon})$. On the other hand, $\nabla (T^{\alpha}u) \in L^{2}(\mathfrak{M})$ remains true for $\alpha = 0$, under the assumptions of the corollary, since $|\nabla u(p)| \lesssim \|p\|^{-3}$ for $p \in \Omega_{\epsilon}$. \end{remark}

\begin{proof}[Proof of Corollary \ref{cor3}] The distributional vertical Fourier transform or $\partial_{x}^{k}\partial_{y}^{m}\partial_{t}^{n}u$ equals $(2\pi i \tau)^{n}\partial_{x}^{k}\partial_{y}^{m}\hat{u}$ by Lemma \ref{lemma2}. Here $\partial_{x}^{k}\partial_{y}^{m}\hat{u}$ \emph{a priori} refers to a distributional derivative of $\hat{u}$, but since $\hat{u}$ is smooth in the $x$ and $y$ variables in $\Omega_{f} \, \setminus \, \{\tau = 0\}$, and both $\hat{u}$ and $\partial_{x}^{k}\partial_{y}^{m}\hat{u}$ are locally integrable in $\Omega_{f}$ (all of this follows immediately from \eqref{uDecay1}-\eqref{uDecay2}), it is straightforward to check that the distribution coincides with the function $\partial_{x}^{k}\partial_{y}^{m}\hat{u} \in L^{1}_{\mathrm{loc}}(\Omega_{f}) \cap C^{\infty}(\Omega_{f} \, \setminus \, \{\tau = 0\})$. Therefore, the vertical Fourier transform of $\partial u$ is a locally integrable function in $\Omega_{f}$, it follows from another application of the decay estimates \eqref{uDecay1}-\eqref{uDecay2} that \emph{a fortiori} $\widehat{\partial u} \in L^{1}(K \times \R)$ for any compact set $K \subset U_{f}$ (since $\dist(K,K_{f}) > 0$). This completes the proof of $u \in \mathcal{G}(\Omega_{f})$.

Next, we fix $\alpha,\epsilon > 0$, and a measure of the form $\mathfrak{M} = \eta \times \mathcal{L}^{1}$ (as in the statement of the corollary), and claim that $T^{\alpha}u \in L^{2}(\mathfrak{M})$. Note that $T^{\alpha}u \in C^{\infty}(\Omega_{f})$ by Proposition \ref{prop2}, so there is no ambiguity in defining $T^{\alpha}u$ as a $\mathfrak{M}$ measurable function. By Plancherel, in the $t$-variable,
\begin{displaymath} \|T^{\alpha}u\|_{L^{2}(\mathfrak{M})}^{2} = \int \int_{\R} |\tau|^{2\alpha}|\hat{u}(z,\tau)|^{2} \, d\tau \, d\eta(z). \end{displaymath}
Recall that $\spt \eta \subset U_{\epsilon}$, and fix $z \in \spt \eta$. For fixed $\beta > 0$ with $2\beta < \nu$ (the growth exponent of $\mu$), we split the $\R$-integration and use the estimates in \eqref{uDecay1} as follows:
\begin{align}\label{form43} \int_{\R} |\tau|^{2\alpha}|\hat{u}(z,\tau)|^{2} \,  d\tau & \lesssim \int_{\{|\tau| < \dist(z,K_{f})^{-2 + \beta}\}} |\tau|^{2\alpha -2\beta} \, d\tau\\
&\label{form44} \qquad + \int_{\{|\tau| \geq \dist(z,K_{f})^{-2 + \beta}\}} e^{-\tfrac{\pi}{2}|\tau|\dist(z,K_{f})^{2}} \, d\tau. \end{align}
We now pick $\beta > 0$ so small that $1 + 2\alpha > (2 + \beta)/(2 - \beta) + 2\beta$, and note that
\begin{displaymath} \eqref{form43} \sim \dist(z,K_{f})^{-(2 - \beta)(1 + 2\alpha - 2\beta)} \lesssim_{\epsilon,\alpha,\beta} \dist(z,K_{f})^{-(2 + \beta)}. \end{displaymath}
This is good enough, since $\int \dist(z,K_{f})^{-(2 + \beta)} \, d\eta(z) < \infty$ by the growth bound \eqref{eta} for $\eta$. To treat the term on line \eqref{form44}, we make the change of variables $\tau \mapsto \xi/\dist(z,K_{f})^{2}$, so
\begin{displaymath} \eqref{form44}= \frac{1}{\dist(z,K_{f})^{2}} \int_{\{|\xi| \geq \dist(z,K_{f})^{\beta}\}} e^{-\tfrac{\pi}{2}|\xi|} \, d\xi \sim \frac{e^{-\tfrac{\pi}{2}\dist(z,K_{f})^{\beta}}}{\dist(z,K_{f})^{2}}. \end{displaymath} 
The RHS, as a function of $z$, is evidently $\eta$-integrable by \eqref{eta}. We have now shown that $T^{\alpha}u \in L^{2}(\mathfrak{M})$ for $\alpha > 0$, as claimed.

We then prove the final claim, namely that also $\nabla (T^{\alpha}u) \in L^{2}(\mathfrak{M})$. Fix $\alpha > 0$. Using the decay estimates \eqref{uDecay2} (with $n = 1$) in place of \eqref{uDecay1}, and $\widehat{\partial_{z}u} = \partial_{z}\hat{u}$ for $z \in \{x,y\}$, it is easy to modify the previous argument to show that $T^{\alpha}(\partial_{z}u) \in L^{2}(\mathfrak{M})$. By Proposition \ref{prop2}, we have $Z(T^{\alpha}u) = T^{\alpha}(Zu)$ for $Z \in \{X,Y\}$, so noting that $|\widehat{Zu}| \leq |\partial_{z}\hat{u}| + \pi |z||\tau||\hat{u}|$, Plancherel's theorem in the $t$-variable gives
\begin{align} \|Z(T^{\alpha}u)\|^{2}_{L^{2}(\mathfrak{M})} & \lesssim \|T^{\alpha}(\partial_{z}u)\|_{L^{2}(\mathfrak{M})}^{2} \notag\\
&\label{form46} \quad + \int |z|^{2}\int_{\R} |\tau|^{2 + 2\alpha}|\hat{u}(z,\tau)|^{2} \, d\tau \, d\eta(z). \end{align} 
The first term is finite by the discussion in the previous paragraph, so it remains to consider the term on line \eqref{form46}. Here we could even have $\alpha = 0$. We again use the decay estimate \eqref{uDecay1} for some $0 < 2\beta < \nu$, and split the inner integral roughly as in \eqref{form43}-\eqref{form44}. More precisely,
\begin{align}\label{form47} \eqref{form46} & \lesssim \int |z|^{2} \int_{\{|\tau| < \dist(z,K_{f})^{-3/2}\}} |\tau|^{2 + 2\alpha - 2\beta} \, d\tau \, d\eta(z)\\
&\label{form48} \qquad + \int |z|^{2} \int_{\{|\tau| \geq \dist(z,K_{f})^{-3/2}\}} |\tau|^{2 + 2\alpha}e^{-\tfrac{\pi}{2}|\tau|\dist(z,K_{f})^{2}} \, d\tau \, d\eta(z). \end{align}
Here, noting that $|z| \lesssim_{\epsilon} \dist(z,K_{f})$ for all $z \in \spt \eta \subset U_{\epsilon}$, we have 
\begin{displaymath} \eqref{form47} \lesssim_{\epsilon} \int \dist(z,K_{f})^{-5/2 - 3\alpha + 3\beta} \, d\eta(z) < \infty \end{displaymath}
for any $\beta > 0$ such that $5/2 - 3\beta > 2$ (again using \eqref{eta}). For \eqref{form48}, we make the change of variables $\tau \mapsto \xi/\dist(z,K_{f})^{2}$ in the inner integral to find that
\begin{displaymath} \eqref{form48} \lesssim \int \frac{1}{(\dist(z,K_{f}))^{2 + 2\alpha}} \int_{\{|\xi| \geq \dist(z,K_{f})^{1/2}\}} |\xi|^{2 + 2\alpha}e^{-\tfrac{\pi}{2}|\xi|} \, d\xi \, d\eta(z). \end{displaymath}
Noting that $|\xi|^{2 + 2\alpha} \lesssim_{\alpha} e^{\tfrac{\pi}{4}|\xi|}$, and then using again the familiar assumptions, namely $\spt \eta \subset U_{\epsilon}$ and \eqref{eta}, the RHS is evidently finite. This completes the proof.  \end{proof}

\bibliographystyle{plain}
\bibliography{references}

\end{document}